\documentclass[12pt]{report} 
\usepackage{amsmath, amssymb, amsthm}
\usepackage{geometry}
\usepackage{setspace}
\usepackage{graphicx} 
\usepackage{fancyhdr} 
\usepackage{titlesec} 
\usepackage{lipsum} 
\usepackage{tikz-cd}
\usepackage[utf8]{inputenc}
\usepackage[T1]{fontenc}
\usepackage{mathptmx}
\usepackage{mathtools}
\usepackage{hyperref}
\usepackage{stmaryrd}
\usepackage{authblk}

\title{Your Thesis Title}
\author{Your Name}
\date{\today} 

\newtheorem{theorem}{Theorem}[section]
\newtheorem*{theoremA}{Theorem A}
\newtheorem*{theoremB}{Theorem B}
\newtheorem*{theoremC}{Theorem C}
\newtheorem{lemma}[theorem]{Lemma}
\newtheorem{sublemma}[theorem]{Sublemma}
\newtheorem{corollary}[theorem]{Corollary}
\newtheorem{proposition}[theorem]{Proposition}
\newtheorem{conjecture}[theorem]{Conjecture}
\newtheorem{problem}[theorem]{Problem}

\theoremstyle{definition}
\newtheorem{definition}[theorem]{Definition}

\newtheorem{remark}[theorem]{Remark}
\newtheorem{example}[theorem]{Example}
\newtheorem{notation}[theorem]{Notation}
\newtheorem{question}[theorem]{Question}

\titleformat{\chapter}[display]
  {\normalfont\huge\bfseries\centering}{\chaptertitlename\ \thechapter}{20pt}{\Huge}

\begin{document}

\begin{titlepage}
    \centering
    \vspace*{1cm}
    
    \includegraphics[width=0.3\textwidth]{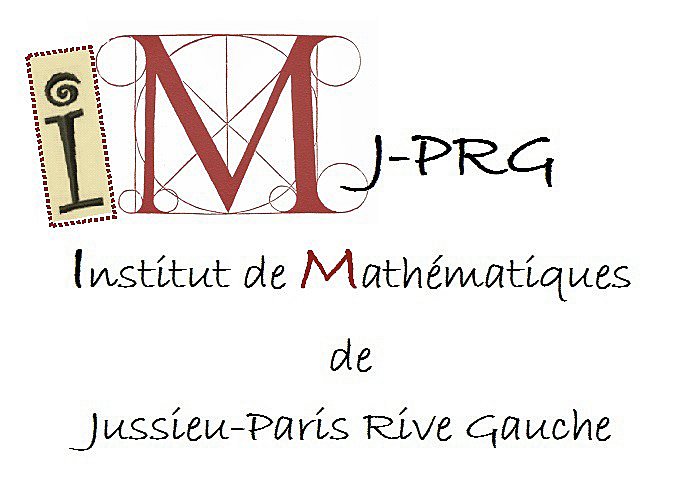}
    
    \vspace{.5cm}
    
    {\Large\textbf{Sorbonne Université}}\\
    \vspace{0.5cm}
    {\large École doctorale de sciences mathématiques de Paris centre}\\
    \vspace{1.5cm}
    
    {\LARGE{THÈSE DE DOCTORAT}}\\
    \vspace{0.5cm}
    {\large Discipline : Mathématiques}\\
    \vspace{0.5cm}
    {\large présentée par}\\
    \vspace{0.5cm}
    {\Large\textbf{Chenyu BAI}}\\
    \vspace{1cm}
    
    \rule{\linewidth}{0.5mm} \\
    \vspace{0.2cm}
    {\LARGE\textbf{Hodge Theory, Algebraic Cycles of Hyper-Kähler Manifolds
}}\\
    \vspace{0.2cm}
    \rule{\linewidth}{0.5mm} \\
    \vspace{1cm}
    
    {\large dirigée par Claire VOISIN}\\
    \vspace{1cm}
    {\large Soutenance le 14 juin 2024 devant le jury composé de:}\\
    \vspace{0.5cm}
    \begin{tabular}{llll}
    Mme Katia AMERIK  & PR & Université Paris-Saclay / NRU HSE & Rapporteure \\
    M. Olivier DEBARRE & PR & Université Paris-Cité & Président du jury \\
    M. Radu LAZA  & PR & Stony Brook University & Examinateur \\
    M. Emanuele MACRÌ & PR & Université Paris-Saclay & Examinateur \\
    M. Gianluca PACIENZA & PR & IECL & Examinateur \\
    Mme Claire VOISIN & DR & CNRS / IMJ-PRG & Directrice
    \end{tabular}
    
    \vfill
\end{titlepage}

\chapter*{Remerciements}
Je suis extrêmement honoré de pouvoir remercier ma directrice de thèse, Claire Voisin. Depuis le M1, son enthousiasme pour les mathématiques, son engagement envers la vérité, son attitude rigoureuse dans la recherche scientifique, et sa générosité dans toutes ses interactions, ont profondément influencé mon propre parcours. Je ne saurais mesurer combien j'ai appris d'elle, mais je peux dire avec certitude que ma compréhension actuelle et ma vision des mathématiques ont été fortement impactées par son influence. Cette thèse reflète également, à bien des égards, l'empreinte de son école de pensée. Je n'oublierai jamais notre première rencontre un après-midi dans son bureau au Collège de France, où elle m'a présenté le théorème d'Ehresmann et la théorie des déformations. 
Rétrospectivement, je vois le temps qui s'est envolé ; ces quatre années sont passées en un clin d'œil. Sous sa tutelle, j'ai parcouru un chemin aussi long que semé d'embûches, compréhensible seulement par ceux qui ont partagé cette route. Sans les encouragements constants de Claire et l'inspiration de nos nombreuses rencontres, je n'aurais certainement pas ouvert la porte du monde des cycles algébriques et de la géométrie algébrique complexe. Je me souviendrai toujours de cet après-midi, après ma soutenance de M2, quand elle m'a remis son livre couleur citrouilles et m'a encouragé à le lire - un excellent texte sur la théorie des cycles, confirmant que je pouvais commencer mon voyage à partir de là - et ainsi est née cette thèse, imprégnée de l'esprit et de la méthodologie de son livre. Des milliers de mots ne pourraient suffire à exprimer ma gratitude envers elle, et c'est avec un cœur plein d'espoir que je dédie cette thèse à Claire Voisin.

Tout au long de ce voyage, l'aide de Sébastien Boucksom a été inestimable, tant sur le plan académique que professionnel. Sa bonté et sa sincérité dans la recherche scientifique m'ont profondément ému. Merci à Bertrand Rémy, dont la compagnie et l'assistance pendant mon temps à l'École Polytechnique de Paris m’ont aidé à façonné la personne que je suis aujourd'hui.

Je tiens à remercier mon superviseur de licence, Baohua Fu, qui m'a initié au monde de la géométrie algébrique avec les 27 droites. La dernière partie de cette thèse, qui étudie les variétés de Calabi-Yau construites à partir de droites, peut être vue comme une continuation à la théorie des 27 droites.

Merci à Maxime, pour ces trois années de camaraderie, les discussions mathématiques avec lui ont toujours été si sereines, et sa bonne humeur a été un moteur constant dans mon avancée. Merci à Elvaz, toujours en train de travailler dur dans la salle d'à côté, nous avons été témoins de notre croissance respective.

Merci à Pietro Beri, Corinne Bedussa et Mauro Varesco, explorer avec eux le sujet des surfaces d'Enriques a été une des plus belles périodes de ma carrière scientifique. J'espère qu’ils apprécieront la section sur les surfaces d'Enriques dans cette thèse. Merci à Matteo, mon petit frère de thèse, les discussions avec lui ont toujours été productives. Merci également à Shengxuan pour les discussions mathématiques et philosophiques. Merci à Mai, Shi, Ruichuan pour les discussions précieuses sur les cubiques.

Je suis reconnaissant envers mes professeurs : Olivier Debarre, Daniel Huybrechts, Emanuele Macrì. J'ai toujours su que je pouvais toujours apprendre davantage d'eux, et c'est ce qui s'est avéré. C'est un honneur d'avoir reçu leurs conseils précieux tout au long de mon parcours de recherche. Merci aussi à Laurent Manivel, que j'ai eu la chance de rencontrer très tôt et de bénéficier de ses conseils. Merci à Zhi Jiang pour m'avoir invité à Shanghai et pour m'avoir donné l'opportunité de parler de mes travaux scientifiques.

Merci à Enrico, Eva, Haohao, Long, Lucas, Mattias, Sacha, Tangi, Tristan, Xiaohan pour leurs conseils, leurs discussions mathématiques et personnels, et les bons moments passés ensemble.

Merci à Chenqin, Godard et Yicheng pour ces trois années passées ensemble près du parc de Sceaux. Un grand merci à Yun-Xiao pour son soutien et ses encouragements, sans lesquels mon parcours scientifique aurait été bien plus difficile. Merci également à Shengkai et Peter Yi, mes amis les plus proches aux États-Unis, et à Yunsong pour nos échanges enrichissants. Je voudrais remercier Zhuo Pan pour m'avoir toujours encouragé. Un sincère merci à Louise et à sa famille pour m'avoir encouragé à explorer de nouvelles possibilités dans la vie et pour avoir enrichi mon expérience.

Un grand merci à Eric Xi Chen pour son calme et son enthousiasme contagieux pour les mathématiques. Je suis toujours inspiré par nos échanges. Merci à Jingyi pour nos discussions scientifiques toujours enrichissantes. Et un spécial remerciement à Ziyi pour m'avoir aidé à découvrir ma passion pour l'escalade, ce qui a profondément changé ma perspective sur la vie.

Pendant mon parcours à Polytechnique, je tiens à exprimer ma reconnaissance à plusieurs personnes. En particulier, un grand merci à Mien, qui a été un soutien moral essentiel pendant le confinement lié au Covid-19. Je lui dois encore des excuses et un profond remerciement. Je souhaite également remercier Daniel, mon compagnon de tennis, qui m'a chaleureusement accueilli lors de mes séjours à Londres.

Un grand merci à Shikang et Yuchen pour m'avoir ouvert la porte vers une autre possibilité professionnelle. Je suis également reconnaissant envers BearishSentiment, Daniel, Kaitong, Tianyuan et Xiao pour leur précieuse aide durant ces procédures.

Et enfin, un immense merci à mes parents, à ma petite sœur et à Anne, vous êtes mon soutien éternel.

\chapter*{Résumé}
Cette thèse est consacrée à l'étude des cycles algébriques dans les variétés hyper-Kähleriennes projectives et les variétés de Calabi-Yau strictes. Elle contribue à la compréhension des conjectures de Beauville et de Voisin sur les anneaux de Chow des variétés hyper-kählériennes projectives et des variétés  de Calabi-Yau strictes. Elle étudie également certains invariants birationnels des variétés hyper-kählériennes projectives.

La première partie de la thèse, parue dans Mathematische Zeitschrift~\cite{Bai23} et présentée dans le chapitre~\ref{ChapterAJ}, étudie si les sous-variétés lagrangiennes dans une variété hyper-kählérienne partageant la même classe cohomologique ont également la même classe de Chow. Nous étudions la notion de familles lagrangiennes et ses applications aux applications d'Abel-Jacobi associées. Nous adoptons une approche infinitésimale pour donner un critère de trivialité de l'application d'Abel-Jacobi d'une famille lagrangienne, et utilisons ce critère pour donner une réponse négative à la question précédente, ajoutant aux subtilités d'une conjecture de Voisin. Nous explorons également comment la maximalité de la variation des structures de Hodge sur la cohomologie de degré $1$ de la famille lagrangienne implique la trivialité de l'application d'Abel-Jacobi.

La deuxième partie de la thèse, parue dans International Mathematics Research Notices~\cite{Bai24} et présentée dans le chapitre~\ref{ChapterBirational}, étudie le degré d'irrationalité, la gonalité fibrante et le genre fibrant des variétés hyper-kählériennes projectives. Nous commençons par donner une légère amélioration d'un résultat de Voisin sur la borne inférieure du degré d'irrationalité des variétés hyper-kählériennes générales de Mumford-Tate. Nous étudions ensuite la relation entre les trois invariants birationnels susmentionnés pour les surfaces K3 projectives de nombre de Picard $1$, rajoutant la compréhension sur une conjecture de Bastianelli, De Poi, Ein, Lazarsfeld, Ullery sur le comportement asymptotique du degré d'irrationalité des surfaces K3 projectives très générales.

La troisième partie de la thèse, parue dans ArXiv~\cite{Bai24II}, présentée dans le chapitre~\ref{ChapterVoisinMaps}, étudie les applications de Voisin de dimension supérieure sur les variétés de Calabi-Yau strictes. Voisin a construit des applications auto-rationnelles de variétés de Calabi-Yau obtenues comme des variétés de $r$-plans dans des hypersurfaces cubiques de dimension adéquate. Cette application a été largement étudiée dans le cas $r=1$, qui est le cas de Beauville-Donagi. Dans les cas de dimensions supérieures, nous étudions d'abord l'action de l'application de Voisin sur les formes holomorphes. Nous demontrons ensuite la conjecture de Bloch généralisée pour l'action des applications de Voisin sur les groupes de Chow dans le cas de $r=2$. Enfin, via l'étude de l'application de Voisin, nous apportons des éléments de preuve à une conjecture de Voisin sur l'existence d'un $0$-cycle spécial sur les variétés de Calabi-Yau strictes.

\textbf{Mots-clés :} Variétés hyper-kählériennes, Variétés de Calabi-Yau strictes, Conjectures de Voisin, Familles lagrangiennes, Degré d'irrationalité, Applications de Voisin

\chapter*{Abstract}
This thesis is devoted to the study of algebraic cycles in projective hyper-Kähler manifolds and strict Calabi-Yau manifolds. It contributes to the understanding of Beauville's and Voisin's conjectures on the Chow rings of projective hyper-Kähler manifolds and strict Calabi-Yau manifolds. It also studies some birational invariants of projective hyper-Kähler manifolds. 

The first part of the thesis, appeared in Mathematische Zeitschrift~\cite{Bai23} and presented in Chapter~\ref{ChapterAJ}, studies whether the Lagrangian subvarieties in a hyper-Kähler manifold sharing the same cohomological class have the same Chow class as well. We study the notion of Lagrangian families and its associated Abel-Jacobi maps. We take an infinitesimal approach to give a criterion for the triviality of the Abel-Jacobi map of a Lagrangian family, and use this criterion to give a negative answer to the above question, adding to the subtleties of a conjecture of Voisin. We also explore how the maximality of the variation of the Hodge structures on the degree $1$ cohomology the Lagrangian family implies the triviality of the Abel-Jacobi map.

The second part of the thesis, appeared in International Mathematics Research Notices~\cite{Bai24} and presented in Chapter~\ref{ChapterBirational}, studies the degree of irrationality, the fibering gonality and the fibering genus of projective hyper-Kähler manifolds, with emphasis on the K3 surfaces case, en mettant l'accent sur le cas des surfaces K3. We first give a slight improvement of a result of Voisin on the lower bound of the degree of irrationality of Mumford-Tate general hyper-Kähler manifolds. We then study the relation of the above three birational invariants for projective K3 surfaces of Picard number $1$, adding the understandinf of a conjecture of Bastianelli, De Poi, Ein, Lazarsfeld, Ullery on the asymptotic behavior of the degree of irrationality of very general projective K3 surfaces.

The third part of the thesis, appeared in ArXiv~\cite{Bai24II}, presented in Chapter~\ref{ChapterVoisinMaps}, studies the higher dimensional Voisin maps on strict Calabi-Yau manifolds. Voisin constructed self-rational maps of Calabi-Yau manifolds obtained as varieties of $r$-planes in cubic hypersurfaces of adequate dimension. This map has been thoroughly studied in the case $r=1$, which is the Beauville-Donagi case. For higher dimensional cases, we first study the action of the Voisin map on the holomorphic forms. We then prove the generalized Bloch conjecture for the action of the Voisin maps on Chow groups for the case of $r=2$. Finally, via the study of the Voisin map, we provide evidence for a conjecture of Voisin on the existence of a special $0$-cycle on strict Calabi-Yau manifolds.

\textbf{Keywords:} Hyper-Kähler manifolds, Strict Calabi-Yau manifolds, Voisin conjectures, Lagrangian families, Degree of irrationality, Voisin maps

\tableofcontents


\chapter{Introduction}
\section{Chow groups}\label{IntroChowGroups}
We closely follow the presentations in~\cite{Fulton}, \cite[Chapter 17]{Voisin}, and \cite[Chapter 2]{VoisinCitrouille}. Let $X$ be an algebraic variety of dimension $n$, defined over a field which will be in this thesis the field of complex numbers. An algebraic cycle $Z\subset X$ of dimension $k$ is defined as a finite formal sum $Z=\sum_i n_iZ_i$ with $n_i\in \mathbb Z$, and $Z_i$ is a closed irreducible subvariety of $X$ of dimension $k$. The cycle group $\mathcal Z_k(X)$ is the abelian group of all algebraic cycles of dimension $k$ in $X$. If $\phi: X\to Y$ is a proper morphism, there is an induced map $\phi_*: \mathcal Z_k(X)\to \mathcal Z_k(Y)$, called the push-forward map, defined as follows: for an irreducible subvariety $Z\subset X$ of dimension $k$ with $\phi|_Z: Z\to Z'$ being a generically finite map of degree $d$, we define $\phi_*Z = d Z'$. In other cases, $\phi_*Z = 0$. Finally, we extend the definition of $\phi_*$ linearly to the entire group $\mathcal Z_k(X)$.

\begin{definition}
    Two algebraic cycles $Z_1, Z_2\subset X$ of dimension $k$ are called rationally equivalent if there exist irreducible subvarieties $W_1, \ldots, W_l$ of dimension $k+1$ in $X$, nonzero rational functions $f_i$ on $\tilde W_i$ where $\tau_i: 
    \tilde W_i\to W_i$ is the normalization of $W_i$, such that 
    \[
    Z_1 - Z_2 = \sum_{i=1}^l \tau_{i*}(\mathrm{div}(f_i)).
    \]
\end{definition}

\begin{notation}
Throughout this thesis, we use the following notation. Let $X$ be an algebraic variety of dimension $n$.
\begin{itemize}
    \item The Chow group $CH_k(X) = \mathcal Z_k(X)/\equiv_{rat}$ is the quotient group of $\mathcal Z_k(X)$ modulo rational equivalence $\equiv_{rat}$.
    \item $\mathcal Z^k(X) = \mathcal Z_{n-k}(X)$.
    \item $CH^k(X) = CH_{n-k}(X)$.
    \item $CH_k(X)_{\mathbb Q} = CH_k(X)\otimes_{\mathbb Z}\mathbb Q$. 
    \item $CH^k(X)_{\mathbb Q} = CH^k(X)\otimes_{\mathbb Z}\mathbb Q$.
\end{itemize}
\end{notation}

In~\cite[Chapter 6]{Fulton}, an intersection product is constructed for smooth algebraic varieties.

\begin{theorem}[Fulton]
    Let $X$ be a smooth quasi-projective variety. There exists a unique product structure on $CH^*(X)$ satisfying the following condition: if two subvarieties $A, B$ are transverse to each other, then $[A].[B] = [A\cap B]\in CH^*(X)$. This product structure makes $CH^*(X)$ a graded ring.
\end{theorem}

\subsection{Push-forward and pull-back maps}
Let $p: X\to Y$ be a flat morphism of relative dimension $l$. Then, there is a naturally defined map $\phi^*: \mathcal Z_k(Y)\to \mathcal Z_{k+l}(X)$, called the pull-back map, defined by taking the preimage. A fundamental fact is the following (See Lemma 2.3 and 2.5 in~\cite{VoisinCitrouille}).

\begin{lemma}
    The push-forward map and the pull-back map, defined at the level of cycle groups, preserve rational equivalence.
\end{lemma}

Therefore, the push-forward map of a proper morphism and the pull-back map of a flat morphism can be defined at the level of Chow groups. Let $p: X\to Y$ be a morphism between smooth varieties (in fact, we only need $Y$ to be smooth). We can also define the pull-back map $p^*: CH^k(Y)\to CH^k(X)$ without flatness conditions (see~\cite[17.2]{Voisin}). We have the following functoriality result.

\begin{proposition}
    Let $p: X\to Y$ be a morphism between smooth varieties.\\
    (a) (Projection formula) For $Z\in CH(Y)$ and $Z'\in CH(X)$, we have
    \[
    p_*(p^*Z.Z') = Z.p_*(Z')\in CH(X).
    \]
    (b) For $Z, Z'\in CH(Y)$, we have 
    \[
    p^*(Z.Z') = p^*Z.p^*Z'\in CH(X).
    \]
\end{proposition}

\begin{proof}
    See~\cite[8.1]{Fulton}.
\end{proof}

    A correpondence between $X$ and $Y$ is an element $Z$ in $CH(X\times Y)$. Let $p_X: X\times Y\to X$ and $p_Y: X\times Y\to Y$ be the two projection maps.  A correspondence $Z$ between smooth varieties $X$ and $Y$ induces two  natural maps $Z_*: CH(X)\to CH(Y)$ and $Z^*: CH(Y)\to CH(X)$ in the following way. For any $z\in CH(X)$, 
    \[Z_*x := p_{Y, *}(p_X^*x.Z)\in CH(Y),\] and for any $w\in CH(Y)$, 
    \[Z^*w := p_{X,*}(p_Y^*w. Z)\in CH(X).\]

\subsection{Localization exact sequence}
Let $X$ be a quasi-projective variety, and let $j: Y\hookrightarrow X$ be a closed algebraic subset. Let $i: U:=X-Y\hookrightarrow X$ be the open embedding of the complement of $Y$. 

\begin{proposition}
    We have the following exact sequence:
    \[
    CH_k(Y)\stackrel{j_*}{\to} CH_k(X)\stackrel{i^*}{\to} CH_k(U)\to 0.
    \]
\end{proposition}
\begin{proof}
    See~\cite[Lemma 17.12]{Voisin}.
\end{proof}

\subsection{Cycle class map}
Let $X$ be a smooth complex algebraic variety. To each irreducible subvariety $Z\subset X$ of codimension $k$, we can associate its cohomological class in $H^{2k}(X,\mathbb Z)$. Extending linearly, we get the cycle class map:
\[
cl: \mathcal Z^k(X)\to H^{2k}(X,\mathbb Z).
\]

\begin{lemma}
    Let $Z_1, Z_2\in \mathcal Z^k(X)$. If $Z_1$ and $Z_2$ are rationally equivalent, then $cl(Z_1) = cl(Z_2)$ in $H^{2k}(X,\mathbb Z)$.
\end{lemma}

Thus, the cycle class map descends to the Chow group level:
\[
cl: CH^k(X)\to H^{2k}(X,\mathbb Z),
\]
and we still call it the cycle class map.

We have the following compatibility results of the cycle class map (see~\cite[17.2.4]{Voisin}).

\begin{proposition}
    The cycle class map $cl: CH^*(X)\to H^{2k}(X,\mathbb Z)$ is a ring homomorphism. 
\end{proposition}

\begin{proposition}
    Let $p: X\to Y$ be a morphism between smooth varieties.\\
    (a) For any $Z\in CH^k(Y)$, we have 
    \[
    p^*cl(Z) = cl(p^*Z)\in H^{2k}(X,\mathbb Z).
    \]
    (b) If $p$ is proper, then for any $Z\in CH_k(X)$, we have 
    \[
    p_*cl(Z) = cl(p_*Z)\in H^{2\dim Y - 2k}(Y,\mathbb Z).
    \]
\end{proposition}

\begin{notation}
    The kernel of the cycle class map $cl: CH^k(X)\to H^{2k}(X,\mathbb Z)$ is denoted by $CH^k(X)_{hom}$. The subscript "hom" signifies "homologous to 0". Similarly, $CH^k(X)_{\mathbb Q, hom}$ denotes the kernel of $cl: CH^k(X)_\mathbb Q\to H^{2k}(X, \mathbb Q)$.
\end{notation}

\subsection{Constant cycle subvarieties}\label{IntroCCSubvarieties}
The notion of constant cycle subvarieties is introduced and developed in~\cite{ConstantCycle, VoisinCoisotrope}, especially in the context of algebraic hyper-Kähler manifold.

\begin{definition}[\cite{ConstantCycle, VoisinCoisotrope}]
    Let $X$ be an algebraic variety. A subvariety $i: Z\hookrightarrow X$ is called a constant cycle subvariety if every two points $z_1, z_2\in Z$ are rationally equivalent in $X$. In other words, the image of the push-forward map
    \[
    i_*: CH_0(Z)\to CH_0(X)
    \]
    is $\mathbb Z$.
\end{definition}

\begin{remark}
    The study of rational curves in an algebraic variety is a powerful method to understand the algebraic variety itself (\cite{Kollar}). Constant cycle subvarieties are a big generalization of rationally connected subvarieties. The existence of such subvarieties often gives many interesting geometric implications (e.g., \cite{VoisinCoisotrope, Lin, Bazhov}). On the other hand, constant cycle subvarieties share similar properties to rationally connected subvarieties. For example, it is classical that there are no rational curves in abelian varieties. One can also prove that there are no constant cycle subvarieties of positive dimension in abelian varieties either.
\end{remark}

Constant cycle subvarieties will appear in Chapter~\ref{ChapterVoisinMaps} where we prove the following theorem (see Theorem~\ref{ThmFixedLocusIsConstantCycle}). 
\begin{theorem}
    Let $Y\subset \mathbb P^9$ be a general cubic eightfold. Let $X = F_2(Y)$ be its Fano variety of planes. Let $F\subset X$ be the closure of the set of the points $x\in X$ parametrizing the planes $P_x\subset Y$ such that there exists a unique linear subspace $H$ of dimension $3$ such that $H
    \cap Y = 3P_x$. Then $F\subset X$ is a constant cycle subvariety of $X$ of codimension $3$.
\end{theorem}
This theorem serves as a crucial ingredient for our understanding of the Voisin maps in Chapter~\ref{ChapterVoisinMaps}.

\subsection{Bloch-Beilinson filtration}
It has been conjectured by Bloch and Beilinson that there exists a decreasing filtration $F^i CH_k(X)_\mathbb Q$ on the Chow groups with rational coefficients of any smooth complex projective variety $X$, satisfying several axioms. The precise statements are as follows. We follow closely the presentation in~\cite{BBFiltration} and \cite[Conjecture 2.19]{VoisinCitrouille} for the statements of the Bloch-Beilinson conjectures.

\begin{conjecture}[Bloch-Beilinson Conjecture]\label{ConjBlochBeilinson}
    For any smooth projective variety $X$, there exists a decreasing filtration $F^\bullet$ on $CH^i(X)_\mathbb Q$, with the following properties:\\
    (i) (Non-Triviality) $F^0CH^i(X)_\mathbb Q = CH^i(X)_\mathbb Q$ and $F^1CH^i(X)_\mathbb Q = CH^i(X)_{\mathbb Q, hom}$.\\
    (ii) (Functoriality) If $Z\in CH^k(X\times Y)_\mathbb Q$, then $Z_*(F^iCH^l(X)_\mathbb Q)\subset F^iCH^{l+k-n}(X)_\mathbb Q$, where $n = \dim X$.\\
    (iii) (Graded) The induced map $Z_*: Gr_F^iCH^l(X)_\mathbb Q\to Gr_F^iCH^{l+k-n}(Y)_\mathbb Q$ vanishes for any $i$ if $[Z]=0$ in $H^{2k}(X\times Y,\mathbb Q)$.\\
    (iv) (Finiteness) One has $F^{k+1}CH^k(X)_\mathbb Q = 0$ for any $X$ and $k$.
\end{conjecture}

There is also a strengthened version of the Bloch-Beilinson conjecture where (iii) above is replace by 

    (iii)' (Graded) For a fixed $i$, the induced map $Z_*: Gr_F^iCH^l(X)_\mathbb Q\to Gr_F^iCH^{l+k-n}(Y)_\mathbb Q$ vanishes if the following induced map 
    \[
        [Z]^*: H^{2m - 2k - 2l + 2n + i}(Y, \mathbb Q) \to H^{2n - 2l + i}(X, \mathbb Q)
    \]
    vanishes, where $m = \dim Y$ and $n = \dim X$.

\section{Interaction between Hodge structures and Chow groups}
\subsection{Hodge Structures and Coniveau}
\begin{definition}
    A weight $k$ Hodge structure (respectively, a rational Hodge structure) on $H$ consists of a free abelian group $H_{\mathbb{Z}}$ (respectively, a $\mathbb{Q}$-vector space $H_{\mathbb{Q}}$) and a decomposition
    \[
    H_{\mathbb{C}} := H \otimes \mathbb{C} = \bigoplus_{p+q = k} H^{p, q},
    \]
    satisfying the Hodge symmetry condition
    \[
    \overline{H^{p,q}} = H^{q, p}.
    \]
\end{definition}

Hodge structures naturally arise from the Betti cohomology groups of compact Kähler manifolds~\cite[7.1]{Voisin}.
\begin{theorem}[Hodge Decomposition Theorem~\cite{Voisin}]
    Let $X$ be a compact Kähler manifold. Then the $k$-th Betti cohomology group $H^k(X, \mathbb{Z})$ of $X$ possesses a weight $k$ Hodge structure. Specifically, there is a decomposition
    \[
    H^k(X, \mathbb{C}) = \bigoplus_{p+q = k} H^{p, q}(X),
    \]
    such that $\overline{H^{p,q}(X)} = H^{q, p}(X)$. Moreover, there is a canonical isomorphism $H^q(X, \Omega_X^p) \cong H^{p, q}(X)$.
\end{theorem}

\begin{definition}
    A (rational) Hodge class in a (rational) Hodge structure $H$ of degree $2k$ is an element in $H_{\mathbb{Z}} \cap H^{k, k}$ (respectively, $H_{\mathbb{Q}} \cap H^{k, k}$).
\end{definition}

\begin{definition}
    The (Hodge) coniveau $c \leq k/2$ of a weight $k$ Hodge structure $(H_{\mathbb{Z}}, H^{p,q})$ is the smallest integer $p$ for which $H^{p,q} \neq 0$.
\end{definition}

A multitude of conjectures relate algebraic cycles and Hodge structures.

\begin{conjecture}[Hodge Conjecture]\label{ConjHodge}
    Let $X$ be a complex smooth projective variety. Then any rational Hodge class in $H^{2k}(X, \mathbb{Q})$ is representable by a rational algebraic cycle of codimension $k$.
\end{conjecture}

\begin{conjecture}[Generalized Hodge Conjecture, Grothendieck~\cite{GeneralizedHC}]\label{ConjGeneralizedHC}
    Let $X$ be a complex smooth projective variety. Suppose $L \subset H^k(X, \mathbb{Q})$ is a rational sub-Hodge structure of Hodge coniveau $\geq c$. Then there exists a closed algebraic subset $Z \subset X$ of codimension $c$ such that $L$ vanishes under the restriction map $H^k(X, \mathbb{Q}) \to H^k(X \setminus Z, \mathbb{Q})$, where $U := X \setminus Z$.
\end{conjecture}

Conjecture~\ref{ConjHodge} is the special case of Conjecture~\ref{ConjGeneralizedHC} when the coniveau $c$ is half the weight of the Hodge structure. The significance and recent advancements in the Hodge conjecture and the generalized Hodge conjecture are detailed in~\cite{HodgeConjecture}. The following conjecture is stated by  in~\cite{Hodge=Bloch} by Voisin who  also noticed that it is in fact a consequence of the Lefschetz standard conjecture (see~\cite[Remark 2.30]{VoisinCitrouille}).

\begin{conjecture}[Voisin~\cite{Hodge=Bloch}]\label{ConjVoisinsStandard}
    Let $X$ be a smooth complex projective variety and $Y \subset X$ be a closed algebraic subset. Suppose $Z \subset X$ is a codimension $k$ algebraic cycle, and assume that the cohomology class $[Z] \in H^{2k}(X, \mathbb{Q})$ vanishes in $H^{2k}(X \setminus Y, \mathbb{Q})$. Then there exists a codimension $k$ cycle $Z'$ on $X$ with $\mathbb{Q}$-coefficients, which is supported on $Y$ and such that $[Z'] = [Z]$ in $H^{2k}(X, \mathbb{Q})$.
\end{conjecture}

     Combining the generalized Hodge conjecture (Conjecture~\ref{ConjGeneralizedHC}) and the Bloch-Beilinson conjecture (Conjecture~\ref{ConjBlochBeilinson}), the following conjecture is expected. 
    \begin{conjecture}\label{ConjGeneralPrinciple}
        Let $Z\subset CH^n(X\times X)$ be a self-correspondence of a smooth projective variety $X$ of dimension $n$. If $[Z]^*|_{H^{i, 0}(X)} = 0$ for some $i$, then $Z_*|_{Gr_F^iCH_0(X)} = 0$ for the same $i$ as well, where $F^\bullet$ is the Bloch-Beilinson filtration and $Gr_F^\bullet$ is the graded part. 
    \end{conjecture} 
    Indeed, By the Künneth decomposition theorem and the Poincaré duality, $[Z]\in H^{2n}(X\times X,\mathbb Q)$ can be identified as a graded map $[Z]^*: H^*(X,\mathbb Q)\to H^*(X,\mathbb Q)$ that preserves Hodge structure. Let us fix a polarization on $X$ and let $\phi: H^*(X,\mathbb Q)\to H^*(X,\mathbb Q)$ be the restriction of $[Z]^*$ onto the subspace $N^1H^i(X,\mathbb Q)$ which is the largest sub-Hodge structure of $H^i(X,\mathbb Q)$ of coniveau at least $1$. To be precise, $\phi(\alpha) := [Z]^*(\alpha)$ if $\alpha\in N^1H^i(X,\mathbb Q)$ and $\phi(\alpha) = 0$ if $\alpha\in N^1H^i(X, \mathbb Q)^\bot$ or if $\alpha\in H^k(X,\mathbb Q)$ for $k\neq i$. The map $\phi$ preserves Hodge-structure, and when viewed as an element of $H^{2n}(X\times X,\mathbb Q)$, is a Hodge class. By the Hodge conjecture, there is an algebraic cycle $Z_1\in CH^n(X\times X)$ that represents $\phi\in H^{2n}(X\times X,\mathbb Q)$. By the construction of $Z_1$, the algebraic cycle $Z_2 : = Z - Z_1$ satisfies the condition $[Z_2]^*|_{H^i(X,\mathbb Q)} = 0$. By the Bloch-Beilinson conjecture (Conjecture~\ref{ConjBlochBeilinson}), we have $Z_{2*}|_{Gr^i_FCH_0(X)} = 0$. On the other hand, we have 
    \begin{lemma}\label{LmmGeneralPrinciple}
        Assuming the generalized Hodge conjecture (Conjecture~\ref{ConjGeneralizedHC}) and the generalized Bloch-Beilinson conjecture (Conjecture~\ref{ConjBlochBeilinson}), $Z_{1*}|_{Gr_F^kCH_0(X)} = 0$ for any $k$.
    \end{lemma}
    Admitting Lemma~\ref{LmmGeneralPrinciple} for the moment, we conclude $Z_* = Z_{1*} + Z_{2*}$ acts as $0$ on $Gr_F^iCH_0(X)$, as desired.

    \begin{proof}[Proof of Lemma~\ref{LmmGeneralPrinciple}]
        By construction, $[Z_1]^*|_{H^{k, 0}(X)} = 0$ for any $k$. This implies that the sub-Hodge structure $[Z_1]^*H^k(X,\mathbb Q)\subset H^k(X,\mathbb Q)$ has Hodge coniveau at least $1$ for any $k$. By the generalized Hodge conjecture (Conjecture~\ref{ConjGeneralizedHC}), there is an open dense subset $U\subset X$ such that $([Z_1]^*H^k(X,\mathbb Q))|_U = 0$ for any $k$. This means, by the Künneth decomposition theorem, that the cohomology class of $Z_1|_{U\times X}$ is $0$. Therefore, by Conjecture~\ref{ConjVoisinsStandard} (which is implied by Conjecture~\ref{ConjGeneralizedHC}), there is a cycle $Z'\in CH^n(X\times X)$, supported on $D\times X$, such that $[Z_1 - Z'] = 0\in H^{2n}(X\times X,\mathbb Q)$, where $D\subset X$ is the complement of $U$ in $X$. Since $Z'$ is supported on $D\times X$ and $D$ is a proper closed subset of $X$, we have $Z'_*CH_0(X) = 0$. On the other hand, since $Z_1-Z'$ is cohomologeous to $0$, the (generalized) Bloch-Beilinson conjecture implies that $((Z_1-Z'))_*|_{Gr_F^kCH^*(X)} = 0$ for any $k$. Taken together, we find that the action of $Z_1 =  Z' + (Z_1 - Z')$ on $Gr_F^kCH_0(X)$ is $0$, for any $k$, as desired.
    \end{proof}

In the literature, Conjecture~\ref{ConjGeneralPrinciple} is often called the generalized Bloch conjecture. It extends the classical Bloch conjecture, stated as follows in~\cite{Bloch} for surfaces.

\begin{conjecture}[Bloch~\cite{Bloch}]
For a correspondence $Z\in CH^2(S \times T)_\mathbb{Q}$ between surfaces that induces a null map $[Z]^*: H^{2,0}(T) \to H^{2,0}(S)$, the induced morphism $Z_*: F^2CH_0(S) \to F^2CH_0(T)$ is identically zero. Here, $F^2CH_0(S)$ is defined as the kernel of the Albanese map from $CH_0(S)_{hom}$ to $Alb(S)$, and similarly for $F^2CH_0(T)$.
\end{conjecture}

If we take $Z=\Delta_X$, the diagonal of $X\times X$, Conjecture~\ref{ConjGeneralPrinciple} predicts that if $H^{i, 0}(X) = 0$ for all $i > 0$, then $CH_0(X)_{\mathbb Q, hom} = 0$. This is part of the following conjecture, also named the generalized Bloch conjecture~\cite[Conjecture 1.9]{VoisinCitrouille}.
\begin{conjecture}[Generalized Bloch Conjecture]\label{ConjGeneralizedBlochConjecture2}
    Let $X$ be a smooth projective variety. Assume that $H^{p,q}(X) = 0$ for $p \neq q$ and $p < c$ (or $q < c$). Then the cycle class map
\[
cl : CH_i(X)_\mathbb Q \to H^{2m-2i}(X, \mathbb Q) 
\]
is injective for $i \leq c - 1$.
\end{conjecture}

Some non-trivial incidences of Conjecture~\ref{ConjGeneralizedBlochConjecture2} are given in Chapter~\ref{ChapterVoisinMaps}. Let $Y\subset\mathbb P^9$ be a general cubic eightfold. Its Fano variety of lines $F_1(Y)$ is a Fano manifold of dimension $12$. It has been established that $H^{p,q}(F_1(Y)) = 0$ for $p \leq 1$ and $p \neq q$~\cite{DebarreManivel}, so Conjecture~\ref{ConjGeneralizedBlochConjecture2} predicts that $CH_i(F_1(Y))_{\mathbb Q, hom} = 0$ for $i\leq 1$. In Chapter~\ref{ChapterVoisinMaps}, we prove the following (see Theorem~\ref{ThmChowOneOfF1})
\begin{theorem}
    We have $CH_i(F_1(Y))_{\mathbb Q, hom} = 0$ for any $i\leq 1$ and for any general cubic eightfold $Y$.
\end{theorem}

\subsection{The case of strict Calabi-Yau manifolds}
\begin{definition}
    A strict Calabi-Yau manifold $X$ is a complex projective manifold of dimension at least $3$ that is simply connected and has trivial canonical bundle, and such that for each $0 < i < \dim X$, there is no nonzero holomorphic forms of degree $i$ on $X$.
\end{definition} 

Conjecture~\ref{ConjGeneralPrinciple} takes the following form for strict Calabi-Yau manifolds, and more generally for smooth projective varieties $X$ with $h^{i,0}(X)=0$ for $0<i<n=dim X$ and $h^{n,0}(X)=1$.
\begin{conjecture}\label{ConjBlochCY}
    Let $X$ be a strict Calabi-Yau manifold of dimension $n$. Let $\omega_X\in H^0(X,K_X)$ be a nowhere zero top degree holomorphic form on $X$. Let $Z\in CH^n(X\times X)_\mathbb Q$ be a self-correspondence such that $[Z]^*\omega = 0$. Then for any $z\in CH_0(X)_{\mathbb Q, hom}$, we have $Z_*z = 0$.
\end{conjecture}    
    Indeed, since $X$ is a strict Calabi-Yau manifold, $H^{k, 0}(X) = 0$ for $0 < k < n$. Let $\Delta_X\in CH^n(X\times X)$ be the diagonal of $X\times X$. Then $[\Delta_X]^*|_{H^{k,0}(X)} = 0$ for $0< k < n$. By Conjecture~\ref{ConjGeneralPrinciple}, $\Delta_{X*}|_{Gr^k_FCH_0(X)} = 0$ for $0 < k < n$, which implies that $Gr^k_FCH_0(X) = 0$ for $0 < k < n$. Therefore, the Bloch-Beilinson filtration on $CH_0(X)$ degenerates into
    \[
    0 = F^{n+1}CH_0(X) \subset F^nCH_0(X)= \ldots = F^1CH_0(X) = CH_0(X)_{hom} \subset F^0CH_0(X) = CH_0(X).
    \]
    
    By the assumption of $Z$, we have $[Z]^*|_{H^{n, 0}(X)} = 0$. By  Conjecture~\ref{ConjGeneralPrinciple} again, we have $Z_*|_{Gr^n_FCH_0(X)} = 0$. Hence, $Z_*$ acts as zero on $Gr^n_FCH_0(X) = F^nCH_0(X) = F^1CH_0(X) = CH_0(X)_{hom}$, and we get the desired result.

\subsection{Voisin's examples of Calabi-Yau manifolds}
In Chapter~\ref{ChapterVoisinMaps}, we aim to give evidence to Conjecture~\ref{ConjBlochCY} for specific families of $K$-trivial varieties as constructed in~\cite{KCorr}. Let $Y \subset \mathbb{P}^n$ be a smooth cubic hypersurface of dimension $n-1$, and let $r \geq 0$ denote a nonnegative integer. Define $X=F_r(Y)$ as the Hilbert scheme that parametrizes the $r$-dimensional linear subspaces in $Y$. As proven in~\cite[(4.41)]{KCorr}, for $n+1=\binom{r+3}{2}$ and a general $Y$, the variety $X$ is a $K$-trivial variety of dimension $N = (r+1)(n-r)-\binom{r+3}{3}$. Specifically, when $r=0$, $X$ is an elliptic curve; and as established in~\cite{BeauvilleDonagi}, $X$ is a hyper-Kähler manifold for $r=1$. It is further shown (see Lemma~\ref{LmmStrictCY}) that for $r \geq 2$, $X$ is a strict Calabi-Yau manifold. As is studied in Theorem~\ref{ThmLocalDeformation} in Chapter~\ref{ChapterVoisinMaps}, the small deformations of $X$ is relatively easy to understand.
\begin{theorem}
    Assume $r\geq 2$. Let $X = F_r(Y)$ be the above strict Calabi-Yau manifold.
    \begin{enumerate}
        \item[(a)] For any small deformation $X'$ of $X$, there is a cubic hypersurface $Y'$ such that $X' = F_r(Y')$.
        \item[(b)] The dimension of deformation space of $X$ is $\dim H^0(\mathbb P^n, \mathcal O_{\mathbb P^n}(3)) - \dim GL_{n+1}(\mathbb C)$. 
    \end{enumerate}
\end{theorem}

The distinct feature of the above manifolds $X = F_r(Y)$ among all $K$-trivial manifolds revolves around the presence of a self-rational map, $\Psi: X \dashrightarrow X$, referred to as the Voisin map. This map was introduced in~\cite{KCorr} through the following construction: Consider a general point $x \in X$, representing an $r$-dimensional linear space $P_x$ within $Y$. As demonstrated in~\cite[Lemma 8]{KCorr}, there exists a unique $(r+1)$-dimensional linear subspace $H_x$ in $\mathbb{P}^n$ tangent to $Y$ along $P_x$. The intersection $H_x \cap Y$ forms a cubic hypersurface containing $P_x$ doubly, leaving a residual linear subspace in $Y$ represented by a point $x' \in X$. This process defines the Voisin map as $\Psi(x) = x'$. 

In Chapter~\ref{ChapterVoisinMaps}, we prove the following fact about the Voisin maps in Theorem A:
\begin{theorem}
Given a nowhere zero top degree holomorphic form $\omega \in H^0(X, K_X)$ on $X$, we have 
\[\Psi^*\omega = (-2)^{r+1}\omega.\]
\end{theorem}
By taking $Z = \Gamma_\Psi - (-2)^*\Delta_X\in CH(X\times X)$ where $\Gamma_\Psi$ is the graph of $\Psi$ and $\Delta_X$ is the diagonal of $X\times X$, it should be expected, by Conjecture~\ref{ConjBlochCY}, that for any $z\in CH_0(X)_{\mathbb Q, hom}$, $\Psi_*z = (-2)^{r+1} z$. In the case $r = 2$, we succeed in proving this result in Theorem B in Chapter~\ref{ChapterVoisinMaps}.

\begin{theorem}
    Let $Y \subset \mathbb{P}^9$ be a general cubic $8$-fold. Let $X = F_2(Y)$ be the Fano variety of planes in $Y$ and let $\Psi: X \dashrightarrow X$ be the Voisin map. Then for any $z \in CH_0(X)_{hom}$, we have
    \[\Psi_*z = -8z \,\textrm{ in }\,
    CH_0(X)_{\mathrm{hom}}.\]
\end{theorem}

\section{Some measures of irrationality}

Let $X$ be a projective variety over $\mathbb C$ of dimension $n$. By applying a general  linear projection to $X\subset \mathbb{P}^N$, there exists a generically finite dominant rational map $\phi: X\dashrightarrow \mathbb P^n$. The following invariant, the degree of irrationality, was proposed and studied in~\cite{Irrationality}.

\begin{definition}[\cite{Irrationality}]
    The degree of irrationality of $X$, denoted as $\mathrm{Irr}(X)$, is the minimal degree of dominant rational maps $\phi: X\dashrightarrow \mathbb P^n$. In other words,
    \[
    \mathrm{Irr}(X):= \min\{\deg \phi: \textrm{there is a dominant rational map $\phi: X\dashrightarrow\mathbb P^n$}\}.
    \]
\end{definition}
The degree of irrationality measures how far a variety is from being rational, in the sense that $X$ is rational if and only if $\mathrm{Irr}(X) = 1$. If $X$ is a smooth proper curve, the degree of irrationality of $X$ is the gonality of the curve $X$, which is a classical invariant of a curve.

 In their seminal work~\cite{Irrationality}, Bastianelli, De Poi, Ein, Lazarsfeld, Ullery propose the following conjecture on the degree of irrationality of K3 surfaces.

 \begin{conjecture}[\cite{Irrationality}]\label{ConjectureBDELDIntro}
Let $\{(S_d, L_d)\}_{d\in \mathbb N}$ be very general polarized $K3$ surfaces such that $L_d^2=2d-2$. Then
\[\limsup_{d\to \infty}\mathrm{Irr}(S_d)=+\infty.
\]
\end{conjecture}

In an attempt to determine the degree of irrationality and to better understand the geometry of the variety, the following two measures of irrationality are proposed and studied by Voisin in~\cite{FibgenFibgon}.
\begin{definition}[\cite{FibgenFibgon}]
    (i) The fibering gonality of $X$,
    $\mathrm{Fibgon}(X)$, is the minimal number $c$ such that there exists a rational dominant map $\pi: X\dashrightarrow B$ such that $\dim B=\dim X-1$ and the general fiber is a connected curve $C$ of gonality $c$.\\
    (ii) The fibering genus of $X$,
    $\mathrm{Fibgen}(X)$, is the minimal number $c$ such that there exists a rational dominant map $\pi: X\dashrightarrow B$ such that $\dim B=\dim X-1$ and the general fiber is a connected curve $C$ of geometric genus $c$.
\end{definition}
The fibering genus of K3 surfaces has already been studied by Ein and Lazarsfeld in~\cite{EinLazarsfeld}, under the name \emph{Konno invariant}. They give a good estimate of the asymptotic behavior of fibering genus of K3 surfaces as the genus tends to infinity. Here the genus $d$ of a K3 surface $S$ is defined by the formula $2d-2=c_1(L)^2$, where $L$ generates  $\mathrm{Pic}(S)$.
\begin{theorem}[Ein--Lazarsfeld~\cite{EinLazarsfeld}]
Let $S_d$ be a polarized K3 surface of Picard rank $1$ and genus $d$. Then 
\[
 \mathrm{Fibgen}(S_d)=\Theta(\sqrt{d}),
\]
which means that there are constants $C_1$, $C_2 > 0$ such that $C_2\sqrt d < \mathrm{Fibgen}(S_d) < C_1 \sqrt d$.
\end{theorem}

\begin{corollary}
    We have 
    \[
    \lim_{d\to\infty} \mathrm{Fibgen}(S_d) = +\infty.
    \]
\end{corollary}

It is elementary to prove that for any smooth projective variety $X$, the following inequlities are satisfied.\\
(i) $\mathrm{Fibgon}(X)\leq \mathrm{Irr}(X)$.\\
(ii) $\mathrm{Fibgon}(X)\leq \frac12(\mathrm{Fibgen}(X)-1)+2$.

In Chapter~\ref{ChapterBirational}, we get a finer comparison for very general projective K3 surfaces (see Theorem~\ref{theoremFibgonDeK3}).
\begin{theorem}
Let $S$ be a projective K3 surface whose Picard number is $1$. Then one of the following two cases holds:\\
(a) $\mathrm{Irr}(S)=\mathrm{Fibgon}(S)$;\\
(b) $\mathrm{Fibgen}(S)^2\leq \mathrm{Fibgon}(S)^{21}$.
\end{theorem}

\begin{corollary}
    Let $\{S_d\}_{d\in\mathbb N}$ be projective K3 surfaces such that the Picard group of $S_d$ is generated by a line bundle with self intersection number $2d-2$. Then 
\[\limsup_{d\to\infty} \mathrm{Irr}(S_d)=+\infty \iff \limsup_{d\to \infty} \mathrm{Fibgon}(S_d)=+\infty.\]
\end{corollary}
Conjecture~\ref{ConjectureBDELDIntro} predicts that $\limsup_{d} \mathrm{Irr}(S_d)=+\infty$, so it should be expected that $\limsup_{d} \mathrm{Fibgon}(S_d)=+\infty$ as well. 

\section{Hyper-Kähler Manifolds}\label{IntroHKManifolds}
\begin{definition}
A hyper-Kähler manifold $X$ is defined as a compact Kähler manifold that is simply connected and for which $H^0(X,\Omega_X^2)$ is generated by an everywhere non-degenerate holomorphic $2$-form $\sigma_X$.
\end{definition}

The presence of an everywhere non-degenerate holomorphic $2$-form implies that $X$ has an even dimension. Hereafter, we denote by $m=2n$  the dimension of $X$ and by $\sigma_X$ a non-degenerate holomorphic $2$-form of $X$ (determined up to a scalar). Notably, a K3 surface is defined as a hyper-Kähler manifold of dimension $2$. The broader aspects of hyper-Kähler varieties are extensively discussed in~\cite{Beauville, Fujiki, Huy99}.

\subsection{Beauville-Bogomolov-Fujiki Form}
Hyper-Kähler manifolds exhibit a non-degenerate symmetric integral quadratic form on $H^2(X,\mathbb Z)$. This form, which extends the Lefschetz intersection form applicable to K3 surfaces, has been discovered in~\cite{Beauville, Fujiki}. It is referred to as the Beauville-Bogomolov-Fujiki form.

\begin{theorem}[\cite{Beauville, Fujiki}]\label{ThmDeFormeBBF}
Consider a hyper-Kähler manifold $X$ of dimension $m=2n$. There exist a unique integral quadratic form $q_X$ on $H^2(X,\mathbb Q)$ and a positive rational coefficient $\lambda_X\in\mathbb Q_{>0}$ such that: \\
(1) for any $\alpha\in H^2(X,\mathbb Q)$,
\begin{equation}\label{FormeDeBBF}
\int_X\alpha^m=\lambda_Xq(\alpha)^n;
\end{equation}
(2) $q_X$ is indivisible, meaning that for any $k>1$, $q_X/k$ is not integral;\\
(3) $q_X(\alpha)$ is positive for any Kähler class $\alpha$.
\end{theorem}

\subsection{Lagrangian Subvarieties}
Consider a hyper-Kähler manifold $X$ of dimension $2n$ with a holomorphic $2$-form $\sigma_X$.
\begin{definition}
A closed subvariety $L\subset X$ is called Lagrangian if:\\
(i) The restriction of $\sigma_X$ to the smooth part $L_{reg}$ of $L$ vanishes, and\\
(ii) The dimension of $L$ is $n$.
\end{definition}

The deformation of \emph{smooth} Lagrangian subvarieties with $X$ has been investigated in~\cite{VoisinLag}.

\begin{theorem}[Voisin~\cite{VoisinLag}]\label{ThmDeVoisin}
For a hyper-Kähler manifold $X$ and a smooth Lagrangian subvariety $L\subset X$, let $\mathcal X\to \mathrm{Def}(X)$ represent the Kuranishi family of $X$, with $o\in \mathrm{Def}(X)$ as the reference point. Also, let $\mathrm{Def}(X,L)$ be the deformation germ of the pair $(X,L)$. Then,\\
(1) $\mathrm{Def}(X,L)$ is smooth, and both its fibers and the image of the natural map
\[\pi: \mathrm{Def}(X,L)\to \mathrm{Def}(X)
\]
are also smooth.\\
(2) Furthermore,
\begin{align*}
\mathrm{Im}\pi & = \{t\in\mathrm{Def}(X): [L]\textrm{ remains a Hodge class in } H^{2n}(\mathcal X_t,\mathbb Q)\}\\
&= \{t\in \mathrm{Def}(X): [\sigma_{X_t}]\in\ker(H^2(X,\mathbb C)\to H^2(L,\mathbb C))\}.
\end{align*}
\end{theorem}

The second equality in (2) implies that small deformations of $X$ with constant Picard number contain a deformation of $L$. Theorem~\ref{ThmDeVoisin} crucially indicates that the deformation of a smooth Lagrangian subvariety within a hyper-Kähler manifold is unobstructed. It is used in Chapter~\ref{ChapterAJ} to justify the naturality of a certain technical condition.

\subsection{Lagrangian fibrations}
Consider a hyper-Kähler manifold $X$. A Lagrangian fibration of $X$ is characterized as a holomorphic surjection $\pi: X\to B$, where the general fibers are Lagrangian subvarieties of $X$. Importantly, a Lagrangian fibration differs from a topological fibration in that it is a topological fibration solely over an open subset of the base and includes singular fibers.

This section describes the most important classical results on Lagrangian fibrations.

\begin{theorem}[Matsushita~\cite{Matsushita99}]
Let $\pi: X\to B$ be a holomorphic surjection  from a hyper-Kähler manifold $X$ of dimension $2n$ to a compact complex variety $B$, assuming $0<\dim B<2n$ and the fibers are connected. Then $\dim B=n$ and $\pi: X\to B$ defines a Lagrangian fibration.
\end{theorem}

\begin{theorem}[Matsushita~\cite{Matsushita99}, Voisin~\cite{VoisinLag}]
For a Lagrangian fibration $\pi: X\to B$, the general fibers are projective, hence are abelian varieties.
\end{theorem}

\begin{theorem}[Hwang~\cite{Hwang}]\label{ThmHwang}
Given a Lagrangian fibration $\pi: X\to B$ with a smooth base $B$, then $B\cong\mathbb P^n$.
\end{theorem}
Theorem~\ref{ThmHwang} also holds under the assumption that $X$ is merely smooth Kähler. This was demonstrated by Greb and Lehn~\cite{Greb} using a result on the deformation of Lagrangian fibrations by Matsushita~\cite{Matsushita16}. It is conjectured that the same holds when the base $B$ is normal. Supporting evidence for this conjecture is provided in~\cite{MatsushitaBase}, which discusses the intersection cohomology of the base. While this conjecture is straightforward in the context of K3 surfaces, it remains unresolved in a broader scope. In the case of four-dimensional hyper-Kähler manifolds, the conjecture has been proved by Huybrechts and Xu in~\cite{HuyXu}.

\subsection{Lagrangian families}
A natural generalization of Lagrangian fibrations is the notion of Lagrangian families, introduced by Voisin in~ \cite{VoisinTriangle, VoisinLefschetz}.

\begin{definition}[Voisin]
A Lagrangian family of a hyper-Kähler manifold $X$ is a diagram
\begin{equation}
    \begin{tikzcd}
    \mathcal L\arrow{r}{q}\arrow{d}{p} & X\\
    B &
    \end{tikzcd}
\end{equation}
In this configuration, $p$ is flat and projective, $\mathcal L$ and $B$ are connected quasi-projective manifolds, and $q$ maps the general fiber $L_b:=p^{-1}(b)$, $b\in B$, birationally to a Lagrangian subvariety of $X$. For practical purposes, we will denote $j_b$ as the composition $L_b\hookrightarrow\mathcal L\to X$.
\end{definition}

As opposed to the Lagrangian fibrations, the existence of Lagrangian families are expected to be a less restrictive condition. In Chapter~\ref{ChapterAJ}, we study the Abel-Jacobi map associated with the Lagrangian families. We briefly review Griffiths' theory on Abel-Jacobi maps in Section~\ref{SectionAbelJacobiMaps}. In Chapter~\ref{ChapterAJ}, we prove the following criterion to determine if the Abel-Jacobi map of a given Lagrangian family is trivial (see Proposition~\ref{CriterionIntroduction}).
\begin{theorem}
    Consider a Lagrangian family of a hyper-Kähler manifold $X$ of dimension $2n$, satisfying the following condition:

\begin{quote}
   $\clubsuit$ For general $b\in B$, the contraction by $q^*\sigma_X$ gives an isomorphism $\lrcorner q^*\sigma_X: T_{B,b}\stackrel{\cong}{\to} H^0(L_b,\Omega_{L_b}).$ 
\end{quote}
 Then the Abel-Jacobi map (\ref{AbelJacobiMap}) is trivial if and only if for general $b\in B$, the restriction map 
\[j_b^*: H^{2n-1}(X,\mathbb Q)\to H^{2n-1}(L_b,\mathbb Q)
\]
is zero.
\end{theorem}

This criterion gives a topological method to determine the triviality of the Abel-Jacobi maps of the Lagrangian families.

\subsection{Mumford-Tate Groups}\label{IntroMTGroup}
Our discussion closely follows the presentations of Mumford-Tate groups as found in~\cite[Section 7]{DeligneK3}, \cite{VoisinvGeemen}, and \cite[Section 4.2]{HodgeLoci}. 

Consider a rational Hodge structure $H = (H_\mathbb Q, H^{p,q})_{p+q = r}$  of weight $r$. An algebraic group action, denoted $\mu$, of $\mathbb S^1$ on $H_\mathbb R$ is defined as follows: For any $v\in H_\mathbb R$, with Hodge decomposition $v = \sum_{p+q= r}v^{p, q}$, and for any $z = e^{i\theta}\in \mathbb S^1\subset \mathbb C^*$, we define
\[
\mu(z).v := \sum_{p,q} z^p\bar z^q v^{p,q}.
\]

\begin{definition}
    The Mumford-Tate group $MT(H)$ for a Hodge structure $H$ is defined as the smallest algebraic subgroup $G$ of $GL(H_\mathbb Q)$, defined over $\mathbb Q$, such that $G(\mathbb R)$ contains $\mu(\mathbb S^1)$.
\end{definition}

It should be noted that the Mumford-Tate group as defined here is referred to as the special Mumford-Tate group of $H$ in~\cite{HodgeLoci}.

Returning to the context of hyper-Kähler manifolds, consider a projective hyper-Kähler manifold $X$ with dimension $2n$, and let $q$ denote the Beauville-Bogomolov-Fujiki form on $H^2(X, \mathbb Q)$. The transcendental part of $H^2(X,\mathbb Q)$, denoted by $H^2(X, \mathbb Q)_{tr}$, is the minimal sub-Hodge structure containing $H^{2,0}$ and, thanks to the projectivity of $X$, can also be defined as the orthogonal complement to the Néron-Severi group $NS(X)_{\mathbb Q}$ in $H^2(X,\mathbb Q)$ with respect to the Beauville-Bogomolov-Fujiki quadratic form. Due to the first Hodge-Riemann bilinear relation, the Hodge structure of $H^2(X,\mathbb Q)$ is compatible with the Beauville-Bogomolov-Fujiki form $q$, which implies that the Mumford-Tate group $MT(H^2(X,\mathbb Q)_{tr})$ is contained in $SO(H^2(X, \mathbb Q)_{tr}, q)$. The following theorem shows that this inclusion is an equality for a very general polarized hyper-Kähler manifold (see, e.g.~\cite[Lemma 9]{VoisinvGeemen}):

\begin{theorem}\label{ThmMaxMTOfHKIsSO}
    Let $X$ be a very general fiber of a complete family of lattice polarized deformations. Then the Mumford-Tate group $MT(H^2(X,\mathbb Q)_{tr})$ is equal to $SO(H^2(X, \mathbb Q)_{tr}, q)$.
\end{theorem}

When $X$ meet the criteria of Theorem~\ref{ThmMaxMTOfHKIsSO}, it is said that the Mumford-Tate group of $X$ is maximal, or that $X$ is a Mumford-Tate very general hyper-Kähler manifold. This is a technical condition that may help us to prove results for very general projective hyper-Kähler manifolds. Here are two examples of how the Mumford-Tate maximality condition appears in this thesis.

The first one gives a condition for the triviality of the Abel-Jacobi map for a Lagrangian family (see Theorem~\ref{MainTheorem} and Proposition~\ref{PropMaxVarIntroduction}).
\begin{theorem}
 Consider a Lagrangian family on a Mumford-Tate very general hyper-Kähler manifold $X$ of dimension $2n$ satisfying the following conditions
 \begin{itemize}
     \item[(i)] At least one of the Lagrangian subvarieties $L_b$ in the Lagrangian family is smooth;
     \item[(ii)] $h^{1,0}(L_b)\leq 2^{\lfloor\frac{b_2(X)_{tr}-3}{2}\rfloor}$.
 \end{itemize}
 Then the Abel-Jacobi map associated with this Lagrangian family is trivial.
\end{theorem}

The second example consists of the lower bound of the fibering genus of a projective hyper-Kähler manifold obtained by Voisin~\cite{FibgenFibgon}.
\begin{theorem}[Voisin~\cite{FibgenFibgon}]
Let $X$ be a Mumford-Tate very general hyper-Kähler manifold of dimension $2n$ with $n\geq 3$ and $b_{2,\mathrm{tr}}(X)\geq 5$. Then 
\[
\mathrm{Fibgen}(X)\geq \min\{n+2, 2^{\lfloor \frac{b_{2, \mathrm{tr}}(X)-3}{2}\rfloor}\}.
\]   
\end{theorem}
This result is further improved as follows in Chapter~\ref{ChapterBirational} (see Theorem~\ref{TheoremFibGen}).
\begin{theorem}
    Let $X$ be a Mumford-Tate very general hyper-Kähler manifold of dimension $2n$ and assume $b_{2,\mathrm{tr}}(X)\geq 5$. Then 
    \[\mathrm{Fibgen}(X)\geq \min\left\{ n+\left\lceil\frac{-1+\sqrt{8n-7}}{2}\right\rceil, 2^{\lfloor\frac{ b_{2,\mathrm{tr}}(X)-3}{2}\rfloor}
\right\}.
\]
\end{theorem}

\section{Beauville's splitting conjecture and Voisin's conjectures}\label{IntroVoisinsConjectures}
We follow closely the article~\cite{VoisinCoisotrope}. In this section, the Chow rings are considered with rational coefficients. Let $X$ be a projective hyper-Kähler manifold. The splitting conjecture of Beauville~\cite{BeauvilleSplitting} predicts that the Bloch-Beilinson filtration on the Chow ring of $X$ has a natural multiplicative splitting. As is explained in the Introduction of \cite{VoisinCoisotrope}, this implies the following conjecture, now often called the ``weak splitting conjecture''.
\begin{conjecture}[Beauville's weak splitting conjecture~\cite{BeauvilleSplitting}]
    Let $X$ be a projective hyper-Kähler manifold. Then the cycle class map is injective on the subalgebra of $CH^*(X)$ generated by divisors.
\end{conjecture}

Beauville's weak splitting conjecture has been enlarged in~\cite{0CycleHK} as follows.

\begin{conjecture}[Voisin~\cite{0CycleHK}]\label{conjVoisinChernIntroThesis}
    Let $X$ be a projective hyper-Kähler manifold. Let $C^*$ be the subalgebra of $CH^*(X)$ generated by divisors and Chern classes. Then the cycle class map is injective on $C^*$.
\end{conjecture}

It has been shown by Beauville~\cite[Examples 1.7]{BeauvilleSplitting} that the analogies of Conjecture~\ref{ConjBeauvilleWeakSplitting} and Conjecture~\ref{conjVoisinChernIntroThesis} for strict Calabi-Yau manifolds are false in general. Beauville constructs in \emph{loc. cit.} a strict Calabi-Yau threefold $X$ such that the cycle-class map $cl: CH_1(X)_\mathbb Q\to H^4(X,\mathbb Q)$ is \emph{not} injective on the subgroup generated by intersections of divisors. However, it is expected that the analogy of Conjecture~\ref{conjVoisinChernIntroThesis} to strict Calabi-Yau manifolds still holds true for $0$-cycles.
\begin{conjecture}[Voisin]\label{ConjVoisinCYIntro}
    Let $X$ be a strict Calabi-Yau manifold of dimension $n$. Let $C\subset CH_0(X)_\mathbb Q$ be the degree $n$ part of the subring of $CH^*(X)$ generated by the intersections of divisors and of Chern classes of $X$. Then the cycle class map 
    \[
    cl: CH_0(X)_\mathbb Q\to H^{2n}(X,\mathbb Q)
    \]
    is injective on $C$.
\end{conjecture}

We give a conditional result on Conjecture~\ref{ConjVoisinCYIntro} for the strict Calabi-Yau $11$-folds $F_2(Y)$ constructed by Voisin~\cite{KCorr} (see Theorem C in Chapter~\ref{ChapterVoisinMaps}) and studied in this thesis.

\begin{theorem}
    Let $X = F_2(Y)$ be the Fano variety of planes of a general cubic $8$-fold $Y\subset \mathbb P^9$. Suppose the indeterminacy locus of the Voisin map is a constant cycle subvariety. Then $X$ satisfies Conjecture~\ref{ConjVoisinCYIntro}.
\end{theorem}

In~\cite{VoisinCoisotrope}, a filtration, now often called the ``Voisin's filtration'', is proposed for the Chow group of $0$-cycles of $X$. As is defined in \emph{loc. cit.}, for $x\in X$, let $O_x\subset X$ be the rational equivalence orbit of $x$, namely, $O_x$ is the set of points in $X$ that are rationally equivalent to $x$. Standard argument~\cite[p. 2]{VoisinCoisotrope} shows that $O_x$ is a countable union of closed algebraic subsets of $X$, which makes it possible to define the dimension of $O_x$ as the maximal dimension among the irreducible components of $O_x$. 
\begin{definition}[Voisin~\cite{Voisin}]\label{DefVoisinsFiltration}
    We define $S_iX \subset X$ to be the set of points in $X$ whose orbit under rational equivalence has dimension $\geq i$.
The filtration $S_\bullet$ is then defined by letting $S_iCH_0(X)$ be the subgroup of $CH_0(X)$ generated by classes of points $x \in S_iX$.
\end{definition}

Highly motivated by the theory developed in \emph{loc. cit.}, Conjecture~\ref{conjVoisinChernIntroThesis} is enlarged in~\cite{VoisinCoisotrope}. Let $2n$ be the dimension of $X$. In the case of $i = n$, the conjectures in~\cite{VoisinCoisotrope} predict that two \emph{constant cycle} subvarieties of dimension $n$, upon sharing an identical cohomological class, ought to share the same Chow class as well. Motivated by this conjecture, in Chapter~\ref{ChapterAJ}, we study the Chow classes of Lagrangian subvarieties, raising the question of whether two Lagrangian subvarieties within the same hyper-Kähler manifold, sharing identical cohomological classes, also possess the same Chow class. Given that a constant-cycle subvariety of dimension $n$ is easily seen to be Lagrangian, see~\cite[Theorem 0.7]{VoisinCoisotrope}, our question emerges as a natural extension of Voisin's conjecture. In Chapter~\ref{ChapterAJ}, we give negative answer to this question by presenting a counter-example, thereby uncovering the intricacies and challenges associated with Voisin's conjecture.

We construct our counter-examples within the framework of the generalized Kummer varieties as introduced in~\cite{Beauville}. Let us remind the construction. Consider an abelian surface $A$ and its Hilbert scheme $A^{[n+1]}$ of length $n+1$ subschemes. The morphism $\mathrm{alb}: A^{[n+1]}\to A$ results from combining the Hilbert-Chow morphism with the summation map as follows:
\[A^{[n+1]}\to A^{(n+1)}\to A.\]
It is important to note that $\mathrm{alb}$ constitutes an isotrivial fibration.
The generalized Kummer variety, denoted as $K_n(A)$, is defined as the fiber of $\mathrm{alb}$ over $0 \in A$. As demonstrated in~\cite{Beauville}, $K_n(A)$ is a hyper-Kähler manifold of dimension $2n$. We then proceed to construct Lagrangian subvarieties in $K_n(A)$. For any element $x \in A$, we identify a subvariety $Z_x$ within $K_n(A)$ that consists of Artinian subschemes of $A$ of length $n+1$, supported on $x$ and $-nx$, with multiplicities $n$ and $1$, respectively. For any curve $C \subset A$, we define $Z_C = \bigcup_{x \in C} Z_x$. In Section~\ref{Examples}, we establish the following theorem:
\begin{theorem}
    \begin{enumerate}
        \item[(a)] The subvariety $Z_C$ is a Lagrangian subvariety in $K_n(A)$.
        \item[(b)] For numerically equivalent very ample curves $C$ and $C'$ in $A$, $Z_C$ and $Z_{C'}$ share the same cohomological class in $H^{2n}(K_n(A),\mathbb{Z})$ but possess distinct Chow classes in $CH^n(K_n(A))$.
    \end{enumerate}
\end{theorem}

\chapter{On Abel-Jacobi Maps of Lagrangian Families}\label{ChapterAJ}

We study in this chapter the cohomological properties of Lagrangian families on projective hyper-Kähler manifolds. First, we give a criterion for the vanishing of Abel-Jacobi maps of Lagrangian families. Using this criterion, we show that under a natural condition, if the moduli map for the fibers of the Lagrangian family is maximal, its Abel-Jacobi map is trivial. We also construct Lagrangian families on generalized Kummer varieties whose Abel-Jacobi map is not trivial, showing that our criterion is optimal. 

This chapter presents the main result of~\cite{Bai23}. The structure is as follows. In Section~\ref{SectionBackgroundInAJ}, we give a panorama of known results in this area. Section~\ref{SectionCriterion}, Section~\ref{SectionLagrangianFibrations} and Section~\ref{SectionMaxVar} are the content of the article~\cite{Bai23}. In Section~\ref{SectionCriterion}, we prove Proposition~\ref{CriterionIntroduction}. In Section~\ref{SectionLagrangianFibrations}, we construct a Lagrangian fibration structure on the relative Albanese variety and use it to prove Theorem~\ref{MainTheorem}. In Section~\ref{SectionMaxVar}, we discuss the condition on the maximality of the variation of Hodge structures. In Section~\ref{Examples}, we construct Lagrangian families satisfying a maximal condition whose Abel-Jacobi map is nontrivial, showing that Theorem~\ref{MainTheorem} is optimal. 

\section{Background}\label{SectionBackgroundInAJ}
The primary objective of this chapter is to address the following question:

\begin{question}\label{ProbBackgroundInAJ}
    Suppose $X$ is a hyper-Kähler manifold of dimension $2n$, and $L_1$ and $L_2$ are two Lagrangian subvarieties, both sharing the same cohomological class in $H^{2n}(X, \mathbb Z)$. Are $L_1$ and $L_2$ necessarily rationally equivalent to each other?
\end{question}

The example presented in Section~\ref{Examples} gives a negative answer to Question~\ref{ProbBackgroundInAJ}.

\subsection{Abel-Jacobi Maps}\label{SectionAbelJacobiMaps}
We begin by revisiting the Abel--Jacobi invariant and the Abel--Jacobi map introduced by Griffiths~\cite{GriffithsAJ}, following the framework outlined in~\cite[Chapter 12]{Voisin}. Throughout this section, $X$ is a compact Kähler manifold of dimension $n$.

\subsubsection*{Intermediate Jacobian}
\begin{definition}[Griffiths]
    The $k$-th intermediate Jacobian, denoted $J^{2k-1}(X)$, is the complex torus defined as:
    \[
    J^{2k-1}(X)=H^{2k-1}(X, \mathbb C)/(F^kH^{2k-1}(X)\oplus H^{2k-1}(X,\mathbb Z)_{\mathrm{tf}}).
    \]
\end{definition}
By applying Poincaré duality, it follows that:
\[
J^{2k-1}(X) \cong F^{n-k}H^{2n-2k+1}(X,\mathbb C)^*/H_{2n-2k+1}(X, \mathbb Z)_{\mathrm{tf}}.
\]
Here the subscript ``tf'' signifies the ``torsion free part'' of the integral cohomology.

\begin{example}
    (i) When $k = 1$, the intermediate Jacobian $J^1(X) = \mathrm{Pic}^0(X)$, representing the degree $0$ part of the Picard group of $X$. \\
    (ii) When $k = n$, the intermediate Jacobian $J^{2n -1}(X) = \mathrm{Alb}(X)$ is the Albanese variety of $X$.
\end{example}

\subsubsection*{Abel-Jacobi Invariant}
Let $\mathcal Z^k(X)$ denote the free abelian group generated by irreducible subvarieties of codimension $k$ in $X$. The cycle class map,
\[ cl: \mathcal Z^k(X)\to H^{2k}(X,\mathbb Z) \]
maps the class of a subvariety of codimension $k$ to its cohomological class in $H^{2k}(X,\mathbb Z)$. Let $\mathcal Z^k(X)_{hom}$ represent the kernel of the cycle class map, where the subscript "hom" signifies "homologous to 0".

Associated with any $k$-cocycle $Z\in \mathcal Z^k(X)_{hom}$ that is homologous to $0$ is an element $\alpha_Z\in J^{2k-1}(X)$, termed the \emph{Abel-Jacobi invariant of $Z$}, in the intermediate Jacobian of the corresponding degree. The construction of $\alpha_Z$ can be briefly described as follows~\cite{GriffithsAJ}, \cite[12.1.2]{Voisin}.

Given that $Z$ is homologous to $0$, there exists a chain $\Gamma$ of codimension $2k-1$ such that $\partial \Gamma = Z$. The integration current along $\Gamma$,
\[ \int_\Gamma: \omega\mapsto \int_\Gamma\omega \]
can be interpreted as an element in $F^{n-k}H^{2n-2k-1}(X,\mathbb C)^*$. In fact, if $\omega$ and $\omega'$ yield the same cohomological class in $F^{n-k}H^{2n-2k-1}(X,\mathbb C)$, then, by a deep result of Hodge theory, they differ by an exact form $d\phi$ with $\phi\in F^{n-k-1}\mathcal A^{2n-2k-2}_{\mathbb C}(X)$. By Stokes' theorem, $\int_\Gamma d\phi = \int_Z \phi$, which vanishes due to type reasons. Additionally, if $\Gamma$ and $\Gamma'$ are two chains such that $\partial \Gamma = \partial \Gamma' = Z$, then modulo the image of $H_{2n-2k+1}(X,\mathbb Z)$ in $F^{n-k}H^{2n-2k+1}(X,\mathbb C)^*$, the linear forms $\int_\Gamma$ and $\int_{\Gamma'}$ are the same. Hence, there exists a uniquely determined element:
\[ \alpha_Z = \left[\int_\Gamma\right]\in F^{n-k}H^{2n-2k-1}(X,\mathbb C)^*/H_{2n-2k+1}(X,\mathbb Z) = J^{2n-2k+1}(X). \]

\begin{proposition}[\cite{Voisin}, Lemme 21.19]\label{PropREImpliesAJ}
    Let $Z$ be an algebraic cycle that is rationally equivalent to $0$ in $X$. Then the Abel--Jacobi invariant of $Z$ is $0$.
\end{proposition}

By Proposition~\ref{PropREImpliesAJ}, we are now equipped with a map
\[
\begin{array}{cccc}
    \Phi^k_X :& CH^n(X)_{hom} & \to & J^{2n-1}(X) \\
    &Z=\partial \Gamma &\mapsto &\int_{\Gamma}
\end{array}
\]

\subsubsection*{Abel-Jacobi Map of a Family}
Let $\mathcal Z\subset B\times X$ be a flat family of subvarieties of codimension $k$ in $X$. Specifically, $B$ is a connected complex manifold, and $\mathcal Z$ is a subvariety of codimension $k$ in $B\times X$, flat over $B$. Denote by $p: \mathcal Z\to B$ and $q:\mathcal Z\to X$ the projection maps to the two components. Let $0\in B$ be a reference point.

\begin{definition}
    The Abel--Jacobi map of the family $\mathcal Z\subset B\times X$ with respect to the reference point $0\in B$ is the map 
    \[
    \begin{array}{cccc}
       \Psi_{\mathcal L}^{AJ}: & B &\to & J^{2k-1}(X)  \\
         & b &\mapsto &\Phi_X^k(q_*p^*(b)-q_*p^*(0)).
    \end{array}
    \]
\end{definition}
  
\begin{theorem}[Griffiths~\cite{GriffithsAJ}]\label{theoremGriffiths}
    \begin{enumerate}
        \item[(i)] The Abel--Jacobi map $\Psi_{\mathcal L}^{AJ}: B\to J^{2k-1}(X)$ is holomorphic.
        \item[(ii)] The image of the differential of $\Psi_{\mathcal L}^{AJ}$ at any point lies in $H^{k-1,k}(X)\subset H^{2k-1}(X,\mathbb C)$.
    \end{enumerate}
\end{theorem}

\subsection{Lagrangian Families}
Consider $X$ as a projective hyper-Kähler manifold of dimension $2n$. Voisin introduced the notion of Lagrangian families in her work \cite{VoisinTriangle, VoisinLefschetz}, which serves as a generalizations of Lagrangian fibrations.

\begin{definition}[Voisin~\cite{VoisinTriangle, VoisinLefschetz}]
A Lagrangian family of a hyper-Kähler manifold $X$ is a diagram
\begin{equation}\label{LagrangianFamily}
    \begin{tikzcd}
    \mathcal L\arrow{r}{q}\arrow{d}{p} & X\\
    B &
    \end{tikzcd}
\end{equation}
In this configuration, $p$ is flat and projective, $\mathcal L$ and $B$ are connected quasi-projective manifolds, and $q$ maps the general fiber $L_b:=p^{-1}(b)$, $b\in B$, birationally to a Lagrangian subvariety of $X$. For practical purposes, we will denote by $j_b$ the composition $L_b\hookrightarrow\mathcal L\to X$.
\end{definition}

Let us give some examples of Lagrangian families documented in the literature:

\begin{example}\label{ExampleLagrangianFamilies}
    \begin{enumerate}
        \item[(i)] Any Lagrangian fibration of $X$ naturally leads to a Lagrangian family. Note that the existence of a Lagrangian fibration necessitates that $X$ has a Picard number of at least $2$, whereas a very general projective hyper-Kähler manifold typically has a Picard number of $1$.
        \item[(ii)] (Voisin~\cite{VoisinLag}) Let $S$ be a K3 surface and $C$ be a curve in $S$. The curve $C$ can move in a family $\{C_b\}_{b\in U}$ where $U$ is an open subset of the Hilbert scheme of deformations of $C$ in $S$. This constitutes a Lagrangian family of $S$. Moreover, consider $S^{[n]}$, the Hilbert scheme of $n$ points of $S$, which is a hyper-Kähler manifold of dimension $2n$. Then $\{C_b^{(n)}\}_{b\in U}$ is a Lagrangian family of $S^{[n]}$.
        \item[(iii)] (Voisin~\cite{VoisinLag}) Let $Y\subset \mathbb P^5$ be a smooth cubic fourfold, and $H\subset \mathbb P^5$ be a general hyperplane. Then the Fano surface $F_1(Y\cap H)$ of lines of $Y\cap H$, is a Lagrangian subvariety of the Fano variety $F_1(Y)$ of lines of $Y$, which is a hyper-Kähler fourfold by~\cite{BeauvilleDonagi}. The general hyperplanes $H\subset \mathbb P^5$ are parameterized by an open subset $U$ of $(\mathbb P^5)^*$. The family $\{F_1(Y\cap H)\}_{H\in U}$ provides a Lagrangian family of $F_1(Y)$.
        \item[(iv)] (Iliev-Manivel~\cite{IlievManivel}) Let us consider a linear inclusion $\mathbb P^5\subset \mathbb P^6$ and a smooth cubic fourfold $Y\subset \mathbb P^5$. Now, let $Z\subset \mathbb P^6$ be a general smooth cubic fivefold containing $Y$. Let $S_Z:=F_2(Z)$ denote the Fano surface of planes in $Z$. The map
        \[
        \begin{array}{cccc}
            j_Z: & S_Z & \to & F_1(Y)   \\
             & P & \mapsto & P\cap H
        \end{array}
        \]
        is generically 1:1 on its image and the image is a Lagrangian subvariety. Consequently, $\{S_Z\}_{Z\in U}$, where $U$ is an open subset of all cubic fivefolds containing $Y$, provides a Lagrangian family of $F_1(Y)$.
    \end{enumerate}
\end{example}

\subsection{Voisin's Conjectures Revisited}
Let $X$ be a projective hyper-Kähler manifold of dimension $2n$. Recall the definition of Voisin's filtration~\cite{VoisinCoisotrope} previously mentioned in definition~\ref{DefVoisinsFiltration}. The subset $S_nX$ contains the points whose rational equivalence orbit has a dimension of $n$. 

\begin{proposition}[\cite{VoisinCoisotrope}, Theorem 0.7]
    An irreducible subvariety of dimension $n$, denoted as $L\subset X$, is a constant-cycle subvariety if and only if $L\subset S_nX$.
\end{proposition}

Based on this proposition, the following conjecture is posed in~\cite{VoisinCoisotrope}:
\begin{conjecture}[\cite{VoisinCoisotrope}]\label{conjVoisinCCSubvariety}
Let $L$ and $L'$ be two $n$-dimensional \emph{constant cycle} subvarieties of a hyper-Kähler manifold $X$. If the cohomological classes $[L]=[L']$ in $H^{2n}(X,\mathbb Q)$, then $L$ is rationally equivalent to $L'$ as algebraic cycles in $X$.
\end{conjecture}

It is established~\cite[Theorem 0.7]{VoisinCoisotrope} that an $n$-dimensional constant cycle subvariety $L$ of $X$ is Lagrangian. Question~\ref{ProbBackgroundInAJ} is the question whether, in Conjecture~\ref{conjVoisinCCSubvariety}, the condition "constant cycle subvarieties" can be substituted with "Lagrangian subvarieties".

Lagrangian families serve as an essential source of Lagrangian subvarieties sharing the same cohomological class. In view of Question~\ref{ProbBackgroundInAJ}, we are motived to study the following
\begin{problem}\label{Problem1}
Consider a Lagrangian family on a hyper-Kähler manifold $X$ of dimension $2n$ given by a diagram as in  (\ref{LagrangianFamily}). What can be said of the map 
\begin{equation}\label{LagrangianChowMap}
    \begin{array}{cccc}
        \psi_{\mathcal L}: & B & \rightarrow & CH^n(X)  \\
        & b & \mapsto & q_*(L_b)
    \end{array}?
\end{equation}
\end{problem}

The condition constant cycle is a strong condition for Lagrangian subvarieties. Notice that in contrast with Lagrangian subvarieties, constant cycle Lagrangian subvarieties cannot deform into families. \begin{lemma}\label{LmmRigid}
Small deformations of constant cycle Lagrangian subvarieties of $X$ are no longer constant cycle subvarieties.
\end{lemma}
\begin{proof}
Following the notations in~\cite{VoisinCoisotrope}, let 
\[S_nX:=\{x\in X: \textrm{ the rational equivalence orbit of } x \textrm{ has dimension} \geq n\}.
\]
As is shown in~\cite[Theorem 1.3]{VoisinCoisotrope}, $S_nX$ is a countable union of irreducible varieties of dimension $\leq n$ and constant cycle Lagrangian subvarieties of $X$ are exactly irreducible components of $S_nX$ of dimension $n$. Hence, constant cycle Lagrangian subvarieties of $X$ are rigid.
\end{proof}

As is described in Section~\ref{SectionBackgroundInAJ}, a weaker invariant of algebraic cycles in a projective manifold is the Abel-Jacobi invariant. Problem~\ref{Problem1} thus motivates the following question. 
\begin{problem}\label{Problem2}
Consider a Lagrangian family of a hyper-Kähler manifold $X$ of dimension $2n$ given by a diagram as in  (\ref{LagrangianFamily}). Let $0\in B$ be a point. Under which conditions is the Abel-Jacobi map
\begin{equation}\label{AbelJacobiMap}
    \begin{array}{cccc}
       \Psi_{\mathcal L}^{AJ}: & B  & \rightarrow & J^{2n-1}(X) \\
       & b & \mapsto & \Phi_X^n(q_*(L_b-L_0)) 
    \end{array}
\end{equation}
trivial?
\end{problem}

In many instances, Lagrangian families provide affirmative answers to Problem~\ref{Problem1}, as highlighted in all examples of Lagrangian families in Example~\ref{ExampleLagrangianFamilies}. However, as detailed in Section~\ref{Examples}, an explicit example has been constructed that offers a negative solution to Problem~\ref{Problem2}. This example is derived from a Lagrangian family of the generalized Kummer varieties. The advantage of considering a Lagrangian family over a mere pair of Lagrangian subvarieties is the utilization of the differential theory of families developed by Griffiths~\cite{GriffithsAJ} to explore the triviality of the Abel-Jacobi map of such a family. Consequently, we have developed a cohomological criterion in Section~\ref{SectionCriterion} to determine if a given Lagrangian family possesses a trivial Abel-Jacobi map. Additionally, an explicit Lagrangian family of generalized Kummer varieties is constructed and shown to exhibit a nontrivial Abel-Jacobi map using the criterion we have developed. This implies that any two Lagrangian subvarieties lacking the same Abel-Jacobi invariant in this family provide a negative response to Question~\ref{ProbBackgroundInAJ}.

It is important to note that this example does not present counter-examples to Voisin's Conjecture~\ref{conjVoisinCCSubvariety}. In fact, it adds depth and intricacy to Voisin's Conjecture, making it even more engaging and nuanced.

\subsection{Organization of the chapter}

In this chapter, we first give a criterion for the vanishing of the Abel-Jacobi map (\ref{AbelJacobiMap}) for Lagrangian families of a hyper-Kähler manifold (see also Proposition~\ref{Criterion}).
\begin{proposition}\label{CriterionIntroduction}
Consider a Lagrangian family on a hyper-Kähler manifold $X$ of dimension $2n$ as in (\ref{LagrangianFamily}), satisfying the following condition  :

\begin{quote}
   $\clubsuit$ For general $b\in B$, the contraction by $q^*\sigma_X$ gives an isomorphism $\lrcorner q^*\sigma_X: T_{B,b}\stackrel{\cong}{\to} H^0(L_b,\Omega_{L_b}).$ 
\end{quote}
 Then the Abel-Jacobi map (\ref{AbelJacobiMap}) is trivial if and only if for general $b\in B$, the restriction map 
\[j_b^*: H^{2n-1}(X,\mathbb Q)\to H^{2n-1}(L_b,\mathbb Q)
\]
is zero.
\end{proposition}

The condition $\clubsuit$ is natural. According to~\cite[Proposition 2.4]{VoisinLag}, the deformations of a \emph{smooth} Lagrangian subvariety are non-obstructed, and a local deformation is still Lagrangian. Therefore, if we take $(B,b)$ to be a germ of the Hilbert scheme of deformations of a \emph{smooth} Lagrangian subvariety $L\subset X$, and $\mathcal L\to B$ the corresponding family, then condition $\clubsuit$ holds since $T_{B,b}\cong H^0(L_b, N_{L_b/X})$ by unobstructedness and $\lrcorner\sigma_X: H^0(L_b,N_{L_b/X})\to H^0(L_b,\Omega_{L_b})$ is an isomorphism for a smooth Lagrangian variety.  

Using this criterion, we give a response to Problem~\ref{Problem2}.
\begin{theorem}\label{MainTheorem}
 Consider a Lagrangian family on a hyper-Kähler manifold $X$ of dimension $2n$ given by a diagram as in (\ref{LagrangianFamily}), satisfying condition $\clubsuit$. Assume that the variation of Hodge structures on the degree $1$ cohomology of the fibers of $p:\mathcal L\to B$ is maximal, i.e., the period map
\begin{equation}\label{PeriodMap}
    \begin{array}{cccc}
         \mathcal P: & B &\to & Gr(h^{1,0}(L), H^1(L,\mathbb C))\\
          & b & \mapsto & H^{1,0}(L_b)\subset H^1(L_b,\mathbb C)\cong H^1(L,
          \mathbb C),
    \end{array}
\end{equation}
where $L$ is a general fiber of $p: \mathcal L\to B$,
is generically a local immersion. Then the Abel-Jacobi map (\ref{AbelJacobiMap}) is trivial.
\end{theorem}

This response to Problem~\ref{Problem2} is conditional. However, it can be shown that the condition ``maximal variation of Hodge structures'' cannot be dropped. In fact, we construct in Section~\ref{Examples} Lagrangian families satisfying $\clubsuit$ for which the Abel-Jacobi map is shown to be nontrivial using Proposition~\ref{CriterionIntroduction}. The variation of weight $1$ Hodge structures of the constructed Lagrangian families is not maximal.

In Section~\ref{SectionMaxVar}, we shall explore under which conditions the variation of weight $1$ Hodge structures is maximal. Let $H^2(X,\mathbb Q)_{tr}$ be the orthogonal complement of $NS(X)_{\mathbb Q}$ in $H^2(X,\mathbb Q)$ with respect to the Beauville-Bogomolov-Fujiki form $q$ of $X$ (see~\cite{Beauville}) and let $b_2(X)_{tr}$ be the dimension of $H^2(X,\mathbb Q)_{tr}$ .We prove the following result (see also Proposition~\ref{PropMaxVar}):
\begin{proposition}\label{PropMaxVarIntroduction}
 Consider a Lagrangian family on a hyper-Kähler manifold $X$ of dimension $2n$ given by a diagram as in (\ref{LagrangianFamily}), satisfying condition $\clubsuit$. Assume that the Mumford-Tate group of the Hodge structure $H^2(X,\mathbb Q)$ is maximal, i.e.\ it is the special orthogonal group of $(H^2(X,\mathbb Q)_{tr},q)$, and assume that $b_2(X)_{tr}\geq 5$.  If $h^{1,0}(L_b)$ is smaller than $2^{\lfloor\frac{b_2(X)_{tr}-3}{2}\rfloor}$, then the variation of weight $1$ Hodge structures of $p$ is maximal. 
\end{proposition}

\begin{corollary}
Under the same assumptions as in Proposition~\ref{PropMaxVarIntroduction}, the Abel-Jacobi map (\ref{AbelJacobiMap}) is trivial.
\end{corollary}

Let $p: \mathcal L\to B$, $q: \mathcal L\to X$ be a Lagrangian family. Up to shrinking $B$, we may assume that the map $p:\mathcal L\to B$ is smooth. Let
\[\pi:\mathcal A:=Alb(\mathcal L/B)\to B
\] 
be the relative Albanese variety of $p: \mathcal L\to B$. In the proof of Theorem~\ref{MainTheorem} and Proposition~\ref{PropMaxVarIntroduction}, we use a similar construction to those in~\cite{LazaSaccaVoisin, VoisinTriangle} to get a holomorphic $2$-form $\sigma_{\mathcal A}$ on $\mathcal A$. If we assume the condition $\clubsuit$,  $\pi:\mathcal A\to B$ is a Lagrangian fibration with respect to $\sigma_{\mathcal A}$(see Section~\ref{SectionLagrangianFibrations}). It is interesting to notice that, by this construction, under condition $\clubsuit$, we can translate the problem concerning Lagrangian families to a problem concerning Lagrangian fibrations. However, the total space of the Lagrangian fibration is no longer a hyper-Kähler manifold, but a completely integrable system over an open subset of the base, as studied for instance in~\cite[Chapter 7]{DonagiMarkman}.

\section{A criterion}\label{SectionCriterion}
In this section, we establish a criterion for the vanishing of the Abel-Jacobi map (\ref{AbelJacobiMap}). Let $X$ be a hyper-Kähler manifold of dimension $2n$. With the notation $j_b: L_b\to X$ as in the introduction, we prove

\begin{proposition}\label{Criterion}
 Consider a Lagrangian family of hyper-Kähler manifold $X$ of dimension $2n$ given by a diagram as in (\ref{LagrangianFamily}).  \\
(a) If for general $b\in B$, the restriction map
\[j_b^*: H^{2n-1}(X,\mathbb Q)\to H^{2n-1}(L_b,\mathbb Q)
\]
is zero, then the Abel-Jacobi map (\ref{AbelJacobiMap}) is trivial.\\
(b) If condition $\clubsuit$ holds (see Proposition~\ref{CriterionIntroduction}), then the converse of (a) holds.
\end{proposition}

\begin{remark}
The cohomology group $H^{2n-1}(L_b,\mathbb Q)$ has a Hodge structure of weight $2n-1$ and level $1$. By Hodge symmetry and using the fact $j_b^*: H^{2n-1}(X,\mathbb Q)\to H^{2n-1}(L_b,\mathbb Q)$ is a morphism of Hodge structures, we conclude that $j_b^*: H^{2n-1}(X,\mathbb Q)\to H^{2n-1}(L_b,\mathbb Q)$ is zero, if and only if $j_b^*: H^{n-1,n}(X)\to H^{n-1,n}(L_b)$ is zero.
\end{remark}

\begin{proof}
Let $d\Psi_{\mathcal L, b}^{AJ}: T_{B,b}\to H^n(X,\Omega_X^{n+1})$ denote the differential of the Abel-Jacobi map $\Psi_{\mathcal L}^{AJ}$ at point $b\in B$ (Theorem~\ref{theoremGriffiths}). Let $j_{b*}: H^0(L_b,\Omega_{L_b})\to H^n(X,\Omega_X^{n+1})$ be the Gysin map, which is the Serre dual to the following composition
\begin{equation}\label{SerreDualOfTheGysinMap}
    j_b^*: H^n(X,\Omega_X^{n-1})\to H^n(L_b,\Omega_{\mathcal L|L_b}^{n-1})\to H^n(L_b,\Omega_{L_b}^{n-1}).
\end{equation}
By the above remark, the proposition follows from the following lemma and the fact that 
\[\cup\sigma_X: H^n(X,\Omega_X^{n-1})\to H^n(X,\Omega_X^{n+1})\]
is an isomorphism since $\wedge\sigma_X: \Omega_X^{n-1}\to \Omega_X^{n+1}$ is a vector bundle isomorphism.
\end{proof}
\begin{lemma}
The following diagram is commutative:
\begin{equation}\label{CommDiagLmm}
    \begin{tikzcd}
    T_{B,b}\arrow{r}{d\Psi_{\mathcal L, b}^{AJ}}\arrow{d}{\lrcorner q^*\sigma_X}& H^n(X,\Omega_X^{n-1})\arrow{d}{\cup \sigma_X}\\
    H^0(L_b,\Omega_{L_b})\arrow{r}{j_{b*}}&    H^n(X,\Omega_X^{n+1}).
    \end{tikzcd}
\end{equation}
\end{lemma}
\begin{proof}
We are going to show that the Serre dual of the diagram (\ref{CommDiagLmm}) is commutative.

Let $L^\bullet\Omega^i_{\mathcal L|L_b}$ be the Leray filtration~\cite[Chapter 16]{Voisin} induced on the vector bundle $\Omega_{\mathcal L|L_b}^i$ by the exact sequence
\[0\to p^*\Omega_{B, b}\to \Omega_{\mathcal L|L_b}\to \Omega_{L_b}\to 0,
\]
and defined by $L^j\Omega^i_{\mathcal L|L_b}=p^*\Omega_{B,b}^j\wedge\Omega_{\mathcal L|L_b}^{i-j}$. Since $L_b$ is supposed to be Lagrangian, $q^*\sigma_X\in H^0(L_b, L^1\Omega_{\mathcal L|L_b}^2)$ and thus the cup product 
\[\cup q^*\sigma_X: \Omega_{\mathcal L|L_b}^{\bullet}\to \Omega_{\mathcal L|L_b}^{\bullet+2}
\]
sends $L^k\Omega_{\mathcal L|L_b}^{\bullet}$ to $L^{k+1}\Omega_{\mathcal L|L_b}^{\bullet+2}$. Denoting $\overline{q^*\sigma_X}$ the image of $q^*\sigma_X$ in $H^0(L_b, Gr_L^1\Omega_{\mathcal L|L_b}^2)\cong H^0(L_b, \Omega_{L_b})\otimes \Omega_{B,b}$, this implies the existence of the following commutative diagram
\begin{equation}\label{CommDiagLerayFilt}
    \begin{tikzcd}
    L^1\Omega_{\mathcal L|L_b}^{n+1}=\Omega_{\mathcal L|L_b}^{n+1}\arrow{r}{}& K_{L_b}\otimes p^*\Omega_{B,b}=Gr_L^1\Omega_{\mathcal L|L_b}^{n+1}\\
    L^0\Omega_{\mathcal L|L_b}^{n-1}=\Omega_{\mathcal L|L_b}^{n-1}\arrow{r}{}\arrow{u}{\cup q^*\sigma_X}&\Omega_{L_b}^{n-1}\arrow{u}{\cup\overline{q^*\sigma_X}}=Gr_L^0\Omega_{\mathcal L|L_b}^{n-1},
    \end{tikzcd}
\end{equation}
where $K_{L_b}$ is the canonical bundle of ${L_b}$.
Taking the $n$-th cohomology of (\ref{CommDiagLerayFilt}) and combine it with $q^*: H^n(X,\Omega_X^\bullet)\to H^n(L_b,\Omega_{\mathcal L|L_b}^\bullet)$, we get the following commutative ladder
\begin{equation}\label{CommLadder}
    \begin{tikzcd}
    H^n(X,\Omega_{X}^{n+1})\arrow{r}{q^*} & H^n(L_b,\Omega_{\mathcal L|L_b}^{n+1})\arrow{r}{}& H^n(L_b,K_{L_b}\otimes p^*\Omega_{B,b})\cong \Omega_{B,b}\\
    H^n(X,\Omega_X^{n-1})\arrow{r}{q^*}\arrow{u}{\cup \sigma_X}& H^n(L_b,\Omega_{\mathcal L|L_b}^{n-1})\arrow{r}{}\arrow{u}{\cup q^*\sigma_X} & H^n(L_b, \Omega_{L_b}^{n-1})\arrow{u}{\cup \overline{q^*\sigma_X}}.
    \end{tikzcd}
\end{equation}

\begin{lemma}\label{LmmDualOfTheDiffOfTheAJMap}
The composite in the first row of the diagram (\ref{CommLadder}) coincides with the dual of $d\Psi_{\mathcal L,b}$.
\end{lemma}
\begin{proof}
    Let $\bar p: \bar{\mathcal L}\to\bar B$, $\bar q: \bar{\mathcal L}\to X$ be a relative completion of $p: \mathcal L\to B$ with respect to the morphism $q: \mathcal L\to X$. More precisely, $\bar B$ is a smooth projective variety, $\bar p$ is a flat morphism extending $p$, and $\bar q$ is a morphism extending $q$. The extended family has a Abel--Jacobi map $\Psi^{AJ}_{\bar{\mathcal L}}: \bar B\to J^{2n-1}(X)$ that induces a morphism of complex tori $\psi: \mathrm{Alb}(\bar B)\to J^{2n-1}(X)$ whose differential is given by  the morphism of Hodge structures (\cite[Théorème 12.17]{Voisin})
    \[
    \bar q_*\bar p^*: H^{2d-1}(\bar B,\mathbb Z)\to H^{2n-1}(X,\mathbb Z),
    \]
    where $d=\dim B$. It is well-known that the differential of the Albanese map $\mathrm{alb}: \bar B\to \mathrm{Alb}(\bar B)$ at a point $b\in B$ is given by the dual of the evaluation map $ev_b: H^0(\bar B, \Omega_{\bar B})\to \Omega_{
\bar B,b}$. Hence, the dual of $d\Psi_{\mathcal{L}, b}^{AJ}$ is given by the correspondance $\bar p_*\bar q^*: H^n(X,\Omega_X^{n+1})\to H^0(\bar B, \Omega_{\bar B})$ composed with the evaluation map $ev_b: H^0(\bar B,\Omega_{\bar B})\to \Omega_{\bar B,b}\cong \Omega_{B,b}$. The domain $\bar B$ can be restricted to $B$ before the evaluation map $ev_b: H^0(\bar B,\Omega_{\bar B})\to\Omega_{B,b}$. Therefore, the dual of $d\Psi_{\mathcal{L}, b}^{AJ}$ is given by the correspondance $p_*q^*: H^n(X,\Omega_X^{n+1})\to H^0(B, \Omega_{B})$ composed with the evaluation map $ev_b: H^0(B,\Omega_{B})\to \Omega_{B,b}$. The Gysin map $p_*: H^n(\mathcal L, \Omega_{\mathcal L}^{n+1})\to H^0(B,\Omega_B)$ is given by the Leray filtration.
 Taken together, the dual of $d\Psi_{\mathcal L,b}^{AJ}$ is given by 
    \[
    H^n(X,\Omega_X^{n+1})\stackrel{q^*}{\to}  H^n(\mathcal L,\Omega_{\mathcal L}^{n+1})\to H^n(\mathcal L, K_{\mathcal L/B}\otimes p^*\Omega_B)\to H^0(B, R^np_*\Omega_{\mathcal L/B}\otimes \Omega_B)\stackrel{ev_b}{\to} \Omega_{B,b}.
    \]
    Since the restriction to the fiber $L_b$ and taking the Leray filtration are commutative processes, the above composite of maps is equal to the one that takes the restriction to the fiber $L_b$ first and then takes the Leray filtration. The latter one is exactly the first row of the diagram (\ref{CommLadder}).
\end{proof}

As in (\ref{SerreDualOfTheGysinMap}), the composite in the second row is $j_b^*$. Taking into account of Lemma~\ref{LmmDualOfTheDiffOfTheAJMap}, the diagram~\ref{CommLadder} is indeed the Serre dual of the diagram (\ref{CommDiagLmm}). This concludes the proof of Lemma~\ref{CommDiagLmm}.
\end{proof}

\section{Lagrangian fibrations}\label{SectionLagrangianFibrations}
In this section, we associate to a Lagrangian family satisfying condition $\clubsuit$ on $X$ a fibration on another variety, which turns out to be holomorphic symplectic in such a way that the fibration is lagrangian. This is constructed with the help of a construction appeared in~\cite{LazaSaccaVoisin, VoisinTriangle}. We use this Lagrangian fibration to prove Theorem~\ref{MainTheorem}.

\subsection*{General Constructions}
Let (\ref{LagrangianFamily}) be a Lagrangian family of a hyper-Kähler manifold $X$ of dimension $2n$. We fix a relative polarization of $\mathcal L\to B$ given by a hyperplane section of $X$. Let \[
\pi: \mathcal A:=Alb(\mathcal L/B)\to B
\]
be the relative Albanese variety of $p: \mathcal L\to B$. 

\begin{lemma}\label{MumfordConstruction}
There exist an open dense subset $B_0\subset B$ and a finite covering $B_0'\to B_0$ such that, denoting $p_0':\mathcal L_0'\to B_0'$ the base change of $p$ under $B_0'\to B_0\hookrightarrow B$ and $\pi_0': \mathcal A_0'\to B_0'$ the relative Albanese variety of $p_0'$, there is a cycle $Z_0\in CH^n(\mathcal A_0'\times_{B_0'}\mathcal L_0')$ such that \[[Z_0]^*: p_{0*}'\Omega_{\mathcal L_0'/B_0'}\to \pi_{0*}'\Omega_{\mathcal A_0'/B_0'}\]
is an isomorphism.
\end{lemma}
\begin{proof}
Let $B_0\subset B$ be the subset of regular points of $p:\mathcal L\to B$. For $b\in B_0$, let $j: C_b\hookrightarrow L_b$ be a complete intersection curve and $J_{C_b}$ the Jacobian variety of $C_b$. By Lefschetz theorem on hyperplane sections, $j_*: J_{C_b}\to A_b:=Alb(L_b)$ is surjective. By the semi-simplicity of polarized Hodge structures, there exists a $\mathbb Q$-section $s: A_b\to J_{C_b}$ of $j_*$, i.e., there exists $N>0$ such that $j_*\circ s=N\cdot id_{A_b}$. On $J_{C_b}\times C_b$, we have the Poincaré divisor $d_b\in CH^1(J_{C_b}\times C_b)$ such that $[d_b]^*: H^1(C_b,\mathbb Q)\to H^1(J_{C_b}, \mathbb Q)$ is an isomorphism of Hodge structures, exhibiting the inverse of the pull-back of the Jacobi map $jac_b: C_b\to J_{C_b}$. Let us consider 
\[\begin{tikzcd}
A_b\times C_b\arrow{r}{(s,id_{C_b})}\arrow{d}{(id_{A_b},j)} & J_{C_b}\times C_b\\
A_b\times L_b & 
\end{tikzcd}
\]
and define $Z_b:=(id_{A_b},j)_*(s,id_{C_b})^*(d_b)\in CH^n(A_b\times L_b)$. Then $[Z_b]^*: H^1(L_b,\mathbb Q)\to H^1(A_b,\mathbb Q)$ is given by the composition 
\begin{equation}\label{EqComposite}
    H^1(L_b,\mathbb Q)\stackrel{j^*}{\to} H^1(C_b,\mathbb Q)\stackrel{d_b^*}{\to} H^1(J_{C_b},\mathbb Q)\stackrel{s^*}{\to} H^1(A_b,\mathbb Q),
\end{equation} which is an isomorphism by the definition of $s$. In fact, let $(j_*)^*: H^1(A_b,\mathbb Q)\to H^1(J_{C_b},\mathbb Q)$ be the pull-back map of $j_*: J_{C_b}\to A_b=Alb(L_b)$. Precomposing the left-hand-side of $(j_*)^*: H^1(A_b,\mathbb Q)\to H^1(J_{C_b},\mathbb Q)$ with the natural identification $H^1(Alb(L_b),\mathbb Q)\cong H^1(L_b,\mathbb Q)$, and post-composing the right hand side with the pull-back of the Jacobi map $jac_b^*: H^1(J_{C_b},\mathbb Q)\to H^1(C_b,\mathbb Q)$, we would get the pull-back map $j^*:H^1(L_b,\mathbb Q)\to H^1(C_b,\mathbb Q)$. Since $jac_b^*: H^1(J_{C_b},\mathbb Q)\to H^1(C_b,\mathbb Q)$ is the inverse of $d_b^*:H^1(C_b,\mathbb Q)\to H^1(J_{C_b},\mathbb Q)$, the composite $d_b^*\circ j^*: H^1(L_b, \mathbb Q)\to H^1(J_{C_b},\mathbb Q)$ of the first two maps of the composition (\ref{EqComposite}) is exactly the pull-back map of $j_*: J_{C_b}\to A_b=Alb(L_b)$ at the level of cohomology, after the natural identification $H^1(Alb(L_b),\mathbb Q)\cong H^1(L_b,\mathbb Q)$. Since $j_*\circ s=N\cdot id_{A_b}$, we get the desired isomorphism.

The cycles $Z_b$ are defined fiberwise, but standard arguments \cite[Chapter 3]{VoisinCitrouille} show that they can be constructed in family over a smooth generically finite cover $B_0'\to B_0$. Let us spell out the standard arguments. By the theory of Hilbert schemes, there are countably many connected projective $B_0$-schemes $H_1, \ldots, H_i, \ldots$ together with the universal families of cycles $\mathcal Z_1,\ldots, \mathcal Z_i,\ldots $ such that for each $i\in\mathbb N$, every fiber of $\mathcal Z_i\to H_i$ is a cycle $Z_b\in \mathcal Z^1(A_b\times L_b)$ such that $[Z_b]^*: H^1(L_b,\mathbb Q)\to H^1(A_b,\mathbb Q)$ is an isomorphism. By the construction of the previous paragraph, the structure map $\bigcup_{i\in\mathbb N}H_i\to B_0$ is surjective, thus there is $i\in \mathbb N$, such that $H_i\to B_0$ is surjective. Since $H_i\to B_0$ is projective, we may take a multisection 
of the map $B_0'\to B_0$, and the universal family $\mathcal Z_i$ restricted to $B_0$ gives the desired $1$-cycle.
\end{proof}

For the sake of simplicity, we shall note $B_0$, $\mathcal L_0$ and $\mathcal A_0$ instead of $B_0'$, $\mathcal L_0'$ and $\mathcal A_0'$. We define a holomorphic $2$-form $\sigma_{\mathcal A_0}$ on $\mathcal A_0$ by setting
\begin{equation}\label{2FormOnA}
    \sigma_{\mathcal A_0}:=[Z_0]^*q_0^*\sigma_X,
\end{equation}
where $q_0:\mathcal L_0\to X$ is the natural map. 
\begin{proposition}\label{PropertiesOfThe2Form}
(a) The $2$-form $\sigma_{\mathcal A_0}$ is closed.\\
(b) $\sigma_{\mathcal A_0}$ vanishes on fibers of $\pi_0:\mathcal A_0\to B_0$.\\
(c) The composite morphism $\kappa: T_{B_0}\stackrel{\lrcorner q_0^*\sigma_X} {\longrightarrow} p_{0*}\Omega_{\mathcal L_0/B_0}\stackrel{[Z_0]^*}{\to} \pi_{0*}\Omega_{\mathcal A_0/B_0}$ is given by the contraction $\lrcorner\sigma_{\mathcal A_0}$.
\end{proposition}
\begin{proof}
(a) Let $Z_q:=(id, q_0)_*Z_0\in CH(\mathcal A_0\times X)$. Then by the projection formula, $\sigma_{\mathcal A_0}=[Z_q]^*\sigma_X$. Let $\mathcal A'$ be a projective completion of $\mathcal A_0$. Then $Z_q$ extends to a cycle $\bar Z_q$ of $\mathcal A'\times X$. $\sigma_{\mathcal A_0}$ extends to a $2$-form $\sigma_{\mathcal A'}:=[\bar Z_q]^*\sigma_X$ which is automatically closed since $\mathcal A'$ is projective. Thus, $\sigma_{\mathcal A_0}=\sigma_{\mathcal A'|\mathcal A_0}$ is also closed.

(b) Since $Z_0$ is a cycle in $\mathcal A_0\times_{B_0}\mathcal L_0\subset \mathcal A_0\times \mathcal L_0$, the morphism $[Z_0]^*: H^*(\mathcal L_0)\to H^*(\mathcal A_0)$ preserves the Leray filtrations on both sides. Therefore, 
$\sigma_{\mathcal A_0}\in H^0(\mathcal A_0, \pi_0^*\Omega_{B_0}\wedge \Omega_{\mathcal A_0})\subset H^0(\mathcal A_0, \Omega_{\mathcal A_0}^2)$ since $q_0^*\sigma_X\in H^0(\mathcal L_0, p_0^*\Omega_{B_0}\wedge \Omega_{\mathcal L_0})$ by the definition of Lagrangian families. Therefore, $\sigma_{\mathcal A_0}$ vanishes on the fibers of $\pi_0:\mathcal A_0\to B_0$.

(c) By Lemma~\ref{MumfordConstruction}, $[Z_0]^*$ induces an isomorphism $H^0(B_0, \Omega_{B_0}\otimes p_{0*}\Omega_{\mathcal L_0/B_0})\to H^0(B_0, \Omega_{B_0}\otimes \pi_{0*}\Omega_{\mathcal A_0/B_0})$ which sends $\lrcorner q_0^*\sigma_X$ to $\lrcorner\sigma_{\mathcal A_0}$.
\end{proof}

By (b) and (c) of the above proposition, we get the following diagram that is commutative up to a sign:
\begin{equation}\label{MainCommDiag}
    \begin{tikzcd}
    0 \arrow{r}{} & T_{\mathcal A_0/B_0} \arrow{r}{}\arrow{d}{(\pi_0^*\kappa)^*} & T_{\mathcal A_0} \arrow{r}{} \arrow{d}{\lrcorner\sigma_{\mathcal A_0}} & \pi_0^*T_{B_0} \arrow{r}{}\arrow{d}{\pi_0^*\kappa} & 0\\
    0 \arrow{r}{}& \pi_0^*\Omega_{B_0} \arrow{r}{} & \Omega_{\mathcal A_0}\arrow{r}{} &\Omega_{\mathcal A_0/B_0}\arrow{r}{} & 0.
    \end{tikzcd}
\end{equation}
Here, the commutativity of the second square is dual to Proposition~\ref{PropertiesOfThe2Form} (c). Since the dual of $\lrcorner\sigma_{\mathcal A_0}: T_{\mathcal A_0}\to \Omega_{\mathcal A_0}$ is given by $-\lrcorner\sigma_{\mathcal A_0}: T_{\mathcal A_0}\to \Omega_{\mathcal A_0}$, the first square is anti-commutative.

\begin{lemma}
If condition $\clubsuit$ (see Proposition~\ref{CriterionIntroduction}) holds for all $b\in B_0$, then $\sigma_{\mathcal A_0}$ is nowhere degenerate on $\mathcal A_0$.
\end{lemma}
\begin{proof}
If condition $\clubsuit$ holds, then $\pi_0^*\kappa: \pi_0^*T_{B_0}\to \Omega_{\mathcal A_0/B_0}$ is an isomorphism.
By the commutativity of (\ref{MainCommDiag}) and the five lemma, $\lrcorner\sigma_{\mathcal A_0}: T_{\mathcal A_0}\to \Omega_{\mathcal A_0}$ is an isomorphism, which means that $\sigma_{\mathcal A_0}$ is nowhere degenerate.
\end{proof}

\subsection*{Symmetry}
Let (\ref{LagrangianFamily}) be a Lagrangian family of a hyper-Kähler manifold $X$ of dimension $2n$. We fix a relative polarization of $\mathcal L\to B$ given by a hyperplane section of $X$. Let $b\in B$ be a general point. The infinitesimal variation of Hodge structures on degree $1$ cohomology of the fibers of $p: \mathcal L\to B$ at $b$ is given by (see~\cite[Lemme 10.19]{Voisin})
\[\bar\nabla: T_{B,b}\to Hom(H^0(L_b,\Omega_{L_b}), H^1(L_b,\mathcal O_{L_b})).
\]
Precomposed with the map $\lrcorner q^*\sigma_X: T_{B,b}\to H^0(L_b,\Omega_{L_b})$, the map $\bar\nabla$ induces a bilinear map
\begin{equation}
    \begin{array}{cccc}
       S:  & T_{B,b}\times T_{B,b} & \to & H^1(L_b,\mathcal O_{L_b})  \\
         & (u,v) & \mapsto & \bar\nabla_u(v\lrcorner q^*\sigma_X).
    \end{array}
\end{equation}
\begin{proposition}\label{SymmetryProp}
The bilinear map $S$ is symmetric in the sense that $S(u,v)=S(v,u)$ for any $u,v\in T_{b,B}$.
\end{proposition}
\begin{proof}
By Griffiths' transversality~\cite[Chapter 17]{Voisin}, $S(u,v)=\rho(u)\lrcorner(v\lrcorner q^*\sigma_X)$, where $\rho: T_{B,b}\to H^1(L_b, T_{L_b})$ is the Kodaira-Spencer map. Therefore, we need to show that the following diagram is commutative
\begin{equation}\label{CommDiagSymm}
  \begin{tikzcd}
   T_{B,b}\otimes T_{B,b} \arrow{r}{\rho\otimes id}\arrow{d}{\lrcorner q^*\sigma_X\otimes id} &  H^1(L_b, T_{L_b})\otimes T_{B,b} \arrow{d}{\beta}\\
   H^0(L_b,\Omega_{L_b})\otimes T_{B,b}\arrow{r}{\alpha} & H^1(L_b,\mathcal O_{L_b}),
  \end{tikzcd}
\end{equation}
where $\alpha(\omega, v)= \rho(v)\lrcorner \omega$ and $\beta(u,\chi)=u\lrcorner ( \chi\lrcorner q^*\sigma_X)$. 

To see the commutativity of (\ref{CommDiagSymm}), restrict the commutative ladder (\ref{MainCommDiag}) to $L_b$ and apply the cohomology on $L_b$, then the commutativity of (\ref{MainCommDiag}) implies the commutativity of (\ref{CommDiagSymm}). Indeed, (\ref{CommDiagSymm}) is the connecting homomorphism of the cohomology of (\ref{MainCommDiag}) tensored by $T_{B,b}$.
\end{proof}

\begin{remark}
When condition $\clubsuit$ is satisfied, the symmetry of $S$ comes from the completely integrable system structure on $(\mathcal A_0, \sigma_{\mathcal A_0})$. What we proved is in fact the symmetry of 
\begin{equation}
    \begin{array}{cccc}
       S':  & T_{B,b}\times T_{B,b} & \to & H^1(A_b,\mathcal O_{A_b})  \\
         & (u,v) & \mapsto & \bar\nabla_u(v\lrcorner \sigma_{\mathcal A_0}).
    \end{array}
\end{equation}
Fixing a relative polarisation on $\mathcal A_0\to B_0$, we have natural isomorphisms (always under $\clubsuit$): $H^1(A_b,\mathcal O_{A_b})\cong H^0(A_b, \Omega_{A_b})^*\cong T_{B,b}^*$ and we can thus view $S'$ as an element in $T_{B,b}^*\otimes T_{B,b}^*\otimes T_{B,b}^*$. If this relative polarisation is principal, Donagi and Markman proved in~\cite[Lemma 7.2]{DonagiMarkman} that $S'$ lies in $Sym^3T_{B,b}^*$. This result is called ``weak cubic condition'' in~\cite[Lemma 7.2]{DonagiMarkman}. See also~\cite[Theorem 4.4]{VoisinTorsion}
\end{remark}

Now we are ready to prove the main theorem.
\begin{proof}[Proof of Theorem~\ref{MainTheorem}]
Under the assumptions of Theorem~\ref{MainTheorem}, assume by contradiction that the Abel--Jacobi map (\ref{AbelJacobiMap}) is not constant. In what follows, we fix a relative polarization on $p_0:\mathcal L_0\to B_0$ induced from a hyperplane section of $X$, so that $R^{2n-1}p_{0*}\mathbb Q\cong R^1p_{0*}\mathbb Q$. By Proposition~\ref{CriterionIntroduction}, the morphism
\[j^*: H^{2n-1}(X,\mathbb Q)\to R^{2n-1}p_{0*}\mathbb Q\cong R^1p_{0*}\mathbb Q
\]
 of variations of Hodge structures on an open subset $B_0\subset B$ containing $b$ is not zero. Hence, there is a non-zero locally constant sub-variation of Hodge structures $I:=\mathrm{Im}j^*\subset R^1p_{0*}\mathbb Q$. Since $I$ is locally constant, for any $\omega\in I^{1,0}_b$ and $u\in T_{B,b}$, $\nabla_u(\omega)=0$. Recall that $\clubsuit$ means that $\lrcorner q^*\sigma_X: T_{B,b}\to H^0(L_b,\Omega_{L_b})$ is bijective. Let $F:=(\lrcorner q^*\sigma_X)^{-1}(I^{1,0})\subset T_{B,b}$. Then by the symmetry of $S$ given by Proposition~\ref{SymmetryProp}, $F$ lies in the kernel of $\bar\nabla$, which contradicts our assuption that the variation of Hodge structures is maximal.

\end{proof}

\section{Maximal variations}\label{SectionMaxVar}
In this section, we study under what conditions could the variation of Hodge structures of a Lagrangian family be maximal. Consider a Lagrangian family of a hyper-Kähler manifold $X$ of dimension $2n$ satisfying the condition $\clubsuit$ given by the diagram as in (\ref{LagrangianFamily}). Let $U\subset B$ be a simply connected open subset of $B_0\subset B$ and let
\begin{equation}\label{ApplicationModules}
    \begin{array}{cccc}
         \mathcal P: & U &\to & Gr(h^{1,0}(L), H^2(L,\mathbb C))\\
          & b & \mapsto & H^{1,0}(L_b)\subset H^1(L_b,\mathbb C)\cong H^1(L,
          \mathbb C),
    \end{array}
\end{equation}
be the local period map of the Lagrangian family. 

In what follows, we are going to use the universal property of the Kuga-Satake construction proved in~\cite[Proposition 6]{VoisinvGeemen}.
\begin{theorem}[\cite{VoisinvGeemen}]\label{universalKS}
Let $(H^2,q)$ be a polarized Hodge structure of hyper-Kähler type of dimension $\geq 5$. Assume that the Mumford–Tate group of the Hodge structure on $H^2$ is maximal, namely the special orthogonal group of $(H^2,q)$. Let $H$ be a simple effective weight-$1$ Hodge structure, such that there exists an injective morphism of Hodge structures of bidegree $(-1,-1)$
\[H^2\hookrightarrow Hom(H, A)
\]
for some weight-$1$ Hodge structure $A$. Then $H$ is a direct summand of the Kuga–Satake Hodge
structure $H^1_{KS}(H^2,q)$. In particular, $\dim H\geq 2^{\lfloor\frac{\dim H^2-1}{2}\rfloor}$.
\end{theorem}

With the same notations as in the introduction, we prove

\begin{proposition}\label{PropMaxVar}
Assume that the Mumford-Tate group of the Hodge structure $H^2(X,\mathbb Q)$ is maximal, i.e. it is the special orthogonal group of $(H^2(X,\mathbb Q)_{tr},q)$ and assume $b_2(X)_{tr}\geq 5$. If the dimension of $H^{0,1}(L_b)$ is smaller than $2^{\lfloor\frac{b_2(X)_{tr}-3}{2}\rfloor}$ for a general fiber $L_b$ of $p: \mathcal L\to B$, then the variation of weight $1$ Hodge structure of $p$ is maximal.
\end{proposition}

\begin{proof}
We use the same argument as in~\cite{VoisinvGeemen} where similar results were proved for Lagrangian fibrations. Assuming that the period map (\ref{ApplicationModules}) is not generically an immersion, we are going to prove that $\dim H^{0,1}(L_b)\geq 2^{\lfloor\frac{b_2(X)_{tr}-3}{2}\rfloor}$. By assumption, the nonempty general fibers of $\mathcal P$ are of dimension $\geq 1$. Let $b\in U$ be a general point and let $B_b$ the fiber of $\mathcal P$ passing through $b$. Let $U_b=B_b\cap U$. Then the fibers of $\pi|_{U_b}: \mathcal A_{U_b}\to U_b$ are isomorphic with each other.  Thus, up to a base change by a finite covering of $U_b$, we may assume $\pi|_{U_b}: \mathcal A_{U_b}\to U_b$ is trivial, i.e., $\mathcal A_{U_b}=U_b\times A_b$. Let $\pi_{F_b}:\mathcal A_{F_b}\to F_b$ be a smooth completion of $\pi|_{U_b}$, then $\mathcal A_{F_b}$ is birational to $F_b\times A_b$, which gives a morphism $H^2(\mathcal A_{F_b})\to H^2(F_b\times A_
b)$. Recall by Lemma~\ref{MumfordConstruction}, we get a morphism $[Z]^*: H^2(X)\to H^2(\mathcal A)$ that sends $\sigma_X$ to a holomorphic $2$-form which is non-degenerate on $\mathcal A_U$. Finally, the rational map $\mathcal A_{F_b}\dashrightarrow \mathcal A$ induces $H^2(\mathcal A)\to H^2(\mathcal A_{F_b})$. Compositing all these maps, we get a morphism
\begin{equation}
\alpha: H^2(X)_{tr}\hookrightarrow H^2(X)\to H^2(\mathcal A)\to H^2(\mathcal A_{F_b})\to H^2(F_b\times A_b)\to H^1(F_b)\otimes H^1(A_b),
\end{equation}
where the last map is given by the projection in the Künneth decomposition. 
\begin{lemma}
$\alpha: H^2(X)_{tr}\to H^1(F_b)\otimes H^1(A_b)$ is injective.
\end{lemma}
\begin{proof}
Since $h^{2,0}(X)=1$ and $H^{2,0}(X)$ is orthogonal to $NS(X)$ with respect to the Beauville-Bogomolov-Fujiki form, $H^2(X)_{tr}$ is a simple Hodge structure. Therefore, to show the injectivity of $\alpha$ it suffices to show that $\alpha$ is not zero. We claim that $\alpha(\sigma_X)\neq 0$. Indeed, Since $A_b$ is Lagrangian with respect to $\sigma_{\mathcal A}$ (Proposition~\ref{PropertiesOfThe2Form} (b)), in the Künneth's decomposition of $H^2(A_b\times F_b)$, the image of $\sigma_X$ in $H^0(F_b)\otimes H^2(A_b)$ is zero. If furthermore $\alpha(\sigma_X)=0$ in $H^1(F_b)\otimes H^1(A_b)$, then the image of $\sigma_X$ on $F_b\times A_b$ comes from a $2$-form on $F_b$, which has rank $\leq \dim F_b$. Therefore, the rank of $\sigma_{\mathcal A}$ has rank $\leq \dim F_b$ on $\mathcal A_{U_b}$. On the other hand, the codimension of $\mathcal A_{U_b}$ in $\mathcal A_U$ is $\dim B-\dim F_b$, and thus the non-degeneration of $\sigma_{\mathcal A_U}$ implies that $\sigma_{\mathcal A}$ has rank $\geq 2\dim F_b$ on $\mathcal A_{U_b}$. This is a contradiction since we are assuming $\dim F_b\geq 1$.
\end{proof}
We are now in the position to use the universal property of the Kuga-Satake construction (see Theorem~\ref{universalKS} above). Since $\alpha: H^2(X)_{tr}\to H^1(F_b)\otimes H^1(A_b)$ is nonzero, there is at least one simple direct factor $A$ of $A_b$ such that $H^2(X)_{tr}\to H^1(F_b)\otimes H^1(A)$ is nonzero thus injective. Taking $H^2$ as $H^2(X)_{tr}$, we conclude by Theorem~\ref{universalKS} that 
\[\dim H^{0,1}(L_b)=\dim A_b\geq \dim A\geq \frac12\times 2^{\lfloor \frac{\dim H^2(X)_{tr}-1}{2}\rfloor}=2^{\lfloor\frac{b_2(X)_{tr}-3}{2}\rfloor},\]
as desired.
\end{proof}

\section{Example of a Lagrangian family with nontrivial Abel-Jacobi map}\label{Examples}
Recall the construction of generalized Kummer varieties introduced in~\cite{Beauville}. Let $A$ be an abelian surface and $A^{[n+1]}$ the Hilbert scheme of length $n+1$ subschemes of $A$. Let $\mathrm{alb}: A^{[n+1]}\to A$ be the composition of the Hilbert-Chow morphism and the summation map
\[A^{[n+1]}\to A^{(n+1)}\to A.\]
Note that $\mathrm{alb}$ is an isotrivial fibration.
The generalized Kummer variety $K_n(A)$ is defined to be the fiber of $\mathrm{alb}$ over $0\in A$. As is shown in~\cite{Beauville}, $K_n(A)$ is a hyper-Kähler manifold of dimension $2n$.

In this section, we are going to construct Lagrangian families of $X:=K_n(A)$ for $n\geq 2$, satisfying condition $\clubsuit$ and whose Abel-Jacobi map is \emph{not} trivial. 

For any $x\in A$, one defines a subvariety $Z_x$ of $K_n(A)$ consisting of Artinian subschemes of $A$ of length $n+1$ supported on $x$ and $-nx$, with multiplicities $n$ and $1$, respectively. By \cite[Proposition VI.1.1]{Briancon}, $Z_x$ is a \emph{rational} variety of dimension $n-1$ if $x$ is \emph{not} an $(n+1)$-torsion point. Let $Z=\bigcup_{x\in A} Z_x$ and let $\pi: Z\to A$ send elements in $Z_x$ to $x$. For any curve $C\subset A$, define $Z_C=\bigcup_{x\in C} Z_x$. 

Now let $B$ be a connected open subset of the Hilbert scheme of deformations of a smooth curve $C\subset A$ and $\mathcal C\to B$ the corresponding family.
\begin{lemma}
$\{Z_C\}_{C\in B}$ is a Lagrangian family of $K_n(A)$ satifying condition $\clubsuit$.
\end{lemma}
\begin{proof}
Since for general $C$, $Z_C$ is a fibration over a curve $C$ whose general fibers are rational, any holomorphic $2$-form on $Z_C$ is $0$. Furthermore, $\dim Z_C=n=\dim K_n(A)/2$. These imply that $\{Z_C\}_{C\in B}$ is a Lagrangian family. We now show that this family satisfies condition $\clubsuit$. Denoting $\mathcal L$ the total space of the family $\{Z_C\}_{C\in B}$ and $L$ a general fiber, and using as before the following notation
\[
\begin{tikzcd}
    \mathcal L\arrow{r}{q}\arrow{d}{p} & X\\
    B &
    \end{tikzcd},
    \]
    we need to show that $\lrcorner q^*\sigma_{K_n(A)}: H^0(L,N_{L/Z})=H^0(L, N_{L/\mathcal L})\to H^0(L,\Omega_L)$ is an isomorphism. Since the general fibers of $\pi$ are rational, $q^*\sigma_{K_n(A)}=\pi^*\sigma_A$, where $\sigma_A$ is the unique (up to coefficients) holomorphic $2$-form on $A$. Therefore, we can conclude by the commutativity of the following diagram
\[\begin{tikzcd}
 H^0(C,N_{C/A})\arrow{r}{\lrcorner\sigma_A}\arrow{d}{\pi^*} & H^0(C,\Omega_C)\arrow{d}{\pi^*}\\
 H^0(L,N_{L/ Z})\arrow{r}{\lrcorner\pi^*\sigma_A}& H^0(L,\Omega_L)
\end{tikzcd}
\]
noting that the two vertical arrows are isomorphims since the fibers of $\pi$ are rational, and that $\lrcorner \sigma_A: H^0(C,N_{C/A})\to H^0(C,\Omega_C)$ is an isomorphism since $\sigma_A$ is nondegenerate.
\end{proof}

\begin{proposition}
The Abel-Jacobi map of the Lagrangian family $\{Z_C\}_{C\in B}$ is \emph{not} trivial.
\end{proposition}
\begin{proof}
Let $i: C\hookrightarrow A$ be a general curve in the family $\mathcal C\to B$. By Proposition~\ref{Criterion}(a), it suffices to show that the restriction map $H^{2n-1}(X,\mathbb Q)\to H^{2n-1}(Z_C,\mathbb Q)$ is nonzero. 

define an injective morphism
\[\begin{array}{cccc}
     \beta:& A  \times  A & \hookrightarrow & A^{(n+1)}\\
     & (x, y) & \mapsto & n\{x\}+\{y\},
\end{array}
\]
where we use the notation $\{x\}\in \mathcal Z_0(A)$ the $0$-cycle of the point $x\in A$. Consider the following pull-back diagram definitionning a subvariety $Z'\subset A^{[n+1]} $
\[\begin{tikzcd}
Z'\arrow[r, hookrightarrow, "\alpha"]\arrow{d}{\pi'} &  A^{[n+1]} \arrow{d}{c}\\
A\times A \arrow[r, hookrightarrow, "\beta"] & A^{(n+1)}
\end{tikzcd},
\]
where $c:A^{[n+1]}\to A^{(n+1)}$ is the Hilbert-Chow morphism.
Then $Z = Z'\cap K_n(A)\subset A^{[n+1]}$. We have the following commutative diagram where all three squares are pull-back diagrams
\[
\begin{tikzcd}
Z\arrow[rr, hookrightarrow]\arrow[dd, "\pi"]\arrow[rd, hookrightarrow] && X= K_n(A) \arrow[rd, hookrightarrow] \\
& Z'\arrow[rr, hookrightarrow, "\alpha"]\arrow[dd,"\pi'"] &&  A^{[n+1]} \arrow[dd, "c"]\arrow[rdd, "\mathrm{alb}"]\\
A \arrow[rd, hookrightarrow, "f"]  \\
&A\times A \arrow[rr, hookrightarrow, "\beta"] && A^{(n+1)}\arrow[r, "\sum"] & A
\end{tikzcd},
\]
Here $f: A\to A\times A$ defined by $x\mapsto (x, -nx)$ is the fiber over $0\in A$ of the trivial fibration $\sum\circ\beta: A\times A\to A$. 

By~\cite[Corollary 5.1.5]{deCataldoMigliorini}, $[Z']^*: H^{2n-1}(A^{[n+1]}, \mathbb Q)\to H^1(A\times A, \mathbb Q)$ is surjective. Furthermore, the restriction map $f^*: H^1(A\times A, \mathbb Q)\to H^1(A, \mathbb Q)$ is surjective since $f$ is the fiber of a trivial fibration. These imply that $[Z]^*: H^{2n-1}(X, \mathbb Q)\to H^1(A, \mathbb Q)$ is surjective. Finally, since the restriction map $i^*: H^1(A, \mathbb Q)\to H^1(C, \mathbb Q)$ is injective by Lefschetz hyperplane theorem, the composition map $i^*\circ [Z]^*: H^{2n-1}(X, \mathbb Q)\to H^1(C, \mathbb Q)$ is nonzero. This implies that the restriction map $H^{2n-1}(X, \mathbb Q)\to H^{2n-1}(Z_C, \mathbb Q)$ is nonzero, as desired.
\end{proof}

\chapter{On Some Birational Invariants of Hyper-Kähler Manifolds}\label{ChapterBirational}
We study in this chapter three birational invariants of projective varieties, the degree of irrationality, the fibering gonality and the fibering genus. We first improve the lower bound in a recent result of Voisin bounding from below the fibering genus of a Mumford-Tate very general projective hyper-Kähler manifold by a constant depending on its dimension and the second Betti number. We also compare the relations between these birational invariants for projective $K3$ surfaces of Picard number $1$ and study the asymptotic behaviors of their degree of irrationality and fibering gonality. 

This chapter presents the main result of~\cite{Bai24}. The structure is as follows. In Section~\ref{SectionBackgroundInBirational}, we give a panorama of known results in this area and a brief introduction to the main results of this chapter. Section~\ref{FibgenOfHK}, Section~\ref{K3} and Section~\ref{Inequalities} are the content of the article~\cite{Bai24}. In Section~\ref{FibgenOfHK}, we prove Theorem~\ref{TheoremFibGen}. In Section~\ref{K3}, we prove Theorem~\ref{theoremFibgonDeK3}. In Section~\ref{Inequalities}, we prove Propositions~\ref{InequalityOfLatticesHK} and \ref{InequalityOfLatticesK3}. In a final Section~\ref{SectionFurtherResultsInBirational}, we present some related results and questions.

\section{Background and Introduction}\label{SectionBackgroundInBirational}
\subsection{Some measures of irrationality}

 Let \(X\) be a projective variety over \(\mathbb{C}\) of dimension \(n\). The following invariant, known as the \emph{degree of irrationality}, was initially defined and studied by Hisao Yoshihara and others before being further explored in~\cite{Irrationality}.

\begin{definition}
The \emph{degree of irrationality} of \(X\), denoted as \(\operatorname{Irr}(X)\), is the minimal degree of dominant rational maps \(\varphi : X \dashrightarrow \mathbb{P}^n\). Formally,
\[
\operatorname{Irr}(X) := \min\{\deg \varphi : \text{there is a dominant rational map } \varphi : X \dashrightarrow \mathbb{P}^n\}.
\]
\end{definition}

In their important work, Bastianelli, De Poi, Ein, Lazarsfeld, and Ullery studied \(\operatorname{Irr}(X)\) for very general hypersurfaces \(X\) of large degree in~\cite{Irrationality}. Note that the following Theorem C from~\cite{Irrationality}, regarding the degree of irrationality for hypersurfaces, has also appeared in~\cite{GonalityTheorem} and was proven earlier in the case of surfacees in the thesis by Renza Cortini.

\begin{theorem}[=Theorem C in~\cite{Irrationality}]
Let \(X \subset \mathbb{P}^{n+1}\) be a very general smooth hypersurface of dimension \(n\) and degree \(d \geq 2n + 1\). Then \(\operatorname{Irr}(X) = d - 1\). Furthermore, if \(d \geq 2n + 2\), then any rational mapping \(f : X \dashrightarrow \mathbb{P}^n\) with \(\deg(f) = d - 1\) is given by the projection from a point of \(X\).
\end{theorem}

In an attempt to determine the degree of irrationality and to better understand the geometry of the variety, the following two measures of irrationality are proposed and studied by Voisin in~\cite{FibgenFibgon}.
\begin{definition}[\cite{FibgenFibgon}]
    (i) The fibering gonality of $X$,
    $\mathrm{Fibgon}(X)$, is the minimal number $c$ such that there exists a rational dominant map $\pi: X\dashrightarrow B$ such that $\dim B=\dim X-1$ and the general fiber is a connected curve $C$ of gonality $c$.\\
    (ii) The fibering genus of $X$,
    $\mathrm{Fibgen}(X)$, is the minimal number $c$ such that there exists a rational dominant map $\pi: X\dashrightarrow B$ such that $\dim B=\dim X-1$ and the general fiber is a connected curve $C$ of geometric genus $c$.
\end{definition}

Following~\cite{FibgenFibgon}, the following inequalities hold between these two measures of irrationality and the degree of irrationality.
\begin{lemma}
For any smooth projective variety $X$, we have\\
(i) $\mathrm{Fibgon}(X)\leq \mathrm{Irr}(X)$.\\
(ii) $\mathrm{Fibgon}(X)\leq \frac12(\mathrm{Fibgen}(X)-1)+2$.
\end{lemma}

\subsection{Measures of irrationality of hyper-Kähler manifolds}
Let $X$ be a projective hyper-Kähler manifold of dimension $2n$. The following result is directly implied by the main result of~\cite{Pirola}.

\begin{theorem}[Alzati-Pirola~\cite{Pirola}]\label{theoremAlzatiPirola}
     The degree of irrationality of $X$ is greater than $n$.
\end{theorem}
A direct proof of Theorem~\ref{theoremAlzatiPirola} goes as follows.
\begin{proof}
Let $X\dashrightarrow Y$ be a dominant generically finite rational map of degree $d$ to a smooth variety $Y$ such that $H^{l,0}(Y)=0$ for any $l>0$. We are going to show that $d> n$. First we desingularize the rational map as $\phi: \tilde X\to Y$ and let $\sigma_{\tilde X}$ be the pull-back of the holomorphic $2$-form of $X$ on $\tilde X$. 

Consider the fibered product $Z'=\tilde X\times_Y\tilde X\times_Y\ldots\times_Y\tilde X\subset \tilde X\times \ldots\times \tilde X$ with $d$ copies of $\tilde X$. Let $\Delta_{ij}$ be the pull back of the diagonal under the natural projection map $pr_{ij}: Z'\to \tilde X\times \tilde X$. Let $Z$ be the Zariski closure of $Z'-\bigcup_{i,j}\Delta_{ij}$ in $\tilde X\times \ldots\times \tilde X$. Since $\phi:\tilde X\to Y$ is of degree $d$, the natural projection map $pr_i: Z\to \tilde X$ is dominant and generically finite for any $i$. Hence, we may choose an irreducible component of $Z$ such that the projection maps are dominant and generically finite. Denote $X'$ a desingularisation of such an irreducible component of $Z$. Let $p_i: X'\to \tilde X$ be the natural projection maps and let $f: X'\to Y$ be the structural map of the fibered product. Let $I\in CH(X'\times \tilde X)$ be the incidence cycle such that $I_*: CH_0(X')\to CH_0(\tilde X)$ is given by $x\mapsto \sum_{i=1}^dp_i(x)$ for general $x\in X'$. Then, $I_*=\phi^*f_*: CH_0(X')\to CH_0(\tilde X)$. By Mumford-type theorems~\cite[1.1.3]{VoisinCitrouille}, $I^*=f^*\phi_*: H^0(\tilde X,\Omega_{\tilde X}^l)\to H^0(X', \Omega_{X'}^l)$ for $l>0$. Since $H^{l,0}(Y)=0$ for $l>0$, and since $I^*$ factorizes through $H^{l, 0}(Y)$, we have $I^*=0: H^0(\tilde X,\Omega_{\tilde X}^l)\to H^0(X', \Omega_{X'}^l)$. Hence, for any $l>0$, we have
\[\sum_{i=1}^d p_i^*\sigma_{\tilde X}^l=0
\]
By the theory of symmetric polynomials, any symmetric polynomial without constant terms in $p_1^*\sigma_{\tilde X},\ldots, p_d^*\sigma_{\tilde X}$ is zero. Hence, 
\[\prod_{i=1}^d(T-p_i^*\sigma_{\tilde X})=T^d.
\]
Taking $T=p_1^*\sigma_{\tilde X}$, one gets $p_1^*\sigma_{\tilde X}^d=0$. This implies in particular that the rank of $\sigma_X$ is less than $d$. Hence, $d>n$.
\end{proof}

For the fibering genus of hyper-Kähler manifolds, Voisin proves the following result in~\cite{FibgenFibgon}.
\begin{theorem}[Voisin~\cite{FibgenFibgon}]
Let $X$ be a Mumford-Tate very general hyper-Kähler manifold of dimension $2n$ with $n\geq 3$ and $b_{2,\mathrm{tr}}(X)\geq 5$. Then 
\[
\mathrm{Fibgen}(X)\geq \min\{n+2, 2^{\lfloor \frac{b_{2, \mathrm{tr}}(X)-3}{2}\rfloor}\}.
\]   
\end{theorem}
The reader is referred to Section~\ref{IntroHKManifolds} and Section~\ref{IntroMTGroup} for details about Mumford-Tate groups. In this chapter, we will present a sharper bound.

\begin{theorem}\label{TheoremFibGen}
    Let $X$ be a Mumford-Tate very general hyper-Kähler manifold of dimension $2n$ and assume $b_{2,\mathrm{tr}}(X)\geq 5$. Then 
    \[\mathrm{Fibgen}(X)\geq \min\left\{ n+\left\lceil\frac{-1+\sqrt{8n-7}}{2}\right\rceil, 2^{\lfloor\frac{ b_{2,\mathrm{tr}}(X)-3}{2}\rfloor}
\right\}.
\]
\end{theorem}

\subsection{Measures of irrationality of K3 surfaces}
We first recollect some known results on the assymptotic behaviors of the measures of irrationality of K3 surfaces.

\subsubsection*{Degree of irrationality}
In~\cite{Irrationality}, the following conjecture on the degree of irrationality of K3 surfaces is proposed.
\begin{conjecture}[\cite{Irrationality}]\label{ConjectureBDELDBackground}
Let $\{(S_d, L_d)\}_{d\in \mathbb N}$ be very general polarized $K3$ surfaces such that $L_d^2=2d-2$. Then
\[\limsup_{d\to \infty}\mathrm{Irr}(S_d)=+\infty.
\]
\end{conjecture}
In the litterature, the $d$ appearing in Conjecture~\ref{ConjectureBDELDBackground} is called the \emph{genus} of the K3 surface. To justify the name, for any smooth curve $C\subset S$ in the linear sytem $|L_d|$, it follows from tne adjunction formula that the arithemetic genus of $C$ is $d$. 

\begin{remark}
    We can calculate the degree of irrationality of $S_d$ when the genus $d$ is small.
    \begin{itemize}
        \item When $d = 2$, the very general K3 surface $S_2$ is a double cover of $\mathbb P^2$ ramified along a sextic curve. Since K3 surfaces are not rational, the degree of irrationality of $S_2$ is $2$.
        \item When $d = 3$, the very general K3 surface $S_3$ is a quartic surface in $\mathbb P^3$ and is not birational to a double cover of $\mathbb P^2$ since $S_3$ does not admit a rational involution. Projecting $S_3\subset \mathbb P^3$ via a general point on $S_3$ gives a rational map of degree $3$ of $S_3$ to $\mathbb P^2$. Hence, the degree of irrationality of $S_3$ is $3$.
    \end{itemize}
    Despite the above examples, it is in general quite hard to determine the degree of irrationality of K3 surfaces.
\end{remark}

One general result we know in the spirit of this conjecture is the following theorem proved in Stapleton's thesis~\cite{ThesisOfStapleton}.
\begin{theorem}[Stapleton~\cite{ThesisOfStapleton}]
    There exists a constant $C > 0$ such that $\mathrm{Irr}(S_d)\leq C\sqrt d$.
\end{theorem}

\subsubsection*{Fibering genus}
The fibering genus of surfaces has already been studied by Ein and Lazarsfeld in~\cite{EinLazarsfeld}, under the name \emph{Konno invariant}. They give a good estimate of the asymptotic behavior of fibering genus of K3 surfaces as the genus tends to infinity.
\begin{theorem}[Ein--Lazarsfeld~\cite{EinLazarsfeld}]\label{EinLazarsfeld}
Let $S_d$ be a polarized K3 surface of Picard rank $1$ and genus $d$. Then 
\[
 \mathrm{Fibgen}(S_d)=\Theta(\sqrt{d}).
\]
More precisely, 
\[ \sqrt{\frac{d}{2}}\leq \mathrm{Fibgen}(S_d)\leq 2\sqrt{2d}.
\]
\end{theorem}

\begin{corollary}
    We have 
    \[
    \lim_{d\to\infty} \mathrm{Fibgen}(S_d) = +\infty.
    \]
\end{corollary}

\subsubsection*{Fibering gonality}
We do not know much about the fibering gonality of K3 surfaces. In this chapter, we will prove the following relation between the degree of irrationality and the fibering gonality of K3 surfaces.

\begin{theorem}\label{TendsToInfty}
    Let $\{S_d\}_{d\in\mathbb N}$ be projective K3 surfaces such that the Picard group of $S_d$ is generated by a line bundle with self intersection number $2d-2$. Then 
\[\limsup_{d\to\infty} \mathrm{Irr}(S_d)=+\infty \iff \limsup_{d\to \infty} \mathrm{Fibgon}(S_d)=+\infty.\]
\end{theorem}

The proof of Theorem~\ref{TendsToInfty} relies on Theorem~\ref{EinLazarsfeld} and the following theorem.

\begin{theorem}\label{theoremFibgonDeK3}
Let $S$ be a projective K3 surface whose Picard number is $1$. Then one of the following two cases holds:\\
(a) $\mathrm{Irr}(S)=\mathrm{Fibgon}(S)$;\\
(b) $\mathrm{Fibgen}(S)^2\leq \mathrm{Fibgon}(S)^{21}$.
\end{theorem}

\begin{proof}[Proof of Theorem~\ref{TendsToInfty} assuming Theorem~\ref{theoremFibgonDeK3}]
The implication $\Leftarrow$ is clear since $\mathrm{Fibgon}(S_d)\leq \mathrm{Irr}(S_d)$ for any $d$. Now let us prove the implication $\Rightarrow$. Theorem~\ref{EinLazarsfeld} (\cite[Theorem B]{EinLazarsfeld}) shows that $\lim_{d}\mathrm{Fibgen}(S_d)=+\infty$. If $\limsup_{d} \mathrm{Fibgon}(S_d)\neq+\infty$, then there exist constants $C$ and $D$ such that for every $d>D$ we have $\mathrm{Fibgon}(S_d)^{21/2}<C<\mathrm{Fibgen}(S_d)$. Hence for $d>D$, we must have $\mathrm{Irr}(S_d)=\mathrm{Fibgon}(S_d)<C^{2/21}$ by Theorem~\ref{theoremFibgonDeK3}, which contradicts the assumption that $\limsup_{d} \mathrm{Irr}(S_d)=+\infty$.
\end{proof}

Conjecture~\ref{ConjectureBDELDBackground} predicts that $\limsup_{d} \mathrm{Irr}(S_d)=+\infty$, so it should be expected that $\limsup_{d} \mathrm{Fibgon}(S_d)=+\infty$ as well.

\subsection{Some lattice theoretic results about hyper-Kähler manifolds}
In the course of the proof of Theorem~\ref{theoremFibgonDeK3}, we obtain the following inequalities about the discriminant of Picard lattices of projective hyper-Kähler manifolds. Although we only use this proposition in the case of $K3$ surfaces and we only use one side of the inequalities, the following general form is of independent interest. Here $\mathrm{disc}(\mathrm{Pic}(X))$ denotes the discriminant of the Picard lattice of $X$ with respect to the Beauville-Bogolov-Fujiki form and similarly for $\mathrm{disc}(\mathrm{Pic}(X'))$.
\begin{proposition}\label{InequalityOfLatticesHK}
Let $X$ and $X'$ be deformation equivalent projective hyper-Kähler manifolds of dimension $2n$. Let $\phi\colon X\dashrightarrow X'$ be a dominant rational map. Then
\[(\deg\phi)^{(\frac{1}{n}-2)b_{2,\mathrm{tr}}(X)}\leq \left|\frac{\mathrm{disc}(\mathrm{Pic}(X))}{\mathrm{disc}(\mathrm{Pic}(X'))}\right|\leq (\deg\phi)^{\frac{b_{2,\mathrm{tr}}(X)}{n}}.
\]
\end{proposition}
If  in the previous proposition $X$ and $X'$ are $K3$ surfaces, we can get a slightly better lower bound.
\begin{proposition}\label{InequalityOfLatticesK3}
Let $S$ and $S'$ be projective $K3$ surfaces and let $\lambda(S)=\min\{\rho(S), b_{2,\mathrm{tr}}(S)\}$ where $\rho(S)$ is the Picard number of $S$. Let $\phi\colon S\dashrightarrow S'$ be a dominant rational map. Then
\[(\deg\phi)^{-\lambda(S)}\leq \left|\frac{\mathrm{disc}(\mathrm{Pic}(S))}{\mathrm{disc}(\mathrm{Pic}(S'))}\right|\leq (\deg\phi)^{b_{2,\mathrm{tr}}(S)}.
\]
\end{proposition}

\section{Fibering genus of very general hyper-Kähler manifolds}\label{FibgenOfHK}
 Let $X$ be a projective hyper-Kähler manifold of dimension $2n$. Let $f\colon X\dashrightarrow B$ be a fibration into curves and let $\tau\colon \tilde X\to X$ and $\tilde f\colon \tilde X\to B$ be a resolution of indeterminacy points. Let $\tilde X_b$ be a smooth fiber of $\tilde f$ over a general point $b\in B$. The exact sequence of vector bundles on $\tilde X_b$
 \[
 0\to N_{\tilde X_b/\tilde X}^*\to \Omega_{\tilde X|\tilde X_b}\to \Omega_{\tilde X_b}\to 0
 \] 
 induces an exact sequence
 \[
 0\to N_{\tilde X_b/\tilde X}^*\wedge \Omega_{\tilde X|\tilde X_b}\to \Omega_{\tilde X|\tilde X_b}^2\to \Omega_{\tilde X_b}^2\to 0.
 \]
 Since $\tilde X_b$ is a curve, we have $\Omega_{\tilde X_b}^2=0$. Therefore, $\Omega_{\tilde X|\tilde X_b}^2\cong N_{\tilde X_b/\tilde X}^*\wedge \Omega_{\tilde X|\tilde X_b}$ as vector bundles on $\tilde X_b$. Let us view $N_{\tilde X_b/\tilde X}^*\wedge \Omega_{\tilde X|\tilde X_b}$ as a subbundle of $N_{\tilde X_b/\tilde X}^*\otimes \Omega_{\tilde X|\tilde X_b}$. Then $(\tau^*(\sigma_X))|_{\tilde X_b}\in H^0(\tilde X_b, \Omega_{\tilde X|\tilde X_b}^2)$ can be viewed as an element in $H^0(\tilde X_b, N_{\tilde X_b/\tilde X}^*\wedge \Omega_{\tilde X|\tilde X_b})=\mathrm{Hom}_{\tilde X_b}(N_{\tilde X_b/\tilde X}, \Omega_{\tilde X|\tilde X_b})$. Therefore, $(\tau^*(\sigma_X))|_{\tilde X_b}$ induces a morphism $\phi_b\colon H^0(\tilde X_b, N_{\tilde X_b/\tilde X})\to H^0(\tilde X_b,\Omega_{\tilde X_b})$. Therefore, we have a morphism
\[\sigma_b\colon T_{B,b}\to H^0(\tilde X_b,
\Omega_{\tilde X_b})=H^0(\tilde X_b, K_{\tilde X_b})\] defined as the composition of the natural morphisms $T_{B, b}\to H^0(\tilde X_b, N_{\tilde X_b/\tilde X})$ and $\phi_b\colon H^0(\tilde X_b, N_{\tilde X_b/\tilde X})\to H^0(\tilde X_b,\Omega_{\tilde X_b})$. Here, $H^0(\tilde X_b, K_{\tilde X_b})\cong \mathbb C^g$ where $g=g(\tilde X_b)$ is the genus of the curve $\tilde X_b$. Let $\rho\colon T_{B,b}\to H^1(\tilde X_b,T_{\tilde X_b})$ be the Kodaira-Spencer map. In~\cite{FibgenFibgon}, Voisin proves the following

\begin{proposition}[Voisin~\cite{FibgenFibgon}]\label{LemmeDeVoisin}
 (i) If $n\geq 2$, then $\tilde X_b$ is \emph{not} hyperelliptic.\\
 (ii) The rank of $\sigma_b$ is $\geq n$.\\
 (iii) Let $I_b\subset H^0(\tilde X_b, K_{\tilde X_b}^{\otimes 2})$ be image of $\mathrm{Im}\sigma_b\otimes H^0(\tilde X_b,K_{\tilde X_b})$ under the multiplication map
 \[\mu\colon H^0(\tilde X_b,K_{\tilde X_b})\otimes H^0(\tilde X_b,K_{\tilde X_b})\to H^0(\tilde X_b,K_{\tilde X_b}^{\otimes 2}).\] Then $\rho(\ker\sigma_b)\subset H^1({\tilde X_b}, T_{\tilde X_b})$ is orthogonal to $I_b$ via the Serre pairing for $T_{\tilde X_b}$:
 \[ H^1({\tilde X_b}, T_{\tilde X_b})\otimes H^0({\tilde X_b},K_{\tilde X_b}^{\otimes 2})\to  H^1({\tilde X_b}, K_{\tilde X_b})\cong H^0(\tilde X_b, \mathcal O_{\tilde X_b})\cong \mathbb C.\]
  (iv) Assume that that $b_{2,\mathrm{tr}}(X)\geq 5$ and that the Mumford-Tate group of $H^2(X,\mathbb Q)_{tr}$ is maximal. If $\rho\colon T_{B,b}\to H^1({\tilde X_b},T_{\tilde X_b})$ is \emph{not} injective, then $g\geq 2^{\lfloor\frac{ b_{2,\mathrm{tr}}(X)-3}{2}\rfloor}$.
 \end{proposition}
 The proofs of (i), (ii) and (iii) are explicitly written in~\cite{FibgenFibgon}. Although (iv) is not stated in~\cite{FibgenFibgon} with this generality, it is essentially proved there (see the proof of Lemma 2.13 and Lemma 2.15 in \textit{loc. cit.}).

\subsection{Proof of Theorem~\ref{TheoremFibGen}}
For $n=1$,  the inequality is obvious, since  $K3$ surfaces cannot be fibered into rational curves. From now on, let us assume $n\geq 2$ to be able to assume $\tilde X_b$ is not hyperelliptic (see Proposition~\ref{LemmeDeVoisin} (i)). Let $k$ be the corank of $\sigma_b$, i.e., $\mathrm{rank}\sigma_b=g-k$. We now prove 
\begin{lemma}\label{LinearAlgebraLemma}
With notation as above, we have the following inequality $$\dim\ker\rho \geq 2n-1-(g-k)-\frac{k(k+1)}{2}.$$
\end{lemma}
\begin{proof}
We use the notation as in Proposition~\ref{LemmeDeVoisin} (iii). By Serre duality, the orthogonality result in Proposition~\ref{LemmeDeVoisin} (iii) implies that $\dim \rho(\ker\sigma_b)+\dim I_b\leq \dim  H^0({\tilde X_b},K_{\tilde X_b}^{\otimes 2} )$, that is,
\begin{equation}\label{Inequality1inFibgen}
    \mathrm{codim}(I_b\subset H^0({\tilde X_b},K_{\tilde X_b}^{\otimes 2} ))\leq \dim \rho(\ker\sigma_b).
\end{equation}
The multiplication map $\mu\colon H^0(\tilde X_b,K_{\tilde X_b})\otimes H^0(\tilde X_b,K_{\tilde X_b})\to H^0(\tilde X_b,K_{\tilde X_b}^{\otimes 2})$ factors through the symmetric product $\mathrm{Sym}^2H^0(\tilde X_b,K_{\tilde X_b})$. Let \mbox{$\mathrm{Im}\,\sigma_b\cdot H^0(\tilde X_b,K_{\tilde X_b})$} denote the image of $\mathrm{Im}\,\sigma_b\otimes H^0(\tilde X_b,K_{\tilde X_b})$ under the canonical symmetrization map $\mathrm{pr}\colon H^0(\tilde X_b,K_{\tilde X_b})\otimes H^0(\tilde X_b,K_{\tilde X_b})\to \mathrm{Sym}^2H^0(\tilde X_b,K_{\tilde X_b})$. Recall that $k$ is the corank of $\sigma_b$, that is, $\mathrm{codim}(\mathrm{Im}\,\sigma_b\subset H^0(\tilde X_b, K_{\tilde X_b}))=k$. Then we have 
\begin{equation}\label{Equality1inFibgen}
\mathrm{codim}(\mathrm{Im}\,\sigma_b\cdot H^0(\tilde X_b,K_{\tilde X_b}))\subset \mathrm{Sym}^2H^0(\tilde X_b,K_{\tilde X_b}))
=\frac{k(k+1)}{2}.
\end{equation}
On the other hand, by Max Noether theorem (see~\cite[Chapter III, §2]{Arbarello} or~\cite{Voisin}), the multiplication map $\mu'\colon \mathrm{Sym}^2H^0(\tilde X_b,K_{\tilde X_b})\to H^0(\tilde X_b,K_{\tilde X_b}^{\otimes 2})$ is surjective, since $\tilde X_b$ is not hyperelliptic. Taking into account the fact that $I_b$ is the image of $\mathrm{Im}\,\sigma_b\cdot H^0(\tilde X_b,K_{\tilde X_b})$ under the map $\mu'$, we have the following inequality
\begin{equation}\label{Inequality2inFibgen}
\mathrm{codim}(I_b\subset H^0(\tilde X_b,K_{\tilde X_b}^{\otimes 2}) ) \leq \mathrm{codim}(\mathrm{Im}\,\sigma_b\cdot H^0(\tilde X_b,K_{\tilde X_b})\subset \mathrm{Sym}^2H^0(\tilde X_b,K_{\tilde X_b})).
\end{equation}
Combining (\ref{Inequality1inFibgen}), (\ref{Equality1inFibgen}) and (\ref{Inequality2inFibgen}), we get $\dim\rho(\ker\sigma_b)\leq \frac{k(k+1)}{2}$, from which we deduce that $$\dim\ker\rho\geq \dim\ker\sigma_b-\frac{k(k+1)}{2}=2n-1-(g-k)-\frac{k(k+1)}{2}.$$
\end{proof}  

\begin{proof}[Proof of Theorem~\ref{TheoremFibGen}] Assuming  that $g(\tilde X_b)< 2^{\lfloor\frac{ b_{2,\mathrm{tr}}(X)-3}{2}\rfloor}$, we have to prove $g\geq n+ \lceil\frac{-1+\sqrt{8n-7}}{2}\rceil$. By Proposition~\ref{LemmeDeVoisin} (ii), (iv) and Lemma~\ref{LinearAlgebraLemma}, we have the following constraints on $g$ and $k$:
\[\left\{\begin{array}{lr}
     k\geq 0\\
     & \cr 
     g-k- n\geq 0  \\
      & \cr 
     2n-1-(g-k) -\frac{k(k+1)}{2} \leq 0.
\end{array}\right.
\]
In order to find the minimal possible value of $g$ under these constraints, we make the following discussion according to the values of $k$. 
\begin{itemize}
    \item When $0\leq k\leq \frac{-1+\sqrt{8n-7}}{2}$, we have $k-\frac{k(k+1)}{2}+2n-1\geq n+k$. Hence, the minimal possible value of $g$ in this domain is the minimum of $k-\frac{k(k+1)}{2}+2n-1$ with $0\leq k\leq  \frac{-1+\sqrt{8n-7}}{2}$, which is $n+ \frac{-1+\sqrt{8n-7}}{2}$.
    \item When $k\geq  \frac{-1+\sqrt{8n-7}}{2}$, we have $k-\frac{k(k+1)}{2}+2n-1\leq n+k$. Hence, the minimal possible value of $g$ in this domain is the minimum of $n+k$ with  $k\geq  \frac{-1+\sqrt{8n-7}}{2}$, which is $n+ \frac{-1+\sqrt{8n-7}}{2}$.
\end{itemize} 

Since $g$ and $k$ are integers, we find $g\geq n+ \lceil\frac{-1+\sqrt{8n-7}}{2}\rceil$, as desired.

\end{proof}

\begin{remark}
Our proof relies on the inequality (\ref{Inequality2inFibgen}) which only uses the surjectivity of the multiplication map $\mu'$. With more information on the geometry of the canonical embedding, and in particular, on the gonality of the fibers, we could get a better estimate in Theorem~\ref{TheoremFibGen}.

\end{remark}

\section{Relations between birational invariants of $K3$ surfaces}\label{K3}
In this section, we are going to prove Theorem~\ref{theoremFibgonDeK3} that relates the three birational invariants, namely, the degree of irrationality, the fibering gonality and the fibering genus, of projective $K3$ surfaces of Picard number $1$.

\subsection{A factorization}\label{ConstructionofPhi}
Let $S$ be a smooth projective surface and let $f\colon S\dashrightarrow B$ be a fibration into curves over a smooth base $B$. After a resolution of inderminacies of $f$ and replacing $S$ by another birational model, we may assume $f\colon S\to B$ is a morphism. Let $d$ be the gonality of the general fiber of $f$, so the general fiber $C$ admits a degree $d$ morphism from $C$ to $\mathbb P^1$. A standard argument shows that we can spread this morphism into a family up to a generically finite base change.
\begin{lemma}\label{StandardArgumentsGonalityK3}
There is a generically finite morphism $\pi\colon B'\to B$ and a degree $d$ dominant rational map $\psi\colon S\times_BB'\dashrightarrow \mathbb P^1\times B'$ over $B'$.
\end{lemma}
\begin{proof}
Let $B_0$ be the smooth locus of $f\colon S\to B$ and let $f_0\colon S_0\to B_0$ be the restriction of $f$ on the smooth locus. Let $p\colon \mathrm{Pic}^d(S_0/B_0)\to B_0$ be the degree $d$ relative Picard variety of $f_0\colon S_0\to B_0$. By the assumption on the general fiber of $f\colon S\to B$, the restriction of the map $p\colon \mathrm{Pic}^d(S_0/B_0)\to B_0$ to the Brill-Noether locus in $\mathrm{Pic}^d(S_0/B_0)$ of the linear systems of degree $d$ and dimension $1$ is dominant.  Let $B_0'$ be a general reduced irreducible subscheme of $\mathrm{Pic}^d(S_0/B_0)$ that is dominant and generically finite over $B_0$ by $p$. Let us take $B'$ to be a completion of $B_0'$. Then by construction, the universal line bundle restricted to $S_0\times_BB_0'$ gives a dominant rational map $\psi\colon S\times_BB'\dashrightarrow \mathbb P_{B'}^1$ of degree $d$, as desired.
\end{proof}

It is natural to ask if $\psi\colon S\times_BB'\dashrightarrow \mathbb P^1\times B'$ over $B'$ descends to a rational map $\psi_B\colon S\dashrightarrow \mathbb P^1\times B$ over $B$. A moment of thinking will convince us that we are asking too much, because $\pi\colon B'\to B$ is in general not a Galois cover. We make the following construction. Let $n$ be the degree of the morphism $\pi\colon B'\to B$. Consider the $n$-th self fibred product of $B'$ over $B$: $B'\times_B\cdots\times_BB'$. define $B''$ to be the closure in $B'\times_B\cdots\times_BB'$ of the set 
\[\{(x_1, \ldots, x_n)\in B'\times_B\cdots\times_BB'\colon x_1,\ldots, x_n 
\textrm{ are distinct }\}.
\]
Then $\pi'\colon B''\to B$ is of degree $n!$ and the symmetric group $\mathfrak S_n$ permuting the components of $B'\times_B\cdots\times_BB'$ acts on an open dense subset of $B''$. The rational map $\psi\colon S\times_BB'\dashrightarrow \mathbb P^1\times B'$ over $B'$ given in Lemma~\ref{StandardArgumentsGonalityK3} can be extended to a rational map 
\begin{equation}\label{EqPsi}
\psi'\colon S\times_BB''\dashrightarrow (\mathbb P^1)^n\times B''
\end{equation}
over $B''$ in a natural way: let $x\in S$ and let $y=(y_1,\ldots, y_n)\in B''$ be general points, we define $\psi'(x,y)=(\psi(x,y_1),\ldots, \psi(x,y_n), y)$. Moreover, the symmetric group $\mathfrak S_n$ acts canonically on both sides of (\ref{EqPsi}) in the following way. To define the action of $\mathfrak S_n$ on $S\times_B B''$, we let $\mathfrak S_n$ act trivially on $S$ and act as permutations of components of $B''$; and to define the action on $(\mathbb P^1)^n\times B''$, we let $\mathfrak S_n$ act as permutations of components for both $(\mathbb P^1)^n$ and $B''$. It is clear from the construction that $\psi'\colon S\times_BB''\dashrightarrow (\mathbb P^1)^n\times B''$ is $\mathfrak S_n$-equivariant. Thus $\psi'$ induces a rational map $ S\dashrightarrow ((\mathbb P^1)^n\times B'')/\mathfrak S_n$ over $B$. Let $S'$ be the image of this map. Thus we get a dominant rational map
\[\phi\colon S\dashrightarrow S'.
\]
\begin{proposition}\label{PropOfPhi}
$S'$ is a surface and the degree of $\phi$ divides $d$. Furthermore, if $n\geq 2$, the general fiber of $S'\dashrightarrow B$ is of geometric genus $\leq (\frac{d}{\deg\phi}-1)^2$.
\end{proposition}
\begin{proof}
Over the general point $b=(b_1,\ldots, b_n)\in B''$, $\psi'$ is given by the morphism $\psi'_b\colon C\to (\mathbb P^1)^n$ induced by the $n$ morphisms $C\to \mathbb P^1$ of degree $d$ corresponding to the points $b_i\in B'$, where $C$ is the fiber of $f\colon S\dashrightarrow B$ oveer $\pi'(b)\in B$. Let $C'$ be the image of $\psi_b'$. Then the fiber of $S'\dashrightarrow B$ over $\pi'(b)\in B$ is $C'$ by construction. Thus $S'$ is a surface, and the degree of $\phi$ is the degree of $C$ over $C'$, which divides $d$. This proves the first statement. To prove the second, we need to prove the geometric genus of $C'$ is $\leq (\frac{d}{\deg\phi}-1)^2$. Since $C'$ is a curve of multi-degree $(\frac{d}{\deg\phi},\ldots, \frac{d}{\deg\phi})$ in $(\mathbb P^1)^n$, we can use Lemma~\ref{GenusOfCurvesInP1n} below concerning algebraic curves in $(\mathbb P^1)^n$.
\end{proof}
\begin{lemma}\label{GenusOfCurvesInP1n}
Let $n\geq 2$. Let $C$ be an integral curve in $(\mathbb P^1)^n$ of multi-degree $(d,\ldots, d)$. Then the geometric genus of $C$ is less than or equal to $(d-1)^2$.
\end{lemma}
\begin{proof}
We prove it by induction on $n$. When $n=2$, it is the adjunction formula. Now assume that any curve $C'\subset (\mathbb P^1)^{n-1}$ of multi-degree $(e,\ldots, e)$ has geometric genus $\leq (e-1)^2$. Consider the projection $C\to C''\subset (\mathbb P^1)^{n-1}$ to the first $n-1$ components. The degree of $C\to C''$ is of the form $d/e$, for some $e$. Hence,  $C''\subset (\mathbb P^1)^{n-1}$ is a curve of multi-degree $(e,\ldots, e)$, hence it has geometric genus $\leq (e-1)^2$ by induction assumption. Let $\tilde C$ and $\tilde C''$ be the normalization of $C$ and $C''$ respectively. Then $\tilde C$ is birational to its image $C'''$ in $\tilde C''\times \mathbb P^1$, where the map to the second component is given by the composition map $\tilde C\to C\stackrel{i_n}{\to}\mathbb P^1$. Here, the map $i_n\colon C\to \mathbb P^1$ is the projection map to the $n$-th component. Note that $C'''$  is of multi-degree $(d/e, d)$ in $\tilde C''\times \mathbb P^1$ and note that $NS(C''\times \mathbb P^1)=NS(C'')\oplus NS(\mathbb P^1)$. Adjunction formula and $g(C'')\leq (e-1)^2$ give us $p_a(C''')=\frac{d(g(C'')+d-1)}{e}-d+1\leq d(e+d/e-2)-d+1\leq (d-1)^2$ since $1\leq e\leq d$. Hence, $p_g(C)=g(\tilde C)=p_g(C''')\leq p_a(C''')\leq (d-1)^2$, as desired.
\end{proof}

\subsection{Rational maps between $K3$ surfaces}\label{RatMapsBetweenK3}
We treat rational maps between $K3$ surfaces in this part. Let $S$ (resp. $S'$) be a projective $K3$ surface whose Picard group is generated by an ample line bundle of degree $2D-2$ (resp. $2D'-2$). Let $\phi\colon S\dashrightarrow S'$ be a dominant rational map. 
\begin{proposition}\label{BoundsForDegreeOfK3}
We have the following inequality \[\frac{1}{(\deg\phi)^{21}}\leq \frac{D-1}{D'-1}\leq (\deg\phi)^{21}.\]
\end{proposition}
\begin{proof}
Let $\tau\colon \tilde S\to S$ and $\tilde\phi\colon \tilde S\to S'$ be a resolution of indeterminacy points of $\phi\colon S\dashrightarrow S'.$ Let $T$ (resp. $T'$) be the lattice $H^2(S,\mathbb Z)_{tr}$ (resp. $H^2(S',\mathbb Z)_{tr}$) endowed with the intersection form. For a positive integer $e$, define $T'(e)$ to be the lattice $T'$ with the quadratic form multiplied by $e$. For example, with this notation, the sublattice $eT'$ of $T'$ is isometric to $T'(e^2)$ as lattices. The image $E$ of the morphism $\tilde\phi^*\colon H^2(S',\mathbb Z)_{tr}\to H^2(\tilde S,\mathbb Z)_{tr}\cong T$, viewed as a sublattice of $T$, is isometric to $T'(\deg\phi)$. The isomorphism $H^2(\tilde S,\mathbb Z)_{tr}\cong T$ is because  $\tau^*\colon H^2(S,\mathbb Z)_{tr}\to H^2(\tilde S,\mathbb Z)_{tr}$ is an isomorphism. We thus get the following equalities
\begin{equation}\label{EquationOfDisc}
    [T\colon E]^2=\left|\frac{\mathrm{disc}(E)}{\mathrm{disc}(T)}\right|=(\deg\phi)^{21}\left|\frac{\mathrm{disc}(T')}{\mathrm{disc}(T)}\right|=(\deg\phi)^{21}\frac{D'-1}{D-1}.
\end{equation}
Here, we have used the fact that the lattice $H^2(S,\mathbb Z)$ is unimodular since $S$ is a surface, and the fact that the orthogonal complement of a primitive sublattice in a unimodular lattice and the sublattice itself have the same discriminant, up to sign (~\cite[Proposition 1.6.1]{Lattice}).
On the other hand, we have
\begin{lemma}\label{philowerstarisinjective}
    The morphism of abelian groups $\tilde\phi_*\colon H^2(\tilde S,\mathbb Z)_{tr}\to H^2(S',\mathbb Z)_{tr}$ is injective and sends $E$ onto $(\deg\phi)T'$.
\end{lemma}
\begin{proof}
The projection formula shows that $\tilde\phi_*\tilde\phi^*=\deg\phi\cdot Id$. Hence, $\tilde\phi_*$ sends $E$ onto $(\deg\phi)T'$. By Lemma~\ref{b2trTheSame} and the fact that $\tilde\phi_*$ is surjective with $\mathbb Q$-coefficients, the kernel of $\tilde\phi_*$ is of torsion. But $H^2(\tilde S,\mathbb Z)_{tr}$ is torsion-free, as $\tilde S$ is simply connected. We conclude that the kernel of $\tilde\phi_*$ is zero, as desired. 
\end{proof}

Now Lemma~\ref{philowerstarisinjective} implies that  \begin{equation}\label{InequalityOfTE}
    1\leq [T\colon E]\leq [T'\colon(\deg\phi)T']=(\deg\phi)^{21}.
\end{equation} 
Proposition~\ref{BoundsForDegreeOfK3} follows by combining (\ref{EquationOfDisc}) and (\ref{InequalityOfTE}).
\end{proof}
\begin{remark}
One can similarly prove the more general result on hyper-Kähler manifolds, namely Proposition~\ref{InequalityOfLatticesHK}. The detailed proof is given in Section~\ref{Inequalities}. A sharper lower bound will also be given there.
\end{remark}

\subsection{Proof of Theorem~\ref{theoremFibgonDeK3}; }\label{ProofCaseA}
Let $f\colon S\dashrightarrow B=\mathbb P^1$ be a fibration into curves realizing the fibering gonality of $S$. After a resolution of indeterminacies of $f$, we get a dominant morphism $\tilde f\colon \tilde S\to B$ whose general fiber is of gonality $d=\mathrm{Fibgon}(S)$. In Section~\ref{ConstructionofPhi}, we constructed a surface $S'$ that is a fibration over $B$ into curves and a dominant rational map $\phi\colon \tilde S\dashrightarrow S'$ over $B$ of degree dividing $d$. The Kodaira dimension $\kappa(S')$, the irregularity $q(S')$ and the geometric genus $p_g(S')$ of $S'$ cannot exceed those of $S$ since $S'$ is dominated by $S$. By Enriques-Kodaira classification of algebraic surfaces~\cite{Enriques, EnriquesKodaira}, $S'$ can only be birational to $\mathbb P^2$, an Enriques surface or a K3 surface. 

If  $S'$ is a rational surface, then $\mathrm{Irr}(S)\leq 
\deg \phi \leq d=\mathrm{Fibgon}(S)$. Here, the inequality $\deg\phi\leq d$ is because of Proposition~\ref{PropOfPhi}. But clearly $\mathrm{Fibgon}(S)\leq \mathrm{Irr}(S)$. We get the equality. This is case (a) of the theorem.

If $S'$ is birational to an Enriques surface $S''$. After a birational modification of $\tilde S$, there is a dominant morphism $g\colon\tilde S\to S''$. Since $\tilde S$ is simply-connected, $g$ factors through the universal covering $S'''$ of $S''$. $S'''$ is the K3 cover of the Enriques surface $S''$. The Picard number of $S'''$ is at least $10$ since the Picard number of the Enriques surface $S''$ is $10$. \begin{lemma}\label{b2trTheSame}
Let $\phi\colon X\dashrightarrow X'$ be a dominant rational map between projective hyper-Kähler manifolds of the same dimension. Then $b_{2,\mathrm{tr}}(X)=b_{2,\mathrm{tr}}(X')$.
\end{lemma}
\begin{proof}
Let $\tau\colon \tilde X\to X$ and $\tilde\phi\colon\tilde X\to X'$ be a resolution of indeterminacy points of $\phi$. Then $\tau^*\colon H^2(X,\mathbb Q)_{tr}\to H^2(\tilde X,\mathbb Q)_{tr}$ is an isomorphism and $\tilde\phi^*\colon H^2( X',\mathbb Q)_{tr}\to H^2(\tilde X,\mathbb Q)_{tr}$ is injective. They are moreover both morphisms of Hodge structures. But $H^2(X,\mathbb Q)_{tr}$ and $H^2(X',\mathbb Q)_{tr}$ are simple Hodge structures. This implies that $\tilde\phi^*\colon H^2( X',\mathbb Q)_{tr}\to H^2(\tilde X,\mathbb Q)_{tr}$ is an isomorphism and hence the result.
\end{proof}
\noindent Lemma~\ref{b2trTheSame} shows that the Picard number of $S'''$ can only be $1$ since it is dominated by $S$. This gives us a contradiction. The case where $S'$ is birational to an Enriques surface is thus excluded.

 In the rest of this section, we discuss the case when $S'$ is birational to a $K3$ surface. By changing the birational model, we may assume $S'$ \emph{is} a K3 surface. By Lemma~\ref{b2trTheSame}, the Picard number of $S'$ is also $1$. The following proposition shows that in our situation the Case (b) of Theorem~\ref{theoremFibgonDeK3} holds, which concludes the proof of Theorem~\ref{theoremFibgonDeK3}.
\begin{proposition}\label{InequalityAsb}
The following inequality holds: $$\mathrm{Fibgen}(S)\leq \mathrm{Fibgon}(S)^{21/2}.$$
\end{proposition}
\begin{proof}
Let $D$ and $D'$ be the degrees of the $K3$ surfaces $S$ and $S'$, respectively. Let $C'$ be the general fiber of $S'\dashrightarrow B$ as in Section~\ref{ConstructionofPhi}. Then we have the following inequalities
\begin{align*}
    \left(\frac{\mathrm{Fibgon}(S)}{\deg\phi}-1\right)^2 & \geq p_g(C')  & \textrm{ by Proposition~\ref{PropOfPhi}}\\
    & \geq \sqrt{\frac{D'}{2}}  & \textrm{ by Ein-Lazarsfeld's theorem (Theorem~\ref{EinLazarsfeld})}\\
    & \geq \sqrt{\frac{D}{2(\deg\phi)^{21}}}  & \textrm{ by Proposition~\ref{BoundsForDegreeOfK3}}\\
    & \geq \frac{\mathrm{Fibgen}(S)}{4(\deg\phi)^{21/2}}  & \textrm{ by Ein-Lazarsfeld's theorem (Theorem~\ref{EinLazarsfeld})}.
\end{align*}
Note that $\mathrm{Fibgon}(S)\geq 2\deg\phi$. Proposition~\ref{InequalityAsb} follows from these inequalities.
\end{proof}

\section{Some inequalities about Picard lattices of hyper-Kähler manifolds}\label{Inequalities}
We prove in this section Propositions~\ref{InequalityOfLatticesHK} and \ref{InequalityOfLatticesK3}. 

\begin{proof}[Proof of Proposition~\ref{InequalityOfLatticesHK}]
Let $\tau\colon \tilde X\to X$ and $\tilde\phi\colon \tilde X\to X'$ be a resolution of indeterminacy points of $\phi\colon X\dashrightarrow X'$. Let $T$ (resp. $T'$) be the lattice $H^2(X,\mathbb Z)_{tr}$ (resp. $H^2(X',\mathbb Z)_{tr}$) endowed with the Beauville-Bogomolov-Fujiki form. As in the proof of Proposition~\ref{BoundsForDegreeOfK3}, for a positive integer $e$, define $T'(e)$ to be the lattice $T'$ with the quadratic form multiplied by $e$. We claim that the image $E$ of the morphism $\tilde\phi^*\colon H^2(X',\mathbb Z)_{tr}\to \tau^*H^2(X,\mathbb Z)_{tr}\cong T$, viewed as a sublattice of $T$, is isometric to $T'((\deg\phi)^{\frac1n})$. This follows from the equalities $[q_X(\tilde\phi^*\alpha)]^n=c_X\cdot(\int_X\tilde\phi^*\alpha^{2n})=(\deg\phi)\cdot c_X\cdot(\int_{X'}\alpha^{2n})=(\deg\phi)\cdot [q_{X'}(\alpha)]^n$, where $c_X=c_{X'}$ is the Fujiki constant for the deformation class of $X$ and we have viewed $\tilde\phi^*\alpha$ as an element in $H^2(X, \mathbb Z)$ via the isomorphism $\tau_*: \tau^*H^2(X,\mathbb Z)_{tr}\to H^2(X, \mathbb Z)_{tr}$. Now the claim implies the following equalities
\begin{equation}\label{EquationOfDiscHK}
    [T\colon E]^2=\left|\frac{\mathrm{disc}(E)}{\mathrm{disc}(T)}\right|=(\deg\phi)^{\frac{b_{2,\mathrm{tr}}(X)}{n}}\cdot\left|\frac{\mathrm{disc}(T')}{\mathrm{disc}(T)}\right|=(\deg\phi)^{\frac{b_{2,\mathrm{tr}}(X)}{n}}\cdot\left|\frac{\mathrm{disc}(\mathrm{Pic}(X'))}{\mathrm{disc}(\mathrm{Pic}(X))}\right|.
\end{equation}
On the other hand, with a similar argument to Lemma~\ref{philowerstarisinjective}, we prove that the morphism of abelian groups $\tilde\phi_*\colon H^2(X,\mathbb Z)_{tr}\to H^2(X',\mathbb Z)_{tr}$ is injective and sends $E$ to $(\deg\phi)T'$. Hence, \begin{equation}\label{InequalityOfTEHK}
    1\leq [T\colon E]\leq [T'\colon(\deg\phi)T']=(\deg\phi)^{b_{2,\mathrm{tr}}(X)}.
\end{equation} 
The proposition follows by combining (\ref{EquationOfDiscHK}) and (\ref{InequalityOfTEHK}).
\end{proof}

\begin{proof}[Proof of Proposition~\ref{InequalityOfLatticesK3}]
The only thing that needs proving, in the view of Proposition~\ref{InequalityOfLatticesHK}, is the following inequality
\[\frac{1}{(\deg\phi)^{\rho(X)}}\leq\left| \frac{\mathrm{disc}(\mathrm{Pic}(S))}{\mathrm{disc}(\mathrm{Pic}(S'))}\right|.
\]
Let $\tau\colon \tilde S\to S$ and $\tilde\phi\colon \tilde S\to S'$ be a resolution of indeterminacy points of $\phi\colon S\dashrightarrow S'$. Via the morphism $\tilde\phi^*\colon \mathrm{Pic}(S')\to \mathrm{Pic}(\tilde S)$, we can view $\tilde\phi^*\mathrm{Pic}(S')$ as a sublattice of $\mathrm{Pic}(\tilde S)$. The sublattice $\tilde\phi^*\mathrm{Pic}(S')$ is isomorphic to $\mathrm{Pic}(S')(\deg\phi)$, since $\tilde\phi^*\alpha\cup\tilde\phi^*\beta=\tilde\phi^*(\alpha\cup\beta)=\deg\phi\cdot(\alpha\cup\beta)$ for $\alpha, \beta\in \mathrm{Pic}(S')$. Thus 
\begin{equation}\label{EquationOfDiscK3enlargie}
    \mathrm{disc}(\tilde\phi^*(\mathrm{Pic}(S')))=(\deg\phi)^{\rho(S)}\mathrm{disc}(\mathrm{Pic}(S')).
\end{equation} 
\begin{lemma}\label{LatticesOrthogonalLemma}
The sublattice $\ker(\tilde\phi_*\colon\mathrm{Pic}(\tilde S)\to \mathrm{Pic}(S'))$ of $\mathrm{Pic}(\tilde S)$ is the orthogonal complement of $\tilde\phi^*(\mathrm{Pic}(S'))$ in $\mathrm{Pic}(\tilde S)$.
\end{lemma} 
\begin{proof}
Let $\alpha\in \mathrm{Pic}(\tilde S)$. Let us show that $\tilde\phi_*\alpha=0$ if and only if for any  $\beta\in\mathrm{Pic}(S')$, we have $\alpha\cup\tilde\phi^*\beta=0$ in $H^4(\tilde S, \mathbb Z)$. The projection formula gives $\tilde\phi_*(\alpha\cup\tilde\phi^*\beta)=(\tilde\phi_*\alpha)\cup\beta$ in $H^4(S', \mathbb Z)$. Hence, if $\tilde\phi_*\alpha=0$, then $\tilde\phi_*(\alpha\cup\tilde\phi^*\beta)=0$. But $\tilde\phi_*\colon H^4(\tilde S, \mathbb Z)\to H^4(S', \mathbb Z)$ is an isomorphism, we must have $\alpha\cup\tilde\phi^*\beta=0$. Conversely, if $\alpha\cup\tilde\phi^*\beta=0$ for every $\beta\in\mathrm{Pic}(S')$, still by the projection formula, we get $(\tilde\phi_*\alpha)\cup\beta=0$, which implies that $\tilde\phi_*\alpha=0$ since the intersection product map is nondegenerate on $\mathrm{Pic}(S')$.
\end{proof}

Taking into account the fact that the intersection map on $H^2(\tilde S,\mathbb Z)$ is nondegenerate on $\tilde\phi^*\mathrm{Pic}(S')$, Lemma~\ref{LatticesOrthogonalLemma} implies that $\ker(\tilde\phi_*)\oplus \tilde\phi^*(\mathrm{Pic}(S'))$ is a direct sum and that $\ker(\tilde\phi_*)\oplus \tilde\phi^*(\mathrm{Pic}(S'))$ is of \emph{finite} index in the abelian group $\mathrm{Pic}(\tilde S)$.  Hence, 
\begin{equation}\label{InequalityOfTEK3}
    \left|\frac{\mathrm{disc}(\ker(\tilde\phi_*))\cdot\mathrm{disc}(\tilde\phi^*(\mathrm{Pic}(S'))}{\mathrm{disc(\mathrm{Pic}(\tilde S))}}\right|=[\mathrm{Pic}(\tilde S)\colon \ker(\tilde\phi_*)\oplus \tilde\phi^*(\mathrm{Pic}(S'))]^2.
\end{equation}
Since $\tilde\phi_*\colon\mathrm{Pic}(\tilde S)\to \mathrm{Pic}(S'))$ sends $\ker(\tilde\phi_*)\oplus \tilde\phi^*(\mathrm{Pic}(S'))$ onto $(\deg\phi)\mathrm{Pic}(S')$, and since the induced morphism \[\overline{\tilde\phi_*}\colon \mathrm{Pic}(\tilde S)/\ker\tilde\phi_*\to \mathrm{Pic}(S')\] is injective, we have 
\begin{equation}\label{InequalityOfThingsK3}
\begin{split}
     [\mathrm{Pic}(\tilde S): (\ker(\tilde\phi_*)\oplus \tilde\phi^*(\mathrm{Pic}(S')))] & =[(\mathrm{Pic}(\tilde S)/\ker\tilde\phi_*): \overline{\tilde\phi^*(\mathrm{Pic}(S'))}] \\
    & \leq [\mathrm{Pic}(S'):(\deg\phi)\mathrm{Pic}(S')] \\
    & =(\deg\phi)^{\rho(S)},
\end{split}
\end{equation}

where $\overline{\tilde\phi^*(\mathrm{Pic}(S'))}$ is the image of $\tilde\phi^*(\mathrm{Pic}(S'))$ in $\mathrm{Pic}(\tilde S)/\ker\tilde\phi_*$.
Notice the following Lemma.
\begin{lemma}\label{BlowUpDoesntChangeDisc}
$|\mathrm{disc}(\mathrm{Pic}(\tilde S))|=|\mathrm{disc}(\mathrm{Pic}(S))|$.
\end{lemma}
\begin{proof}
$\tilde S$ is obtained by a sequence of blowing-ups of points from $S$. Therefore, 
\[\mathrm{Pic}(\tilde S)=\tau^*\mathrm{Pic}(S)\oplus \bigoplus_i\mathbb Z E_i,\]
where $E_i$ is the \emph{total transform} in $\tilde S$ of the exceptional divisor of the $i$-th blowing-up. We have the following formula for the intersection numbers of $E_i$:
\[E_i.E_j=\left\{\begin{array}{cc}
   -1  & \textrm{ if } i=j \\
   0  & \textrm{ otherwise. }
\end{array}\right.
\]
Hence, $\mathrm{disc}(\mathrm{Pic}(\tilde S))=\mathrm{disc}(\mathrm{Pic}( S))\cdot\mathrm{disc}(\oplus_i\mathbb ZE_i)=\pm\mathrm{disc}(\mathrm{Pic}(S))$, as desired.
\end{proof}

The proposition now follows from Lemma~\ref{BlowUpDoesntChangeDisc} and from inequalities (\ref{EquationOfDiscK3enlargie}), (\ref{InequalityOfTEK3}) and (\ref{InequalityOfThingsK3}), noticing that $\mathrm{disc}(\ker(\tilde\phi_*))\geq 1$.
\end{proof}

\section{Related results}\label{SectionFurtherResultsInBirational}
The study of rational maps between K3 surfaces is an intriguing question. One of the circulated expectations in the area is the following question.
\begin{problem}\label{ProbRatBetweenK3}
    Let $S$ be a very general projective K3 surface. Let $\phi: S\dashrightarrow S'$ be a dominant rational map to another K3 surface $S'$. Is that true that $\phi$ can only be an isomorphism (i.e. the degree of $\phi$ is $1$)?
\end{problem}
As is clear from the proof of Theorem~\ref{theoremFibgonDeK3}, if Problem~\ref{ProbRatBetweenK3} has a positive answer, then the Case (b) in Theorem~\ref{theoremFibgonDeK3} can be improved to be $\mathrm{Fibgon}(S)^2\geq \mathrm{Fibgen}(S)$.
We do not expect a purely lattice theoretic answer.

We can also study self-rational maps of hyper-Kähler manifolds, the higher dimensional generalization of K3 surfaces. As we can see from the following theorem, self-rational maps of hyper-Kähler manifolds has very restrictive numeric properties.
\begin{theorem}\label{theoremNumericRestrictionOfSelfRationalMapsBetweenHK}
    Let $X$ be a projective hyper-Kähler manifold of dimension $2n$ such that $b_{2,tr}(X)$ is odd. Let $\sigma_X$ be its holomorphic symplectic form. Let $\phi: X\dashrightarrow X$ be a dominant self-rational map. Then the degree of $\phi$ is of the form $k^{2n}$ where $k$ is an integer. In this case, $\phi^*\sigma_X = \pm k\sigma_X$. 
\end{theorem}

\begin{proof}
Let $\tau: \tilde X\to X$ and $\tilde\phi:\tilde X\to X$ be a resolution of indeterminacies of $\phi: X\dashrightarrow X$. Let $T$ be the lattice $H^2(X,\mathbb Z)_{tr}$ endowed with the Beauville--Bogomolov--Fujiki form $q_X$. For a positive integer $e$, define $T(e)$ to be the lattice $T$ with the quadratic form multiplied by $e$.  We claim that the image $E$ of the morphism $\tilde\phi^*: H^2(X,\mathbb Z)_{tr}\to \tau^*H^2(X,\mathbb Z)_{tr}\cong T$, viewed as a sublattice of $T$, is isometric to $T((\deg\phi)^{\frac1n})$. This follows from the equalities $[q_X(\tilde\phi^*\alpha)]^n=c_X\cdot(\int_X\tilde\phi^*\alpha^{2n})=(\deg\phi)\cdot c_X\cdot(\int_{X}\alpha^{2n})=(\deg\phi)\cdot [q_{X}(\alpha)]^n$, where $c_X$ is the Fujiki constant for the deformation class of $X$ and we have viewed $\tilde\phi^*\alpha$ as an element in $H^2(X, \mathbb Z)$ via the isomorphism $\tau_*: \tau^*H^2(X,\mathbb Z)_{tr}\to H^2(X, \mathbb Z)_{tr}$.. The claim already implies that $\deg\phi$ is of the form $m^n$ with $m$ an integer. Let us show that $m$ is a perfect square. Note that the claim also implies the following equalities
\begin{equation}
    [T: E]^2=\left|\frac{\mathrm{disc}(E)}{\mathrm{disc}(T)}\right|=(\deg\phi)^{\frac{b_{2,tr}(X)}{n}}\cdot\left|\frac{\mathrm{disc}(T)}{\mathrm{disc}(T)}\right|=(\deg\phi)^{\frac{b_{2,tr}(X)}{n}}=m^{b_{2,tr}(X)}.
\end{equation}
Since $b_{2,tr}(X)$ is assumed to be odd, comparing the two sides of the equality shows that $m$ is a perfect square. Hence, the degree of $\phi$ is of the form $k^{2n}$.

Therefore, for each $\alpha\in H^2(X,\mathbb Z)_{tr}$, we have $q_X(\tilde\phi^*\alpha)=k^2q_X(\alpha)$. Therefore, the morphism $\psi:=\frac1k\phi^*: H^2(X,\mathbb Q)_{tr}\to H^2(X,\mathbb Q)_{tr}$ is an Hodge isometry. Since $b_{2,tr}(X)$ is assumed to be odd, $\psi$ has an eigenvector $v$ with $\pm 1$ as its eigenvalue. But now the fact that $H^2(X,\mathbb Q)_{tr}$ is a \emph{simple} Hodge structure implies that $\psi=\pm Id$. Theorem~\ref{theoremNumericRestrictionOfSelfRationalMapsBetweenHK} then follows.
\end{proof}

\begin{example}
(i) Assume that $X$ admits a Lagrangian fibration $f:X\to B$ with a (rational) zero section. Let $x\in X$ be a general point of $X$ and $b=f(x)\in B$. The fact that the general fiber of the fibration is Lagrangian implies that 
\[\sigma_{X,x}=\alpha\wedge f^*\beta,
\]
where $\alpha\in \Omega_{X_b,x}$ and $\beta\in \Omega_{B,b}$. Since the general fiber $X_b$ is an abelian variety of dimension $n$, the self map $m_{k,b}: X_b\to X_b$ defined as the multiplication by a natural number $k$ is of degree $k^{2n}$, and $m_{k,b}^*\alpha=k\alpha$. The existence the global zero section of this Lagrangian fibration implies that $m_{k,b}: X_b\to X_b$ can be defined globally as a self-rational map 
$m_k: X\dashrightarrow X$. The degree of $m_k$ is $k^{2n}$ and $m_k^*\sigma_X=k\sigma_X$.\\
(ii) Let $X$ be the Fano variety of lines of a cubic fourfold $Y$. The following self-rational map
$\phi: X\dashrightarrow X$, called the Voisin map, is constructed in~\cite{KCorr} as follows. Let $\ell\in X$ be a general line in $Y$. Then there is a unique plane in $\mathbb P^5$ that passes $\ell$ twice. The intersection of $Y$ and this plane is of the form $2\ell+\ell'$ where $\ell'$ is another line in $Y$. Then $\phi(\ell)$ is defined to be $\ell'$. As is shown in~\cite{KCorr}, $\phi: X\dashrightarrow X$ is of degree $16=2^4$ and $\phi^*\sigma_X=-2\sigma_X$.
\end{example}

\chapter{On the Geometry of the Higer Dimension Voisin Maps}\label{ChapterVoisinMaps}
Voisin constructed self-rational maps of Calabi-Yau manifolds obtained as varieties of $r$-planes in cubic hypersurfaces of adequate dimension. This map has been thoroughly studied in the case $r=1$, which is the Beauville-Donagi case. In this chapter, we compute the action of $\Psi$ on holomorphic forms for any $r$. For $r=2$, we compute the action of $\Psi$ on the Chow group of $0$-cycles,  and confirm  that it is as expected from the generalized Bloch conjecture.

\section{Introduction}
Let $X$ be a smooth projective variety of dimension $n$ defined over the field of complex numbers $\mathbb C$. The Chow group of $k$-cycles of $X$, denoted as $CH_k(X)$,  is the quotient by the rational equivalence of the free abelian group generated by closed irreducible subvarieties of dimension $k$. In this work, we will ignore the complexities introduced by the torsion part of the Chow groups, focusing instead on $CH_k(X)_\mathbb Q$, which is the Chow group after tensoring with $\mathbb Q$.

Hodge theory emerges as a pivotal framework for investigating complex smooth projective varieties. The interaction between Hodge structures and the Chow groups is both enlightening and enigmatic. The cycle class map $cl: CH_k(X)_{\mathbb Q}\to H^{2n-2k}(X,\mathbb Q)$ plays a central role in this interplay. A plethora of conjectures have been proposed, relating the complexity of Chow groups and that of Hodge structures. Among these, the Hodge conjecture stands out predicting that the image of $cl$ coincides with the set of Hodge classes in $H^{2n-2k}(X,\mathbb Q)$. 

This chapter is primarily motivated by a conjecture put forth by Voisin, which we will explore in detail:
\begin{conjecture}[Voisin]\label{ConjVoisinCY}
    Let $X$ be a strict Calabi-Yau manifold of dimension $n$. Let $C\subset CH_0(X)_\mathbb Q$ be the subgroup generated by the intersections of divisors and of Chern classes of $X$. Then the cycle class map 
    \[
    cl: CH_0(X)_\mathbb Q\to H^{2n}(X,\mathbb Q)
    \]
    is injective on $C$.
\end{conjecture}

Subsequent sections will delve into the terminology, underlying motivations, and the conceptual backdrop of Conjecture~\ref{ConjVoisinCY}, setting the stage for a comprehensive exploration of its implications and relevance to the broader field of algebraic geometry.

\subsection{Strict Calabi-Yau manifolds}
\begin{definition}
A \emph{strict Calabi-Yau manifold} is a simply connected compact Kähler manifold  $X$ of dimension at least $3$, with trivial canonical bundle and such that $H^0(X,\Omega_X^k)=0$ for $0<k<{\rm dim}\,X$.
\end{definition}

\begin{remark}
It follows from Kodaira's embedding theorem that such $X$ is automatically projective.
\end{remark}

Strict Calabi-Yau manifolds are pivotal in the study of algebraic varieties with trivial canonical bundles, known as $K$-trivial varieties. This class also includes hyper-Kähler manifolds, which are simply connected manifolds and possess a non-degenerate holomorphic $2$-form, for which Conjecture~\ref{ConjVoisinCY} is also expected to be true when they are projective (this is the Beauville-Voisin conjecture~\cite{BeauvilleSplitting, 0CycleHK}), and complex tori, which are defined as quotients of $\mathbb{C}^n$ by a lattice $\Gamma$, for which Conjecture~\ref{ConjVoisinCY} is definitely wrong, when they are projective.

\subsubsection{Voisin's examples of Calabi-Yau manifolds}\label{SectionVoisinExample}
We focus on specific families of $K$-trivial varieties as constructed in~\cite{KCorr}, generalizing the Beauville-Donagi construction~\cite{BeauvilleDonagi}. Let $Y \subset \mathbb{P}^n$ be a smooth cubic hypersurface of dimension $n-1$, and let $r \geq 0$ denote a nonnegative integer. Define $X=F_r(Y)$ as the Hilbert scheme that parametrizes the $r$-dimensional linear subspaces in $Y$. As proved in~\cite[(4.41)]{KCorr}, for $n+1=\binom{r+3}{2}$ and a general $Y$, the variety $X$ is a $K$-trivial variety of dimension $N = (r+1)(n-r)-\binom{r+3}{3}$. Specifically, when $r=0$, $X$ is an elliptic curve; and as established in~\cite{BeauvilleDonagi}, $X$ is a hyper-Kähler manifold for $r=1$. We have the following

\begin{lemma}\label{LmmStrictCY}
For $r \geq 2$ and $n+1 = \binom{r+3}{2}$, the constructed variety $X=F_r(Y)$ is a strict Calabi-Yau manifold.
\end{lemma}
\begin{proof}
Reference \cite[Proposition 3.1(a)]{DebarreManivel} indicates that, generally, if $n \geq \frac{2}{r+1}\binom{r+3}{r} + r + 1$, then $F_r(Y)$ is simply connected. This inequality is consistently satisfied for $r \geq 1$ when $n+1 = \binom{r+3}{2}$. Consequently, $X$ is simply connected in our context. The Beauville-Bogomolov decomposition theorem~\cite[Théorème 1]{Beauville} gives a decomposition of $X \cong T \times W \times CY$, where $T$ is a complex torus, $W$ a product of hyper-Kähler manifolds, and $CY$ a product of strict Calabi-Yau manifolds. As shown in \cite[Théorème 3.4]{DebarreManivel}, the restriction morphism $H^i(\mathrm{Gr}(r+1, n+1), \mathbb{Q}) \to H^i(F_r(Y), \mathbb{Q})$ is an isomorphism for $i < \min\{\dim F_r(Y), n-2r-1\}$. Under the condition $n+1 = \binom{r+3}{2}$, $2 < \min\{\dim F_r(Y), n-2r-1\}$ for any $r \geq 2$. Thus, $H^{2,0}(X) = H^{2,0}(\mathrm{Gr}(r+1, n+1)) = 0$, and the Picard number of $X$ is $1$ for $r \geq 2$. Hence, in the Beauville-Bogomolov decomposition of $X$, only the strict Calabi-Yau manifold component remains, and it is irreducible since the Picard number is $1$.
\end{proof}

\subsection{Conjectures on Chow groups}
We begin by examining some foundational conjectures related to the Chow groups of smooth projective varieties. This discussion is then extended to include conjectures specifically concerning the Chow groups of projective hyper-Kähler and strict Calabi-Yau manifolds. These sections provide the necessary background and motivation for Conjecture~\ref{ConjVoisinCY}, with our analysis closely following the insights and frameworks presented in \cite{BBFiltration}, \cite{VoisinCitrouille}, and \cite{VoisinCoisotrope}.

\subsubsection{Bloch-Beilinson filtration and the generalized Bloch conjecture}
Bloch and Beilinson have conjectured the existence of a descending filtration, denoted as $F^i CH_k(X)_\mathbb{Q}$, on the Chow groups with rational coefficients for any smooth complex projective variety $X$. This filtration is conjectured to satisfy a set of axioms.

\begin{conjecture}[Bloch-Beilinson Conjecture~\cite{BBFiltration}]
For every smooth projective variety $X$, there exists a descending filtration $F^\bullet$ on $CH^i(X)_\mathbb{Q}$ characterized by:
\begin{itemize}
\item[(i)] (Non-Triviality) $CH^i(X)_\mathbb{Q}= F^0CH^i(X)_\mathbb{Q}$, and $F^1CH^i(X)\mathbb{Q} = CH^i(X)_{\mathbb{Q}, hom}$.
\item[(ii)] (Functoriality) For a cycle $Z$ in $CH^k(X \times Y)_\mathbb{Q}$, the pushforward of $F^iCH^l(X)_\mathbb{Q}$ by $Z_*$ is included within $F^iCH^{l+k-n}(Y)_\mathbb{Q}$, where $n$ represents the dimension of $X$.
\item[(iii)] (Graded Component) If $[Z] = 0$ in $H^{2k}(X \times Y, \mathbb{Q})$, the induced map $Z_*: Gr_F^iCH^l(X)_\mathbb{Q}\to Gr_F^iCH^{l+k-n}(Y)_\mathbb{Q}$ vanishes for any $i$.
\item[(iv)] (Finiteness) The filtration terminates with $F^{k+1}CH^k(X)_\mathbb{Q} = 0$ for all varieties $X$ and integers $k$.
\end{itemize}
\end{conjecture}

The conjecture further suggests that the filtration of $CH_0(X)$ is intricately linked to the Hodge structures modulo Hodge substructures of coniveau $\geq 1$, that is, to holomorphic forms.

\begin{conjecture}[Generalized Bloch Conjecture for $0$-Cycles]\label{ConjGBC0Cycle}
Given a correspondence $Z\in CH^n(X \times Y)_\mathbb{Q}$ between smooth projective varieties $X$ and $Y$, both of dimension $n$, if the map $[Z]^*: H^{i, 0}(Y) \to H^{i, 0}(X)$ vanishes for some $i \leq n$, then the pushforward $Z_*: Gr_F^iCH_0(X)_\mathbb{Q} \to Gr_F^iCH_{m-k}(Y)_\mathbb{Q}$ also vanishes for that $i$. Here, $F^\bullet$ represents the Bloch-Beilinson filtration and $Gr_F^\bullet$ its graded part.
\end{conjecture}

This conjecture is named the generalized Bloch conjecture as it extends the classical Bloch conjecture, stated as follows in~\cite{Bloch}.

\begin{conjecture}[Bloch~\cite{Bloch}]
For a correspondence $Z\in CH^2(S \times T)_\mathbb{Q}$ between surfaces that induces a null map $[Z]^*: H^{2,0}(T) \to H^{2,0}(S)$, the induced morphism $Z_*: F^2CH_0(S) \to F^2CH_0(T)$ is identically zero. Here, $F^2CH_0(S)$ is defined as the kernel of the Albanese map from $CH_0(S)_{hom}$ to $Alb(S)$, and similarly for $F^2CH_0(T)$.
\end{conjecture}

\subsubsection{Conjectures of Beauville and Voisin}
Let $X$ be a projective hyper-Kähler manifold. Beauville's splitting conjecture, as introduced in~\cite{BeauvilleSplitting}, suggests that the Bloch-Beilinson filtration on the Chow ring of $X$ undergoes a natural multiplicative splitting. One weak version is now often referred to as the "weak splitting conjecture," detailed in the introduction of~\cite{VoisinCoisotrope}.

\begin{conjecture}[Beauville's Weak Splitting Conjecture~\cite{BeauvilleSplitting}]\label{ConjBeauvilleWeakSplitting}
For a projective hyper-Kähler manifold $X$, the cycle class map is injective on the subalgebra of $CH^*(X)$ that is generated by divisors.
\end{conjecture}

This conjecture was further expanded by Voisin in~\cite{0CycleHK} to include not only divisors but also Chern classes into the generating elements of the subalgebra.

\begin{conjecture}[Voisin~\cite{0CycleHK}]\label{conjVoisinChern}
In the case of a projective hyper-Kähler manifold $X$, consider $C^*$ to be the subalgebra of $CH^*(X)$ generated by divisors and Chern classes. The cycle class map is injective on $C^*$.
\end{conjecture}

However, generalizing these conjectures to strict Calabi-Yau manifolds often leads to contradictions. Following~\cite[Examples 1.7]{BeauvilleSplitting}, denote by $Y$ the blow-up of $\mathbb{P}^3$ along a smooth curve of genus 2 and degree 5. Take a smooth divisor $D$ in $|-2K_Y|$ and let $X$ be the double covering of $Y$ ramified along $D$, $X$ is shown to be a strict Calabi-Yau threefold. However, the cycle class map $cl: CH_1(X)_\mathbb{Q} \to H^4(X,\mathbb{Q})$ fails to be injective on the subgroup generated by intersections of divisors,  invaliding the ``strict Calabi-Yau version" of Conjectures~\ref{ConjBeauvilleWeakSplitting} and~\ref{conjVoisinChern}.

Despite these counterexamples concerning the ``strict Calabi-Yau version" of Conjecture~\ref{conjVoisinChern} for $1$-cycles, it remains anticipated that the conjecture holds true for $0$-cycles. This expectation takes the form of Conjecture~\ref{ConjVoisinCY}, with compelling evidence provided by the constructions in~\cite[Theorem 1.2]{Bazhov}.

\begin{theorem}[Bazhov~\cite{Bazhov}]\label{ThmBazhov}
Let $Y$ be a projective homogeneous variety of dimension $n+1\geq 4$, and let $X$ be a general element of the anti-canonical system $|-K_Y|$. Then $X$ is a strict Calabi-Yau manifold satisfying Conjecture~\ref{ConjVoisinCY}.
\end{theorem}

\subsection{Voisin maps}\label{SectionVoisinMaps}
The distinct feature of the manifolds $X = F_r(Y)$ of Section~\ref{SectionVoisinExample} among all $K$-trivial manifolds revolves around the presence of a self-rational map, $\Psi: X \dashrightarrow X$, referred to as the Voisin map. This map was introduced in~\cite{KCorr} through the following construction: Consider a general point $x \in X$, representing an $r$-dimensional linear space $P_x$ within $Y$. As demonstrated in~\cite[Lemma 8]{KCorr}, there exists a unique $(r+1)$-dimensional linear subspace $H_x$ in $\mathbb{P}^n$ tangent to $Y$ along $P_x$. The intersection $H_x \cap Y$ is a cubic hypersurface containing $P_x$ doubly, leaving a residual $r$-dimensional linear subspace in $Y$ represented by a point $x' \in X$. This process defines the Voisin map as $\Psi(x) = x'$.

The result from~\cite[Corollaire 2.2]{FibrationsMeromorphes} indicates that for $r \geq 2$, the Voisin map does not preserve any non-trivial fibrations. Given this, exploring the dynamics of the Voisin map becomes a compelling avenue of study, although it falls outside the scope of this work. Notably, investigating the dynamics of the Voisin map in the hyper-Kähler case ($r=1$) has yielded significant insights into the geometry of $X$, as reported in~\cite{KCorr, AmerikVoisin, AmerikBogomolovRovinsky}.

Our first result is a computation of two basic invariants of the map $\Psi$.

\begin{theoremA}\label{theoremA}
Let $X=F_r(Y),\, r\geq 0$ be as in \ref{SectionVoisinExample}, and let $\Psi: X\dashrightarrow X$ be the Voisin map. Then
\begin{itemize}
\item[(i)] For any $\omega \in H^0(X,K_X)$, we have 
\[\Psi^*\omega = (-2)^{r+1}\omega.\]
\item[(ii)] The map $\Psi: X \dashrightarrow X$ is of degree $4^{r+1}$.
\end{itemize}
\end{theoremA}

\begin{remark}
Theorem A is trivial for $r=0$. For $r=1$, the results have been previously established in~\cite{KCorr}.
\end{remark}


Having Theorem A,  Conjecture~\ref{ConjGBC0Cycle} leads us to the following

\begin{conjecture}\label{PropConditionalActionOfPsiOnChow}
    Let $X = F_r(Y)$ with $r \geq 2$. Then for any $z \in CH_0(X)_{hom}$,
    \[
    \Psi_*z = (-2)^{r+1}z.
    \]
\end{conjecture}

Indeed, let us explain how Conjecture~\ref{ConjGBC0Cycle} implies Conjecture~\ref{PropConditionalActionOfPsiOnChow}. Let $N = \dim X$. As $X$ is a strict Calabi-Yau manifold, we have $H^{k, 0}(X) = 0$ for all $0 < k < N$. Consider $\Delta_X \in CH^n(X \times X)$ as the diagonal of $X \times X$. Then, for $0 < k < N$, $[\Delta_X]^*|_{H^{k,0}(X)} = 0$. Conjecture~\ref{ConjGBC0Cycle} implies that $\Delta_{X*}|_{Gr^k_FCH_0(X)} = 0$ for $0 < k < N$, leading to $Gr^k_FCH_0(X) = 0$ for these values of $k$. Thus, the Bloch-Beilinson filtration on $CH_0(X)$ simplifies to:
    \[
    0 = F^{N+1}CH_0(X) \subseteq F^NCH_0(X) = \ldots = F^1CH_0(X) = CH_0(X)_{hom} \subseteq F^0CH_0(X) = CH_0(X).
    \]
    Now, define $Z = \Gamma_\Psi - (-2)^{r+1}\Delta_X$. Given Theorem A, $[Z]^*|_{H^{N, 0}(X)} = 0$. Applying Conjecture~\ref{ConjGBC0Cycle} once more, we find $Z_*|_{Gr^N_FCH_0(X)} = 0$, confirming that $Z_*$ is null on  $Gr^N_FCH_0(X) = F^NCH_0(X) = F^1CH_0(X) = CH_0(X)_{hom}$, yielding the desired conclusion.

Our second main result is the proof of Conjecture~\ref{PropConditionalActionOfPsiOnChow} when $r=2$
\begin{theoremB}\label{theoremB}
    Let $Y \subset \mathbb{P}^9$ be a general cubic $8$-fold, and let $X = F_2(Y)$ be the Fano variety of planes in $Y$. Let $\Psi: X \dashrightarrow X$ be the Voisin map. Then for any $z \in CH_0(X)_{hom}$:
    \[\Psi_*z = -8z.\]
\end{theoremB}


In the proof of Theorem B, the notion of constant-cycle subvarieties, as described in~\cite{ConstantCycle, VoisinCoisotrope}, plays a pivotal role. Let us revisit the definition for clarity.

\begin{definition}[\cite{ConstantCycle, VoisinCoisotrope}]
    Let $X$ be a smooth algebraic variety. A closed subvariety $j: Z \hookrightarrow X$ is a \emph{constant-cycle subvariety} if every two points $z_1, z_2 \in Z$ are rationally equivalent in $X$. Equivalently, the image of the morphism, $j_*: CH_0(Z) \to CH_0(X)$, is $\mathbb{Z}$.
\end{definition}

Let $F$ denote the closure of the fixed locus under the Voisin map $\Psi: X \dashrightarrow X$. We next observe that   Conjecture~\ref{PropConditionalActionOfPsiOnChow} immediately  leads us to   the following

\begin{conjecture}\label{conjFixedLocusConstantCycle}
     The variety $F\subset X$ is a constant-cycle subvariety for $r\geq 2$.
\end{conjecture}

Let us explain how  Conjecture~\ref{conjFixedLocusConstantCycle} is implied by  Conjecture~\ref{PropConditionalActionOfPsiOnChow}. Consider two points $x, y \in F$. Since the rational equivalence class is a countable union of closed algebraic subsets~\cite[Section 1.1.1]{VoisinCitrouille}, we may assume that $x, y$ are general. Assuming $x, y$ are within the fixed locus of $\Psi$, we have $\Psi(x) = x$ and $\Psi(y) = y$, leading to $\Psi_*(x - y) = x - y$. However, Proposition~\ref{PropConditionalActionOfPsiOnChow} indicates $\Psi_*(x - y) = (-2)^r(x - y)$ under Conjecture~\ref{ConjGBC0Cycle}. Consequently, $x - y = 0 \in CH_0(X)_\mathbb{Q}$. Roitman's theorem~\cite{Roitman} implies $CH_0(X)_{hom}$ is torsion-free due to the triviality of $Alb(X)$, so it follows that $x = y \in CH_0(X)$. Thus, $F \subset X$ forms a constant-cycle subvariety.

In the present paper, we will prove  directly  Conjecture~\ref{conjFixedLocusConstantCycle} for $r=2$, and this will be   one step in our proof of Theorem B, that is, Conjecture~\ref{PropConditionalActionOfPsiOnChow} for $r=2$.

\begin{remark}
    Interestingly, even in the scenario of $r=1$—not addressed in Conjecture~\ref{conjFixedLocusConstantCycle}—it has been proved~\cite{0CycleHK} that $F \subset X$ is a constant-cycle subvariety.
\end{remark}


Now, let us examine the indeterminacy locus of the Voisin map $\Psi: X \dashrightarrow X$. This locus comprises two components, described as follows. The first component, $\mathrm{Ind}_0$, contains points $x \in X$ that represent $P_x \subset Y$, such that there is more than one linear subspace of dimension $r+1$ tangent to $Y$ along $P_x$. The second component, $\mathrm{Ind}_1$, includes points $x \in X$ representing $P_x \subset Y$, where an $(r+1)$-dimensional linear space exists that contains $P_x$ and is contained within $Y$. It is not hard to prove (Proposition~\ref{PropCodimAndClassOfInd0} and Proposition~\ref{PropCodimAndClassOfInd1}) that $\mathrm{Ind}_0$ has codimension $2$, while $\mathrm{Ind}_1$ has codimension $r+2$ for $r \geq 2$ (and $\mathrm{Ind}_1$ is empty for $r = 1$). We prove as a consequence of Theorem B the following result

\begin{theoremC}
    Assume $r=2$. Then any $0$-cycle of $X$ which is a polynomial in the Chern classes of $X$ and divisor classes is rationally equivalent to a cycle supported on Ind.
\end{theoremC}
\begin{corollary}
    If $\mathrm{Ind}\subset X$ is a constant cycle subvariety, Conjecture~\ref{ConjVoisinCY} holds true for $X = F_2(Y)$.
\end{corollary}

This leaves us with an open
\begin{question}\label{QuestionConstantCycleInd}
    Is $\mathrm{Ind} \subset X$ a constant cycle subvariety for $r \geq 2$?
\end{question}

\begin{remark}
    In the hyper-Kähler case ($r=1$), the indeterminacy locus $\mathrm{Ind}_0$ is \emph{not} a constant-cycle subvariety, nor is it Lagrangian, as indicated by~\cite[Lemma 2]{Amerik} and~\cite{KCorr}. However, our understanding of why $\mathrm{Ind}_0$ is not a constant-cycle subvariety in this scenario does not contradict the evidence - that we will present in Section~\ref{SectionDiscussionsOnQuestion3} - for the positive answer to Question~\ref{QuestionConstantCycleInd}.
\end{remark}

\subsection{Notations and supplementary results}

For the continuity and coherence of this chapter, we establish some notations that will be used throughout, unless specified otherwise:

\begin{itemize}
    \item[$Y$] denotes a cubic hypersurface, which is typically smooth.
    \item[$X$] refers to a strict Calabi-Yau or hyper-Kähler manifold. Concrete examples of strict Calabi-Yau and hyper-Kähler manifolds in this chapter are constructed as the Fano variety of $r$-linear spaces in a smooth cubic hypersurface $Y$. The dimension of $X$ is $(r+1)(n-r)-\binom{r+3}{3}$ and is denoted as $N$.
    \item[$P_x$] denotes the $r$-dimensional linear subspace in $Y$ (resp. $\mathbb P^n$) for $x\in X$ (resp. $x\in \mathrm{Gr}(r+1, n+1)$).
    \item[$\Psi$]: $X\dashrightarrow X$ denotes the Voisin map.
    \item[$F$] represents the closure of the fixed locus of the Voisin map $\Psi: X \dashrightarrow X$.
\end{itemize}

Throughout the chapter, we establish several auxiliary results that not only support the proof of the main theorems but also hold intrinsic interest. Here is a summary of these findings:

\begin{itemize}
    \item (Referencing Theorem~\ref{ThmFixedLocusIsConstantCycle} and Theorem~\ref{ThmChowClassOfF}): For $X = F_2(Y)$, the Fano variety of planes in a general cubic eightfold, it is established that $F \subset X$ is a constant cycle subvariety with a Chow class of $-404 c_1^3 + 110 c_1c_2 + 49 c_3$.
    \item (Referencing Theorem~\ref{ThmChowOneOfF1}): For a general cubic eightfold $Y$, it is demonstrated that $CH_1(F_1(Y))_{hom,\mathbb{Q}} = 0$.
    \item (Referencing Theorem~\ref{ThmLocalDeformation}): For the strict Calabi-Yau manifolds $X = F_r(Y)$, with $r\geq 2$, as presented in Section~\ref{SectionVoisinExample}, the map $F_r: \mathbb PH^0(\mathbb P^n, \mathcal O_{\mathbb P^n}(3))/PSL_{n+1}(\mathbb C)\dashrightarrow \mathrm{Def}(X)$ is a local isomorphism.
\end{itemize}

\section{Action of the Voisin map on top degree holomorphic forms}

In this section, we compute the action $\Psi^*: H^0(X,K_X)\rightarrow H^0(X,K_X)$ of the Voisin map and its degree, proving Theorem A.

\begin{remark}
    For $r=0$, where $Y$ is a plane cubic curve and $X=F_0(Y)=Y$ denotes an elliptic curve, $\Psi: X \to X$ acts by mapping $x$ to $-2x$, in accordance with the addition law of the elliptic curve. Consequently, the degree of $\Psi: X \to X$ is $4$. In the scenario where $r=1$ and $X$ is a hyper-Kähler fourfold, the result $\deg \Psi = 16$ is first discovered by Voisin~\cite{KCorr}, and can be derived using either Chow-theoretic techniques~\cite{KCorr}, \cite[Corollary 1.7]{AmerikVoisin} or vector bundle methods \cite[Lemma 4.12 and Proposition 4.17]{Cubic}.
\end{remark}

Let us first note the following
\begin{lemma}\label{LmmNearlyEquiv}
    The two assertions in Theorem A are equivalent, up to sign.
\end{lemma}
 \begin{proof}
     (i) leads to (ii) as follows: Given that $\sigma \wedge \bar\sigma$ constitutes a volume form on $X$, it follows that
    \[
    \deg \Psi \int_X \sigma \wedge \bar\sigma = \int_X \Psi^* \sigma \wedge \Psi^* \bar\sigma = \int_X (-2)^{r+1} \sigma \wedge (-2)^{r+1} \bar\sigma = 4^{r+1} \int_X \sigma \wedge \bar\sigma.
    \]
    Conversely, let us demonstrate that (ii) implies $\Psi^* \sigma = \pm 2^{r+1} \sigma$, aligning closely with (i). The rational map can be consistently defined across the family of cubic hypersurfaces, implying the degree of $\Psi$ remains invariant across different choices of the generic cubic hypersurface $Y$. Here, ``generic'' means that the cubic hypersurface is chosen outside a proper Zariski closed subset. Therefore, we may presume $Y$ is defined over $\mathbb{Q}$. Consequently, the rational map $\Psi: X \dashrightarrow X$ is also defined over $\mathbb{Q}$. Thus, $\Psi^*: H^0(X, K_{X/\mathbb{Q}}) \to H^0(X, K_{X/\mathbb{Q}})$ operates by multiplication by a rational number, given that $H^0(X, K_{X/\mathbb{Q}})$ is a one-dimensional $\mathbb{Q}$-vector space. If $\Psi^*\sigma = \lambda \sigma$ with $\lambda \in \mathbb{Q}$, then $\lambda^2 = 4^{r+1}$, leading to $\lambda = \pm 2^{r+1}$.
 \end{proof}

\subsection{Proof of Theorem A (i)}
Let $\mathrm{Fix}(\Psi):=\{x\in X: \Psi \textrm{ is defined at } x \textrm{ and } \Psi(x) = x\}$ be the fixed locus of $\Psi: X\dashrightarrow X$.
\begin{proposition}
    For $Y$ general, the fixed locus $\mathrm{Fix}(\Psi)$ is not empty and is of codimension $r+1$ in $X$.
\end{proposition}
\begin{proof}
    We will suppose that $r\geq 2$ since the case $r=1$ is shown in~\cite[Proposition 3.1]{AmerikBogomolovRovinsky}  already (see also~\cite[Corollary 3.13]{FixedLocus}). Let $B=\mathbb PH^0(\mathbb P^n, \mathcal O(3))$ be the parametrizing space of cubic hypersurfaces. Let $\mathrm{Fl} = \{(t, s)\in \mathrm{Gr}(r+1, n+1)\times \mathrm{Gr}(r+2, n+1): P_t\subset \Pi_s\}$ be the flag variety of pairs of linear subspaces in $\mathbb P^n$. Let 
    \[\mathcal I = \{(f, t, s)\in B\times \mathrm{Fl}: \Pi_s\textrm{ is the only }\mathbb P^{r+1}\textrm{ such that } Y_f\cap \Pi_s = 3P_t\},
    \]
    \[
    \tilde{\mathcal I} = \{
    (f, t, s)\in B\times \mathrm{Fl}: Y_f\cap \Pi_s \supset 3P_t
    \}.
    \]
     Let $p: \mathcal I\to B$ (resp. $\tilde p: \tilde{\mathcal I}\to B$) and $q: \mathcal I\to \mathrm{Fl}$ (resp. $\tilde q: \tilde{\mathcal I}\to \mathrm{Fl}$) be the canonical projection maps. By definition of $\Psi$, for any $f\in B$, the fiber $p^{-1}(f)\subset \mathcal I$ coincides with the fixed locus of $\Psi$ for the cubic hypersurface $Y_f$. Hence, it suffices to show that the map $p:
    \mathcal I\to B$ is dominant and that the dimension of the general fibers is $\dim X - r - 1$.
    \begin{lemma}\label{LmmVarietyMathcalI}
        The variety $\mathcal I$ is open dense in $\tilde{\mathcal I}$. The dimension of $\mathcal I$ is $\dim X + \dim H^0(\mathbb P^n, \mathcal O(3)) - r - 2$.
    \end{lemma}
    \begin{proof}
        Let us consider the fibers of the map $\tilde q: \tilde{\mathcal I}\to \mathrm{Fl}$. For any element $(t,s)\in \mathrm{Fl}$ representing linear subspaces $P$ and $\Pi$ respectively, we may assume without loss of generality that 
        \[
    P=\{(x_0, x_1,\ldots, x_r, 0,\ldots, 0)\},
    \]
    \[
    \Pi=\{(x_0, x_1, \ldots, x_r, x_{r+1}, 0,\ldots,0)\}.
    \]
    The fiber $\tilde q^{-1}((t, s))$ parametrizes the cubic hypersurfaces $Y$ such that $\Pi\cap Y\supset 3P$. The last condition implies that the defining equation $f$ of $Y$ in $\mathbb P^n$ is given by 
    \[
    f(Y_0,\ldots, Y_n) = \alpha Y_{r+1}^3 + Y_{r+2}Q_{r+2} + \ldots + Y_nQ_n,
    \]
    where $\alpha\in \mathbb C$ is a constant and $Q_{r+1},\ldots, Q_n$ are quadratic polynomials. To write the above fact more formally, the fiber $\tilde q^{-1}((t, s))$ is identified with the image of the following map
    \[
    \begin{array}{cccc}
        \Phi: & (\mathbb C\oplus H^0(\mathbb P^n, \mathcal O(2))^{\oplus n-r-1}) - \{0\} & \to & B  \\
         & (\alpha, Q_{r+1}, \ldots, Q_n) & \mapsto & \alpha Y_{r+1}^3 + Y_{r+2}Q_{r+2} + \ldots + Y_nQ_n.
    \end{array}
    \]
    It is not hard to see that $\mathcal I\cap \tilde q^{-1}((t, s))$ parametrizes the cubic hypersurfaces $Y$ whose defining equation $f$ is given by 
    \[ f(Y_0,\ldots, Y_n) = \alpha Y_{r+1}^3 + Y_{r+2}Q_{r+2} + \ldots + Y_nQ_n,
    \]
    subject to $\alpha \neq 0$ and $Q_{r+1}(Y_0, \ldots, Y_r, 0, \ldots, 0), \ldots, Q_n(Y_0, \ldots, Y_r, 0, \ldots, 0)$ are linearly independent in $H^0(\mathbb P^r, \mathcal O(2))$. Notice that $n - r - 1 = \dim H^0(\mathbb P^r, \mathcal O(2))$, the above conditions give an open dense subset in $(\mathbb C\oplus H^0(\mathbb P^n, \mathcal O(2))^{\oplus n-r-1}) - \{0\}$, which implies that $\mathcal I\cap \tilde q^{-1}((t, s))$ is open dense in $\tilde q^{-1}((t, s))$ for any $(t, s)\in \mathrm{Fl}$. Therefore, $\mathcal I$ is open dense in $\tilde{\mathcal I}$.

    For the next step, let us calculate the dimension of $\mathcal I$. To this end, we use now the projection $\mathcal I \rightarrow \mathrm{Fl}$ and compute the dimensions of its fibers. Since $\mathrm{Fl}$ is a $\mathbb P^{r+1}$-bundle over $\mathrm{Gr}(r+2, n+1)$, it is clear that $\dim \mathrm{Fl} = (r+2)(n-r-1) + (r+1)$. By the above description of the fiber $q^{-1}((t, s))$, the dimension of the fiber of $q: \mathcal I\to \mathrm{Fl}$ is the dimension of the space of cubic polynomials in $Y_0, \ldots, Y_n$ such that each monomial of it contains either one of the variables $Y_{r+2}, \ldots, Y_n$, and the latter is $\dim H^0(\mathbb P^n, \mathcal O(3)) - \dim H^0(\mathbb P^{r+1}, \mathcal O(3))$. Taken together, we have
    \[
    \dim {\mathcal I} = (r+2)(n-r-1) + (r+1) + \dim H^0(\mathbb P^n, \mathcal O(3)) - \dim H^0(\mathbb P^{r+1}, \mathcal O(3)).
    \]
    Taking into account of the relation $n+1 = \dim H^0(\mathbb P^{r+1}, \mathcal O(2))$ and the fact that $\dim X = (n-r)(r+1) - \dim H^0(\mathbb P^r, \mathcal O(3))$, we find that $\dim {\mathcal I} = \dim X + \dim H^0(\mathbb P^n, \mathcal O(3)) - r - 2$, as desired.
    \end{proof}
    \begin{lemma}
        The map $\tilde p: \tilde{\mathcal I}\to B$ is surjective.
    \end{lemma}
    \begin{proof}
        It is equivalent to showing that for any cubic hypersurface $Y$, there are linear subspaces $P, \Pi\subset\mathbb P^n$ of dimension $r$, $r+1$ respectively such that $\Pi\cap Y\supset 3P$. To construct such examples, it is worthwhile to notice that for $r\geq 2$, the variety $F_{r+1}(Y)$ is non-empty by a dimension counting argument. For any $\Pi\in F_{r+1}(Y)$ and any $P\subset \Pi$, we have $\Pi\cap Y\supset 3P$. That concludes the proof.
    \end{proof}
    By the two lemmas above, we conclude that the map $p: \mathcal I\to B$ is dominant. Since the map $p: \mathcal I\to B$ is dominant, the dimension of the general fiber is $\dim\mathcal I - \dim B $, which is equal to $\dim X - r - 1$.
\end{proof}

\begin{proposition}\label{PropEigenPoly}
Let $x\in X$ be a generic fixed point of $\Psi: X\dashrightarrow X$ representing an $r$-dimensional linear subspace $P\subset Y$. Then the linear map
\[
\Psi_{*,x}: T_{X,x}\to T_{X, x}
\]
 has $N-r-1$ eigenvalues equal to $1$, corresponding to the tangent space to the fixed locus $F$ of $\Psi$ at $x$, and $r+1$ eigenvalues equal to $-2$, corresponding to the action of $\Psi_*$ in the normal direction to $F$ at $x$.
\end{proposition}
\begin{proof}
    Since $x\in X$ is a generic fixed point of $\Psi$ representing a plane $P\subset Y$, there is a unique $(r+1)$-plane $\Pi\subset \mathbb P^n$ such that $\Pi\cap Y=3P$ as algebraic cycles. Without loss of generality, we may assume that 
    \[
    P=\{(x_0, x_1,\ldots, x_r, 0,\ldots, 0)\},
    \]
    \[
    \Pi=\{(x_0, x_1, \ldots, x_r, x_{r+1}, 0,\ldots,0)\}.
    \]
    The fact $\Pi\cap Y=3P$ implies that the defininig equation $f$ of $Y$ in $\mathbb P^n$ is given by
    \begin{equation}\label{EqDefiningEquationOfY}
        f(Y_0, \ldots, Y_n)=Y_{r+1}^3+Y_{r+2}Q_{r+2}+\ldots+ Y_nQ_n,
    \end{equation}
    where $Q_{r+1}, \ldots, Q_n$ are quadratic polynomials. Let $v_1\in T_{X, x}$ be a nonzero tangent vector given by a path $\{P_{1,t}\}_{t\in \Delta}$ with $P_{1,0}=P$, where $\Delta$ is the unit disc in the complex plane. To understand $\Psi_{*, x}(v_1)\in T_{X, x}$, we have to determine $P_{2,t}:=\Psi(P_{1,t})$ for each small $t\in \Delta$. By the construction, for each small $t\in \Delta$, there is a unique $(r+1)$-plane $\Pi_t\subset \mathbb P^n$ such that $\Pi_t\cap Y=2P_{1,t}+P_{2,t}$. Since the inclusions $P_{1,t}\subset \Pi_t\subset \mathbb P^n$ vary holomorphically with respect to $t\in \Delta$, we may assume that $\Pi_t$ is given by the image of a linear map $G_t: \mathbb P^{r+1}\to \mathbb P^n$ varying holomorphically with $t\in \Delta$ and that 
    \begin{equation}
        P_{1,t}=\{G_t(x_0,\ldots, x_r, 0)\}.
    \end{equation}
    Then the relation $\Pi_t\cap Y= 2 P_{1,t}+P_{2,t}$ can be translated as 
    \begin{equation}\label{EqFormOfFcircGt}
        f\circ G_t(x_0, \ldots, x_{r+1})=x_{r+1}^2L_t(x_0, \ldots, x_{r+1})
    \end{equation}
    for some linear function $L_t$ and that 
    \begin{equation}\label{EqDefiningEquationOfP2t}
    P_{2,t}=\{G_t(x_0,\ldots, x_{r+1}): L_t(x_0, \ldots, x_{r+1})=0\}.
    \end{equation}
    Let us simplify the notation and write $\mathbf x=(x_0,\ldots, x_{r+1})$. Let us write $G_t(\mathbf x)$ as 
    \begin{equation}\label{EqExpansionOfGt}
        G_t(\mathbf x) = (x_0+tY_0(\mathbf x), x_1+tY_1(\mathbf x), \ldots, x_{r+1}+tY_{r+1}(\mathbf x),\\ tY_{r+2}(\mathbf x), \ldots, tY_n(\mathbf x))+O(t^2).
    \end{equation}
    Then by (\ref{EqDefiningEquationOfY}), we have
    \begin{equation}\label{EqExpansionOfFcircGt}
        f\circ G_t(\mathbf x)=x_{r+1}^3+ 3tx_{r+1}^2Y_{r+1}(\mathbf x)+t\sum_{i=r+2}^{n}Y_i(\mathbf x)Q_i(\mathbf x, \Vec{0})+O(t^2).
    \end{equation}
    Comparing (\ref{EqFormOfFcircGt}) and (\ref{EqExpansionOfFcircGt}), we get that $x_{r+1}^2$ divides $\sum_{i=r+2}^n Y_i(\mathbf x)Q_i(\mathbf x, \Vec{0})$, and that 
    \begin{equation}
        L_t(\mathbf x)=x_{r+1}+t\left\{3Y_{r+1}(\mathbf x)+\frac{Y_i(\mathbf x)Q_i(\mathbf x, \Vec{0})}{x_{r+1}^2}\right\}+O(t^2).
    \end{equation}
    Hence, the equation $L_t(\mathbf x)=0$ has an explicit expression as 
    \begin{equation}\label{EqExplicitFormOfLt}
        x_{r+1}=-t\left\{
        3Y_{r+1}(\mathbf x', 0)+\frac{\sum_{i=r+2}^n Y_iQ_i}{X_{r+1}^2}(\mathbf x', 0)
        \right\}+O(t^2),
    \end{equation}
    where $\mathbf x'=(x_0,\ldots, x_r)$. Comparing (\ref{EqDefiningEquationOfP2t})(\ref{EqExpansionOfGt})(\ref{EqExplicitFormOfLt}), we get
    \begin{align*}
        P_{2,t}=&\{(x_0+tY_0(\mathbf x', 0), \ldots, x_r+tY_r(\mathbf x', 0), -2tY_{r+1}(\mathbf x',0)-t\frac{\sum Y_iQ_i}{X_{r+1}^2}(\mathbf x, 0), \\
        &tY_{r+2}(\mathbf x',0), \ldots, tY_n(\mathbf x',0))+O(t^2)\}.
    \end{align*}
    We view now $v_1\in T_{X, x}$ as a $(r+1)\times (n-r)$-matrix via the following natural inclusion and identification $T_{X,x}\subset T_{\mathrm{Gr}(r+1, n+1), x}=Hom(\mathbb C^{r+1}, \mathbb C^{n+1}/\mathbb C^{r+1})$. Then $v_1$ is represented by the matrix 
    \[
    \begin{pmatrix}
     Y_{r+1}\\
     Y_{r+2}\\
     \vdots\\
     Y_{n}
    \end{pmatrix},
    \]
    where we view each $Y_i$ as a row vector $(a_0, a_1,\ldots, a_r)$ if $Y_i(x_0, \ldots, x_r, 0)=a_0x_0+\ldots +a_rx_r$. By the above calculations, we see that $\Psi_{*, x}(v_1)$ is represented by the matrix 
    \[
    \begin{pmatrix}
        -2Y_{r+1}-\frac{\sum_{i=r+2}^n Y_iQ_i}{X_{r+2}^2}\\
        Y_{r+2}\\
        \vdots\\
        Y_n
    \end{pmatrix}
    \]
    with the same explanation about the notations. Hence, $\Psi_{*,P}$ sends 
    $
    \begin{pmatrix}
     Y_{r+1}\\
     Y_{r+2}\\
     \vdots\\
     Y_{n}
    \end{pmatrix}
    $
    to 
    $
    \begin{pmatrix}
        -2Y_{r+1}-\frac{\sum_{i=r+2}^n Y_iQ_i}{X_{r+2}^2}\\
        Y_{r+2}\\
        \vdots\\
        Y_n
    \end{pmatrix}.
    $
    The representing matrix of $\Psi_{*, x}$ is uppper triangular with the diagonal $r+1$ copies of $-2$ and $N - r - 1$ copies of $1$. Therefore, the eigenpolynomial of $\Psi_{*,x}:T_{X, x}\to T_{X,x}$ is $(t+2)^{r+1}(t-1)^{N-r-1}$, as desired.
\end{proof}

\begin{proof}[Proof of Theorem A (i)]
    By Proposition~\ref{PropEigenPoly}, the eigenpolynomial of $\Psi^*_P: \Omega_{X,P}\to \Omega_{X, P}$ is $(t+2)^{r+1}(t-1)^{N-r-1}$ for a generic fixed point $P$ of $\Psi$. Hence, the map $\Psi^*_P: K_{X,P}\to K_{X,P}$ is given as the multiplication by $(-2)^{r+1}$. Let $\omega\in H^0(X,K_X)$ be a nowhere zero top degree holomorphic differential form on $X$. Then, as $H^0(X,K_X)$ has dimension $1$,
 \begin{equation}\label{EqARelation}
        \Psi^*\omega=\lambda\omega
    \end{equation} 
    for some $\lambda\in \mathbb C$. By evaluating the equation (\ref{EqARelation}) at the point $P\in X$, we find $\lambda=(-2)^{r+1}$, and we finish the proof of Theorem A (i).
\end{proof}

\subsection{A direct proof of Theorem A (ii)}
While Theorem A (ii) has been established as a consequence of Theorem A (i) using Lemma~\ref{LmmNearlyEquiv}, we present a direct proof, using an enumerative geometry viewpoint.
\begin{proof}[Proof of Theorem A
(ii)]
The beginning of the proof is similar to that of \cite[Lemma 4.12]{Cubic}. Let $P'\in X$ be a generic $r$-plane in $Y$. The preimage of $P'$ via $\Psi$ is the set of $r$-planes $P$ in $Y$ such that there is an $(r+1)$-plane $\Pi$ in $\mathbb P^n$ such that $\Pi\cap Y=2P+P'$ as algebraic cycles. In general, let $\Pi\subset \mathbb P^n$ be an $(r+1)$-plane containing $P'$. Then $\Pi\cap Y=Q\cup P'$ where $Q\subset \Pi$ is a quadratic hypersurface that corresponds to a quadratic form $q_{\Pi}$. The preimage of $P'$ via $\Psi$ corresponds to the $(r+1)$-planes $\Pi$ such that the quadratic form $q_{\Pi}$ is of rank $1$. Now let $\pi: \mathbb P^n\dashrightarrow \mathbb P^{n-r-1}$ be the projection map induced by the $r$-plane $P'$. One can resolve the indeterminacies of $\pi:\mathbb P^n\dashrightarrow \mathbb P^{n-r-1}$ by blowing up $\mathbb P^n$ along $P'$. Let $p:\mathcal P\to \mathbb P^n$ be the blowup map. Then the induced map $q: \mathcal P\to \mathbb P^{n-r-1}$ is a projective bundle induced by a vector bundle $\mathcal E$ over $\mathbb P^{n-r-1}$. We now prove
\begin{lemma}
    (i) We have $\mathcal E\cong \mathcal O_{\mathbb P^{n-r-1}}^{r+1}\oplus \mathcal O_{\mathbb P^{n-r-1}}(-1)$.\\
    (ii) There is a section $q\in H^0(\mathbb P^{n-r-1}, Sym^2\mathcal E^*\otimes \mathcal O_{\mathbb P^{n-r-1}}(1))$ such that on each point $\Pi\in\mathbb P^{n-r-1}$, $q$ coincides with $q_{\Pi}$.
\end{lemma}
\begin{proof}
    See \cite[Section 1.5.1-1.5.2]{Cubic}.
\end{proof}
Hence, to calculate the degree of $\Psi: X\dashrightarrow X$, we need to calculate the number of points in $\mathbb P^{n-r-1}$ over which $q$ is of rank $1$. Since $P'$ is generic, by a dimension counting argument we find that there is no plane $\Pi$ containing $P'$ that is contained in $Y$. Therefore, we are reduced to calculate the degree of the locus where a generic section of $Sym^2\mathcal E^*\otimes \mathcal O(1)$ has rank $\leq 1$. Let $f: \mathbb P^{n-r-1}\to \mathbb P^{n-r-1}$ be a morphism of degree $2^{n-r-1}$ (e.g., $f: (x_0,\ldots, x_{n-r-1})\mapsto (x_0^2, \ldots, x_{n-r-1}^2)$). Let $Z_q\subset \mathbb P^{n-r-1}$ be the locus where $q\in H^0(\mathbb P^{n-r-1}, Sym^2\mathcal E^*\otimes \mathcal O(1))$ is of rank $\leq 1$. Then $f^*Z_q$ is the locus where $f^*q\in H^0(Sym^2\mathcal F)$ is of rank $\leq 1$. Here, $\mathcal F=\mathcal O(1)^{r+1}\oplus \mathcal O(3)$. Then 
\begin{equation}\label{EqRelationOfF*}
    \deg f^*Z_q=\deg f\cdot\deg Z_q=2^{n-r-1}\deg Z_q.
\end{equation}
Let $p: \mathbb P(\mathcal F)\to \mathbb P^{n-r-1}$ and $\pi: \mathbb P(\mathrm{Sym}^2\mathcal F)\to \mathbb P^{n-r-1}$ be the projective bundles over $\mathbb P^{n-r-1}$ corresponding to the vector bundles $\mathcal F$ and $\mathrm{Sym}^2\mathcal F$, respectively. Consider the Veronese map of degree $2$ over the projective space $\mathbb P^{n-r-1}$.
\[\begin{array}{cccc}
v_2: &\mathbb P(\mathcal F) &\to &\mathbb P(\mathrm{Sym}^2\mathcal F)\\
& \alpha &\mapsto &\alpha^2.
\end{array}\]
The section $f^*q$ induces a section $\sigma_q$ of the projective bundle $\pi: \mathbb P(Sym^2\mathcal F)\to\mathbb P^{n-r-1}$ such that $f^*Z_q$ coincides with $\pi_*(\mathrm{Im}\sigma_q\cap \mathrm{Im}(v_2)) = p_*(v_2^*(\mathrm{Im}(\sigma_q)) $. Let $\mathcal S$ (resp. $\mathcal S'$) be the tautological subbundle of $\mathbb P(\mathcal F)$ (resp. $\mathbb P(Sym^2\mathcal F)$). Let $\mathcal Q'$ be the tautological quotient bundle of $\mathbb P(Sym^2\mathcal F)$. Then the pull-back of $f^*q\in H^0(\mathbb P^{n-r-1}, Sym^2\mathcal F)$ gives a section of $H^0(\mathbb P(Sym^2\mathcal F), \mathcal Q')$ whose zero locus coincides with $\mathrm{Im}(\sigma_q)$. Writing $c(\mathcal Q')$ as $\pi^*c(\mathrm{Sym}^2\mathcal F)\cdot c(\mathcal S')^{-1}$ and noticing that $v_2^*\mathcal S'=\mathcal S^{\otimes 2}$, we see that $f^*Z_q$ is the part of degree $n$ of $p_*(c(S^{\otimes 2})^{-1}\cdot p^* c(\mathrm{Sym}^2\mathcal F))$. Let $s(\mathcal F)$ be the formal series of the Segre classes of $\mathcal F$~\cite[Chapter 3]{Fulton} and let $h=c_1(\mathcal O_{\mathbb P^n}(1))$. By a direct calculation, we find that
\begin{align*}
    f^*Z_q & = 2^{r+1}\sum_i2^{i}s_i(\mathcal F)\cdot c_{n-r-1-i}(\mathrm{Sym}^2\mathcal F)\\
    & = \textrm{The degree $n-r-1$ part of } 2^{r+1} \frac{(1+2h)^{\frac{(r+1)(r+2)}{2}}\cdot (1+4h)^{r+1}\cdot (1+6h)}{(1+2h)^{r+1}\cdot (1+6h)}\\
    &= 2^{r+1}\cdot 2^{\frac{(r+1)r}{2}}\cdot 4^{r+1} h^{n-r-1}.
\end{align*}
Taking into account of the relation (\ref{EqRelationOfF*}), we find that $\deg Z_q=4^{r+1}$, as desired.
\end{proof}

\section{Action of the Voisin map on $CH_0(X)_{hom}$}
The main purpose of this section is to prove Theorem B. Part of the argument will work for any $r$ and will be given in Section~\ref{SectionDecompOfTheAction}.

\subsection{Decomposition of the action of $\Psi_*$}\label{SectionDecompOfTheAction}
Let $X=F_r(Y)$ be as presented in Section~\ref{SectionVoisinExample}. Let $I_r = \{(x, x')\in X\times X: \dim (P_x\cap P_{x'})\geq r - 1 \}$.
\begin{theorem}\label{ThmDecOfActionOfPsiGeneralCase}
    There is an $(r+1)$-cocycle $\gamma\in CH^{ r+1}(X)$ lying in the image of the restriction map $CH^{ r+1}(\mathrm{Gr}(r+1,n+1))\to CH^{ r+1}(X)$ such that for any $z\in CH_0(X)$, we have in $CH_0(X)_{\mathbb Q}$
    \begin{equation}\label{EqDecOfTheActionOfPsi}
        \Psi_*z=(-2)^{r+1}z+\gamma\cdot I_{r*}(z).
    \end{equation}
\end{theorem}
\begin{proof}
    The argument is quite similar to that in the proof of~\cite[Theorem 4.16]{VoisinCitrouille}. Let $B=\mathbb PH^0(\mathbb P^n, \mathcal O(3))$ be the projective space parametrizing all cubic hypersurfaces in $\mathbb P^n$. For $f\in B$, we denote $Y_f\subset \mathbb P^9$ the hypersurface defined by $f$. There is a universal Fano variety of $r$-spaces defined as follows.
    \[
    \mathcal X:=\{(f,x)\in B\times \mathrm{Gr}(r+1,n+1): P_x\subset Y_f\}.
    \]
    The fiber of the projection map $\pi:\mathcal X\to B$ over a point $f\in B$ is the Fano variety of $r$-spaces of $Y_f$. The {Voisin} map $\Psi: X\dashrightarrow X$ can also be defined universally. In fact, we define the universal graph
    $\Gamma_{\Psi_{\mathrm{univ}}}$ as the closure of the set of pairs $(x, x')\in \mathcal X\times_B\mathcal X$ such that there exists a $(r+1)$-dimensional linear subspace $H_{r+1}\subset \mathbb P^n$ such that $H_{r+1}\cap Y_{\pi(x)}=2P_x+P_{x'}$. Denote
    \[
    \begin{tikzcd}
       \Gamma_{\Psi_{\mathrm{univ}}}\arrow[r, hookrightarrow, "i"] &\mathcal X\times_B\mathcal X\arrow{r}{q}\arrow{d}{p}& \mathrm{Gr}(r+1,n+1)\times \mathrm{Gr}(r+1,n+1)\\
        &B
    \end{tikzcd}
    \]
    the inclusion map and the natural projection maps. Then over a point $f\in B$ corresponding to a general smooth cubic hypersurface $Y_f$, the fiber $(p\circ i)^{-1}(f)$ is the graph of the {Voisin} map on $F_r(Y_f)$. There is a stratification of $\mathrm{Gr}(r+1,n+1)\times \mathrm{Gr}(r+1,n+1)$ given as follows. 
    \[
    \begin{tikzcd}
        \mathrm{Gr}(r+1,n+1)\times \mathrm{Gr}(r+1,n+1)\arrow[d, hookleftarrow]  \\
        I^G_1 \arrow[d, hookleftarrow] := \{(x,x'): P_x\cap P_{x'}\neq \emptyset\}\\
        I^G_2 \arrow[d, hookleftarrow]  :=\{(x,x'): P_x\cap P_{x'} \textrm{ contains a line }\}\\
        \vdots\arrow[d, hookleftarrow]\\
        I_r^G\arrow[d, hookleftarrow]:=\{(x,x'): P_x\cap P_{x'} \textrm{ contains an } (r-1) \textrm{-space}\}\\
        I_{r+1}^G=\Delta_G  :=\{(x,x)\}
    \end{tikzcd}
    \]
    {In other words, the subvariety $I_k^G$ is defined as} \[I_k^G=\{(x,x')\in \mathrm{Gr}(r+1, n+1)\times \mathrm{Gr}(r+1, n+1): \dim (P_x\cap P_{x'})\geq k-1\}.\]
    
    The map $q: \mathcal X\times_B\mathcal X\to \mathrm{Gr}(r+1, n+1)\times \mathrm{Gr}(r+1, n+1)$ { has a projective bundle structure over each stratum}. Precisely, let $d=\dim B$. Let $I_k^{\mathcal X}$ be the preimage of $I_k^G$ under the map $q: \mathcal X\times_B\mathcal X\to \mathrm{Gr}(r+1,n+1)\times \mathrm{Gr}(r+1,n+1)$. Over the open subset $I_{k}^G-I_{k+1}^G$, the map $q_{|I_{k}^\mathcal X}$ is a $\mathbb P^{d-\delta_k}$-bundle where $\delta_k=2h^0(\mathbb P^r,\mathcal O(3))-h^0(\mathbb P^{k-1},\mathcal O(3))$.  Let $i_k^{\mathcal X}: I_k^\mathcal X\hookrightarrow \mathcal X\times_B \mathcal X$ be the inclusion maps. The Chow ring of $\mathcal X\times_B\mathcal X$ is thus equal to
    \[\bigoplus_ih^i\cdot \left(\bigoplus_k \left(q_{|I_k^{\mathcal X}}\right)^*CH^*(I_k^G)\right),
    \]
    where $h=p^*\mathcal O_B(1)$.
    Now let us consider the Chow class of $\Gamma_{\Psi_{\mathrm{univ}}}$ in $CH^{N}(\mathcal X\times_B\mathcal X)$, where $N=(r+1)(n-r)-\binom{r+3}{3}$ is the dimension of $X$. By the construction, we have $\Gamma_{\Psi_{\mathrm{univ}}}\subset I_r^\mathcal X$, thus \[
    \Gamma_{\Psi_{\mathrm{univ}}}\in \bigoplus_i h^i\cdot \left(i_{r+1,*}^\mathcal X(q_{|\Delta_{\mathcal X}})^*CH^{-i}(\Delta_G) + i_{r*}^\mathcal X\left(q_{|I_k^{\mathcal X}}\right)^*CH^{m-i}(I_r^G)\right),
    \]
where $m$ is the relative dimension of $pr_1: I_r\to X$ that can be calculated as follows. Over a point $x\in X$, the fiber of $pr_1: I_r\to X$ is the set of points $x'\in X$ such that { the intersection $P_x\cap P_{x'}$ contains a $\mathbb P^{r-1}$ in $P_x$}. {The set of $\mathbb P^{r-1}$ contained in $P_x$ is a $\mathbb P^r$. Furthermore, for each given $\mathbb P^{r-1}\subset Y$, the set of $r$-spaces $P_{x'}\subset \mathbb P^n$ that contain the given $\mathbb P^{r-1}$ is a $\mathbb P^{n-r}$}, and the condition that $P_{x'}\subset Y$ is equivalent to saying that the defining equation of the residual quadric is {identically} zero, which gives $\frac{(r+1)(r+2)}{2}$ independent conditions. Taking everything into account, we find that $m=n-\frac{(r+1)(r+2)}{2}=r+1$.

    Now since $\Gamma_\Psi$ is a fiber of $p\circ i: \Gamma_{\Psi_{\mathrm{univ}}}\to B$, we conclude that 
    \[ \Gamma_\Psi=\Gamma_{\Psi_{\mathrm{univ}}|{X\times X}}\in (i_{r+1,*}^\mathcal Xq^*CH^0(\Delta_G) + i_{r*}^\mathcal X q^*CH^{r+1}(I_r^G))|_{X\times X}.
    \]
    Therefore, we can write 
    \begin{equation}\label{EqDecompOfGammaPsi}
        \Gamma_\Psi=\alpha \Delta_X+\delta,
    \end{equation} 
    where $\alpha\in \mathbb Z$ is a coefficient and $\delta\in (i_{{r}*}^\mathcal X q^*CH^{r+1}(I_r^G))|_{X\times X}$. {Write $\delta= i_{r*}^\mathcal X (q^*\delta_G)$ for some $\delta_G\in CH^{r+1}(I_r^G)$}. Notice that $I_r^G$ is a fiber bundle over $\mathrm{Gr}(r+1, n+1)$ whose fiber is a closed Schubert subvariety $\Sigma$ of $\mathrm{Gr}(r+1, n+1)$, that is defined as the closure of the set of $r$-spaces in $\mathbb P^n$ that intersects with a given $r$-space along a $(r-1)$-space. This fiber bundle has a universal cellular decomposition into affine bundles in the sense of~\cite[Example 1.9.1]{Fulton}. By~\cite[Example 19.1.11]{Fulton} and~\cite[Théorème 7.33]{Voisin}, we have
    \[
    CH^*(I_r^G)=CH^*(\mathrm{Gr}(r+1, n+1))\otimes CH^*(\Sigma).
    \]
    Therefore, we can write 
    \[
    \delta_G=\sum_{i=0}^{r+1} (p_1^*\alpha_i).\beta_i,
    \]
    where $\alpha_i\in CH^i(\mathrm{Gr}(r+1, n+1)$ and $\beta_i\in CH^{r+1-i}(\Sigma)$. {One easily checks that the morphism $pr_2^*: CH^*(Gr(r+1, n+1))\to CH^*(\Sigma)$ is surjective.} Hence, there exists $\gamma_i\in CH^{r+1-i}(\mathrm{Gr}(r+1, n+1))$ such that $\beta_i=pr_2^*\gamma_i$. Therefore, 
    {
    \begin{equation}\label{EqDecompOfDeltaG}
    \delta_G=\sum_{i=0}^{r+1}p_1^*\alpha_i.p_2^*\gamma_i.
    \end{equation}
    }
    {Now we prove
    \begin{lemma}\label{LmmCoeffAlpha}
        The coefficient $\alpha$ in the decomposition (\ref{EqDecompOfGammaPsi}) equals $(-2)^{r+1}$ .
    \end{lemma}
    }
    \begin{proof}
         For $\omega\in H^{N,0}(X)$, we have $\delta^*\omega=p_{2*}(\sum_i p_1^*([\alpha_i]\cup \omega)\cup \gamma_{i|X}\cup [I_r])$. Since $\omega$ is a top degree form,  $[\alpha_i]\cup \omega=0$ unless $i=0$ and thus, $\delta^*\omega=c[\gamma_{0|X}]\cup I_r^*\omega$, where $c$ is some constant number. Now let us prove that $I_r^*\omega=0$. Let $\mathcal P_{r,r-1}=\{(x,
\lambda)\in F_r(Y)\times F_{r-1}(Y): P_\lambda\subset P_x\}$ be the flag variety. Since $I_r=\mathcal P_{r,r-1}^t\circ \mathcal P_{r,r-1}$ as {correspondences}, we have $I_r^*\omega=\mathcal P_{r,r-1}^*(\mathcal P_{r,r-1*}(\omega))=0$ since $\mathcal P_{r,r-1*}\omega\in H^{N-1,-1}(F_1(Y))$. Therefore, $\Psi^*\omega = \alpha\omega + \delta^*\omega = \alpha\omega$. By Theorem A, we get $\alpha=(-2)^{r+1}$.
    \end{proof}
    Since $z\in CH_0(X)_{\mathbb Q}$ is a {$0$-cycle}, $\delta_*z=p_{2*}(\sum_i p_1^*(\alpha_i\cdot z)\cdot p_2^*\gamma_{i|X}\cdot I_r)=p_{2*}(p_1^*(\alpha_0\cdot z)\cdot p_2^*\gamma_{0|X}\cdot I_r)=\gamma\cdot I_{r*}z$ where $\gamma$ is some multiple of $\gamma_{0|X}$, which is an element in the image of the restriction map $CH^{r+1}(\mathrm{Gr}(r+1,n+1))\to CH^{r+1}(X)$. {Taking Lemma~\ref{LmmCoeffAlpha}, formula (\ref{EqDecompOfGammaPsi}) and formula (\ref{EqDecompOfDeltaG}) into account, we prove the formula 
    \begin{equation}\label{EqFormulaForPsiz}
        \Psi_*z=(-2)^{r+1}z+\gamma\cdot I_{r*}z,
    \end{equation}
    as desired.}
\end{proof}
{
The cycle gamma appearing in (\ref{EqFormulaForPsiz}) is a polynomial in the Chern classes $c_i$ of the tautological bundle of the Grassmannian, restricted to $X$. Let us now prove}
\begin{proposition}\label{PropVanishingOfc3}
    We have $c_{r+1}\cdot I_{r*}z$=0 for $z\in CH_0(X)_{hom}$.
\end{proposition}
\begin{proof}
    Let $H_{n-1}\subset \mathbb P^n$ be a general hyperplane. Let us define the Schubert variety { $\Sigma_{r+1}^{H_{n-1}}$ as}
    \[
    \Sigma_{r+1}^{H_{n-1}}:=\{y\in \mathrm{Gr}(r+1,n+1): P_y\subset H_{n-1}\}.
    \]
    Then $c_{r+1}$ is represented by the class of $\Sigma_{r+1}^{H_{n-1}}$. Let $x\in X$ be a general point and let 
    \[
    \Theta_{x}:=\{y\in X: \dim (P_y\cap P_x)\geq r-1
    \}.
    \]
    Then {by the definition of $I_r$, }$\Theta_x$ represents $I_{r *}x$. {By definition, the intersection  $\Sigma_{r+1}^{H_{n-1}} \cap \Theta_x$ is the set
    $$\{y\in X: P_y \subset H_{n-1}\cap Y, \,{\rm dim}\,P_x\cap P_y\geq r-1\}, $$
    which we can also rewrite, if $P_x$ is not contained in $H_{n-1}$, as
    $$\Sigma_{r+1}^{H_{n-1}} \cap \Theta_x=\{y\in X:\, Y\cap H_{n-1}\supset P_y\supset P_x\cap H_{n-1}\}.$$}
    Let $\Delta_x=P_x\cap H_{n-1}$. {Then $\Delta_x$ provides a point $\delta_x$ of the variety $F_{r-1}(Y\cap H_{n-1})$}. Let $i: F_r(Y\cap H_{n-1})\hookrightarrow X$ be the natural embedding map. Let $\mathcal P_{r,r-1}^{H_{n-1}}:=\{(y, \lambda)\in F_r(Y\cap H_{n-1})\times F_{r-1}(Y\cap H_{n-1}): P_\lambda\subset P_y\}$ be the {incidence} variety. By the above description of $ \Sigma_{r+1}^{H_{n-1}}\cap \Theta_x$, one finds that 
    \[
    c_{r+1}\cdot I_{r*}x=i_*(\mathcal P_{r,r-1}^{H_{n-1}*}({\delta_x})),
    \]
    where $\Delta_x$ is viewed as an element in $CH_0(F_{r-1}(Y\cap H_{n-1}))$. One can verify that $F_{r-1}(Y\cap H_{n-1})$ is a Fano manifold and thus, $CH_0(F_{r-1}(Y\cap H_{n-1})) = \mathbb Z$. Therefore, the Chow class of $c_{r+1}\cdot I_{r*}x$ does not depend on $x$. That is, for any $z\in CH_0(X)_{hom}$, we get $c_{r+1}\cdot I_{r*}z=0$ in $CH_0(X)$, as desired. 
\end{proof}

\begin{proposition}\label{PropVanishingOfc2}
    If $CH_1(F_{r-1}(Y))_{\mathbb Q}$ is trivial, then $c_r\cdot I_{r*}(z)=0 $ in $CH_1(X)$ for any $z\in CH_0(X)_{hom}$.
\end{proposition}

\begin{proof}
    We only write down the case $r=2$, {for which the assumption of Proposition~\ref{PropVanishingOfc2} will be proved in the next section}. The general case is similar. Let $H_7$ be a $7$-dimensional linear subspace in $\mathbb P^9$. In the Grassmannian $\mathrm{Gr}(3,10)$, the class $c_2$ is represented by the subvariety $\Sigma_2^{H_7}$ defined as 
    \[
    \Sigma_2^{H_7}:=\{y\in \mathrm{Gr}(3,10): P_y\cap H_7 \textrm{ contains a line }{ \Delta_{y,H}}\}.
    \]
    For general $x\in X$, the plane $P_x$ intersects $H_7$ at a single point {$z_x$}, and the class $I_{2,*}x$ is represented by the following subvariety $\Theta_x$ in $X$ defined as
    \[
    \Theta_x:=\{y\in X: P_y\cap P_x \textrm{ contains a line }{\Delta_{xy}}\}.
    \]
    Let $H_6\subset H_7$ be a linear subspace of dimension $6$ not containing the point ${z_x}$.
    \begin{lemma}\label{LmmEquivOfC2I2z}
        { For general $x\in X$, and any $y\in \Theta_x$, we have $z_x\in \Delta_{xy}$.}
    \end{lemma}
    \begin{proof}
        Let $P_y$ be a plane in $Y$ that intersects $P_x$ along the line {$\Delta_{xy}$} and intersects $H_7$ along the line {$\Delta_{y, H}$}. Then {$\Delta_{xy}$} and {$\Delta_{y, H}$} must intersect, since they are two lines in a projective plane; and the intersection point of {$\Delta_{xy}$} and {$\Delta_{y, H}$} must be ${z_x}$ since ${z_x}$ is the only intersection point of $P_x$ and $H_7$. {Therefore, $z_x\in \Delta_{xy}$, as desired.} 
    \end{proof}
        { Let 
    $\Sigma^{H_6}_1:=\{y\in \mathrm{Gr}(3,10): P_y\cap H_6\neq \emptyset \}.$
    Let $\Xi_x$ be the variety of points $y\in X$ such that $P_y\cap P_x$  contains a line $\Delta_{xy}$ containing ${z_x}$. By Lemma~\ref{LmmEquivOfC2I2z}, we now conclude that
    \begin{equation}\label{EqEquivOfC2I2z}
        \Sigma_2^{H_7}\cap \Theta_x
 = \Sigma_1^{H_6} \cap \Xi_x.
    \end{equation}
    Indeed, by Lemma~\ref{LmmEquivOfC2I2z}, $\Sigma_2^{H_7}\cap \Theta_x$ contains all the points $y\in X$ such that $P_y\cap P_x$ contains a line $\Delta_{xy}$ containing ${z_x}$  and such that $P_y\cap H_7$ contains a line $\Delta_{y, H}$. This variety coincides with $\Sigma_1^{H_6}\cap \Xi_x$ since knowing that ${z_x}\in P_y$ and that $z_x\not\in H_6$, $P_y\cap H_7$ contains a line if and only if $P_y\cap H_6$ is nonempty.
    }
    
    Let $\mathcal P_{2,1}:=\{(x', [l])\in F_2(Y)\times F_1(Y): l\subset P_{x'}\}$ be the flag variety. Let $\Delta_{P_x, {z_x}}^\vee=\{[l]\in F_1(Y): {z_x}\in l\subset P_x\}$. { Then $\Delta_{P_x, z_x}^\vee$ provides a class $\delta_{P_x, z_x}^\vee\in CH_1(F_1(Y))$. The equation (\ref{EqEquivOfC2I2z}) shows that $c_2\cdot I_{2,*}x=c_1\cdot \mathcal P_{2,1}^*(\delta_{P_x,{z_x}}^\vee)$}. Therefore, for $z\in CH_0(X)_{\mathbb Q,hom}$, we have $c_2\cdot I_{2,*}z=c_1\cdot \mathcal P_{2,1}^*(Z)$ {for some} $Z\in CH_1(F_1(Y))_{\mathbb Q, hom}$. {But we have $CH_1(F_1(Y))_{\mathbb Q, hom}=0$ by Theorem~\ref{ThmChowOneOfF1} below (or by assumption for $r>2$)}. Hence, $c_2\cdot I_{2,*}z=0\in CH_1(X)_{\mathbb Q}$, as desired.
\end{proof}

\subsection{Triviality of $CH_1(F_1(Y))_{\mathbb{Q}}$}

Let $Y\subset \mathbb P^9$ be a cubic eightfold. It has been established that $H^{p,q}(F_1(Y)) = 0$ for $p \leq 1$ and $p \neq q$~\cite{DebarreManivel}, indicating that the coniveau of $F_1(Y)$ is at least $2$. According to the generalized Bloch conjecture, this suggests that $CH_i(F_1(Y))_{\mathbb{Q}, hom} = 0$ for $i \leq 1$. {This section is devoted to proving this statement, namely
\begin{theorem}\label{ThmChowOneOfF1}
The Chow group $CH_1(F_1(Y))_{\mathbb{Q}}$ is isomorphic to $\mathbb{Q}$.
\end{theorem}
}
\subsubsection{Lines of lines}
Let $P$ be a plane contained in $Y$ and let $x\in P$ be a point. Let $\Delta_{P,x}^\vee$ be the variety of lines in $P$ passing through $x$. Viewed as a $1$-cycle of $F_1(Y)$, it is clear that the Chow class of $\Delta_{P,x}^\vee$ does not depend on the choice of $x\in P$ and we let $\Delta_P^\vee\in CH_1(F_1(Y))$ denote the Chow class of $\Delta_{P,x}^\vee$ for some (and thus any) $x\in P$.

\begin{proposition}\label{PropTrianglePlanesConstant}
    Let $L\cong\mathbb P^3\subset \mathbb P^9$ be a {$3$}-dimensional linear subspace whose intersection with $Y$ contains three planes $P_1, P_2$ and $P_3$ (not necessarily distinct). Then the Chow class
    \[
\Delta_{P_1}^\vee+\Delta_{P_2}^\vee+\Delta_{P_3}^\vee\in CH_1(F_1(Y))
    \]
    does not depend on the choice of $L$.
\end{proposition}
\begin{remark}
    The linear spaces $L\cong\mathbb P^3$ whose intersection with $Y$ is the union of three planes form a projective variety of general type, which is not rationally connected. To prove the proposition, we consider a larger space, namely the variety of cubic surfaces in $Y$ which are cones and prove that this space is rationally connected. 
\end{remark}

\begin{proof}[Proof of Proposition~\ref{PropTrianglePlanesConstant}]
    Let $\mathcal M$ be the space of cubic surfaces in $Y$ which are cones. An element of $\mathcal M$ is of the form $(S, y_0)$ where $S$ is a cubic surface that is a cone from a vertex $y_0$. If $P_1, P_2, P_3$ are planes as in the Proposition, and $y_0$ is a point in the intersection of the three planes, then $P_1\cup P_2\cup P_3$ is a cone with vertex $y_0$, hence $(P_1\cup P_2\cup P_3, y_0)$ is an element of $\mathcal M$. Consider the incidence variety 
\[
\mathcal E=\{((S,y_0), l)\in \mathcal M\times F_1(Y): l \textrm{ is a line in } S \textrm{ passing through } y_0\}.
\]
Then $\Delta_{P_1}^\vee+\Delta_{P_2}^\vee+\Delta_{P_3}^\vee$ is the $1$-cycle given by $\mathcal E_*(P_1\cup P_2\cup P_3, y_0)$. Therefore, to prove the Proposition, it suffices to show that $\mathcal M$ is rationally connected.

Let $\pi: \mathcal M\to Y$ be the rational map that sends a cubic surface $(S, y_0)$ that is a cone onto its vertex $y_0$. The fiber of $\pi$ over $y_0$ parametrizes the cubic surfaces in $Y$ that are cones with vertex $y_0$. Then $S$ is contained in the singular hyperplane section $Y_{y_0}:=Y\cap \bar T_{Y,y_0}$ where $\bar T_{Y,y_0}$ is the projective tangent space of $Y$ at the point $y_0$. We may take $y_0$ a general point so that the cubic hypersurface $Y_{y_0}$ of dimension $7$ has an ordinary double point at $y_0$, and does not contain any $\mathbb P^3$, and so that no $3$-dimensional linear subspaces in $\mathbb P^9$ passing through $y_0$ is contained in $Y$. The last condition is satisfied since the $3$-dimensional linear subspaces contained in $Y$ form a divisor in $Y$, as can be shown by a simple dimension counting argument. 

\begin{lemma}\label{LmmFiberOfPiOverGeneralPoint}
    Under the assumptions above on $y_0$, the fiber of $\pi$ over $y_0$ is in bijection with the set of planes contained in the Hessian quadric $Q_{y_0}$ of $Y_{y_0}$ at the point $y_0$.
\end{lemma}
\begin{proof}
    A cubic surface  $S=\mathbb P^3\cap Y_{y_0}$ in $Y_{y_0}$ passing through $y_0$ is a cone  with vertex $y_0$ if and only if the equation $f_{y_0}$ defining $Y_{y_0}$ has vanishing Hessian, which means that the Hessian quadric vanishes on the  tangent space $T_{\mathbb P^3, y_0}$ .
\end{proof}

One knows that the variety of planes in a quadric sixfold is a connected Fano manifold. Thus, by Lemma~\ref{LmmFiberOfPiOverGeneralPoint}, the map $\pi$ is a fibration whose base and general fiber are rationally connected. Therefore, the total space $\mathcal M$ is rationally connected \cite{GraberHarrisStarr}, as desired.
\end{proof}

\begin{definition}\label{DefOfDelta}
    Let $\Delta^\vee$ be an element of $CH_1(F_1(Y))_{\mathbb Q}$ defined by $\frac13(\Delta_{P_1}^\vee+\Delta_{P_2}^\vee+\Delta_{P_3}^\vee)$ where $P_1, P_2, P_3$ are the planes as in Proposition~\ref{PropTrianglePlanesConstant}.
\end{definition}

\begin{corollary}
    Let $L\subset Y$ be a $3$-dimensional linear space in $Y$. For any plane $P\subset L$, we have $\Delta^\vee_P=\Delta^\vee$ in $CH_1(F_1(Y))_{\mathbb Q}$.
\end{corollary}
\begin{proof}
    This is because the triple $(P,P,P)$ satisfies Proposition~\ref{PropTrianglePlanesConstant} and thus $3\Delta^\vee_P=3\Delta^\vee$, from which we conclude.
\end{proof}
\begin{corollary}\label{CorImOfFanoOf3SpacesDoesNotDependOn3Spaces}
    Let $L\subset Y$ be a $3$-dimensional linear space in $Y$. Then the image of the natural morphism 
    \[
    CH_1(F_1(L))_{\mathbb Q}\to CH_1(F_1(Y))_{\mathbb Q}
    \]
    is of dimension $1$ and is generated by $\Delta^\vee$. In particular, the image does not depend on the choice of $L\in F_3(Y)$.
\end{corollary}

\begin{remark}
    We will see, in Proposition~\ref{PropLinesOfLinesAreConstant}, that for any $P\subset Y$, we have $\Delta^\vee_P=\Delta^\vee$.
\end{remark}

\subsubsection{The geometry of $F_1(Y)$  and its $1$-cycles}
The following Lemma is proved by an easy dimension count.

\begin{lemma}\label{LmmP3CoverADivisor}
    Assume $Y$ is general. The linear subspaces  $\mathbb P^3\subset Y$ cover a divisor $D$  in $Y$.
\end{lemma}
\begin{corollary}\label{CorLineIntersectFiniteP3}
    A general  line $\Delta\subset Y$ meets finitely many $\mathbb P^3\subset Y$.
\end{corollary}
Indeed, if $\Delta \not \subset D$, the  $\mathbb P^3\subset Y$ intersecting $\Delta$ are in bijection with the intersection points of $\Delta$ and $D$.
Let $F_{3,1}\subset  F_3(Y)\times F_1(Y)$ be the set of pairs $(x,s)$ such that $P_x\cap \Delta_s\not=\emptyset$. Let $p: F_{3,1}\rightarrow F_3(Y)$ (resp. $q:F_{3,1}\rightarrow F_1(Y)$)  be the first (resp. second) projection. The second projection $q:F_{3,1}\rightarrow F_1(Y)$ is dominant by Lemma~\ref{LmmP3CoverADivisor} and  generically finite by  Corollary~\ref{CorLineIntersectFiniteP3}.

\begin{lemma}\label{LmmMorphismF1F31}
     (i) The morphism  $q^*: CH_1(F_1(Y))\rightarrow CH_1(F_{3,1})$ is injective.\\
(ii) The morphism  $p_*\circ q^*: CH_1(F_1(Y))_{hom}\rightarrow CH_1(F_3(Y))_{hom}$ is  zero.
\end{lemma} 
\begin{proof} (i) follows from the fact that $q$ is dominant. For (ii), we observe that there is a natural birational map 
\[ \pi: \mathcal P_3\times_{Y} \mathcal P_1\to F_{3,1},\]
where $$\mathcal P_3\subset F_3(X)\times Y,\,\, \mathcal P_1\subset F_1(X)\times Y$$ are the incidence correspondences. The morphism $\pi$ is commutative with the projections to $F_3(Y)$ and to $F_1(Y)$. By direct calculations,
\begin{align*}
    \mathcal P_3^*\circ \mathcal P_{1*} & = p_*\circ \pi_* \circ \pi^* \circ q^* \\
    & = p_* \circ q^*
\end{align*}
Thus  the  morphism  $p_*\circ q^*$ factors through the morphism $(\mathcal P_1)_* :  CH_1(F_1(Y))_{hom}\rightarrow CH_2(Y)_{hom}$, which is zero because  $CH_2(Y)_{hom}=0$ by~\cite{Otwinowska}.
\end{proof}

The variety $F_{3,1}$ admits a rational map $f$  to the projective bundle $\mathcal{P}_5$ over $F_3(Y)$ whose fiber over $x\in F_3(Y)$ is the set of $\mathbb{P}^4$ containing $P_x$. The variety $\mathcal P_5$ is of dimension $9$ since it is a $\mathbb P^5$-bundle over a $4$-fold $F_3(Y)$. To a pair $(P_x,\Delta_t)$ of a $\mathbb P^3$ and a line which intersect, this map associates $\langle P_x,\Delta_t\rangle$. We introduce now a desingularization $\tau: \widetilde{F_{3,1}} \rightarrow  F_{3,1}$ on which $f$ desingularizes to a morphism $\tilde{f}: \widetilde{F_{3,1}}\rightarrow \mathcal{P}_5$.  We observe now the following:  for each $x\in F_3(Y)$ and $4$-dimensional space $P'_x$ containing $P_x$, the intersection $P'_x\cap Y$ is the union  $P_x\cup Q_x$ where $Q_x$ is a $3$-dimensional quadric intersecting $P_x$ along a $2$-dimensional quadric.  The general fiber  of  $\tilde{f}$ over $(x, P'_x)$ is birational to the family of lines in the $3$-dimensional quadric $Q_x$. We are now going to  prove
\begin{proposition}\label{ThmF31ToF1}
    Let $z\in CH_1(\widetilde{F_{3,1}})_{\mathbb Q}$ be a $1$-cycle such that $\tilde f_*(z)=0\in CH_1(\mathcal P_5)$. Then $\tilde q_*(z)=\alpha \Delta^\vee\in CH_1(F_1(Y))_{\mathbb Q}$ for some $\alpha\in \mathbb Q$.
\end{proposition}
We will use for this the following  general result of Bloch-Srinivas type~\cite{VoisinCitrouille}:
\begin{lemma}\label{ThmSpreading}
    Let $f: Z\to B$ be a surjective projective morphism between algebraic varieties. Let $B^0$ be an open dense subset of $B$ such that\\
    (a) $Z-f^{-1}(B^0)$ is of codimension at least $2$, and that\\
    (b) Every fiber of $f$ over $B^0$ has trivial $CH_0$, i.e., $CH_{0, hom} =  0$.\\
    Let $z\in CH_1(Z)_{\mathbb Q}$ be a $1$-cycle in $Z$ such that $f_*(z)=0\in CH_1(B)_{\mathbb Q}$. Then $z$ is supported on the fibers of $f$ over $B^0$. More precisely, there are points $b_1, \ldots, b_r\in B^0$, such that $z$ is $\mathbb Q$-rationally equivalent to a $1$-cycle supported on $f^{-1}(b_1)\cup\ldots\cup f^{-1}(b_r)$.
\end{lemma}

\begin{proof}
    The case when $\dim B=1$ is rather trivial, so let us assume $\dim B\geq 2$. Let $W\subset Z$ be a multisection of degree $N$ of $f:Z\to B$. Let $f_W: W\to B$ be the restriction of $f$ to $W$. Since $Z-f^{-1}(B^0)$ is of codimension at least $2$, by Chow moving lemma, we may assume that $z$ is supported on $f^{-1}(B^0)$. We may also assume that none of the components of $z$ lies entirely in the branched locus of $f_W$. Write $z=z_1+z_2+\ldots+z_s-z_{s+1}-\ldots -z_t$ with $z_i$ irreducible curves in $f^{-1}(B^0)$. Consider $B_{z_i}=f(z_i)\subset B^0$. We may assume $B_{z_i}$ is dimension $1$ since otherwise the component $z_i$ is supported on the fibers of $f$. Let $f_i: f^{-1}(B_{z_i})\to B_{z_i}$ be the restriction of $f$. Since $\dim B\geq 2$, we may further assume, by Chow moving lemma, that the map $f_i: f^{-1}(B_{z_i})\to B_{z_i}$ restricted to $z_i$ is a birational map. Since every fiber of $f$ over $B^0$ has trivial $CH_0$, the $1$-cycle $z_i-\frac1N f^{-1}_W(B_{z_i})$ in $CH_1(f^{-1}(B_{z_i}))$ restricts to $0$ for the general fiber of $f_i: f^{-1}(B_{z_i})\to B_{z_i}$, so by the Bloch-Srinivas construction \cite{Voisin}, the cycle $z_i-\frac1N f^{-1}_W(B_{z_i})$ is supported on the fibers of $f_i$. Summing up the components, we find that the $1$-cycle
    $z-\frac1N f^*_Wf_*(z)\in CH_1(Z)_{\mathbb Q}
    $
    is supported on the fibers of $f$ over $\bigcup_i B_{z_i}\subset B^0$. However, $f_*(z)=0$ by assumption. Hence, $z\in CH_1(Z)_{\mathbb Q}$ is supported on the fibers of $f$ over $B^0$.
\end{proof}
We are going to apply Lemma~\ref{ThmSpreading} to $\tilde f: \widetilde{F_{3,1}}\to \mathcal P_5$ and to the following open set $\mathcal P_5^0$ defined as the set parametrizing the pairs $(L_3, P_4)$ such that \\
    (a) either $Q$ is smooth,\\
    (b) or $Q$ is singular at only one point $y$ and $L$ does not contain $y$.\\
It is clear that the fibers of $\tilde f$ over $\mathcal P^0_5$ are $CH_0$ trivial. Our first step is thus to check assumption (a) in Lemma~\ref{ThmSpreading}. We prove by a case by case analysis the following
\begin{lemma}\label{LmmCodim}
    $\tilde f^{-1}(\mathcal P_5^0)\subset \widetilde{F_{3,1}}$ has complement of codimension $\geq 2$.
\end{lemma}
\begin{proof}
    The complement $R$ of $\mathcal P^0_5$ in $\mathcal P_5$ is stratified by the following subsets $R_4, R_3, R_2, R_1$, where $R_i$ parametrizes $(L, P_4)\in R$ such that the residual $Q$ is of rank $i$.
        
         \textbf{Analysis of $R_4$.} The stratum $R_4$ parametrizes $(L, P_4)\in \mathcal P_5$ such that  $Q$ is singular at one point (equivalently, the rank of $Q$ is $4$) and that $L$ contains the singular point of $Q$. In this case, $Q$ is a cone from its singular point over a smooth quadric surface. 
        \begin{sublemma}\label{LmmLocusOfSimpleSingularQuadric}
            For any $3$-dimensional linear subspace $L$ contained in $Y$, the set of $4$-dimensional subspaces $P_4\subset \mathbb P^9$ containing $L$ such that $Q$ has a single singular point lying on $L$ has codimension at least $2$ in $\mathbb P^5=$\{$4$-dimensional subspaces $P_4\subset \mathbb P^9$ containing $L$\}.
        \end{sublemma}
        \begin{proof}
            Without loss of generality, let us assume $L=\{(x_0:\ldots: x_3: 0:\ldots: 0)\}\subset \mathbb P^9$. The fact that $L\subset Y$ implies that the defining equation of $Y$ is of the form
            \[
            f(x_0,\ldots, x_9)=\sum_{i=4}^9 x_iQ_i(x_0, \ldots, x_9),
            \]
            where $Q_i$ is a quadratic polynomial in the variables $x_0, \ldots, x_9$, for each $i\in \{4,\ldots, 9\}$. Let $\mathbf a:=(a_0: a_1: \ldots: a_5)\in \mathbb P^5$. Then $\mathbf a$ determines a dimension $4$ linear subspace $P_4$ containing $L$ as 
            \[
            P_4=\{(x_0: x_1: x_2: x_3: ta_0, ta_1:\ldots: ta_5): (x_0: x_1: x_2: x_3: t)\in \mathbb P^4\}.
            \]
            The corresponded residual quadric hypersuface $Q_{\mathbf a}$ is thus defined by 
            \[
            Q_{\mathbf a}(x_0, x_1, x_2, x_3, t)=\sum_{i=4}^9 a_{i-4} Q_i(x_0, x_1, x_2, x_3, ta_0, ta_1,\ldots, ta_5).
            \]
            We can identify the quadratic form $Q_{\mathbf a}(x_0, x_1, x_2, x_3, t)$ with a $5\times 5$ symmetric matrix that we still denote as $Q_{\mathbf a}$. Viewed as a function of $\mathbf a\in \mathbb P^5$, the matrix $Q_{\mathbf a}$ is a section of $\mathrm{Sym}^2(\mathcal O(1)_{\mathbb P^5}\oplus \mathcal O_{\mathbb P^5}^4)\otimes \mathcal O_{\mathbb P^5}(1)$, and thus the locus of $\mathbf a\in\mathbb P^5$ where $Q_{\mathbf a}$ degenerates is a degree $7$ hypersurface in $\mathbb P^5$. On the other hand, the quadratic polynomial of $Q_{\mathbf a}\cap L$ is $Q_{\mathbf a}(x_0, x_1, x_2, x_3, t=0)$, which, viewed as a $4\times 4$ symmetric matrix varing with $\mathbf a$, is a section of $\mathrm{Sym}^2( \mathcal O_{\mathbb P^5}^4)\otimes \mathcal O_{\mathbb P^5}(1)$, so that the degenerate locus of  $Q_{\mathbf a}\cap L$ is of degree $4$. Therefore, there exists a point $\mathbf a\in \mathbb P^5$ such that $Q_{\mathbf a}$ is singular whereas $Q_{\mathbf a}\cap L$ is smooth. Therefore, the locus of $\mathbf a\in\mathbb P^5$ such that $Q_{\mathbf a}$ is singular at exactly one point and that this point lies in $L$ is strictly contained, as a closed subset, in the locus where $Q_{\mathbf a}$ is singular at only one point. The latter locus is of codimension $1$ in $\mathbb P^5$ since it is defined by the vanishing of the determinant of the matrix $Q_{\mathbf a}$. Hence, the set of $4$-dimensional subspaces $P_4\subset \mathbb P^9$ containing $L$ such that $Q$ has a single singular point lying on $L$ has codimension at least $2$ in $\mathbb P^5$, as desired.
        \end{proof}
        \begin{sublemma}\label{LmmFiberOverASimpleSingularQuadric}
            The fiber of $\tilde f: \widetilde{F_{3,1}}\to \mathcal P_5$ over an element $(L,P_4)$ such that the corresponding residual quadric $Q$ has only one singularity is isomorphic to the union 
            \[
            \mathcal P_{2, 1}^\vee\cup \mathcal P_{2,2}^\vee,
            \]
            where $\mathcal P_{2,i}^\vee$ ($i = 1, 2$) is a $\mathbb P^2$-bundle over $\mathbb P^1$, and $\mathcal P_{2, 1}^\vee\cap \mathcal P_{2, 2}^\vee$ is a smooth quadric surface.
        \end{sublemma}
        \begin{proof}
            The singular quadric hypersurface $Q$ is a cone with vertex at its only singular point $y$ over a smooth quadric surface $Q'$. Each line in $Q'$, together with $x$ determines a plane, and every line in $Q$ lies in some of these planes. There are two pencils of lines in $Q'$, each of which induces a pencil of planes in $Q'$. Each plane contains a $\mathbb P^2$ of lines. Hence, the set of lines in $Q$ is the union of two $\mathbb P^2$-bundles over $\mathbb P^1$. The intersection of these $\mathbb P^2$-bundles is the set of lines in $Q$ that pass through the singular point $y$, which is in turn isomorphic to $Q'$. 
        \end{proof}
        Sublemma~\ref{LmmLocusOfSimpleSingularQuadric} implies that $R_4\subset \mathcal P_5$ is of codimension at least $2$. Sublemma~\ref{LmmFiberOverASimpleSingularQuadric} shows that the dimension of the fiber of $\tilde f$ over $R_4$ is $3$, the same with that of the general fiber of $\tilde f: \widetilde{F_{3,1}}\to \mathcal P_5$. Hence, the codimension of $\tilde f^{-1}(R_4)$ is at least $2$. This complete the analysis of $R_4$.
        
        \textbf{Analysis of $R_3$.} The stratum $R_3$ parametrizes $(L, P_4)\in \mathcal P_5$ such that $Q$ is of rank $3$. In this case, $Q$ is singular along a line and is a cone with its vertex one of its singular point over a quadric cone surface.
        \begin{sublemma}\label{LmmCodimOfStrata}
            The stratum $R_3, R_2, R_1$ is of codimension $3, 6, 10$, respectively, in $\mathcal P_5$. In particular, $R_1=\emptyset$.
        \end{sublemma}
        \begin{proof}
            Let $\mathcal E_4$ and $\mathcal Q_6$ be the tautological subbundle, respectively, tautological quotient bundle over $F_3(Y)$. Then $\mathcal P_5$ is nothing but the projectivization of $\mathcal Q_6$. Let $\mathcal H$ be the Hopf bundle of $\mathbb P(\mathcal Q_6)$, which is a subbundle of $\pi^*\mathcal Q_6$. Let $\mathcal F_5$ be the kernel of the composite of the canonical maps $V_{10}\to \pi_*\mathcal Q_6\to \pi_*\mathcal Q_6/\mathcal H$, where the first map is the universal quotient map of bundles of $\mathbb P(\mathcal Q_6)$. Then there is a natural exact sequence of vector bundles on $\mathbb P(\mathcal Q_6)$:
            \[
            0\to \pi^*\mathcal E_4\to \mathcal F_5\to \mathcal H\to 0.
            \]
            The defining equation of $Y\subset \mathbb P^9$ gives a section of the bundle $\mathrm{Sym}^3\mathcal F_5^*$ whose image in $\pi^*\mathrm{Sym}^3\mathcal E_4^*$ vanishes since the fibers of $\mathbb P(\mathcal E_4)$ are by definition $3$-dimensional linear subspaces in $Y$. Hence, the universal residual quadric hypersurface is defined by a section of the bundle $\mathrm{Sym}^2\mathcal F_5^*\otimes \mathcal H^*$. Notice that $\mathrm{Sym}^2\mathcal F_5^*\otimes \mathcal H^*$ is generated by global sections. Since $Y$ is general, the locus where such a quadratic form is of rank $\leq r$ is $\binom{6-r}{2}$, for $r\in\{0,\ldots, 4\}$, as desired.
        \end{proof}
        \begin{sublemma}\label{LmmFiberOverR3}
            The fiber of $\tilde f: \widetilde{F_{3,1}}\to \mathcal P_5$ over a point $(L, P_4)$ in $R_3$ is of dimension $3$.
        \end{sublemma}
        \begin{proof}
            The residual quadric hypersurface $Q$ is singular along a line and is a cone from one of its singular point over a quadric cone surface. By the same argument as in Sublemma~\ref{LmmFiberOverASimpleSingularQuadric}, the variety of lines in $Q$ is isomorphic to a $\mathbb P^2$-bundle over a smooth conic, thus has dimension $3$.
        \end{proof}
        By Sublemma~\ref{LmmCodimOfStrata} and Sublemma~\ref{LmmFiberOverR3}, the codimension of $\tilde f^{-1}(R_3)$ is $3$. This completes the analysis of $R_3$.
        
        \textbf{Analysis of $R_2$.} The stratum $R_2$ parametrizes $(L, P_4)\in \mathcal P_5$ such that $Q$ is of rank $2$. In this case, $Q$ is a union of two $3$-dimensional linear subspaces intersecting along a plane. It is clear that in this case, the variety of lines in $Q$ is isomorphic to the union of two $\mathrm{Gr}(2,4)$ intersecting along a $\mathbb P^2$. Hence, the dimension of fibers of $\tilde f$ over $R_2$ is $4$. By Sublemma~\ref{LmmCodimOfStrata}, the codimension of $\tilde f^{-1}(R_2)$ is $5$. This completes the analysis of $R_2$.
        
        \textbf{Analysis of $R_1$.} The stratum $R_1$ parametrizes $(L, P_4)\in \mathcal P_5$ such that $Q$ is of rank $1$. By Sublemma~\ref{LmmCodimOfStrata}, this case does not happen.
        
        The case where $Q$ is of rank $0$ does not happen, since in this case, $P_4$ is contained in $Y$, which cannot happen if $Y$ is a general cubic eightfold.
\end{proof}

\begin{lemma}\label{LmmSurjectivity}
    For any $(L, P_4)$ in $\mathcal P_5^0$, with associated $3$-dimensional quadric $Q$, the natural morphism 
\[
CH_1(F_1(L\cap Q))_{\mathbb Q}\to CH_1(F_1(Q))_{\mathbb Q}
\]
is surjective.
\end{lemma}

\begin{proof}
Recall that after the proof of Lemma~\ref{ThmSpreading}, we have divided the points in $\mathcal P^0_5$ into two cases.

\textbf{Case (a).} In this case, $CH_1(F_1(Q))_{\mathbb Q}=\mathbb Q$ since $F_1(Q)$ is a connected Fano manifold. Therefore, $CH_1(F_1(L\cap Q))_{\mathbb Q}\to CH_1(F_1(Q))_{\mathbb Q}$ has to be surjective. 

\textbf{Case (b).} In this case, $Q$ is a cone from its singular point $y$ over a smooth quadric surface $Q'$. By Lemma~\ref{LmmFiberOverASimpleSingularQuadric}, the Fano variety of lines $F_1(Q)$ is isomorphic to $\mathcal P_{2,1}^\vee\cup \mathcal P_{2,2}^\vee$ where $\mathcal P_{2,i}^\vee$ is a $\mathbb P^2$-bundle over $\mathbb P^1$ and the $\mathcal P_{2,1}^\vee\cap \mathcal P_{2,2}^\vee$ is a smooth quadric surface $Q'$. The variety $F_1(L\cap Q)$ is a disjoint union of two lines $\ell_1$ and $\ell_2$, each representing a pencil of lines on the smooth quadric surface $L\cap Q$. Hence, $CH_1(F_1(L\cap Q))$ is generated freely by two cycles $z_1$ and $z_2$ where $z_i$ represents $\ell_i$ for $i=1,2$. Let us analyze $CH_1(F_1(Q))$. Let us consider $\mathcal P_{2,1}^\vee$ which is a $\mathbb P^2$-bundle over $\ell_1$. Let $p_1: \mathcal P_{2,1}^\vee\to \ell_1$ be the corresponding map, and let $h$ be the auti-tautological class of this projective bundle. The Chow group $CH_1(\mathcal P_{2,1}^\vee)$ is then generated by $h^2\cdot p_1^*(\ell_1)$ and $h\cdot p_1^*(pt)$. The image of the class $h^2\cdot p_1^*(\ell_1)$ in $CH_1(F_1(Q))$ is the same as the class of $z_1$ in $CH_1(F_1(Q))$ and the image of the class $h\cdot p_1^*(pt)$ is the as the class of $z_2$ in $CH_1(F_1(Q))$. Hence, the image of $CH_1(F_1(L\cap Q))\to CH_1(F_1(Q))$ contains the image of $CH_1(\mathcal P_{2,1}^\vee)\to CH_1(F_1(Q))$. Similarly, the image of $CH_1(F_1(L\cap Q))\to CH_1(F_1(Q))$ contains the image of $CH_1(\mathcal P_{2,2}^\vee)\to CH_1(F_1(Q))$. Therefore, $CH_1(F_1(L\cap Q))\to CH_1(F_1(Q))$ is surjective.
\end{proof}

We finally conclude the proof of Proposition~\ref{ThmF31ToF1}.
\begin{proof}[Proof of Proposition~\ref{ThmF31ToF1}]
    Let $z\in CH_1(\widetilde{F_{3,1}})_{\mathbb Q}$ be a $1$-cycle such that $\tilde f_*(z)=0\in CH_1(\mathcal P_5)_{\mathbb Q}$. Lemma~\ref{LmmCodim} and the discussion above shows that Lemma~\ref{ThmSpreading} can be applied to the map $\tilde f:\widetilde{F_{3,1}}\to \mathcal P_5$ with $\mathcal P_5^0$ the open dense subset, so that we conclude that $z$ is supported on fibers of $\tilde f$ over $\mathcal P_5^0$. Write $z=z_1+\ldots +z_r\in CH_1(\widetilde{F_{3,1}})_{\mathbb Q}$ where $z_i$ is supported on $F_1(Q_i)$ with $Q_i$ is a residual quadric hypersurface coming from a pair $(L_i, P_{4,i})$ in $\mathcal P_5^0$. By Lemma~\ref{LmmSurjectivity}, for each $i$, the the natural morphism $CH_1(F_1(Q_i\cap L_i))_{\mathbb Q}\to CH_1(F_1(Q_i))_{\mathbb Q}$ is surjective. Hence, $\tilde q_*(z)$ lies in the sum of the images of $CH_1(F_1(L_i))_{\mathbb Q}\to CH_1(F_1(Y))_{\mathbb Q}$. By Corollary~\ref{CorImOfFanoOf3SpacesDoesNotDependOn3Spaces}, the cycle $\tilde q_*(z)$ is a multiple of $\Delta^\vee$, as desired.
\end{proof}

We next prove that all $\Delta_P^\vee$ has the same Chow class in $CH_1(F_1(Y))_{\mathbb Q}$.

\begin{proposition}\label{PropLinesOfLinesAreConstant}
    Let $P\subset Y$ be a plane. Then $\Delta_P^\vee=\Delta^\vee$ in $CH_1(F_1(Y))_{\mathbb Q}$.
\end{proposition}

\begin{proof}
    Since the rational equivalence class is a countable union of closed algebraic subsets~\cite[Section 1.1.1]{VoisinCitrouille}, we may assume $P\subset Y$ is general. Let $P'\subset Y$ be another plane, general among all planes in $Y$ intersecting $P$ along a line $l$. By Corollary~\ref{CorLineIntersectFiniteP3}, there is a $3$-dimensional linear space $L\subset Y$ such that $l$ intersects $L$ at a point $y$. As $P$ and $P'$ are general, we may assume that both $P$ and $P'$ intersect $L$ at only one point $y$. The line of lines in $P$ passing through $y$ naturally lifts to a curve $Z_1\subset \widetilde{F_{3,1}}$ contained in the fiber of $\pi\circ\tilde f$ over the point $l_3$ of $F_3(Y)$ parameterizing $L_3$. We consider the $1$-cycle $z:=[Z_1] - [Z_1']\in CH_1(\widetilde{F_{3,1}})_{\mathbb Q}$. It is clear that $\tilde f_*(z)=0$ in $CH_1(\mathcal P_5)_{\mathbb Q}$ since $\tilde f_*([Z_1])$ and $\tilde f_*([Z_1'])$ is represented by two lines in the fiber of the $\mathbb P^5$-bundle $\mathcal P_5$ over the point $L\in F_3(Y)$. By Proposition~\ref{ThmF31ToF1}, we have $\tilde q_*(z)=\alpha \Delta^\vee$ in $CH_1(F_1(Y))_{\mathbb Q}$. But it is clear that $\tilde q_*([Z_1])=\Delta_P^\vee$ and $\tilde q_*([Z_1'])=\Delta_{P'}^\vee$, and that $\alpha=0$ by degree reasons. Hence, $\Delta_P^\vee=\Delta_{P'}^\vee$. There exists a $P_3\cong\mathbb P^3$ containing $P$ and $P'$ and the intersection $P_3\cap Y$ is the union of three planes $P, P', P''$. The same argument as above shows that $\Delta_P^\vee=\Delta_{P'}^\vee=\Delta_{P''}^\vee$. Finally, Proposition~\ref{PropTrianglePlanesConstant} shows that $\Delta_P^\vee=\Delta^\vee$, as desired.
\end{proof}

We now conclude the proof of Theorem~\ref{ThmChowOneOfF1} in this section. 
\begin{proof}[Proof of Theorem~\ref{ThmChowOneOfF1}]
It is not hard to show that $H^2(F_1(Y),\mathbb Q)=\mathbb Q$ since the restriction map $H^2(\mathrm{Gr}(2,10),\mathbb Q)\to H^2(F_1(Y),\mathbb Q)$ is an isomorphism \cite{DebarreManivel}. It suffices to prove that $CH_1(F_1(Y))_{\mathbb Q, hom}=0$. Let $\alpha\in CH_1(F_1(Y))_{\mathbb Q, hom}$ and let $z=\tilde q^*\alpha\in CH_1(\widetilde{F_{3,1}})_{\mathbb Q, hom}$. Since $\tilde q_*(z)=\deg \tilde q\cdot \alpha$, it suffices to prove that $\tilde q_*(z)=0$. 

Since $\mathcal P_5$ is a $\mathbb P^5$-bundle over $F_3(Y)$, we have $CH_1(\mathcal P_5)=h^5\cdot \pi^*CH_1(F_3(Y))\oplus h^4\cdot \pi^*CH_0(F_3(Y))$. Lemma~\ref{LmmMorphismF1F31} shows that $\tilde p_*z = 0 \in CH_1(F_3(Y))_{\mathbb Q, hom}$. Hence, $\tilde f_*(z)\in h^4\cdot \pi^*CH_0(F_3(Y))_{\mathbb Q, hom}$. Write $\tilde f_*(z)=w_1+\ldots+w_s$ where $w_i$ is a $1$-cycle supported on the fiber of $\pi$ over a point $L_i\in F_3(Y)$. Let $P_i$ be a plane in $Y$ that intersects with $L_i$ at only one point $y_i$. Let $P_{5,i}$ be the $5$-dimensional linear subspace spanned by $L_i$ and $P_i$. Let $z_i$ be the class in $CH_1(\widetilde{F_{3,1}})$ represented by the variety $Z_i:=\{(L_i, l): y_i\in l\subset P_i\}$. Since the fibers of $\pi: \mathcal P_5\to F_3(Y)$ are projective spaces $\mathbb P^5$, the cycle $w_i$ is proportional to the class represented by the variety $\{(L_i, P_4)\in \mathcal P_5: P_4\subset P_{5,i}\}$, which is the image of $Z_i$ under $\tilde f$. Hence, with an appropriate choice of coefficients $a_i\in\mathbb Q$, we have \[
\tilde f_*(z)=\sum_i w_i=\sum_i a_i \tilde f_*z_i=\tilde f_*(\sum_i a_iz_i).
\] 
By Proposition~\ref{ThmF31ToF1}, we conclude that 
\[\tilde q_*(z)=\tilde q_*(\sum_i a_iz_i)=\sum_i a_i\Delta_{P_i}^\vee=a \Delta^\vee,
\]
where $a=\sum_ia_i$. The last equality is due to Proposition~\ref{PropLinesOfLinesAreConstant}. Since $\tilde q_*(z)$ is homologue to $0$, the coefficient $a=0$. Hence, $\tilde q_*(z)=0$, as desired. This terminates the proof of Theorem~\ref{ThmChowOneOfF1}.
\end{proof}

\subsection{Proof of Theorem B}
We prove in this section Theorem B from the introduction. 
\begin{theoremB}\label{theoremB}
    Let $Y \subset \mathbb{P}^9$ be a general cubic $8$-fold, and let $X = F_2(Y)$ be the Fano variety of planes in $Y$. Let $\Psi: X \dashrightarrow X$ be the Voisin map. Then for any $z \in CH_0(X)_{hom}$:
    \[\Psi_*z = -8z.\]
\end{theoremB}
Putting together Theorem~\ref{ThmDecOfActionOfPsiGeneralCase}, Proposition~\ref{PropVanishingOfc3}, Proposition~\ref{PropVanishingOfc2} and Theorem~\ref{ThmChowOneOfF1}, we conclude that formula (\ref{EqDecOfTheActionOfPsi}) becomes 
\begin{equation}
    \Psi_*z = -8z + \gamma' I_{2*}z
\end{equation}
for any cycle $\gamma'$ of the form $ac_1^3 + b'c_1c_2 + c'c_3$ on $X$, where the number $a$ is determined by the class $\gamma$ of (\ref{EqDecOfTheActionOfPsi}) by $\gamma = ac_1^3 + bc_1c_2 + cc_3$, and $b', c'$ are arbitrary. We take for $\gamma'$ a multiple of the class of the fixed locus of $F$ of $\Psi$. Indeed, Proposition~\ref{ThmChowClassOfF} proved in Section~\ref{SectionChowClassOfFInX} says that the class of $F$ in $\mathrm{CH}^3(X)$ has a nonzero coefficient in $c_1^3$. Theorem B then follows form
\begin{theorem}\label{ThmFixedLocusIsConstantCycle}
     The fixed locus $F$ is a constant cycle subvariety in $X$.
 \end{theorem}
 Indeed, Theorem~\ref{ThmFixedLocusIsConstantCycle} says that the natural morphism $CH_0(F)_{hom}\to CH_0(X)_{hom}$ is zero.
 \begin{remark}
     Theorem~\ref{ThmFixedLocusIsConstantCycle} had been proved in~\cite{0CycleHK} in the case $r=1$.
 \end{remark}

\subsubsection{Proof of Theorem~\ref{ThmFixedLocusIsConstantCycle}}
Let $\mathcal P_{2,1}:=\{(x, l)\in X\times F_1(Y): l\subset P_x\}$ be the incidence variety. 

\begin{remark}\label{RmkI2AsSelfCorr}
    We have $I_2 = ^t\mathcal P_{2,1}\circ \mathcal P_{2,1}$ as self-correspondence of $X$.
\end{remark}

Let $x\in X=F_2(Y)$ be a general point and let $x'=\Psi(x)$ where $\Psi: X\dashrightarrow X$ is the Voisin map. In what follows, we use the following notation: for a plane $P$, the dual of $P$, defined as the set of lines in $P$, is denoted by $P^\vee$.
 \begin{proposition}\label{PropPx'-4PxConstant}
     The cycle $P_{x'}^{\vee} - 4 P_x^{\vee}\in CH_2(F_1(Y))_{\mathbb Q}$ does not depend on the choice of $x\in X$.
 \end{proposition}

Admitting Proposition~\ref{PropPx'-4PxConstant}, we conclude the proof of Theorem~\ref{ThmFixedLocusIsConstantCycle}.
\begin{proof}[Proof of Theorem~\ref{ThmFixedLocusIsConstantCycle}]
    By Remark~\ref{RmkI2AsSelfCorr}, we have 
    \[
    \mathcal P_{2,1}^*\circ \mathcal P_{2,1*} = I_{2*}: CH_0(X)\rightarrow CH_3(X).
    \]
     If $x\in F\subset X$, then in the statement of Proposition~\ref{PropPx'-4PxConstant}, $x' = x$ and thus $P_x^{\vee}\in CH_2(F_1(Y))_{\mathbb Q}$ is independent of the choice of $x\in F$. Hence, for any $x_1, x_2\in F$, we have $\mathcal P_{2,1*}(x_1-x_2) = P_{x_1}^{\vee} - P_{x_2}^{\vee} = 0 \in CH_2(F_1(Y))_{\mathbb Q}$. Therefore, $I_{2*}(x_1-x_2) = \mathcal P_{2,1}^*\mathcal P_{2,1*}(x_1 - x_2) = 0\in CH_0(X)_{\mathbb Q}$. Now, if in Equation (\ref{EqDecOfTheActionOfPsi})
     \[
     \Psi_*z = -8z + \gamma\cdot I_{2*}(z),
     \]
     we take $z = x_1 - x_2$, we get $z = -8z\in CH_0(X)_{\mathbb Q}$. Therefore, $z = 0 \in CH_0(X)$ as $CH_0(X)$ is torsion-free. This implies $x_1 = x_2 \in CH_0(X)$. Since $x_1, x_2\in F$ are arbitrarily chosen, we conclude that $F$ is a constant cycle subvariety.
\end{proof}

The rest of this section is devoted to the proof of Proposition~\ref{PropPx'-4PxConstant}. The proof relies on the geometry of cycles in a cubic fourfold $Y_4$ and its variety of lines $F_1(Y_4)$, which has been studied in~\cite{0CycleHK} and \cite{ChowHK}. the following relation is established in~\cite{0CycleHK} (see also \cite{ChowHK}).

\begin{theorem}[Voisin~\cite{0CycleHK}]\label{ThmVoisinQuadraticRelationOnDP}
    For a cubic fourfold $Y_4 \subset \mathbb{P}^5$ containing a plane $P$, let $P^\vee \subset F_1(Y_4)$ denote the variety of lines within $P$, and let $D_P \subset F_1(Y_4)$ represent the divisor comprising lines in $Y_4$ intersecting $P$. With $l \subset CH^1(F_1(Y_4))$ being the restriction of the Plücker line bundle class from $\mathrm{Gr}(2,6)$, there exist constants $\alpha, \beta \in \mathbb{Q}$, and $\gamma \in CH^2(F_1(Y_4))_{\mathbb{Q}}$, where $\alpha \neq 0$ and $\gamma$ is a restriction of a class $\delta \in CH^2(\mathrm{Gr}(2,6))_{\mathbb{Q}}$ that is independent of the chose of the plane $P$ and the cubic fourfold $Y_4$, such that:
    \[
    P^\vee = \alpha D_P^2 + \beta D_P \cdot l + \gamma
    \]
    within $CH^2(F_1(Y_4))_{\mathbb{Q}}$.
\end{theorem}

We will partially generalize this relation to the case where $Y_4$ has mild singularities.
\begin{corollary}\label{CorVoisinQuadraticRelationOnDP}
    Consider a cubic hypersurface $Y_4 \subset \mathbb{P}^5$ with at most simple double points as singularities, containing a plane $P$. Let $F_1(Y_4)_{sm}$ denote the smooth part of $F_1(Y_4)$. Define $P^\vee \subset F_1(Y_4)_{sm}$, $D_P \subset F_1(Y_4)_{sm}$, and $l \subset CH^1(F_1(Y_4)_{sm})$ as in Theorem~\ref{ThmVoisinQuadraticRelationOnDP}, but restricted to the smooth part of $F_1(Y_4)$. There exist constants $\alpha, \beta \in \mathbb{Q}$, and $\gamma \in CH^2(F_1(Y_4)_{sm})_{\mathbb{Q}}$, with $\alpha \neq 0$ and $\gamma$ as a restriction of a class $\delta \in CH^2(\mathrm{Gr}(2,6))_{\mathbb{Q}}$ that is independent of the chose of the plane $P$ and the cubic fourfold $Y_4$, such that:
    \begin{equation}\label{EqVoisinQuadraticRelationOnDP}
         P^\vee = \alpha D_P^2 + \beta D_P \cdot l + \gamma
    \end{equation}
    in $CH^2(F_1(Y_4)_{sm})_{\mathbb{Q}}$.
\end{corollary}

\begin{proof}
Let us consider the construction of the family $\mathcal{F}$ over $B$, where $B$ parametrizes pairs $(P, f)$ with $P$ being a plane in $\mathbb{P}^5$ and $f$ a cubic polynomial such that the hypersurface $Y_f$ defined by $f$ has at most simple double points as singularities, together with the condition that $f_{|P} = 0$. Let \[\mathcal F = \{
    ((P, f), l)\in B\times \mathrm{Gr}(2, 6): f|_l = 0\}
    \}.
    \]
In such a way, we make the Fano variety of lines $F_1(Y_f)$ into family over $B$.

 \begin{lemma}
        The family $p: \mathcal F\to B$ is flat.
    \end{lemma}
    \begin{proof}
        Each fiber of $p$ is a subvariety that is defined as the zero locus of the vector bundle $\mathrm{Sym}^3\mathcal E^*$ on $\mathrm{Gr}(2,6)$, with the expected dimension, where $\mathcal E$ is the tautological subbundle of $\mathrm{Gr}(2,6)$. By Koszul's resolution, each fiber has the same Hilbert polynomial. This implies that the family is flat. 
    \end{proof}
 Within the family $\mathcal F$, we have the following subvarieties. 
    $\mathcal P^\vee := \{ 
    ((P, f), l)\in B\times \mathrm{Gr}(2,6): l\subset P
    \}$ and $\mathcal D : = \{((P, f), l)\in \mathcal F: l\cap P\neq \emptyset\}$, representing lines within $P$ and lines intersecting $P$, respectively. Additionally, let $q: \mathcal F\to \mathrm{Gr}(2,6)$ be the second projection. Let $\mathcal L = q^*\mathcal O_{\mathrm{Gr}(2,6)}(1)$ be the pull-back of the Plücker line bundle. Let $\Gamma = q^*\delta\in CH^2(\mathcal F)_{\mathbb Q}$, where $\delta\in CH^2(\mathrm{Gr}(2,6))_{\mathbb Q}$ be the constant class as defined in Theorem~\ref{ThmVoisinQuadraticRelationOnDP}. Let $\alpha, \beta\in\mathbb Q$ be as in Theorem~\ref{ThmVoisinQuadraticRelationOnDP}.
    We consider the algebraic cycle $\mathcal Z = \mathcal P^{*} - \alpha \mathcal D^2 - \beta \mathcal D\cdot \mathcal L - \Gamma$. Theorem~\ref{ThmVoisinQuadraticRelationOnDP} implies that $\mathcal Z|_{\mathcal F_t} = 0 \in CH^2(\mathcal F_t)_{\mathbb Q}$ for $t\in B$ with smooth fibers. By the specialization of algebraic cycles~\cite[Proposition 1.4]{VoisinUnirational}, we conclude that $\mathcal Z|_{\mathcal F_t} = 0\in CH^2(\mathcal F_t)$ for all $t\in B$. For a singular fiber $\mathcal F_t$, we can restrict further to the smooth part of $\mathcal F_t$ and we get the desired result.
\end{proof}

\begin{remark}
    The reason we do not achieve the relation (\ref{EqVoisinQuadraticRelationOnDP}) for the whole of $F_1(Y_4)$ is that the divisor $D_P$ might encompass the singular locus of $F_1(Y_4)$, rendering it not a Cartier divisor, hence $D_P^2$ is not well-defined. However, upon restriction to the smooth part, all components are well-defined, and the restriction of $\mathcal{D}^2$ equates to $D_P^2$.
\end{remark}

We will also need the following

\begin{lemma}\label{Lmmj*DPIsConstant}
    Let $Y\subset \mathbb P^9$ be a general cubic eightfold and let $P\subset Y$ be a general plane contained in $Y$. Let $H_5$ be a general linear subspace of $\mathbb P^9$ containing $P$ such that $H_5\cap Y =: Y_4$ is a cubic hypersurface containing the plane $P$. Let $j: F_1(Y_4)\hookrightarrow F_1(Y)$ be the natural inclusion map. Let $D_P\subset F_1(Y_4)$ be as in Theorem~\ref{ThmVoisinQuadraticRelationOnDP}. Then the class $j_*D_P\in CH_3(F_1(Y))$ is independent of the choice of the plane $P\subset Y$ and of the linear subspace $H_5\subset \mathbb P^9$.
\end{lemma}
\begin{proof}
    Let $q: \mathcal P_{1,0}\to Y$ and $p: \mathcal P_{1,0}\to F_1(Y)$ be the universal {correspondence} of $Y$ and $F_1(Y)$.
    Let $\Sigma_P\subset F_1(Y)$ be the variety of lines in $Y$ that intersects the plane $P$. The class $\Sigma_P\in CH^5(F_1(Y))$ does not depend on the choice of $P\subset Y$, since $\Sigma_P = \mathcal P_{1,0}^*(P)$ and since $CH_2(Y) = \mathbb ZP$ by~\cite{Otwinowska}. Let $\tilde\Sigma_P\subset \mathcal P_{1,0}$ be the preimage of $P\subset Y$ via $q$. Then similarly, the Chow class of $\tilde\Sigma_P$ in $\mathcal P_{1,0}$ does not depend on the choice of $P$. Let us define two vector bundles $\mathcal E$ and $\mathcal H$ on $\mathcal P_{1,0}$ as follows. $\mathcal E$ is the pull-back of the universal subbundle over $F_1(Y)\subset \mathrm{Gr}(2,10)$ via $p$ and $\mathcal H$ is the pull-back of the Hopf bundle over $Y\subset \mathbb P^9$ via $q$. On $\tilde\Sigma_P$, we have a natural inclusion map $\mathcal H|_{\tilde\Sigma_P}\hookrightarrow \mathcal E|_{\tilde\Sigma_P}$ that induces a sujective morphism of vector bundles $\phi: \mathcal E|_{\tilde\Sigma_P}^*\to \mathcal H|_{\tilde\Sigma_P}^*$ that fits into the short exact sequence
    \begin{equation}\label{EqShortExactSequenceInCubicFourfold}
         0\to (\det\mathcal E^*\otimes \mathcal H)|_{\tilde\Sigma_P} \to \mathcal E|_{\tilde\Sigma_P}^*\to \mathcal H|_{\tilde\Sigma_P}^* \to 0.
    \end{equation}
   
    Over $\tilde\Sigma_P$, the defining equations of $H_5\subset \mathbb P^9$ induces a section $s$ of $\mathcal E^*|_{\tilde\Sigma_P}^{\oplus 4}$ that is zero when projected to $H^0(\tilde\Sigma_P, \mathcal H^*|_{\tilde\Sigma_P}^{\oplus 4})$. Hence, we can view $s$ as a section $\sigma_s$ of $(\det\mathcal E^*\otimes \mathcal H)|_{\tilde\Sigma_P}^{\oplus 4}$. Let $\tilde D_P$ be the zero locus of $\sigma_s$. Then $\tilde D_P$ parametrizes the pairs $(l, y)\in F_1(Y)\times Y$ such that $l\subset H_5$ and $l\cap P = y$. By the projection formula, we find that the Chow class of $\tilde D_P$ in $\mathcal P_{1,0}$ is $\tilde \Sigma_P\cdot c_1((\det\mathcal E^*\otimes \mathcal H))^4$, which is independent of the choice of $P$ and $H_5$. Since $D_P = p_*\tilde D_P$, the Chow class of $D_P$ in $F_1(Y)$ is independent of the choice of $P$ and $H_5$ as well.
\end{proof}

 For a plane $P\subset Y$, let $P^{\vee}\subset F_1(Y)$ be the subvariety of lines contained in $P$. Let $l\in CH^1(F_1(Y))$ be the restriction of the Plücker line bundle class of $\mathrm{Gr}(2, 10)$. Let $Y_4\subset Y$ be a linear section of $Y$ that has at most simple double points as singularities. Let $\Sigma\subset F_1(Y_4)$ be the singular locus of $F_1(Y_4)$ and let $j: F_1(Y_4)_{sm}\hookrightarrow F_1(Y)-\Sigma=: F_1(Y)^0$ be the inclusion map. Let $D_P$ be defined as in Corollary~\ref{CorVoisinQuadraticRelationOnDP}. Then Corollary~\ref{CorVoisinQuadraticRelationOnDP} and Lemma~\ref{Lmmj*DPIsConstant} imply the following
 \begin{corollary}\label{CorPIsConstantUpToDP2}
      In $CH_2(F_1(Y)^0)_{\mathbb Q}$, we have the following relation
     \[
     P^{\vee} = \alpha j_*(D_P^2) + c|_{F_1(Y)^0},
     \]
     where $\alpha\neq 0$ is a rational number and $c\in CH_2(F_1(Y))_{\mathbb Q}$ is a Chow class that is independent of the choice of the plane $P$.
 \end{corollary}

 \begin{proof}[Proof of Proposition~\ref{PropPx'-4PxConstant}]
     By the definition of the Voisin map, there is a unique linear subspace $H\subset \mathbb P^9$ of dimension $3$, such that $H\cap Y = 2 P_x + P_{x'}$. Let $H_5\subset \mathbb P^9$ be a linear subspace containing $H$ (thus containing both $P_x$ and $P_{x'}$), and let $Y_4 = Y\cap H_5$. 
     \begin{lemma}\label{LmmSingularityOfY4}
         For a general choice of $H_5$, the cubic hypersurface $Y_4$ has $4$ simple double points as the only singularities.
     \end{lemma}
     \begin{proof}
         The base locus of the linear system $L = \{H_5'\cap Y: H\subset H_5'\subset \mathbb P^9\}$ is $P_x\cup P_{x'}$. Hence, by Bertini's theorem, for a general $H_5$, the cubic hypersurface $Y_4 := Y\cap H_5$ is smooth outside $P_x\cup P_{x'}$. Let us write $H_5 = \{(x_0: x_1: \ldots: x_5)\}$ and we can assume that $H = \{(x_0: x_1: x_2: x_3: 0: 0)\}$, $P_x = \{(x_0: x_1: x_2: 0: 0: 0)\}$ and $P_{x'} = \{(0: x_1: x_2: x_3: 0: 0)\}$. The fact that $H\cap Y_4 = 2 P_x + P_{x'}$ means that the defining equation of $Y_4\subset H_5$ can be written as
         \[
         f(x_0,\ldots, x_5) = x_3^2x_0 + x_4Q_1(x_0, \ldots, x_5) + x_5Q_2(x_0, \ldots, x_5),
         \]
         where $Q_1$ and $Q_2$ are quadratic polynomials. Let $y\in Y_4$ be a singular point. Since $y\in H$, one must have $x_4(y) = x_5(y) = 0$. Writing $f_{x_i} = \frac{\partial f}{\partial x_i}$, we find by direct calculations that $f_{x_0}(y) = x_3(y)^2$,  $f_{x_1}(y) = f_{x_2}(y) = 0 $, $f_{x_3}(y) = 2x_0(y)x_3(y)$, $f_{x_4}(y) = Q_1(y)$ and $f_{x_5}(y) = Q_2(y)$. The fact $y\in Y_4$ is singular implies that $f_{x_i}(y) = 0$ for any $x_i$. Taken together, we find that $y\in H_5$ satisfies the following system of equations
         \[
         \left\{\begin{array}{cc}
              x_4 &= 0   \\
               x_5 &= 0 \\
               x_3 &= 0 \\
               Q_1(x_0, \ldots, x_5) &= 0 \\
               Q_2(x_0, \ldots, x_5) &= 0
         \end{array}\right.
         \]
         By the generality of $Y\subset \mathbb P^9$ and $H_5$, the solutions of this system of equations are four points in $H_5$. To show that the four singular points are simple double points. We do a local check. Let $y$ be one of the singular points. Up to a change of coordinates of $P_x$, we may assume without loss of generality that $x_0(y) \neq 0$ and $x_1(y) = x_2(y) =0$, namely, $y = (1: 0: 0: 0: 0: 0)\in H_5$. On the open affine subset $U_0$ of $H_5$ defined by $x_0 = 1$, the cubic hypersurface $Y_4\cap U_0$ is defined by the equation $x_3^2 + x_4Q_1(1, x_1, \ldots, x_5) + x_5Q_2(1, x_1, \ldots, x_5)$. The Taylor expansion of this polynomial around the point $y = (0, 0, 0,0,0)\in \mathbb A^5\cong U_0$ is the polynomial itself. The fact that this polynomial does not have degree $1$ term is simply because $y$ is a singular point. To show that $y$ is a double point, we only need to make sure that the degree $2$ term of the expression $x_3^2 + x_4Q_1(1, x_1, \ldots, x_5) + x_5Q_2(1, x_1, \ldots, x_5)$ is non-degenerate, and this condition is clearly satisfied for general $Q_1$ and $Q_2$.
     \end{proof}
     Let $y_1, y_2, y_3, y_4$ be the four singular points of $Y_4$. Let $\Sigma\subset F_1(Y_4)$ be the singular locus of $F_1(Y_4)$. Then $\Sigma$ is the union of four surfaces $\Sigma_1$, $\Sigma_2$, $\Sigma_3$ and $\Sigma_4$, parametrizing the lines in $Y_4$ that pass through the point $y_1, y_2, y_3$ and $y_4$, respectively. Let $F_1(Y)^0= F_1(Y)-\Sigma$
     Let $j: F_1(Y_4)_{sm}\hookrightarrow F_1(Y)^0$ be the natural inclusion. By Corollary~\ref{CorPIsConstantUpToDP2}, we have \begin{equation}\label{EqPxDPx}
         P_x^{*} = \alpha j_*(D_{P_x}^2) + c'|_{F_1(Y)^0}\in CH_2(F_1(Y))_{\mathbb Q}
     \end{equation}
     and 
     \begin{equation}\label{EqPx'DPx'}
         P_{x'}^{*} = \alpha j_*(D_{P_{x'}}^2) + c'|_{F_1(Y)^0}\in CH_2(F_1(Y)^0)_{\mathbb Q}
     \end{equation}
     for some constant $c'\in CH_2(F_1(Y))_{\mathbb Q}$. Now let $\mathcal P = \{(l, y)\in F_1(Y_4)_{sm}\times Y_4: y\in l\}$ be the {incidence} {correspondence}, then for any plane $P\subset Y_4$, we have $D_P = \mathcal P^*(P)$. It is clear that $2P_x + P_{x'} = h^2$ in $CH_2(Y_4)$, so $2D_{P_x} + D_{P_{x'}} = \mathcal P^*(h^2) = l\in CH_3(F_1(Y_4)_{sm})$. Taken into account of the equations (\ref{EqPxDPx}) and (\ref{EqPx'DPx'}), we find that 
     \[
     P_{x'}^{*} = 4 P_x^{*} - 4\alpha j_*D_{P_x}\cdot l + \alpha l^2 - 3 c' \in CH_2(F_1(Y)^0)_{\mathbb Q}.
     \]
     Taking into account of Lemma~\ref{Lmmj*DPIsConstant}, $P_{x'}^{*} - 4 P_x^{*} = c|_{F_1(Y)^0}\in CH_2(F_1(Y)^0)_{\mathbb Q}$, where $c\in CH_2(F_1(Y))_{\mathbb Q}$ is a constant $2$-cycle on $F_1(Y)$. Hence, in $F_1(Y)$, the cycle $P_{x'}^{*} - 4 P_x^{*} - c$ is supported on $\Sigma$, the singular locus of $F_1(Y_4)$. Thus, we need to understand the geometry and Chow classes of $\Sigma$. 

     Following Lemma~\ref{LmmSingularityOfY4}, let $y_1, y_2, y_3, y_4$ be the four singular points of $Y_4$. Then $\Sigma$ is the union of four surfaces $\Sigma_1$, $\Sigma_2$, $\Sigma_3$ and $\Sigma_4$, parametrizing the lines in $Y_4$ that pass through the point $y_1, y_2, y_3$ and $y_4$, respectively. We have the following two lemmas about the geometry of the surfaces $\Sigma_i$ that we will prove later. 
     \begin{lemma}\label{LmmIrreducibilityOfTheFourSurfaces}
         The surfaces $\Sigma_1$, $\Sigma_2$, $\Sigma_3$ and $\Sigma_4$ are irreducible. In particular, $CH_2(\Sigma) = \oplus_{i=1}^4 CH_2(\Sigma_i)$.
     \end{lemma}
     \begin{lemma}\label{LmmChowClassOfTheFourSurfaces}
         Let $H_5\subset \mathbb P^9$ be a general linear subspace of dimension $5$ such that the linear section $Y_4:= Y\cap H_5$ has only simple double singularities and let $y\in Y_4$ be a singular point. Let $S$ be the surface of lines in $Y_4$ that pass through the point $y$. Then the Chow class of $S$ in $F_1(Y)$ does not depend on the choice of $H_5$ and $y$.
     \end{lemma}
     By Lemma~\ref{LmmIrreducibilityOfTheFourSurfaces} and the fact that $P_{x'} - 4P_x - c$ is supported on $\Sigma$, we conclude that $P_{x'} - 4P_x - c = \sum_{i=1}^4 a_i\Sigma_i$, with $a_i\in\mathbb Q$. The cohomological class of $P_{x'} - 4P_x - c$ is clearly a constant, thus the cohomological class of $\sum_{i=1}^4 a_i\Sigma_i$ is constant. Lemma~\ref{LmmChowClassOfTheFourSurfaces} then implies that the Chow class of $\sum_{i=1}^4 a_i\Sigma_i$ is also constant. Therefore, $P_{x'} - P_x = c + \sum_{i = 1}^4 a_i\Sigma_i$ is a constant in $CH_2(F_1(Y))_{\mathbb Q}$, as desired. This concludes the proof of Proposition~\ref{PropPx'-4PxConstant}.
 \end{proof}
 \begin{proof}[Proof of Lemma~\ref{LmmIrreducibilityOfTheFourSurfaces}]
         To check the irreducibility, we return to the proof of Lemma~\ref{LmmSingularityOfY4}. Using the notation there, let 
         \[g(x_1, \ldots, x_5) = f(0, x_1, \ldots, x_5) =  x_4Q_1(0, x_1, \ldots, x_5) + x_5 Q_2(0, x_1, 
         \ldots, x_5)\]
         and let $q(x_1, \ldots, x_5)$ be the degree $2$ part of the polynomial $x_3^2 + x_4Q_1(1, x_1, \ldots, x_5) + x_5Q_2(1, x_1, \ldots, x_5)$. Then
         the surface $\Sigma_1$ of lines passing through the singular point $y = (1: 0: 0: 0: 0: 0)$ is the subvariety in $\mathbb P^4 = \{(x_1: \ldots: x_5)\}$ cut by the  equations $g(x_1, \ldots, x_5) = 0$ and $q(x_1, \ldots, x_5) = 0$. From the expression, we see that $g(x_1, \ldots, x_5)$ depends on the coefficients of the terms $x_ix_j$ with $1\leq i, j\leq 5$ in $Q(x_0, x_1, \ldots, x_5)$ whereas $q(x_1, \ldots, x_5)$  depends on the coefficients of the terms $x_0x_i$ for $1\leq i\leq 5$. So the choice of coefficients of $g(x_1, \ldots, x_5)$ and $q(x_1, \ldots, x_5)$ does not have influence on each other. Now we fix one choice of smooth $Q=\{q=0\}$. Then varying $C = \{g = 0\}$, the base points of $C\cap Q$ is the line $(s: t: 0: 0: 0)$ corresponding to the lines on the plane $P_x$ passing through the point $y$. Hence, for a general choice of $q$ and $g$, the surface $\Sigma_1 = C\cap Q$ is smooth outside the line $(s: t: 0: 0: 0)$. Writing $Q_1(x_0, \ldots, x_5) = \sum_{i, j = 0}^5 a_{ij}x_ix_j$ and $Q_2(x_0, \ldots, x_5) = \sum_{i, j = 0}^5 b_{ij}x_ix_j$, the Jacobian matrix of the polynomials $q(x_1, \ldots, x_5)$ and $g(x_1, \ldots, x_5)$ at the point $(s: t: 0: 0: 0)$ is
    \[
    \begin{pmatrix}
        0 & 0 & 0 &  a_{01} s + a_{02} t &  b_{01} s + b_{02} t\\
        0 & 0 & 0 & 2 a_{12} st & 2 b_{12} st 
    \end{pmatrix}.
    \]
    This matrix does not have full rank only if $s = 0$ or $t = 0$ or $(a_{01} - b_{01}) s + (a_{02} - b_{02}) t = 0$, corresponding to the three lines passing through the other three singular points. Therefore, the singular locus of $\Sigma_1$ is of codimension $2$. But a reducible complete intersection of dimension $\geq 1$ in the projective space always have codimension $1$ singular locus by the Fulton–Hansen connectedness theorem~\cite{Connectedness}. Hence, $\Sigma_1$ is irreducible, as desired.
    \end{proof}

\begin{proof}[Proof of Lemma~\ref{LmmChowClassOfTheFourSurfaces}]
         Let $H$ be the projective tangent space of $Y$ at point $y$. It is a linear subspace of dimension $8$ in $\mathbb P^9$. The linear section $Y_4$ being singular at $y$ is equivalent to the relation $H_5\subset H$. Let $F_{1, y}(Y)$ be the variety of lines in $Y$ that passes through $y$. Let $\mathcal P = \{(l, y)\in F_{1}(Y)\times Y: y\in l\}$ be the {incidence} variety. Then $F_{1,y}(Y)$ can be identified with a subvariety of $\mathcal P$ given by the preimage of the point $y$ under the projection map $q: \mathcal P\to Y$.
         Since every line in $Y$ passing through $y$ is contained in $H$, the subvariety $S$ of $F_{1, y}(Y)$ parametrizing the lines that is furthermore included in $H_5$ is given by the zero locus of a section $\sigma$ of the vector bundle $(\mathcal E^*)^{\oplus 3}$. However, as the point $y$ already lies in $H_5$, if we restrict $\sigma$ on $y$ via the following morphism
         $(\mathcal E^*)^{\oplus 3} \to \mathcal O^{\oplus 3}$,
         we get zero. Thus, $S$ can be viewed as the zero locus of a section of $(\det(\mathcal E^*))^3$, with expected dimension. Therefore, the Chow class of $S$ in $F_1(Y)$ is given by $c_1(\mathcal E^*)^3\cdot F_{1, y}(Y)$. Note that $F_{1, y}(Y) = \mathcal P^*(y) \in CH(F_1(Y))$ with $\mathcal P$ the canonical {correspondence} between $F_1(Y)$ and $Y$. Since $CH_0(Y)$ is trivial, the class $F_{1, y}(Y)\in CH(F_1(Y))$ is independent of $y$. In conclusion, the constant Chow class of $S$ in $F_1(Y)$ is independent of the choice of $H_5$ and $y$. 
     \end{proof}

\subsubsection{Chow Class of $F$ in $X$}\label{SectionChowClassOfFInX}
In this section, we determine the Chow class of $F$ within $X$. Denote by $\mathcal{F}$ the tautological subbundle on $\mathrm{Gr}(4, 10)$ (resp. $\mathrm{Gr}(r+2, n+1)$ in the general case), and $\mathcal{E}$ as the tautological subbundle of $X = F_2(Y)$ (resp. $X = F_r(Y)$ in the general case). Let $c_i$ represent $c_i(\mathcal{E}^*)$ over $X$. In the case $r = 2$, we have

\begin{proposition}\label{ThmChowClassOfF}
The Chow class of $F$ within $X$ is expressed as $-20 c_1^3 + 110c_1c_2 + 49 c_3$ in $CH^3(X)$.
\end{proposition}
\begin{proof}
    Let us start by giving the general method for any $r$ and then specialize to $r=2$ for explicit calculations. Consider the following stratification of $X\times \mathrm{Gr}(r+2,n+1)$:
    \[
    \begin{tikzcd}
        X\times \mathrm{Gr}(r+2, n+1)\arrow[d, hookleftarrow, "\gamma"]\\
        \tilde{\tilde X}\arrow[d, hookleftarrow, "\beta"] :=\{(x, P_{r+1})\in X\times \mathrm{Gr}(r+2,n+1)| P_{r+1}\cap Y\supset P_x\}\\
        \tilde X \arrow[d, hookleftarrow,"\alpha"] :=\{(x,P_{r+1})\in X\times \mathrm{Gr}(r+2, n+1)| P_{r+1}\cap Y\supset 2P_x\}\\
        \tilde F:=\{(x, P_{r+1})\in X\times \mathrm{Gr}(r+2,n+1)|P_{r+1}\cap Y\supset 3P_x\}
    \end{tikzcd}
    \]
    Let $pr_1: X\times \mathrm{Gr}(r+2,n+1)\to X$ and $pr_2: X\times \mathrm{Gr}(r+2, n+1)\to \mathrm{Gr}(r+2, n+1)$ be the projection maps. It is not hard to see that $F=(pr_1\circ\gamma\circ\beta\circ\alpha)(\tilde F)$. We will see shortly that the Chow class of $\tilde F$ in $X\times \mathrm{Gr}(r+2, n+1)$ lies in a subring generated by the Chern classes of $pr_1^*\mathcal E^*$ and $pr_2^*\mathcal F^*$, say 
    \begin{equation}\label{EqExpOfTildeF}
        [\tilde F]=\sum_i pr_1^*P_i(c_k(\mathcal E^*))\cdot pr_2^*Q_i(c_l(\mathcal F^*))
    \end{equation}
    in $CH^{r+1+(n-r-1)(r+2)}(X\times \mathrm{Gr}(r+2, n+1))$, where $P_i$ and $Q_i$ are polynomials with adequate degrees. Then the class of $F$ in $X$ is $pr_{1*}([\tilde F])=\sum_iP_i(c_k(\mathcal E^*))\cdot pr_{1*}pr_2^*Q_i(c_l(\mathcal F^*))$. In this expression, 
    $pr_{1*}pr_2^*Q_i(c_l(\mathcal F^*))\neq 0$ only if the weighted degree of $Q_i(c_l(\mathcal F^*))$ is $(n-r-1)(r+2)$, namely the dimension of $\mathrm{Gr}(r+2, n+1)$ (or equivalently, the weignted degree of $P_i(c_k(\mathcal E^*))$ is $r+1$). In this case, in $CH_0(\mathrm{Gr}(r+2,n+1))$, the class $Q_i(c_l(\mathcal F^*))$ is a multiple of a point $o\in \mathrm{Gr}(r+2, n+1)$. The coefficient, denoted by $q_i$, is calculable using Schubert calculus, once we know the expression of $Q_i(c_l(\mathcal F^*))$. Taken together, the class of $F$ in $CH^{r+1}(X)$ is given by
    \[
    [F]=\sum_{i: \deg P_i=r+1} q_i P_i(c_k(\mathcal E^*)).
    \]
    Now let us prove (\ref{EqExpOfTildeF}), i.e., $[\tilde F]$ is generated by $pr_1^*\mathcal E^*$ and $pr_2^*\mathcal F^*$. The subvariety $\tilde{\tilde X}$ of $X\times \mathrm{Gr}(r+2, n+1)$ is defined by the condition $P_x\subset P_{r+1}$. Hence, we can view $\tilde{\tilde X}$ as the locus where the composite map of vector bundles $pr_1^*\mathcal E\to V_{n+1}\to V_{n+1}/pr_2^*\mathcal F $ is zero, where $V_{n+1}$ is the trivial bundle of rank $n+1$ over $X\times \mathrm{Gr}(r+2, n+1)$. Hence, $\tilde{\tilde X}$ is the zero locus of a section of $pr_1^*\mathcal E^*\otimes V_{n+1}/pr_2^*\mathcal F$. Now over $\tilde{\tilde X}$, we have a natural injection $pr_1^*\mathcal E|_{\tilde{\tilde X}}\to pr_2^*\mathcal F|_{\tilde{\tilde X}}$. The defining polynomial $f$ of $Y$ induces a section of $pr_2^*\mathrm{Sym}^3\mathcal F^*|_{\tilde{\tilde X}}$. Since $P_x\subset Y$ for any $x\in X$, this section vanishes when passing to $pr_1^*\mathrm{Sym}^3\mathcal E^*|_{\tilde{\tilde X}}$. Hence, the section is actually a section $\tilde{\tilde f}$ of $pr_2^*\mathrm{Sym}^2\mathcal F^*|_{\tilde{\tilde X}}\otimes (pr_2^*\mathcal F|_{\tilde{\tilde X}}/pr_1^*\mathcal E|_{\tilde{\tilde X}})^*$. The subvariety $\tilde X$ is  the locus of $\tilde{\tilde X}$ where, when passing to $pr_1^*\mathrm{Sym}^2\mathcal E^*|_{\tilde{\tilde X}}\otimes (pr_2^*\mathcal F|_{\tilde{\tilde X}}/pr_1^*\mathcal E|_{\tilde{\tilde X}})^*$, the section $\bar{\tilde{\tilde f}}$ is zero. Hence, $\tilde X$ is the zero locus of a section of the vector bundle $pr_1^*\mathrm{Sym}^2\mathcal E^*|_{\tilde{\tilde X}}\otimes (pr_2^*\mathcal F|_{\tilde{\tilde X}}/pr_1^*\mathcal E|_{\tilde{\tilde X}})^*$. Similar argument shows that $\tilde F$ is the zero locus of a section of the vector bundle $pr_1^*\mathcal E^*|_{\tilde X}\otimes (pr_2^*\mathcal F|_{\tilde X}/pr_1^*\mathcal E|_{\tilde X})^{*\otimes 2}$. Taken together, the class of $\tilde F$ in $CH(X\times \mathrm{Gr}(r+2,n+1))$ is given by the following class
    \[
    e(pr_1^*\mathcal E^*\otimes (V_{n+1}-pr_2^*\mathcal F))\cdot e((pr_2^*\mathcal F^*-pr_1^*\mathcal E^*)\otimes pr_1^*\mathrm{Sym}^2\mathcal E^*)\cdot e((pr_2^*\mathcal F^*-pr_1^*\mathcal E^*)^{\otimes 2}\otimes pr_1^*\mathcal E^*)
    \]
    where $e(...)$ is the Euler class and we have written the vector bundles as their class in K groups to avoid any confusions. This gives an expression of the class $\tilde F$ as generated by the Chern classes of $pr_1^*\mathcal E^*$ and $pr_2^*\mathcal F^*$.

    In what follows, we do the explicit calculations in the case $r=2$, with the help of Mathematica and SageMath.

    \subsubsection*{Calculations of $e(pr_1^*\mathcal E^*\otimes (V_{10}-pr_2^*\mathcal F))$}
    
    We may assume that $pr_1^*\mathcal E^*$ is formally a direct sum of three line bundles whose classes are $l_1, l_2, l_3$ and that $V_{10}/pr_2^*\mathcal F$ is formally a direct sum of $m_1, \ldots, m_6$. Then the Euler class of $pr_1^*\mathcal E^*\otimes V_{10}/pr_2^*\mathcal F$ is $\prod_{i,j}(l_i+m_j)$. The Euler class of $pr_1^*\mathcal E^*\otimes V_{10}/pr_2^*\mathcal F$ expressed by $c_i$, the Chern classes of $pr_1^*\mathcal E^*$, and by $d_j$, the Chern classes of $V_{10}/pr_2^*\mathcal F$ is calculated, using Mathematica, in Appendix~\ref{SectionMathematicaI}. The result is tremendously long. Fortunately, we do not really need the full form of the result. As our general method indicates, we only need to expand the Chern classes of $pr_1^*\mathcal E^*$ up to degree $3$. With this in mind, the Euler class of $pr_1^*\mathcal E^*\otimes (V_{10}-pr_2^*\mathcal F)$ is 
\begin{verbatim}
    d6^3 + c1 d5 d6^2 + c1^2 d4 d6^2 + c2 (d5^2 d6 - 2 d4 d6^2) + 
    c1 c2 (d4 d5 d6 - 3 d3 d6^2) + c1^3 d3 d6^2 +
    c3 (d5^3 - 6 d3 d6^2 + 3 (-d4 d5 d6 + 3 d3 d6^2))+ higher terms in ci
\end{verbatim}

\subsubsection*{Calculations of $ e((pr_2^*\mathcal F^*-pr_1^*\mathcal E^*)\otimes pr_1^*\mathrm{Sym}^2\mathcal E^*)$ and $e(((pr_2^*\mathcal F^*-pr_1^*\mathcal E^*)^{\otimes 2}\otimes pr_1^*\mathcal E^*)$}
Let $C_i=c_i(pr_2^*\mathcal F^*)$ be the Chern classes of $pr_2^*\mathcal F^*$ and let $S_i=c_i(pr_1^*\mathrm{Sym}^2\mathcal E^*)$. Then $S_i$ can be represented by the $c_i$, the Chern classes of $pr_1^*\mathcal E^*$ as follows.
\begin{verbatim}
    S1[c1_, c2_, c3_] := 4 c1;
    S2[c1_, c2_, c3_] := 5 c1^2 + 5 c2;
    S3[c1_, c2_, c3_] := 2 c1^3 + 11 c1 c2 + 7 c3;
\end{verbatim}  
It is a general fact~\cite[Example 3.2.2]{Fulton} that if $L$ is a line bundle and $V$ is a vector bundle of rank $s$, then 
\[
e(L\otimes V)=\sum_{i=0}^{s}c_1(L)^ic_{s-i}(V).
\]
Since $c_1(pr_2^*\mathcal F^*-pr_1^*\mathcal E^*)=c_1(pr_2^*\mathcal F^*)-c_1(pr_1^*\mathcal E^*)$, the class $e((pr_2^*\mathcal F^*-pr_1^*\mathcal E^*)\otimes pr_1^*\mathrm{Sym}^2\mathcal E^*)$ is given by the expression \begin{verbatim}
    (C1 - c1)^6 + (C1 - c1)^5 S1[c1, c2, c3] + (C1 - c1)^4 S2[c1, c2, c3] 
    + (C1 - c1)^3 S3[c1, c2, c3]+ higher terms in ci
\end{verbatim}
and the class $e(((pr_2^*\mathcal F^*-pr_1^*\mathcal E^*)^{\otimes 2}\otimes pr_1^*\mathcal E^*)$ is given by
\begin{verbatim}
    8 (C1 - c1)^3 + 4 (C1 - c1)^2 c1 + 2 (C1 - c1) c2 + c3.
\end{verbatim}

\subsubsection*{Calculation of $[F]$}
Now we can calculate $e(pr_1^*\mathcal E^*\otimes (V_{n+1}-pr_2^*\mathcal F))\cdot e((pr_2^*\mathcal F^*-pr_1^*\mathcal E^*)\otimes pr_1^*\mathrm{Sym}^2\mathcal E^*)\cdot e(((pr_2^*\mathcal F^*-pr_1^*\mathcal E^*)^{\otimes 2}\otimes pr_1^*\mathcal E^*)$ (up to degree $3$ in $c_i$). The detailed Mathematica calculations and results are presented in Appendix~\ref{SectionMathematicaII}. We only need the expression up to degree $3$ in $c_i$. The expression of $[\tilde F]$ is 
\begin{verbatim}
    c1^3 (8 C1^9 d3 d6^2 - 36 C1^8 d4 d6^2 + 56 C1^7 d5 d6^2 
            - 20 C1^6 d6^3)+
    c3 (8 C1^9 d5^3 - 24 C1^9 d4 d5 d6 + 24 C1^9 d3 d6^2 + 57 C1^6 d6^3)+
    c1 c2 (8 C1^9 d4 d5 d6 - 36 C1^8 d5^2 d6 - 24 C1^9 d3 d6^2 + 
            72 C1^8 d4 d6^2 + 42 C1^7 d5 d6^2 - 178 C1^6 d6^3)+
    + higher terms in ci.
\end{verbatim}
Now we need to calculate the coefficients in the above expression. To this ends, we do the Schubert calculus with the help of SageMath. We replace $C_i$ by $s[1,...,1]$ with $i$ copies of $1$ and replace $d_i$ by $(-1)^is[i]$, where $s[1,..., 1]$ and $s[i]$ are the Schur's polynomials with the corresponding weights. We do the Schubert calculus using Schur's polynomials and find the coefficients of $s[6,6,6,6]$. The output result of the SageMath is exactly $-20 c_1^3 + 110c_1c_2 + 49 c_3$, as desired.
\end{proof}
\begin{remark}
    Using the same argument (but with much simpler calculations), we find that for the Fano variety of lines of cubic fourfold, the fixed locus of the Voisin map is $21c_2$—a result that coincides with the result of~\cite[Theorem A]{FixedLocus}.
\end{remark}

\section{Indeterminacy Locus of the Voisin Map}
In this section, we explore the indeterminacy locus of the Voisin map $\Psi: X \dashrightarrow X$. As previously discussed in Section~\ref{SectionVoisinMaps}, the indeterminacy locus comprises two components, described as follows. The first component, $\mathrm{Ind}_0$, parametrizes points $x \in X$ that represent $P_x \subset Y$, where there are more than one linear subspaces of dimension $r+1$ tangent to $Y$ at $P_x$. The second component, $\mathrm{Ind}_1$, includes points $x \in X$ representing $P_x \subset Y$, where an $(r+1)$-dimensional linear space both contains $P_x$ and is contained within $Y$. The main purpose of this section is to prove Theorem C.

\subsection{Proof of Theorem C}\label{SectionProofOfThmC}
Let us analyse the action of the Voisin map $\Psi$ on the Chern classes of $X$ and on the divisor classes of $X$ respectively. Note that for $r\geq 2$, $CH^1(X)$ is generated by only one divisor $h$, which is the restriction of Plücker line bundle on $\mathrm{Gr}(r+1, n+1)$.
\begin{lemma}\label{LmmActionOfPsiOnDivisor}
    For any $r\geq 1$, we have $\Psi^*h=(3r+4)h$.
\end{lemma}
\begin{proof}
    Let $\mathcal E$ be the tautological subbundle of $X\subset \mathrm{Gr}(r+1, n+1)$. This is a rank $r+1$ subbundle of the trivial vector bundle $V_{n+1}\otimes \mathcal O_X$. Let $\mathcal F$ be the kernel of the following morphism of $\mathcal O_X$-modules
    \[
    \begin{array}{cccc}
       \mathcal G: & V_{n+1}\otimes \mathcal O_X & \to & \mathcal Sym^2\mathcal E^*  \\
         & (a_0, \ldots, a_n) &\mapsto & (\sum_i a_i\frac{\partial f}{\partial x_i}|_{P_x})_{x\in X}.
    \end{array}
    \]
    The codimension of $\mathrm{Ind}_0\subset X$ is $2$. Hence, the result will not change if we replace $X$ by $X':=X-\mathrm{Ind}_0$. We use the same notations for the $\mathcal O_X$-modules restricted to $X'$. Then $\mathcal F$ becomes a vector bundle of rank $r+2$ on $X'$ which fits in the following short exact sequence
    \begin{equation}\label{EqShortExactDefF}
        0\to \mathcal F\to V_{n+1}\otimes \mathcal O_X\to \mathrm{Sym}^2\mathcal E^*\to 0.
    \end{equation}
    From this short exact sequence, we deduce that 
    \begin{equation}
        c_1(\mathcal F)=-(r+2)h.
    \end{equation}
    The homogeneous polynomial $f$ induces a section $\sigma_f$ of the vector bundle $\mathrm{Sym}^3\mathcal F^*$, and the fact that $\mathbb P(\mathcal F|_x)$ is the $(r+1)$-space that is tangent to $Y$ along $P_x$ shows that $\sigma_f$ can be viewed as an element in $H^0(X', (\mathcal F/\mathcal E)^*\otimes (\mathcal F/\mathcal E)^*\otimes \mathcal F^*)$ via the following short exact sequences
    \[
    0\to (\mathcal F/\mathcal E)^*\otimes \mathrm{Sym}^2\mathcal F^*\to \mathrm{Sym}^3\mathcal F^*\to \mathrm{Sym}^3\mathcal E^*\to 0,
    \]
    \[
    0\to (\mathcal F/\mathcal E)^*\otimes (\mathcal F/\mathcal E)^*\otimes \mathcal F^*\to (\mathcal F/\mathcal E)^*\otimes \mathrm{Sym}^2\mathcal F^*\to (\mathcal F/\mathcal E)^*\otimes \mathrm{Sym}^2\mathcal E^*\to 0.
    \]
    Therefore, $\sigma_f$ induces a morphism $\phi: \mathcal F\to (\mathcal F/\mathcal E)^*\otimes (\mathcal F/\mathcal E)^*$. Then the locus of $X$ where $\phi$ is not surjective is $\mathrm{Ind}_1$. By Proposition~\ref{PropCodimAndClassOfInd1} proved in Section~\ref{SectionNumericalData}, the codimension of $\mathrm{Ind}_1$ in $X$ is $r+2$.
    Therefore, the result will not change if we replace $X'$ by $X''=X-\mathrm{Ind}_0-\mathrm{Ind}_1$ and we use the same notations for the restrictions of $\mathcal O_X$-modules to $X''$. On $X''$, the vector bundle $\Psi^*\mathcal E$ is exactly the kernel of $\phi$ and $\Psi^*\mathcal E$ fits into the following short exact sequence
    \[
    0\to \Psi^*\mathcal E\to \mathcal F\to (\mathcal F/\mathcal E)^*\otimes (\mathcal F/\mathcal E)^*\to 0.
    \]
    Therefore, $\Psi^*c_1(\mathcal E)=c_1(\Psi^*\mathcal E)=c_1(\mathcal F)-2c_1((\mathcal F/\mathcal E)^*)=-(r+2)h-2(-(-(r+2)h+h))=-(3r+4)h$. Since $c_1(\mathcal E)=-h$, we find that $\Psi^*h=(3r+4)h$, as desired.
\end{proof}

\begin{remark}
    When $r=1$ and $n=5$, $Y$ is a cubic fourfold and $X$ is a hyper-Kähler fourfold. Our result recovers the result of~\cite{Amerik} and~\cite[Proposition 21.4]{ChowHK} which states $\Psi^*h=7h$.
\end{remark}

\begin{lemma}\label{LmmActionOfPsiOnTangent}
    Let $X^0 = X-\mathrm{Ind}$ be the locus where $X$ is defined. We have $(\Psi^*T_X)|_{X^0} = T_{X^0}$.
\end{lemma}

\begin{proof}
    This result has been implicitly proven in~\cite{KCorr}. Let $\tau: \tilde X\to X$, $\tilde\Psi: \tilde X\to X$ be the desingularisation of the indeterminacies of the Voisin map $\Psi: X\dashrightarrow X$. Since $X$ is a $K$-trivial variety, the exceptional divisor of $\tau$ coincides with the ramification locus of $\tilde\Psi$ (see~\cite[Lemma 4]{KCorr}). Therefore, the Voisin map restricted to the defined domain $\Psi|_{X^0}: X^0\to X$ is étale, which implies that the relative tangent bundle $T_{X^0/X}$ by the map $\Psi|_{X^0}$ is zero. Hence the result. 
\end{proof} 

\begin{proof}[Proof of Theorem C]
    We only consider the case $r=2$. Let $M\in CH_0(X)_{hom}$ be a class
belonging to the subring generated by $h$ and by $c_i(X)$. Let us assume $M$ is a monomial of the form $h^k \prod c_i(X)$. By Lemma~\ref{LmmActionOfPsiOnDivisor} and Lemma~\ref{LmmActionOfPsiOnTangent}, we have
    \begin{equation}\label{EqPsiOnM1}
        (\Psi^*M)|_{X^0} = 10^k\cdot M|_{X^0},
    \end{equation} 
    where $X^0 = X-\mathrm{Ind}$ is the locus where $\Psi$ is defined. On the other hand, by Theorem B, we have \begin{equation}\label{EqPsiOnM2}
        \Psi^*M = -8 M.
    \end{equation} 
    By (\ref{EqPsiOnM1}) and (\ref{EqPsiOnM2}), we find that $M|_{X^0} = 0$, which means that $M$ is supported on $\mathrm{Ind}$, as desired.
\end{proof}

\subsection{Some Numerical Data}\label{SectionNumericalData}
\begin{proposition}\label{PropCodimAndClassOfInd0}
    The first component, $\mathrm{Ind}_0$, has a codimension of $2$ in $X$, and its Chow class in $X$ is expressed as:
    \[
    \left(\frac{1}{2} (r+2)(r+1)+2\right)c_1^2-(r+4)c_2.
    \]
\end{proposition}

\begin{proof}
    With the same notations as in the proof of Lemma~\ref{LmmActionOfPsiOnDivisor} in Section~\ref{SectionProofOfThmC}, 
    Then at point $x\in X$, the projectivation of the vector space $\mathcal F|_x$ is the intersection of the projective tangent spaces $T_{Y,y}$ for all $y\in P_x$. We conclude that $\mathrm{Ind}_0$ is the locus in $X$ where $\mathcal G$ is not surjective. The map $\mathcal G$ sends the subbundle $\mathcal E$ to $0$. Hence, $\mathcal G$ factorizes through the following morphism 
        \[
        \bar{\mathcal G}: (V_{n+1}\otimes \mathcal O_X)/\mathcal E\to \mathrm{Sym}^2\mathcal E^*.
        \]
        For $Y$ general, the locus where $\bar{\mathcal G}$ is not surjective is of codimension $\mathrm{rank}((V_{n+1}\otimes \mathcal O_X)/\mathcal E)-\mathrm{rank}(\mathrm{Sym}^2\mathcal E^*)+1$, which equals $2$ with the relation $n+1 = \binom{r+3}{2}$ taken into account.

        Now let us calculate the Chow class of $\mathrm{Ind}_0$ in $X$. We have viewed $\mathrm{Ind}_0$ as the locus where the map $\bar{\mathcal G}: (V_{n+1}\otimes \mathcal O_X)/\mathcal E\to \mathrm{Sym}^2\mathcal E^*$ is not surjective. By the Porteous formula (see the Introduction part of~\cite{Symmetric}), the class of $\mathrm{Ind}_0$ is represented by the degree $2$ part of \[\frac{c((V_{n+1}\otimes \mathcal O_X)/\mathcal E)}{c(\mathrm{Sym}^2\mathcal E^*)}=\frac{1}{c(\mathcal E)c(\mathrm{Sym}^2\mathcal E^*)}.
        \]
        Suppose $\mathcal E^*$ formally splits into $l_1,\ldots, l_{r+1}$. Then \[c(\mathrm{Sym}^2\mathcal E^*)=\prod_{1\leq i\leq j\leq r+1} (1+l_i+l_j).
    \]
    The degree $1$ part of this expression is given by $\sum_{1\leq i\leq j\leq r+1}(l_i+l_j)=(r+2)c_1$. The degree $2$ part of this expression is given by $\frac12 \sum_{(i,j)\neq (k,l)\in \Delta_{r+1}}(l_i+l_j)(l_k+l_l)$, where $\Delta_{r+1}=\{(i,j): 1\leq i\leq j\leq r+1\}$. We calculate
    \begin{align*}
        &\sum_{(i,j)\neq (k,l)\in \Delta_{r+1}}(l_i+l_j)(l_k+l_l)\\
        =& \sum_{(i,j), (k,l)\in \Delta_{r+1}}(l_i+l_j)(l_k+l_l)-\sum_{(i,j)\in \Delta_{r+1}}(l_i+l_j)^2\\
        =& (\sum_{(i,j)\in\Delta_{r+1}}(l_i+l_j))^2-\sum_{(i,j)\in \Delta_{r+1}}(l_i^2+l_j^2)-2\sum_{(i,j)\in\Delta_{r+1}}l_il_j\\
        =& (r+2)^2c_1^2-(r+4)\sum_il_i^2-2\sum_{i<j}l_il_j\\
        =& (r+2)^2c_1^2-(r+4)(c_1^2-2c_2)-2c_2\\
        =& (r^2+3r)c_1^2+2(r+3)c_2.
    \end{align*}
    Hence, 
    \[
    c(\mathrm{Sym}^2\mathcal E^*)=1+(r+2)c_1+\frac12 (r^2+3r)c_1^2+(r+3)c_2+\ldots
    \]
    On the other hand, we have $c(\mathcal E)=1-c_1+c_2+\ldots$. Finally, by a formal calculation, we find that the degree $2$ part of the expression $\frac{1}{c(\mathcal E)c(\mathrm{Sym}^2\mathcal E^*)}$ is given by $(\frac12 (r+2)(r+1)+2)c_1^2-(r+4)c_2$, as desired.
\end{proof}

\begin{proposition}\label{PropCodimAndClassOfInd1}
    The second component, $\mathrm{Ind}_1$, has a codimension of $r+2$ in $X$ for $r\geq 2$. For $r\leq 1$, the set $\mathrm{Ind}_1$ is empty.
\end{proposition}
\begin{proof}
Recall that $\mathrm{Ind}_1$ is described as the locus of points $x\in X$ such that there is an $r+1$ space in $Y$ containing $P_x$. For $r\geq 2$, the variety $F_{r+1}(Y)$ is nonempty. There is a natural incidence variety 
        \[
        \mathcal P_{r+1,r}:=\{(\lambda, x)\in F_{r+1}(Y)\times F_r(Y): P_\lambda\supset P_x\}.
        \]
        The variety $\mathcal P_{r+1, r}$ is a $\mathbb P^{r+1}$-bundle over $F_{r+1}(Y)$, so $\dim \mathcal P_{r+1, r} = \dim F_{r+1}(Y) + (r + 1)$.
        It is clear that  $\mathrm{Ind}_1$ is the image of the second projection map $\mathcal P_{r+1, r}\to F_r(Y)$. Hence, $\mathrm{codim}\,\mathrm{Ind}_1=\dim F_r(Y)-\dim F_{r+1}(Y)-(r+1)=r+2$. For $r\leq 1$, the variety $F_{r+1}(Y)$ is empty.
    \end{proof}

\subsection{Discussions on Question~\ref{QuestionConstantCycleInd}}\label{SectionDiscussionsOnQuestion3}
In this part, we would like to give some evidence on Question~\ref{QuestionConstantCycleInd}. Let $X = F_r(Y)$ be the strict Calabi-Yau manifold constructed as in Section~\ref{SectionVoisinExample} for any $r\geq 2$. Let $\mathcal P = \{(x, y)\in X\times Y: y\in P_x\}$ be the incidence correspondence of $X$ and $Y$. Let $\tau: \tilde X\to X$, $\tilde\Psi: \tilde X\to X$ be the desingularization of the indeterminacy locus of the Voisin map $\Psi: X\dashrightarrow X$. 

\begin{proposition}\label{PropIfInjThenDivisorConstantCycle}
    If the induced map
    \[
    \mathcal P_*: CH_1(X)_{hom}\to CH_{r+1}(Y)_{hom}
    \]
    is injective, then the divisor $\tilde\Psi(\tau^{-1}(\mathrm{Ind}_0))$ is a constant-cycle subvariety in $X$.
\end{proposition}
\begin{proof}
    Let $x\in Ind_0$ be a general point representing an $r$-linear subspace $P_x$ in $Y$. Then there is a unique $(r+2)$-linear subspace $\Theta_x$ that is tangent to $Y$ along $P_x$. Then the preimage $\tau^{-1}(x)$ is a rational curve $\mathbb P^1_x$ parametrizing the $(r+1)$-linear subspaces $H$ containing $P_x$ and contained in $\Theta_x$. Each $H$ intersects $Y$ with a double $P_x$ and a residual $r$-linear subspace $P_{x'}$ represented by a point $x'\in X$. The image $\Gamma_x:= \tilde\Psi(\mathbb P^1_x)$ is the rational curve in $X$ parametrizing all these residual planes $P_{x'}$. Applying the {correspondence} $\mathcal P_*$ on $\Gamma_x$, it is not hard to see that $\mathcal P_*(\Gamma_x)$ is the algebraic cycle represented by $\Theta_x\cap Y$, and that this class is independent of the choice of $x\in \mathrm{Ind}_0$ since it is a linear section of $Y$. By our assumption that $\mathcal P_*: CH_1(X)_{hom}\to CH_{r+1}(Y)_{hom}$, the Chow classes of the rational curves $\Gamma_x\in X$ are independent of the choice of $x\in \mathrm{Ind}_0$. Notice that the divisor $\tilde\Psi(\tau^{-1}(\mathrm{Ind}_0))$ is the closure of the union of these rational curves $\Gamma_x$. Let $y_x\in \Gamma_x$ and $y_{x'}\in \Gamma_{x'}$ and let $D\subset X$ be an ample divisor such that $d = \deg(D.\Gamma_x) = \deg(D.\Gamma_{x'})$. Then $d.y_x = D.\Gamma_x = D.\Gamma_{x'} = d.y_{x'}$. Since $CH_0(X)_{hom}$ is torsion free, we find that $y_x = y_{x'}$ in $CH_0(X)$, as desired.
\end{proof}

Now let us focus on the case $r = 2$. 
\begin{proposition}\label{PropIfInjThenConstantCycle}
    In the case $r = 2$, if the induced map
    \[
    \mathcal P_*: CH_1(X)_{hom}\to CH_{3}(Y)_{hom}
    \]
    is injective, then the first component $\mathrm{Ind}_0$ of the indeterminacy locus is a constant-cycle subvariety.
\end{proposition}
\begin{proof}
    In the case $r=2$, Theorem B implies that $\Psi_*(x_1 - x_2) = -8(x_1 - x_2)$. But Proposition~\ref{PropIfInjThenDivisorConstantCycle} implies that $\Psi_*(x_1 - x_2) = 0$ since the divisor $\tilde\Psi(\tau^{-1}(\mathrm{Ind}_0))$ is a constant-cycle subvariety in $X$. Therefore, $x_1 = x_2\in CH_0(X)$ since $CH_0(X)$ is torsion-free.
\end{proof}

\begin{remark}
    If we assume Conjecture~\ref{PropConditionalActionOfPsiOnChow}, and if we suppose that the induced map
    \[
    \mathcal P_*: CH_1(X)_{hom}\to CH_{r+1}(Y)_{hom}
    \]
    is injective, then using the same argument as in Proposition~\ref{PropIfInjThenConstantCycle}, we can actually show that $\mathrm{Ind}_0$ is a constant-cycle subvariety for any $r\geq 2$. Notice that $\mathcal P_*: CH_1(X)_{hom}\to CH_2(Y)_{hom}$ is \emph{not} injective in the hyper-Kähler case, since in that case, we know that $CH_2(Y)_{hom} = 0$ (see~\cite[Theorem 1 (ii)]{BlochSrinivas}) and \cite[Section 5.0.1]{Cubic} while $CH_1(F_1(Y))_{hom}$ is big (see~\cite[Theorem 21.9]{ChowHK}). This phenomenon is also reflected by the fact that there is not any constant-cycle divisor in a hyper-Kähler manifold of dimension $\geq 4$ (see~\cite{VoisinCoisotrope}) and that $\mathrm{Ind}_0$ is not a constant-cycle subvariety in this case (see~\cite[Lemma 1]{Amerik}, \cite{KCorr}).
\end{remark}

In the spirit of Proposition~\ref{PropIfInjThenDivisorConstantCycle} and Proposition~\ref{PropIfInjThenConstantCycle}, it is reasonable to propose the following
\begin{question}\label{QuestionInjectivityOfChow}
    If $r\geq 2$, is the induced map \[
    \mathcal P_*: CH_1(X)_{hom}\to CH_{r+1}(Y)_{hom}
    \]
    injective?
\end{question}

In alignment with the principles of the generalized Bloch conjecture and the generalized Hodge conjecture, the Chow group of $1$-cycles on a smooth projective variety $X$ is expected to be ``governed" by the quotient $N^1/N^2H^*(X, \mathbb{Q})$, where $N^i$ denotes the Hodge coniveau filtration on the cohomology $H^*(X, \mathbb{Q})$. Given that $X$ is a strict Calabi-Yau manifold, the derivative of the period map
\[
\mathcal{D}: \mathrm{Def}(X) \to \mathbb{P}H^N(X,\mathbb{C}),
\]
that associates any point $b \in \mathrm{Def}(X)$ with the $1$-dimensional subspace $H^{N, 0}(X_b) \subset H^N(X_b, \mathbb{C}) \cong H^N(X, \mathbb{C})$, has as its image 
\[
\mathrm{Hom}(H^{N, 0}(X), H^{N-1, 1}(X)) \subset \mathrm{Hom}(H^{N, 0}(X), H^N(X,\mathbb{C})/H^{N-1, 1}(X)),
\]
according to Griffiths' theory on period maps. Question~\ref{QuestionInjectivityOfChow} would likely have an affirmative answer if the following conditions are met:
\begin{itemize}
    \item The map $F_r: \mathbb{P}H^0(\mathbb{P}^n, \mathcal{O}_{\mathbb{P}^n}(3)) \dashrightarrow \mathrm{Def}(X)$ is dominant, implying that for a general $Y$, the Fano variety of $r$-linear spaces within $Y$, denoted by $X$, is also general in the moduli space of $X$. Consequently, $N^1/N^2H^N(X, \mathbb{Q}) = 0$.
    \item The induced map $[P]^*: H^{p+r, 1+r}(Y) \to H^{p, 1}(X)$ is surjective for $1 < p < N$. Considering that $[P]^*$ is injective (refer to, for example,~\cite[Lemma 4.6]{VoisinCitrouille}), it is sufficient to compute the dimension of $H^{p,1}(X)$ and compare it to that of $H^{p+1, 1+r}(Y)$. Notice that the dimension of $H^{p+1, 1+r}(Y)$ is well-established (see, for example, \cite[Section 1.1]{Cubic} or~\cite[Chapitre 18]{Voisin}). 
\end{itemize}

The following provides some (partial) affirmative responses regarding the two aforementioned desired conditions.
\begin{theorem}\label{ThmLocalDeformation}
    Let $X = F_r(Y)$ be the strict Calabi-Yau manifold as described in Section~\ref{SectionVoisinExample}, with $r \geq 2$, then the map $F_r: \mathbb{P}H^0(\mathbb{P}^n, \mathcal{O}_{\mathbb{P}^n}(3)) \dashrightarrow \mathrm{Def}(X)$, which assigns to $Y \in \mathbb{P}H^0(\mathbb{P}^n, \mathcal{O}_{\mathbb{P}^n}(3))$ its Fano variety of $r$-spaces $F_r(Y)$, is dominant and has a relative dimension of $n^2+2n = \dim \mathrm{PSL}_{n+1}(\mathbb{C})$.
\end{theorem}

\begin{remark}
    Theorem~\ref{ThmLocalDeformation} does not apply if $r = 1$, due to the existence of non-projective deformations of hyper-Kähler manifolds. This limitation is further highlighted by our proof of Theorem~\ref{ThmLocalDeformation}, which crucially relies on the condition that $r \geq 2$.
\end{remark}

For the explicit calculations of Hodge numbers $h^{p,1}(X)$, we only manage to get the result in the case $r = 2$ due to the complexity of the computation. 
\begin{proposition}\label{PropHodgeNumberCY}
    In the case of $r=2$, we have the following
    \[
    h^{p,1}(X) = \left\{\begin{array}{cc}
       1,  & p = 1  \\
       45,  & p = 3 \\
       120, & p = 10 \\
       0, & \textrm{others}.
    \end{array}\right.
    \]
\end{proposition}
Notice that for $p > 1$, the only non-zero $H^{p+2, 3}(Y)$ for a cubic eightfold $Y$ is $H^{5, 3}(Y)$ whose dimension is $45$ (see~\cite[Section 1.1]{Cubic}). Thus, Proposition~\ref{PropHodgeNumberCY} indeed confirms the second condition mentioned above in the case of $r = 2$.

The proofs of Theorem~\ref{ThmLocalDeformation} and Proposition~\ref{PropHodgeNumberCY} are detailed in the following sections.
\subsubsection{Proof of Theorem~\ref{ThmLocalDeformation}}
\begin{proof}[Proof of Theorem~\ref{ThmLocalDeformation}]
    For the brevity of the notation, in what follows, we will denote $G$ as the Grassmannian $\mathrm{Gr}(r+1, n+1)$ and $S^3\mathcal E^*$ as the symmetric product $\mathrm{Sym}^3\mathcal E^*$. 
    Let 
    \[\alpha: H^0(\mathbb P^n, \mathcal O_{\mathbb P^n}(3))\cong H^0(G, S^3\mathcal E^*)\to H^0(X, S^3\mathcal E^*|_X)\cong H^0(X, N_{X/G})\]
    be the restriction map composed with some canonical isomorphisms. Let 
    \[
    \beta: H^0(X, N_{X/G})\to H^1(X, T_X)
    \]
    be the connection map of the normal exact sequence. 
    With the above notations, we have the following two lemmas that we will prove later.
    \begin{lemma}\label{LmmAlphaIdentifyTangentSpace}
        $H^0(X, N_{X/G})$ is identified with the tangent space of $\mathbb PH^0(\mathbb P^n, \mathcal O_{\mathbb P^n}(3))$ at the point $f$ representing the cubic hypersurface $Y_0$. 
    \end{lemma}
    \begin{lemma}\label{LmmBetaCalcul}
        The map $\beta$ is surjective and $\dim \ker \beta =  n^2 + 2n$.
    \end{lemma}
    Therefore, the map $\beta: H^0(X, N_{X/G})\to H^1(X, T_X)$ identifies with the tangent map of $F_r: \mathbb PH^0(\mathbb P^n, \mathcal O_{\mathbb P^n}(3))\dashrightarrow \mathrm{Def}(X))$, and Theorem~\ref{ThmLocalDeformation} follows, knowing that $\mathrm{Def}(X)$ is smooth since $X$ is a strict Calabi-Yau manifold (see~\cite{BogomolovBTT, TianBTT, TodorovBTT} and~\cite{Ran}).
\end{proof}
The proof of Lemma~\ref{LmmAlphaIdentifyTangentSpace} and of Lemma~\ref{LmmBetaCalcul} is a direct calculation using the following theorem of Bott~\cite{Bott, DemazureSimple}. We use the following formulation presented in~\cite[Proposition 2]{Manivel}. In the statement, $L_\lambda$ represents the Schur's functor of weight $\lambda$.

\begin{theorem}[Bott~\cite{Bott, DemazureSimple, Manivel}]\label{ThmBott}
    Given decreasing weights $\lambda_1\in \mathbb Z^{n-r}$ and $\lambda_2\in \mathbb Z^{r+1}$. Let $\lambda = (\lambda_1, \lambda_2)\in\mathbb Z^{n+1}$. Let $c(n+1) = (1, 2,\ldots, n+1)$. If $\lambda - c(n+1)$ has common components, then $H^q(\mathrm{Gr}(r+1, n+1), L_{\lambda_1}\mathcal Q\otimes L_{\lambda_2}\mathcal E) = 0$ for any $q$ (and in this case, we say that $\lambda - c(n+1)$ is irregular). Otherwise, we have
    \[
    H^q(\mathrm{Gr}(r+1, n+1), L_{\lambda_1}\mathcal Q\otimes L_{\lambda_2}\mathcal E) = \delta_{q, i(\lambda)}L_{\xi(\lambda)}V_{n+1}.
    \]
    Here $i(\lambda)$, $\xi(\lambda)$ are defined as follows. The weight $\xi(\lambda)$ is defined as $(\lambda - c(n+1))^{\geq} + c(n+1)$, where $(\lambda - c(n+1))^{\geq}$ is the decreasing integer sequence obtained by permutation of the sequence $\lambda - c(n+1)$. The number $i(\lambda)$ is the inversion number of the sequence $\lambda - c(n+1)$.
\end{theorem}
The following technical lemma, due to Debarre and Manivel~\cite[Lemma 3.9]{DebarreManivel}, is also used frequently.
\begin{lemma}[Debarre-Manivel~\cite{DebarreManivel}]\label{LmmDM}
    Let $V$ be a complex vector space, $m$ and $d$ be integers. For any irreducible component $L_\lambda V$ of $\bigwedge^j\mathrm{Sym}^dV$, we have 
    \[
    |\lambda|_{>m} \geq j - \binom{m+d-1}{d}.
    \]
    Here, we write $\lambda$ as a decreasing sequence of integers $(\lambda_1, \ldots, \lambda_{\dim V})$ and $|\lambda|_{>m} = \sum_{i>m}\lambda_i$.
\end{lemma}

\begin{proof}[Proof of Lemma~\ref{LmmAlphaIdentifyTangentSpace}]
Let $\mathcal I_X$ be the ideal sheaf of $X$ in $G$ that fits into the following short exact sequence
\[
0\to \mathcal I_X\otimes S^3\mathcal E^*\to S^3\mathcal E^*\to N_{X/G}\to 0.
\]
To prove Lemma~\ref{LmmAlphaIdentifyTangentSpace}, it suffices to prove that $H^0(G, \mathcal I_X\otimes S^3\mathcal E^*)$ is of dimension $1$ while $H^1(G, \mathcal I_X\otimes S^3\mathcal E^*) = 0$. By the Koszul resolution
\[
\ldots \to \bigwedge^{i+1}S^3\mathcal E\otimes S^3\mathcal E^*\to \bigwedge^{i}S^3\mathcal E\otimes S^3\mathcal E^*\to \ldots \to S^3\mathcal E\otimes S^3\mathcal E^*\to \mathcal I_X\otimes S^3\mathcal E^*\to 0,
\]
it suffices to prove
\begin{itemize}
    \item[(i)] $\dim H^0(G, S^3\mathcal E\otimes S^3\mathcal E^*) = 1$ and $H^1(G, S^3\mathcal E\otimes S^3\mathcal E^*) = 0$.
    \item[(ii)] $H^{i-d}(G, \bigwedge^iS^3\mathcal E\otimes S^3\mathcal E^*) = 0 $ for any $i\geq 2$ and for any $d = 0, 1, 2$.
\end{itemize}
To prove (i), by the Littlewood-Richardson rule~\cite[Section 2.1.1]{Manivel}, we find $S^3\mathcal E\otimes S^3\mathcal E^* = \mathcal O_G\oplus L_{(3,0,\ldots, 0, -3)}\mathcal E$. Since $n-r\geq 3$ in our setting, the sequence $(0,\ldots, 0, 3, 0, \ldots, 0, -3) - c(n+1)$ is irregular. Hence, $H^q(G, S^3\mathcal E\otimes S^3\mathcal E^*) = H^q(G, \mathcal O_G)$. The latter has dimension $1$ when $q = 0$ and $0$ otherwise. This proves (i).

To prove (ii), let $L_\lambda \mathcal E\subset \bigwedge^iS^3\mathcal E\otimes S^3\mathcal E^*$ be an irreducible component. Then there is a weight $\mu$ such that $L_\lambda\mathcal E\subset L_\mu\mathcal E\otimes S^3\mathcal E^*$. By the Littleword-Richardson rule, for any $m$, we have $|\lambda|_{>m}\geq |\mu|_{>m} - 3$. Therefore, by Lemma~\ref{LmmDM}, we find
\begin{equation}\label{Equation1}
    |\lambda|_{>m} \geq i - \binom{m+2}{3} - 3.
\end{equation}
By Theorem~\ref{ThmBott}, $H^{i-d}(L_\lambda \mathcal E)\neq 0$ implies that $(-1, -2, \ldots, -(n-r), \lambda_1 - (n-r) -1, \ldots, \lambda_{r+1} - (n+1))$ is regular and that $i-d$ is its inversion number. Equivalently, it means that there exists $0\leq h\leq r+1$, such that 
\begin{equation}\label{Equation2}
    \lambda_h - (n - r) - h \geq 0,
\end{equation}
\begin{equation}\label{Equation3}
    \lambda_{h+1} - (n - r) - (h+1)\leq - (n - r + 1),
\end{equation}
\begin{equation}\label{Equation4}
    i - d = h(n - r).
\end{equation}
Here, we use the convention that $\lambda_0 = +\infty$ and that $\lambda_{r+2} = -\infty$ to simplify the notations. We divide into three cases and conclude none of them is possible.
\begin{itemize}
    \item If $h = r+1$, then by (\ref{Equation4}), $i = (r+1)(n-r) + d > \mathrm{rank} S^3\mathcal E$. Hence, $H^{i-d}(\bigwedge^iS^3\mathcal E\otimes S^3\mathcal E^*) = 0$.
    \item If $0 < h < r+1$, then we have $i = h(n-r)+d$ by (\ref{Equation4}), and $|\lambda|_{>h}\leq h(r+1-h)$ by (\ref{Equation3}). Combining with (\ref{Equation1}), we obtain
    \[
    h(n-r) + d - \binom{h+2}{3} - 3 \leq h(r+1-h),
    \]
    or equivalently,
    \begin{equation}\label{Equation5}
        n - 2r - 1 \leq \frac16(h-1)(h-2) + \frac{3-d}{h} =: \phi_d(h).
    \end{equation}
    It is not hard to see that 
    \[
    \max_{0<h<r+1}\phi_d(h) \leq \left\{
    \begin{array}{cc}
        3-d, & r = 2  \\
        \frac16(r-1)(r-2)+1, & r\geq 3.
    \end{array}\right.
    \]
    A direct comparition with (\ref{Equation5}) gives a contradiction when $r\geq 2$.
    \item If $h = 0$, then $i = d$ by (\ref{Equation4}). But we have assumed $i\geq 2$, so there is only the case $i=d=2$ to discuss. In this case, (\ref{Equation3}) shows that $\lambda_1\leq 0$. We will show that this is impossible. In fact, the irreducible components of $\bigwedge^2S^3\mathcal E$ are $L_{(3,3)}\mathcal E$ and $L_{(5, 1)}\mathcal E$. By the Littlewood-Richardson rule, the $\lambda_1$ in $L_{(3,3)}\mathcal E\otimes S^3\mathcal E^*$ would be $\geq 1$ and that in $L_{(5,1)}\mathcal E\otimes S^3\mathcal E^*$ would be $\geq 3$, contradicting the constraints that $\lambda_1\leq 0$. 
\end{itemize}
\end{proof}

\begin{proof}[Proof of Lemma~\ref{LmmBetaCalcul}]
    Let us consider the normal exact sequence
    \[
    0\to T_X\to T_{G|X}\to N_{X/G}\to 0. 
    \]
    Since $X$ is strict Calabi-Yau, $H^0(X, T_X) = 0$. Hence, the kernel of $\beta: H^0(X, N_{X/G})\to H^1(X, T_X)$ is $H^0(X, T_{G|X})$ and the cokernel is contained in $H^1(X, T_{G|X})$. Hence, to prove Lemma~\ref{LmmBetaCalcul}, it suffices to show that $H^0(X, T_{G|X})\cong H^0(G, T_G)$ and that $H^1(X, T_{G|X}) = 0$. Considering the short exact sequence
    \[
    0\to \mathcal I_X\otimes T_G\to T_G\to T_{G|X}\to 0,
    \]
    and knowing that $H^1(G, T_G) = 0$,
    it suffices to show that $H^i(G, \mathcal I_X\otimes T_G) = 0$ for $i = 0, 1, 2$. Now we take the Koszul resolution of $\mathcal I_X\otimes T_G$, noticing that $T_G=\mathcal E^*\otimes \mathcal Q$:
    \[
    \ldots \to \bigwedge^{i+1}S^3\mathcal E\otimes \mathcal E^*\otimes \mathcal Q\to \bigwedge^{i}S^3\mathcal E\otimes \mathcal E^*\otimes \mathcal Q\to \ldots \to S^3\mathcal E\otimes \mathcal E^*\otimes \mathcal Q\to \mathcal I_X\otimes E^*\otimes \mathcal Q\to 0.
    \]
    Therefore, it suffices to show that $H^{i+d}(G, \bigwedge^iS^3\mathcal E\otimes \mathcal E^*\otimes \mathcal Q) = 0$ for any $i\geq 1$ and any $d\in\{ -1, 0, 1\}$. To prove this, let $L_\lambda\mathcal E\subset \bigwedge^iS^3\mathcal E\otimes \mathcal E^*$ be an irreducible component. By the Littlewood-Richardson rule and Lemma~\ref{LmmDM}, we have 
    \begin{equation}\label{Equation6}
        |\lambda|_{>m} \geq i - \binom{m+2}{3} - 1.
    \end{equation}
    In order that the sequence $(1, 0, \ldots, 0, \lambda_1, \ldots, \lambda_{r+1}) - c(n+1) = (0, -2, -3,\ldots, -(n-r), \lambda_1 - (n-r) -1, \ldots, \lambda_{r+1}-(n+1))$ to be regular, and that $H^{i+d}(G, \bigwedge^iS^3\mathcal E\otimes \mathcal E^*\otimes \mathcal Q)\neq 0$, one of the following two cases needs to happen:
    \begin{itemize}
        \item[(a)] There exists $0\leq h\leq r+1$ such that 
        \begin{equation}\label{Equation7}
            \lambda_h - (n-r) - h \geq 1,
        \end{equation}
        \begin{equation}\label{Equation8}
            \lambda_{h+1} - (n-r) - (h+1) \leq -(n-r) - 1,
        \end{equation}
        \begin{equation}\label{Equation9}
            i+d = h(n-r).
        \end{equation}
        \item[(b)] There exists $0\leq h\leq r$ such that 
        \begin{equation}\label{Equation10}
            \lambda_h - (n - r) - h \geq 1,
        \end{equation}
        \begin{equation}\label{Equation11}
            \lambda_{h+1} - (n-r)-(h+1) = -1,
        \end{equation}
        \begin{equation}\label{Equation12}
            \lambda_{h+2} - (n-r) - (h+2) \leq - ( n-r) -1,
        \end{equation}
        \begin{equation}\label{Equation13}
            i+d = h(n-r) + (n-r-1).
        \end{equation}
    \end{itemize}
    We will show that in both cases, the cohomology is always zero.

    For Case (a),
    \begin{itemize}
        \item If $h = r+1$, then (\ref{Equation9}) shows that $i = (r+1)(n-r)-d > \mathrm{rank}S^3\mathcal E$, we have $\bigwedge^iS^3\mathcal E = 0$.
        \item If $h = 0$, then (\ref{Equation9}) shows that $i = -d$. Since we suppose $i\geq 1$, only the case $i=-d=1$ remains to be checked. In this case, (\ref{Equation8}) becomes $\lambda_1\leq 0$. However, by the Littlewood-Richardson rule, in the decomposition of $S^3\mathcal E\otimes \mathcal E^*$, $\lambda_1\geq 2$.
        \item If $0 < h < r+1$, then $|\lambda|_{>h}\leq h(r+1-h)$ by (\ref{Equation8}). Combining (\ref{Equation6}) and (\ref{Equation9}), we get
        \[
        h(n-r) - d - \binom{h+2}{3} - 1 \leq h(r+1-h),
        \]
        and this inequality can be shown to be impossible using similar method as in (\ref{Equation5}). 
    \end{itemize}

    For Case (b), we have $|\lambda|_{>h}\leq n-r+h+(r-h)(h+1) = n + (r - h)h$. Combining with (\ref{Equation6}) and (\ref{Equation13}), we get
    \[
    h(n-r) + (n - r - 1) - 1 - \binom{h+2}{3} - 1 \leq n + (r-h)h,
    \]
    or equivalently,
    \begin{equation}\label{Equation14}
        \binom{h+2}{3} - h(h+n - 2r) + r + 3 \geq 0.
    \end{equation}
    It is easy to check that (\ref{Equation14}) is impossible if $r=2, h=2$, or if $r\geq 3, 0< h\leq r$. In fact, define $\phi(h) = \binom{h+2}{3} - h(h+n -2r) + r + 3$. Then $\phi(0) = r+3 > 0\geq \phi(1) = - n + 3r +3$. Since $\phi(h)$ is a cubic function, it suffices to show that $\phi(r) < 0$ for $r\geq 2$ and that $\phi(1) < 0$ for $r\geq 3$. The latter is easy to check. For the former, we have 
    \[
    \phi(r) = \binom{r+2}{3}- r(n-r) + r + 3 = r(-\frac13 r^2 - r - \frac53 + \frac3r) < 0
    \]
    for any $r\geq 2$. The remaining cases that have not been checked are $r=2, h=1$ or $h=0$. To this ends, we consider $|\lambda|_{>h+1}$, by (\ref{Equation6}), (\ref{Equation12}) and (\ref{Equation13}), we get 
    \begin{equation}\label{Equation15}
        (h+1)(n-r) - 3 - \binom{h+3}{3} \leq (r-h)(h+1).
    \end{equation}
    If $r=2, h=1$, then the left-hand-side of (\ref{Equation15}) is $7$ whereas the right-hand-side is $2$—the inequality (\ref{Equation15}) does not hold. If $h=0$, the left-hand-side of (\ref{Equation15}) is $n - r - 4$ whereas the right-hand-side is $r$—the inequality (\ref{Equation15}) does not hold either. We have eliminated all the possibilities and Lemma~\ref{LmmBetaCalcul} is proven.
\end{proof}

\subsubsection{Proof of Proposition~\ref{PropHodgeNumberCY}}\label{SectionProofOfHodgeNumberCY}
\begin{proof}[Proof of Proposition~\ref{PropHodgeNumberCY}]
    By Hodge symmetry, let us calculate instead the dimension of $H^q(X, \Omega_X)$ for any $q$. Consider the conormal exact sequence
    \[
    0\to S^3\mathcal E|_X\to (\mathcal E\otimes \mathcal Q^*)|_X\to \Omega_X\to 0,
    \]
    we need to know the cohomology groups of $S^3\mathcal E|_X$ and of $S^3\mathcal E|_X$.

    \textbf{Calculation of $H^i(X, S^3\mathcal E|_X)$}: by the Koszul resolutioin, we need to calculate the cohomology groups of $\bigwedge^kS^3\mathcal E\otimes S^3\mathcal E$. By the calculations using SageMath, the only non-zero cohomology groups of the form $H^j(G, \bigwedge^kS^3\mathcal E\otimes S^3\mathcal E)$ are listed as below.
    \begin{itemize}
        \item $h^7(G, \bigwedge^2S^3\mathcal E\otimes S^3\mathcal E) = 10$.
        \item $h^7(G, \bigwedge^3S^3\mathcal E\otimes S^3\mathcal E) = 55$.
        \item $h^{14}(G, \bigwedge^6S^3\mathcal E\otimes S^3\mathcal E) = 10$.
        \item $h^{21}(G, \bigwedge^9S^3\mathcal E\otimes S^3\mathcal E) = 1$.
        \item $h^{21}(G, \bigwedge^{10}S^3\mathcal E\otimes S^3\mathcal E) = 220$.
    \end{itemize}
    The above calculation gives us the following consequences.
    \begin{itemize}
        \item There is an exact sequence
        \[
        0\to H^4(X, S^3\mathcal E|_X)\to H^7(G,\bigwedge^3S^3\mathcal E\otimes S^3\mathcal E) \to H^7(G, \bigwedge^2S^3\mathcal E\otimes S^3\mathcal E)\to H^5(X, S^3\mathcal E|_X)\to 0.
        \]
        The calculation using SageMath, together with Theorem~\ref{ThmBott}, shows that $H^7(G,\bigwedge^3S^3\mathcal E\otimes S^3\mathcal E)\cong H^7(G, L_{(10, 1, 1)}\mathcal E)\cong L_{(3, 1,\ldots, 1)}V_{10}\cong S^2V_{10}\otimes \det V_{10}$, that $H^7(G,\bigwedge^2S^3\mathcal E\otimes S^3\mathcal E)\cong H^7(G, L_{(8,1)}\mathcal E)\cong L_{(1,1,\ldots, 1,0)}V_{10}\cong V_{10}^*\otimes \det V_{10}$, and that the map $H^7(G,\bigwedge^3S^3\mathcal E\otimes S^3\mathcal E) \to H^7(G, \bigwedge^2S^3\mathcal E\otimes S^3\mathcal E)$ is viewed as $\lrcorner \sigma_f: S^2V_{10}\otimes \det V_{10}\to V_{10}^*\otimes \det V_{10}$, where $\sigma_f$ is the global section of $H^0(\mathbb P^9, \mathcal O(3))\cong S^3H^0(\mathbb P^9, \mathcal O(1))= S^3V_{10}^*$, and a generic choice of $\sigma_f$ gives surjective contraction map $$\lrcorner \sigma_f: S^2V_{10}\otimes \det V_{10}\to V_{10}^*\otimes \det V_{10}.$$ Therefore, we conclude that $h^4(X, S^3\mathcal E|_X) = 45$ and that $h^5(X, S^3\mathcal E|_X) = 0$.
        \item $h^8(X, S^3\mathcal E|_X) = h^{14}(G, \bigwedge^6S^3\mathcal E\otimes S^3\mathcal E) = 10$.
        \item There is an exact sequence
        \[
        0\to H^{11}(X, S^3\mathcal E|_X)\to H^{21}(G, \bigwedge^{10}S^3\mathcal E\otimes S^3\mathcal E)\to H^{21}(G, \bigwedge^{9}S^3\mathcal E\otimes S^3\mathcal E)\to H^{12}(X, S^3\mathcal E|_X) \to 0.
        \]
        But $H^{12}(X, S^3\mathcal E|_X)=0$ since $\dim X = 11$. Hence, $h^{11}(X, S^3\mathcal E|_X) = 219$. 
        \item $H^i(X, S^3\mathcal E|_X) = 0$ for other $i$.
    \end{itemize}

    \textbf{Calculation of $H^i(X, (\mathcal E\otimes \mathcal Q^*)|_X)$}: by the Koszul resolution, we need to calculate the cohomology group of $\bigwedge^kS^3\mathcal E\otimes \mathcal E\otimes \mathcal Q^*$. This is done using SageMath. The only non-zero cohomology groups of the form $H^j(G, \bigwedge^kS^3\mathcal E\otimes \mathcal E\otimes \mathcal Q^*)$ are listed as follows.
    \begin{itemize}
        \item $h^1(G, \bigwedge^0S^3\mathcal E\otimes \mathcal E\otimes \mathcal Q^*) = 1$.
        \item $h^{15}(G, \bigwedge^7S^3\mathcal E\otimes \mathcal E\otimes \mathcal Q^*) = 10$.
        \item $h^{21}(G, \bigwedge^{10}S^3\mathcal E\otimes \mathcal E\otimes \mathcal Q^*) = 99$.
    \end{itemize}
    The above calculation gives us the following consequences.
    \begin{itemize}
        \item $h^1(X, (\mathcal E\otimes \mathcal Q^*)|_X) = 1$.
        \item $h^8(X, (\mathcal E\otimes \mathcal Q^*)|_X) = 10$.
        \item $h^{11}(X, (\mathcal E\otimes \mathcal Q^*)|_X) = 99$.
        \item $h^i(X, (\mathcal E\otimes \mathcal Q^*)|_X) = 0$ for other $i$. 
    \end{itemize}

\textbf{Calculation of $H^q(X, \Omega_X)$}: We use the conormal exact sequence
\[
0\to S^3\mathcal E|_X\to (\mathcal E\otimes \mathcal Q^*)|_X\to \Omega_X\to 0
\]
and the above calculation to obtain:
\begin{itemize}
    \item $h^1(X, \Omega_X) = 1$.
    \item $h^3(X, \Omega_X) = 45$.
    \item There is an exact sequence
    \[
    0\to H^7(X, \Omega_X)\to H^8(X, S^3\mathcal E|_X)\to H^8(X, (\mathcal E\otimes \mathcal Q^*)|_X)\to H^8(X, \Omega_X)\to 0.
    \]
    By chasing the diagram, we find that $H^8(X, S^3\mathcal E|_X)\cong H^{14}(G, L_{(10, 9, 2)}\mathcal E)\cong L_{(3,1,\ldots, 1)}V_{10}$ and $H^8(X, \Omega_X)\cong H^{15}(G, L_{(0,\ldots, 0, -1)}\mathcal Q\otimes L_{(10, 9, 3)}\mathcal E)\cong L_{(3,1,\ldots, 1)}V_{10}$ and the map in between is given by the identity map. Therefore, $H^7(X, \Omega_X)=H^8(X, \Omega_X)=0$.
    \item There is an exact sequence
    \[
    0\to H^{10}(X, \Omega_X)\to H^{11}(X, S^3\mathcal E|_X)\to H^{11}(X, (\mathcal E\otimes \mathcal Q^*)|_X)\to H^{11}(X, \Omega_X)\to 0.
    \]
    But $h^{11}(X, \Omega_X)\cong h^1(X, K_X)\cong h^{1,0}(X) = 0$. Hence, $h^{10}(X, \Omega_X) = 219 - 99 = 120$. This coincides with our expectation since $h^{10}(X, \Omega_X) = h^1(X, \Omega_X^{10}) = h^1(X, T_X)$ and $h^1(X, T_X) = 120$ by our calculation in Theorem~\ref{ThmLocalDeformation}.
    \item $h^i(X, \Omega_X) = 0$ for other $i$.
\end{itemize}
\end{proof}

\appendix
\chapter{Computational Verification in Chapter~\ref{ChapterVoisinMaps}}
\section{Computational Verification in Section~\ref{SectionChowClassOfFInX}}
This appendix provides a comprehensive breakdown of the calculations underpinning the results introduced in Section~\ref{SectionChowClassOfFInX}, employing Mathematica for computational support.

\subsection{Computational verification using Mathematica I}\label{SectionMathematicaI}

\begin{doublespace}
\noindent\(\pmb{\text{(*Define the variables*)}}\\
\pmb{\text{li} = \text{Array}[l, 3]; \text{(*l1, l2, l3*)}}\\
\pmb{\text{mj} = \text{Array}[m, 6]; \text{(*m1, m2, m3, m4, m5, m6*)}}\\
\pmb{\text{(*Create the product expression*)}}\\
\pmb{\text{productExpr} = \text{Product}[\text{li}[[i]] + \text{mj}[[j]], \{i, 1, 3\}, \{j, 1, 6\}];}\\
\pmb{\text{sym1} = \text{SymmetricReduction}[\text{productExpr}, \text{mj}, \{\text{d1}, \text{d2}, \text{d3}, \text{d4}, \text{d5}, \text{d6}\}][[1]];}\\
\pmb{\text{sym2} = \text{SymmetricReduction}[\text{sym1}, \text{li}, \{\text{c1}, \text{c2}, \text{c3}\}][[1]]}\)
\end{doublespace}

\begin{doublespace}
\noindent\(\text{c3}^6+\text{c2} \text{c3}^5 \text{d1}+\text{c1} \text{c3}^5 (\text{d1}^2-2 \text{d2})+\text{c2}^2 \text{c3}^4 \text{d2}+\text{c1}
\text{c2} \text{c3}^4 (\text{d1} \text{d2}-3 \text{d3})+\text{c2}^3 \text{c3}^3 \text{d3}+\text{c3}^5 (\text{d1}^3-6 \text{d3}+3 (-\text{d1}
\text{d2}+3 \text{d3}))+\text{c1} \text{c2}^2 \text{c3}^3 (\text{d1} \text{d3}-4 \text{d4})+\text{c2}^4 \text{c3}^2 \text{d4}+\text{c1}^2 \text{c3}^4
(\text{d2}^2-6 \text{d4}+2 (-\text{d1} \text{d3}+4 \text{d4}))+\text{c2} \text{c3}^4 (\text{d1}^2 \text{d2}-12 \text{d4}+5 (-\text{d1}
\text{d3}+4 \text{d4})+2 (-\text{d2}^2+6 \text{d4}-2 (-\text{d1} \text{d3}+4 \text{d4})))+\text{c1} \text{c2}^3 \text{c3}^2 (\text{d1}
\text{d4}-5 \text{d5})+\text{c2}^5 \text{c3} \text{d5}+\text{c1}^2 \text{c2} \text{c3}^3 (\text{d2} \text{d3}-10 \text{d5}+3 (-\text{d1} \text{d4}+5
\text{d5}))+\text{c2}^2 \text{c3}^3 (\text{d1}^2 \text{d3}-20 \text{d5}+7 (-\text{d1} \text{d4}+5 \text{d5})+2 (-\text{d2} \text{d3}+10 \text{d5}-3
(-\text{d1} \text{d4}+5 \text{d5})))+\text{c1} \text{c3}^4 (\text{d1} \text{d2}^2-30 \text{d5}+12 (-\text{d1} \text{d4}+5 \text{d5})+5
(-\text{d2} \text{d3}+10 \text{d5}-3 (-\text{d1} \text{d4}+5 \text{d5}))+2 (-\text{d1}^2 \text{d3}+20 \text{d5}-7 (-\text{d1} \text{d4}+5 \text{d5})-2
(-\text{d2} \text{d3}+10 \text{d5}-3 (-\text{d1} \text{d4}+5 \text{d5}))))+\text{c1} \text{c2}^4 \text{c3} (\text{d1} \text{d5}-6 \text{d6})+\text{c2}^6
\text{d6}+\text{c1} \text{c2}^5 \text{d1} \text{d6}+\text{c1}^2 \text{c2}^4 \text{d2} \text{d6}+\text{c1}^3 \text{c2}^3 \text{d3} \text{d6}+\text{c1}^4
\text{c2}^2 \text{d4} \text{d6}+\text{c1}^5 \text{c2} \text{d5} \text{d6}+\text{c1}^6 \text{d6}^2+\text{c1}^5 \text{d1} \text{d6}^2+\text{c1}^4 \text{d2}
\text{d6}^2+\text{c1}^3 \text{d3} \text{d6}^2+\text{c1}^2 \text{d4} \text{d6}^2+\text{c1} \text{d5} \text{d6}^2+\text{d6}^3+\text{c1}^2 \text{c2}^3
\text{c3} (\text{d2} \text{d5}-5 \text{d1} \text{d6})+\text{c1}^3 \text{c2}^2 \text{c3} (\text{d3} \text{d5}-4 \text{d2} \text{d6})+\text{c2}^5 (\text{d1}^2
\text{d6}-2 \text{d2} \text{d6})+\text{c1}^4 \text{c2} \text{c3} (\text{d4} \text{d5}-3 \text{d3} \text{d6})+\text{c1} \text{c2}^4 (\text{d1}
\text{d2} \text{d6}-3 \text{d3} \text{d6})+\text{c1}^2 \text{c2}^3 (\text{d1} \text{d3} \text{d6}-4 \text{d4} \text{d6})+\text{c1}^5 \text{c3} (\text{d5}^2-2
\text{d4} \text{d6})+\text{c1}^3 \text{c2}^2 (\text{d1} \text{d4} \text{d6}-5 \text{d5} \text{d6})+\text{c1}^4 \text{c2} (\text{d1} \text{d5}
\text{d6}-6 \text{d6}^2)+\text{c1}^3 \text{c2} (\text{d2} \text{d5} \text{d6}-5 \text{d1} \text{d6}^2)+\text{c1}^2 \text{c2} (\text{d3}
\text{d5} \text{d6}-4 \text{d2} \text{d6}^2)+\text{c1} \text{c2} (\text{d4} \text{d5} \text{d6}-3 \text{d3} \text{d6}^2)+\text{c2}
(\text{d5}^2 \text{d6}-2 \text{d4} \text{d6}^2)+\text{c1}^2 \text{c2}^2 \text{c3}^2 (\text{d2} \text{d4}-15 \text{d6}+4 (-\text{d1} \text{d5}+6
\text{d6}))+\text{c2}^4 \text{c3} (\text{d1}^2 \text{d5}-11 \text{d1} \text{d6}+2 (-\text{d2} \text{d5}+5 \text{d1} \text{d6}))+\text{c1}^3
\text{c2} \text{c3}^2 (\text{d3} \text{d4}-10 \text{d1} \text{d6}+3 (-\text{d2} \text{d5}+5 \text{d1} \text{d6}))+\text{c1}^4 \text{c3}^2 (\text{d4}^2-6
\text{d2} \text{d6}+2 (-\text{d3} \text{d5}+4 \text{d2} \text{d6}))+\text{c1} \text{c2}^3 \text{c3} (\text{d1} \text{d2} \text{d5}-14
\text{d2} \text{d6}+5 (-\text{d1}^2 \text{d6}+2 \text{d2} \text{d6})+3 (-\text{d3} \text{d5}+4 \text{d2} \text{d6}))+\text{c1}^2
\text{c2}^2 \text{c3} (\text{d1} \text{d3} \text{d5}-15 \text{d3} \text{d6}+4 (-\text{d4} \text{d5}+3 \text{d3} \text{d6})+4 (-\text{d1} \text{d2}
\text{d6}+3 \text{d3} \text{d6}))+\text{c2}^4 (\text{d2}^2 \text{d6}-6 \text{d4} \text{d6}+2 (-\text{d1} \text{d3} \text{d6}+4 \text{d4} \text{d6}))+\text{c1}^3
\text{c2} \text{c3} (\text{d1} \text{d4} \text{d5}-14 \text{d4} \text{d6}+5 (-\text{d5}^2+2 \text{d4} \text{d6})+3 (-\text{d1} \text{d3}
\text{d6}+4 \text{d4} \text{d6}))+\text{c1}^4 \text{c3} (\text{d1} \text{d5}^2-11 \text{d5} \text{d6}+2 (-\text{d1} \text{d4} \text{d6}+5
\text{d5} \text{d6}))+\text{c1} \text{c2}^3 (\text{d2} \text{d3} \text{d6}-10 \text{d5} \text{d6}+3 (-\text{d1} \text{d4} \text{d6}+5 \text{d5}
\text{d6}))+\text{c1}^2 \text{c2}^2 (\text{d2} \text{d4} \text{d6}-15 \text{d6}^2+4 (-\text{d1} \text{d5} \text{d6}+6 \text{d6}^2))+\text{c1}
\text{c2}^2 (\text{d3} \text{d4} \text{d6}-10 \text{d1} \text{d6}^2+3 (-\text{d2} \text{d5} \text{d6}+5 \text{d1} \text{d6}^2))+\text{c2}^2
(\text{d4}^2 \text{d6}-6 \text{d2} \text{d6}^2+2 (-\text{d3} \text{d5} \text{d6}+4 \text{d2} \text{d6}^2))+\text{c3} (\text{d5}^3-6
\text{d3} \text{d6}^2+3 (-\text{d4} \text{d5} \text{d6}+3 \text{d3} \text{d6}^2))+\text{c1}^3 \text{c3}^3 (\text{d3}^2-20 \text{d6}+6
(-\text{d1} \text{d5}+6 \text{d6})+2 (-\text{d2} \text{d4}+15 \text{d6}-4 (-\text{d1} \text{d5}+6 \text{d6})))+\text{c2}^3 \text{c3}^2 (\text{d1}^2
\text{d4}-30 \text{d6}+9 (-\text{d1} \text{d5}+6 \text{d6})+2 (-\text{d2} \text{d4}+15 \text{d6}-4 (-\text{d1} \text{d5}+6 \text{d6})))+\text{c1}
\text{c2}^2 \text{c3}^2 (\text{d1} \text{d2} \text{d4}-35 \text{d1} \text{d6}+11 (-\text{d2} \text{d5}+5 \text{d1} \text{d6})+3 (-\text{d3}
\text{d4}+10 \text{d1} \text{d6}-3 (-\text{d2} \text{d5}+5 \text{d1} \text{d6}))+4 (-\text{d1}^2 \text{d5}+11 \text{d1} \text{d6}-2 (-\text{d2}
\text{d5}+5 \text{d1} \text{d6})))+\text{c1}^2 \text{c2} \text{c3}^2 (\text{d1} \text{d3} \text{d4}-32 \text{d2} \text{d6}+10 (-\text{d1}^2
\text{d6}+2 \text{d2} \text{d6})+11 (-\text{d3} \text{d5}+4 \text{d2} \text{d6})+3 (-\text{d1} \text{d2} \text{d5}+14 \text{d2} \text{d6}-5
(-\text{d1}^2 \text{d6}+2 \text{d2} \text{d6})-3 (-\text{d3} \text{d5}+4 \text{d2} \text{d6}))+4 (-\text{d4}^2+6 \text{d2}
\text{d6}-2 (-\text{d3} \text{d5}+4 \text{d2} \text{d6})))+\text{c1}^3 \text{c3}^2 (\text{d1} \text{d4}^2-24 \text{d3} \text{d6}+9
(-\text{d4} \text{d5}+3 \text{d3} \text{d6})+6 (-\text{d1} \text{d2} \text{d6}+3 \text{d3} \text{d6})+2 (-\text{d1} \text{d3} \text{d5}+15 \text{d3}
\text{d6}-4 (-\text{d4} \text{d5}+3 \text{d3} \text{d6})-4 (-\text{d1} \text{d2} \text{d6}+3 \text{d3} \text{d6})))+\text{c2}^3 \text{c3} (\text{d2}^2
\text{d5}-24 \text{d3} \text{d6}+6 (-\text{d4} \text{d5}+3 \text{d3} \text{d6})+9 (-\text{d1} \text{d2} \text{d6}+3 \text{d3} \text{d6})+2 (-\text{d1}
\text{d3} \text{d5}+15 \text{d3} \text{d6}-4 (-\text{d4} \text{d5}+3 \text{d3} \text{d6})-4 (-\text{d1} \text{d2} \text{d6}+3 \text{d3} \text{d6})))+\text{c1}
\text{c2}^2 \text{c3} (\text{d2} \text{d3} \text{d5}-32 \text{d4} \text{d6}+10 (-\text{d5}^2+2 \text{d4} \text{d6})+11 (-\text{d1}
\text{d3} \text{d6}+4 \text{d4} \text{d6})+3 (-\text{d1} \text{d4} \text{d5}+14 \text{d4} \text{d6}-5 (-\text{d5}^2+2 \text{d4} \text{d6})-3
(-\text{d1} \text{d3} \text{d6}+4 \text{d4} \text{d6}))+4 (-\text{d2}^2 \text{d6}+6 \text{d4} \text{d6}-2 (-\text{d1} \text{d3} \text{d6}+4
\text{d4} \text{d6})))+\text{c1}^2 \text{c2} \text{c3} (\text{d2} \text{d4} \text{d5}-35 \text{d5} \text{d6}+11 (-\text{d1} \text{d4}
\text{d6}+5 \text{d5} \text{d6})+3 (-\text{d2} \text{d3} \text{d6}+10 \text{d5} \text{d6}-3 (-\text{d1} \text{d4} \text{d6}+5 \text{d5} \text{d6}))+4
(-\text{d1} \text{d5}^2+11 \text{d5} \text{d6}-2 (-\text{d1} \text{d4} \text{d6}+5 \text{d5} \text{d6})))+\text{c2}^3 (\text{d3}^2
\text{d6}-20 \text{d6}^2+6 (-\text{d1} \text{d5} \text{d6}+6 \text{d6}^2)+2 (-\text{d2} \text{d4} \text{d6}+15 \text{d6}^2-4 (-\text{d1}
\text{d5} \text{d6}+6 \text{d6}^2)))+\text{c1}^3 \text{c3} (\text{d2} \text{d5}^2-30 \text{d6}^2+9 (-\text{d1} \text{d5}
\text{d6}+6 \text{d6}^2)+2 (-\text{d2} \text{d4} \text{d6}+15 \text{d6}^2-4 (-\text{d1} \text{d5} \text{d6}+6 \text{d6}^2)))+\text{c1}^2
\text{c3} (\text{d3} \text{d5}^2-20 \text{d1} \text{d6}^2+7 (-\text{d2} \text{d5} \text{d6}+5 \text{d1} \text{d6}^2)+2 (-\text{d3}
\text{d4} \text{d6}+10 \text{d1} \text{d6}^2-3 (-\text{d2} \text{d5} \text{d6}+5 \text{d1} \text{d6}^2)))+\text{c1} \text{c3}
(\text{d4} \text{d5}^2-12 \text{d2} \text{d6}^2+5 (-\text{d3} \text{d5} \text{d6}+4 \text{d2} \text{d6}^2)+2 (-\text{d4}^2 \text{d6}+6
\text{d2} \text{d6}^2-2 (-\text{d3} \text{d5} \text{d6}+4 \text{d2} \text{d6}^2)))+\text{c1} \text{c2} \text{c3}^3 (\text{d1}
\text{d2} \text{d3}-60 \text{d6}+22 (-\text{d1} \text{d5}+6 \text{d6})+8 (-\text{d2} \text{d4}+15 \text{d6}-4 (-\text{d1} \text{d5}+6 \text{d6}))+3
(-\text{d1}^2 \text{d4}+30 \text{d6}-9 (-\text{d1} \text{d5}+6 \text{d6})-2 (-\text{d2} \text{d4}+15 \text{d6}-4 (-\text{d1} \text{d5}+6 \text{d6})))+3
(-\text{d3}^2+20 \text{d6}-6 (-\text{d1} \text{d5}+6 \text{d6})-2 (-\text{d2} \text{d4}+15 \text{d6}-4 (-\text{d1} \text{d5}+6 \text{d6}))))+\text{c1}^2
\text{c3}^3 (\text{d1} \text{d3}^2-50 \text{d1} \text{d6}+18 (-\text{d2} \text{d5}+5 \text{d1} \text{d6})+7 (-\text{d3} \text{d4}+10 \text{d1}
\text{d6}-3 (-\text{d2} \text{d5}+5 \text{d1} \text{d6}))+6 (-\text{d1}^2 \text{d5}+11 \text{d1} \text{d6}-2 (-\text{d2} \text{d5}+5 \text{d1}
\text{d6}))+2 (-\text{d1} \text{d2} \text{d4}+35 \text{d1} \text{d6}-11 (-\text{d2} \text{d5}+5 \text{d1} \text{d6})-3 (-\text{d3} \text{d4}+10
\text{d1} \text{d6}-3 (-\text{d2} \text{d5}+5 \text{d1} \text{d6}))-4 (-\text{d1}^2 \text{d5}+11 \text{d1} \text{d6}-2 (-\text{d2} \text{d5}+5
\text{d1} \text{d6}))))+\text{c2}^2 \text{c3}^2 (\text{d2}^2 \text{d4}-53 \text{d2} \text{d6}+20 (-\text{d1}^2 \text{d6}+2
\text{d2} \text{d6})+18 (-\text{d3} \text{d5}+4 \text{d2} \text{d6})+7 (-\text{d1} \text{d2} \text{d5}+14 \text{d2} \text{d6}-5 (-\text{d1}^2
\text{d6}+2 \text{d2} \text{d6})-3 (-\text{d3} \text{d5}+4 \text{d2} \text{d6}))+6 (-\text{d4}^2+6 \text{d2} \text{d6}-2 (-\text{d3}
\text{d5}+4 \text{d2} \text{d6}))+2 (-\text{d1} \text{d3} \text{d4}+32 \text{d2} \text{d6}-10 (-\text{d1}^2 \text{d6}+2 \text{d2}
\text{d6})-11 (-\text{d3} \text{d5}+4 \text{d2} \text{d6})-3 (-\text{d1} \text{d2} \text{d5}+14 \text{d2} \text{d6}-5 (-\text{d1}^2
\text{d6}+2 \text{d2} \text{d6})-3 (-\text{d3} \text{d5}+4 \text{d2} \text{d6}))-4 (-\text{d4}^2+6 \text{d2} \text{d6}-2 (-\text{d3}
\text{d5}+4 \text{d2} \text{d6}))))+\text{c1} \text{c2} \text{c3}^2 (\text{d2} \text{d3} \text{d4}-60 \text{d3} \text{d6}+22
(-\text{d4} \text{d5}+3 \text{d3} \text{d6})+22 (-\text{d1} \text{d2} \text{d6}+3 \text{d3} \text{d6})+8 (-\text{d1} \text{d3} \text{d5}+15 \text{d3}
\text{d6}-4 (-\text{d4} \text{d5}+3 \text{d3} \text{d6})-4 (-\text{d1} \text{d2} \text{d6}+3 \text{d3} \text{d6}))+3 (-\text{d2}^2 \text{d5}+24
\text{d3} \text{d6}-6 (-\text{d4} \text{d5}+3 \text{d3} \text{d6})-9 (-\text{d1} \text{d2} \text{d6}+3 \text{d3} \text{d6})-2 (-\text{d1} \text{d3}
\text{d5}+15 \text{d3} \text{d6}-4 (-\text{d4} \text{d5}+3 \text{d3} \text{d6})-4 (-\text{d1} \text{d2} \text{d6}+3 \text{d3} \text{d6})))+3
(-\text{d1} \text{d4}^2+24 \text{d3} \text{d6}-9 (-\text{d4} \text{d5}+3 \text{d3} \text{d6})-6 (-\text{d1} \text{d2} \text{d6}+3 \text{d3}
\text{d6})-2 (-\text{d1} \text{d3} \text{d5}+15 \text{d3} \text{d6}-4 (-\text{d4} \text{d5}+3 \text{d3} \text{d6})-4 (-\text{d1} \text{d2} \text{d6}+3
\text{d3} \text{d6}))))+\text{c1}^2 \text{c3}^2 (\text{d2} \text{d4}^2-53 \text{d4} \text{d6}+20 (-\text{d5}^2+2 \text{d4}
\text{d6})+18 (-\text{d1} \text{d3} \text{d6}+4 \text{d4} \text{d6})+7 (-\text{d1} \text{d4} \text{d5}+14 \text{d4} \text{d6}-5 (-\text{d5}^2+2
\text{d4} \text{d6})-3 (-\text{d1} \text{d3} \text{d6}+4 \text{d4} \text{d6}))+6 (-\text{d2}^2 \text{d6}+6 \text{d4} \text{d6}-2
(-\text{d1} \text{d3} \text{d6}+4 \text{d4} \text{d6}))+2 (-\text{d2} \text{d3} \text{d5}+32 \text{d4} \text{d6}-10 (-\text{d5}^2+2
\text{d4} \text{d6})-11 (-\text{d1} \text{d3} \text{d6}+4 \text{d4} \text{d6})-3 (-\text{d1} \text{d4} \text{d5}+14 \text{d4} \text{d6}-5
(-\text{d5}^2+2 \text{d4} \text{d6})-3 (-\text{d1} \text{d3} \text{d6}+4 \text{d4} \text{d6}))-4 (-\text{d2}^2 \text{d6}+6
\text{d4} \text{d6}-2 (-\text{d1} \text{d3} \text{d6}+4 \text{d4} \text{d6}))))+\text{c2}^2 \text{c3} (\text{d3}^2 \text{d5}-50
\text{d5} \text{d6}+18 (-\text{d1} \text{d4} \text{d6}+5 \text{d5} \text{d6})+7 (-\text{d2} \text{d3} \text{d6}+10 \text{d5} \text{d6}-3 (-\text{d1}
\text{d4} \text{d6}+5 \text{d5} \text{d6}))+6 (-\text{d1} \text{d5}^2+11 \text{d5} \text{d6}-2 (-\text{d1} \text{d4} \text{d6}+5 \text{d5} \text{d6}))+2
(-\text{d2} \text{d4} \text{d5}+35 \text{d5} \text{d6}-11 (-\text{d1} \text{d4} \text{d6}+5 \text{d5} \text{d6})-3 (-\text{d2} \text{d3} \text{d6}+10
\text{d5} \text{d6}-3 (-\text{d1} \text{d4} \text{d6}+5 \text{d5} \text{d6}))-4 (-\text{d1} \text{d5}^2+11 \text{d5} \text{d6}-2 (-\text{d1}
\text{d4} \text{d6}+5 \text{d5} \text{d6}))))+\text{c1} \text{c2} \text{c3} (\text{d3} \text{d4} \text{d5}-60 \text{d6}^2+22
(-\text{d1} \text{d5} \text{d6}+6 \text{d6}^2)+8 (-\text{d2} \text{d4} \text{d6}+15 \text{d6}^2-4 (-\text{d1} \text{d5} \text{d6}+6
\text{d6}^2))+3 (-\text{d2} \text{d5}^2+30 \text{d6}^2-9 (-\text{d1} \text{d5} \text{d6}+6 \text{d6}^2)-2 (-\text{d2}
\text{d4} \text{d6}+15 \text{d6}^2-4 (-\text{d1} \text{d5} \text{d6}+6 \text{d6}^2)))+3 (-\text{d3}^2 \text{d6}+20 \text{d6}^2-6
(-\text{d1} \text{d5} \text{d6}+6 \text{d6}^2)-2 (-\text{d2} \text{d4} \text{d6}+15 \text{d6}^2-4 (-\text{d1} \text{d5} \text{d6}+6
\text{d6}^2))))+\text{c2} \text{c3} (\text{d4}^2 \text{d5}-30 \text{d1} \text{d6}^2+12 (-\text{d2} \text{d5} \text{d6}+5
\text{d1} \text{d6}^2)+5 (-\text{d3} \text{d4} \text{d6}+10 \text{d1} \text{d6}^2-3 (-\text{d2} \text{d5} \text{d6}+5 \text{d1} \text{d6}^2))+2
(-\text{d3} \text{d5}^2+20 \text{d1} \text{d6}^2-7 (-\text{d2} \text{d5} \text{d6}+5 \text{d1} \text{d6}^2)-2 (-\text{d3} \text{d4}
\text{d6}+10 \text{d1} \text{d6}^2-3 (-\text{d2} \text{d5} \text{d6}+5 \text{d1} \text{d6}^2))))+\text{c3}^4 (\text{d2}^3-90
\text{d6}+36 (-\text{d1} \text{d5}+6 \text{d6})+15 (-\text{d2} \text{d4}+15 \text{d6}-4 (-\text{d1} \text{d5}+6 \text{d6}))+6 (-\text{d1}^2
\text{d4}+30 \text{d6}-9 (-\text{d1} \text{d5}+6 \text{d6})-2 (-\text{d2} \text{d4}+15 \text{d6}-4 (-\text{d1} \text{d5}+6 \text{d6})))+6 (-\text{d3}^2+20
\text{d6}-6 (-\text{d1} \text{d5}+6 \text{d6})-2 (-\text{d2} \text{d4}+15 \text{d6}-4 (-\text{d1} \text{d5}+6 \text{d6})))+3 (-\text{d1}
\text{d2} \text{d3}+60 \text{d6}-22 (-\text{d1} \text{d5}+6 \text{d6})-8 (-\text{d2} \text{d4}+15 \text{d6}-4 (-\text{d1} \text{d5}+6 \text{d6}))-3
(-\text{d1}^2 \text{d4}+30 \text{d6}-9 (-\text{d1} \text{d5}+6 \text{d6})-2 (-\text{d2} \text{d4}+15 \text{d6}-4 (-\text{d1} \text{d5}+6 \text{d6})))-3
(-\text{d3}^2+20 \text{d6}-6 (-\text{d1} \text{d5}+6 \text{d6})-2 (-\text{d2} \text{d4}+15 \text{d6}-4 (-\text{d1} \text{d5}+6 \text{d6})))))+\text{c2}
\text{c3}^3 (\text{d2}^2 \text{d3}-80 \text{d1} \text{d6}+31 (-\text{d2} \text{d5}+5 \text{d1} \text{d6})+12 (-\text{d3} \text{d4}+10 \text{d1}
\text{d6}-3 (-\text{d2} \text{d5}+5 \text{d1} \text{d6}))+12 (-\text{d1}^2 \text{d5}+11 \text{d1} \text{d6}-2 (-\text{d2} \text{d5}+5 \text{d1}
\text{d6}))+5 (-\text{d1} \text{d2} \text{d4}+35 \text{d1} \text{d6}-11 (-\text{d2} \text{d5}+5 \text{d1} \text{d6})-3 (-\text{d3} \text{d4}+10
\text{d1} \text{d6}-3 (-\text{d2} \text{d5}+5 \text{d1} \text{d6}))-4 (-\text{d1}^2 \text{d5}+11 \text{d1} \text{d6}-2 (-\text{d2} \text{d5}+5
\text{d1} \text{d6})))+2 (-\text{d1} \text{d3}^2+50 \text{d1} \text{d6}-18 (-\text{d2} \text{d5}+5 \text{d1} \text{d6})-7 (-\text{d3}
\text{d4}+10 \text{d1} \text{d6}-3 (-\text{d2} \text{d5}+5 \text{d1} \text{d6}))-6 (-\text{d1}^2 \text{d5}+11 \text{d1} \text{d6}-2 (-\text{d2}
\text{d5}+5 \text{d1} \text{d6}))-2 (-\text{d1} \text{d2} \text{d4}+35 \text{d1} \text{d6}-11 (-\text{d2} \text{d5}+5 \text{d1} \text{d6})-3
(-\text{d3} \text{d4}+10 \text{d1} \text{d6}-3 (-\text{d2} \text{d5}+5 \text{d1} \text{d6}))-4 (-\text{d1}^2 \text{d5}+11 \text{d1} \text{d6}-2
(-\text{d2} \text{d5}+5 \text{d1} \text{d6})))))+\text{c1} \text{c3}^3 (\text{d2} \text{d3}^2-80 \text{d2} \text{d6}+30
(-\text{d1}^2 \text{d6}+2 \text{d2} \text{d6})+31 (-\text{d3} \text{d5}+4 \text{d2} \text{d6})+12 (-\text{d1} \text{d2} \text{d5}+14
\text{d2} \text{d6}-5 (-\text{d1}^2 \text{d6}+2 \text{d2} \text{d6})-3 (-\text{d3} \text{d5}+4 \text{d2} \text{d6}))+12 (-\text{d4}^2+6
\text{d2} \text{d6}-2 (-\text{d3} \text{d5}+4 \text{d2} \text{d6}))+5 (-\text{d1} \text{d3} \text{d4}+32 \text{d2} \text{d6}-10 (-\text{d1}^2
\text{d6}+2 \text{d2} \text{d6})-11 (-\text{d3} \text{d5}+4 \text{d2} \text{d6})-3 (-\text{d1} \text{d2} \text{d5}+14 \text{d2} \text{d6}-5
(-\text{d1}^2 \text{d6}+2 \text{d2} \text{d6})-3 (-\text{d3} \text{d5}+4 \text{d2} \text{d6}))-4 (-\text{d4}^2+6 \text{d2}
\text{d6}-2 (-\text{d3} \text{d5}+4 \text{d2} \text{d6})))+2 (-\text{d2}^2 \text{d4}+53 \text{d2} \text{d6}-20 (-\text{d1}^2
\text{d6}+2 \text{d2} \text{d6})-18 (-\text{d3} \text{d5}+4 \text{d2} \text{d6})-7 (-\text{d1} \text{d2} \text{d5}+14 \text{d2} \text{d6}-5
(-\text{d1}^2 \text{d6}+2 \text{d2} \text{d6})-3 (-\text{d3} \text{d5}+4 \text{d2} \text{d6}))-6 (-\text{d4}^2+6 \text{d2}
\text{d6}-2 (-\text{d3} \text{d5}+4 \text{d2} \text{d6}))-2 (-\text{d1} \text{d3} \text{d4}+32 \text{d2} \text{d6}-10 (-\text{d1}^2
\text{d6}+2 \text{d2} \text{d6})-11 (-\text{d3} \text{d5}+4 \text{d2} \text{d6})-3 (-\text{d1} \text{d2} \text{d5}+14 \text{d2} \text{d6}-5
(-\text{d1}^2 \text{d6}+2 \text{d2} \text{d6})-3 (-\text{d3} \text{d5}+4 \text{d2} \text{d6}))-4 (-\text{d4}^2+6 \text{d2}
\text{d6}-2 (-\text{d3} \text{d5}+4 \text{d2} \text{d6})))))+\text{c3}^3 (\text{d3}^3-93 \text{d3} \text{d6}+36 (-\text{d4}
\text{d5}+3 \text{d3} \text{d6})+36 (-\text{d1} \text{d2} \text{d6}+3 \text{d3} \text{d6})+15 (-\text{d1} \text{d3} \text{d5}+15 \text{d3} \text{d6}-4
(-\text{d4} \text{d5}+3 \text{d3} \text{d6})-4 (-\text{d1} \text{d2} \text{d6}+3 \text{d3} \text{d6}))+6 (-\text{d2}^2 \text{d5}+24 \text{d3}
\text{d6}-6 (-\text{d4} \text{d5}+3 \text{d3} \text{d6})-9 (-\text{d1} \text{d2} \text{d6}+3 \text{d3} \text{d6})-2 (-\text{d1} \text{d3} \text{d5}+15
\text{d3} \text{d6}-4 (-\text{d4} \text{d5}+3 \text{d3} \text{d6})-4 (-\text{d1} \text{d2} \text{d6}+3 \text{d3} \text{d6})))+6 (-\text{d1}
\text{d4}^2+24 \text{d3} \text{d6}-9 (-\text{d4} \text{d5}+3 \text{d3} \text{d6})-6 (-\text{d1} \text{d2} \text{d6}+3 \text{d3} \text{d6})-2 (-\text{d1}
\text{d3} \text{d5}+15 \text{d3} \text{d6}-4 (-\text{d4} \text{d5}+3 \text{d3} \text{d6})-4 (-\text{d1} \text{d2} \text{d6}+3 \text{d3} \text{d6})))+3
(-\text{d2} \text{d3} \text{d4}+60 \text{d3} \text{d6}-22 (-\text{d4} \text{d5}+3 \text{d3} \text{d6})-22 (-\text{d1} \text{d2} \text{d6}+3
\text{d3} \text{d6})-8 (-\text{d1} \text{d3} \text{d5}+15 \text{d3} \text{d6}-4 (-\text{d4} \text{d5}+3 \text{d3} \text{d6})-4 (-\text{d1} \text{d2}
\text{d6}+3 \text{d3} \text{d6}))-3 (-\text{d2}^2 \text{d5}+24 \text{d3} \text{d6}-6 (-\text{d4} \text{d5}+3 \text{d3} \text{d6})-9 (-\text{d1}
\text{d2} \text{d6}+3 \text{d3} \text{d6})-2 (-\text{d1} \text{d3} \text{d5}+15 \text{d3} \text{d6}-4 (-\text{d4} \text{d5}+3 \text{d3} \text{d6})-4
(-\text{d1} \text{d2} \text{d6}+3 \text{d3} \text{d6})))-3 (-\text{d1} \text{d4}^2+24 \text{d3} \text{d6}-9 (-\text{d4} \text{d5}+3 \text{d3}
\text{d6})-6 (-\text{d1} \text{d2} \text{d6}+3 \text{d3} \text{d6})-2 (-\text{d1} \text{d3} \text{d5}+15 \text{d3} \text{d6}-4 (-\text{d4} \text{d5}+3
\text{d3} \text{d6})-4 (-\text{d1} \text{d2} \text{d6}+3 \text{d3} \text{d6})))))+\text{c2} \text{c3}^2 (\text{d3}^2 \text{d4}-80
\text{d4} \text{d6}+30 (-\text{d5}^2+2 \text{d4} \text{d6})+31 (-\text{d1} \text{d3} \text{d6}+4 \text{d4} \text{d6})+12 (-\text{d1}
\text{d4} \text{d5}+14 \text{d4} \text{d6}-5 (-\text{d5}^2+2 \text{d4} \text{d6})-3 (-\text{d1} \text{d3} \text{d6}+4 \text{d4} \text{d6}))+12
(-\text{d2}^2 \text{d6}+6 \text{d4} \text{d6}-2 (-\text{d1} \text{d3} \text{d6}+4 \text{d4} \text{d6}))+5 (-\text{d2} \text{d3} \text{d5}+32
\text{d4} \text{d6}-10 (-\text{d5}^2+2 \text{d4} \text{d6})-11 (-\text{d1} \text{d3} \text{d6}+4 \text{d4} \text{d6})-3 (-\text{d1}
\text{d4} \text{d5}+14 \text{d4} \text{d6}-5 (-\text{d5}^2+2 \text{d4} \text{d6})-3 (-\text{d1} \text{d3} \text{d6}+4 \text{d4} \text{d6}))-4
(-\text{d2}^2 \text{d6}+6 \text{d4} \text{d6}-2 (-\text{d1} \text{d3} \text{d6}+4 \text{d4} \text{d6})))+2 (-\text{d2} \text{d4}^2+53
\text{d4} \text{d6}-20 (-\text{d5}^2+2 \text{d4} \text{d6})-18 (-\text{d1} \text{d3} \text{d6}+4 \text{d4} \text{d6})-7 (-\text{d1}
\text{d4} \text{d5}+14 \text{d4} \text{d6}-5 (-\text{d5}^2+2 \text{d4} \text{d6})-3 (-\text{d1} \text{d3} \text{d6}+4 \text{d4} \text{d6}))-6
(-\text{d2}^2 \text{d6}+6 \text{d4} \text{d6}-2 (-\text{d1} \text{d3} \text{d6}+4 \text{d4} \text{d6}))-2 (-\text{d2} \text{d3} \text{d5}+32
\text{d4} \text{d6}-10 (-\text{d5}^2+2 \text{d4} \text{d6})-11 (-\text{d1} \text{d3} \text{d6}+4 \text{d4} \text{d6})-3 (-\text{d1}
\text{d4} \text{d5}+14 \text{d4} \text{d6}-5 (-\text{d5}^2+2 \text{d4} \text{d6})-3 (-\text{d1} \text{d3} \text{d6}+4 \text{d4} \text{d6}))-4
(-\text{d2}^2 \text{d6}+6 \text{d4} \text{d6}-2 (-\text{d1} \text{d3} \text{d6}+4 \text{d4} \text{d6})))))+\text{c1}
\text{c3}^2 (\text{d3} \text{d4}^2-80 \text{d5} \text{d6}+31 (-\text{d1} \text{d4} \text{d6}+5 \text{d5} \text{d6})+12 (-\text{d2} \text{d3}
\text{d6}+10 \text{d5} \text{d6}-3 (-\text{d1} \text{d4} \text{d6}+5 \text{d5} \text{d6}))+12 (-\text{d1} \text{d5}^2+11 \text{d5} \text{d6}-2
(-\text{d1} \text{d4} \text{d6}+5 \text{d5} \text{d6}))+5 (-\text{d2} \text{d4} \text{d5}+35 \text{d5} \text{d6}-11 (-\text{d1} \text{d4}
\text{d6}+5 \text{d5} \text{d6})-3 (-\text{d2} \text{d3} \text{d6}+10 \text{d5} \text{d6}-3 (-\text{d1} \text{d4} \text{d6}+5 \text{d5} \text{d6}))-4
(-\text{d1} \text{d5}^2+11 \text{d5} \text{d6}-2 (-\text{d1} \text{d4} \text{d6}+5 \text{d5} \text{d6})))+2 (-\text{d3}^2 \text{d5}+50
\text{d5} \text{d6}-18 (-\text{d1} \text{d4} \text{d6}+5 \text{d5} \text{d6})-7 (-\text{d2} \text{d3} \text{d6}+10 \text{d5} \text{d6}-3 (-\text{d1}
\text{d4} \text{d6}+5 \text{d5} \text{d6}))-6 (-\text{d1} \text{d5}^2+11 \text{d5} \text{d6}-2 (-\text{d1} \text{d4} \text{d6}+5 \text{d5} \text{d6}))-2
(-\text{d2} \text{d4} \text{d5}+35 \text{d5} \text{d6}-11 (-\text{d1} \text{d4} \text{d6}+5 \text{d5} \text{d6})-3 (-\text{d2} \text{d3} \text{d6}+10
\text{d5} \text{d6}-3 (-\text{d1} \text{d4} \text{d6}+5 \text{d5} \text{d6}))-4 (-\text{d1} \text{d5}^2+11 \text{d5} \text{d6}-2 (-\text{d1}
\text{d4} \text{d6}+5 \text{d5} \text{d6})))))+\text{c3}^2 (\text{d4}^3-90 \text{d6}^2+36 (-\text{d1} \text{d5}
\text{d6}+6 \text{d6}^2)+15 (-\text{d2} \text{d4} \text{d6}+15 \text{d6}^2-4 (-\text{d1} \text{d5} \text{d6}+6 \text{d6}^2))+6
(-\text{d2} \text{d5}^2+30 \text{d6}^2-9 (-\text{d1} \text{d5} \text{d6}+6 \text{d6}^2)-2 (-\text{d2} \text{d4} \text{d6}+15
\text{d6}^2-4 (-\text{d1} \text{d5} \text{d6}+6 \text{d6}^2)))+6 (-\text{d3}^2 \text{d6}+20 \text{d6}^2-6 (-\text{d1}
\text{d5} \text{d6}+6 \text{d6}^2)-2 (-\text{d2} \text{d4} \text{d6}+15 \text{d6}^2-4 (-\text{d1} \text{d5} \text{d6}+6 \text{d6}^2)))+3
(-\text{d3} \text{d4} \text{d5}+60 \text{d6}^2-22 (-\text{d1} \text{d5} \text{d6}+6 \text{d6}^2)-8 (-\text{d2} \text{d4} \text{d6}+15
\text{d6}^2-4 (-\text{d1} \text{d5} \text{d6}+6 \text{d6}^2))-3 (-\text{d2} \text{d5}^2+30 \text{d6}^2-9 (-\text{d1} \text{d5}
\text{d6}+6 \text{d6}^2)-2 (-\text{d2} \text{d4} \text{d6}+15 \text{d6}^2-4 (-\text{d1} \text{d5} \text{d6}+6 \text{d6}^2)))-3
(-\text{d3}^2 \text{d6}+20 \text{d6}^2-6 (-\text{d1} \text{d5} \text{d6}+6 \text{d6}^2)-2 (-\text{d2} \text{d4} \text{d6}+15
\text{d6}^2-4 (-\text{d1} \text{d5} \text{d6}+6 \text{d6}^2)))))\)
\end{doublespace}

\subsection{Computational verification using Mathematica II}\label{SectionMathematicaII}

\begin{doublespace}
\noindent\(\pmb{\text{(*Define the Chern classes of the symmetric product*)}}\\
\pmb{}\\
\pmb{\text{S1}[\text{c1$\_$},\text{c2$\_$},\text{c3$\_$}]\text{:=}4 \text{c1};}\\
\pmb{\text{S2}[\text{c1$\_$},\text{c2$\_$},\text{c3$\_$}]\text{:=}5 \text{c1}{}^{\wedge}2+5 \text{c2};}\\
\pmb{\text{S3}[\text{c1$\_$},\text{c2$\_$},\text{c3$\_$}]\text{:=}2 \text{c1}{}^{\wedge}3+11 \text{c1} \text{c2}+7 \text{c3};}\\
\pmb{}\\
\pmb{\text{(*Define the Euler classes up to degree 3 in ci*)}}\\
\pmb{}\\
\pmb{\text{e1}=\text{d6}{}^{\wedge}3+\text{c1} \text{d5} \text{d6}{}^{\wedge}2+\text{c1}{}^{\wedge}2 \text{d4} \text{d6}{}^{\wedge}2+\text{c2} (\text{d5}{}^{\wedge}2 \text{d6}-2 \text{d4} \text{d6}{}^{\wedge}2)+\text{c1} \text{c2} (\text{d4} \text{d5} \text{d6}-3 \text{d3} \text{d6}{}^{\wedge}2)}\\
\pmb{+\text{c3} (\text{d5}{}^{\wedge}3-6 \text{d3} \text{d6}{}^{\wedge}2+3 (-\text{d4} \text{d5} \text{d6}+3 \text{d3} \text{d6}{}^{\wedge}2))+\text{c1}{}^{\wedge}3 \text{d3} \text{d6}{}^{\wedge}2;}\\
\pmb{\text{e2}=(\text{C1}-\text{c1}){}^{\wedge}6+(\text{C1}-\text{c1}){}^{\wedge}5 \text{S1}[\text{c1},\text{c2},\text{c3}]+(\text{C1}-\text{c1}){}^{\wedge}4
\text{S2}[\text{c1},\text{c2},\text{c3}]+}\\
\pmb{(\text{C1}-\text{c1}){}^{\wedge}3 \text{S3}[\text{c1},\text{c2},\text{c3}];}\\
\pmb{\text{e3}=8 (\text{C1}-\text{c1}){}^{\wedge}3+4 (\text{C1}-\text{c1}){}^{\wedge}2 \text{c1}+2 (\text{C1}-\text{c1}) \text{c2}+\text{c3};}\\
\pmb{}\\
\pmb{\text{(*The class of } [\text{Tilde } F] \text{ up to degree } 3\text{*)}}\\
\pmb{}\\
\pmb{\text{final}=\text{e1} \text{e2} \text{e3};}\\
\pmb{\text{final}=\text{Expand}[\text{final}];}\\
\pmb{\text{Collect}[\text{final},\{\text{c1},\text{c2},\text{c3}\}]}\)
\end{doublespace}

\begin{doublespace}
\noindent\(8 \text{C1}^9 \text{d6}^3+\text{c3}^3 (7 \text{C1}^3 \text{d5}^3-21 \text{C1}^3 \text{d4} \text{d5} \text{d6}+21 \text{C1}^3 \text{d3}
\text{d6}^2)+\text{c2}^3 (10 \text{C1}^5 \text{d5}^2 \text{d6}-20 \text{C1}^5 \text{d4} \text{d6}^2)+\text{c1}^9 (4 \text{C1}^2
\text{d4} \text{d6}^2+24 \text{c2} \text{d4} \text{d6}^2)+\text{c3}^2 (57 \text{C1}^6 \text{d5}^3-171 \text{C1}^6 \text{d4} \text{d5}
\text{d6}+171 \text{C1}^6 \text{d3} \text{d6}^2+7 \text{C1}^3 \text{d6}^3)+\text{c3} (8 \text{C1}^9 \text{d5}^3-24 \text{C1}^9 \text{d4}
\text{d5} \text{d6}+24 \text{C1}^9 \text{d3} \text{d6}^2+57 \text{C1}^6 \text{d6}^3)+\text{c2}^2 (42 \text{C1}^7 \text{d5}^2 \text{d6}-84
\text{C1}^7 \text{d4} \text{d6}^2+10 \text{C1}^5 \text{d6}^3+\text{c3} (10 \text{C1}^5 \text{d5}^3-30 \text{C1}^5 \text{d4} \text{d5} \text{d6}+19
\text{C1}^4 \text{d5}^2 \text{d6}+30 \text{C1}^5 \text{d3} \text{d6}^2-38 \text{C1}^4 \text{d4} \text{d6}^2))+\text{c1}^8 (-24 \text{C1}^3
\text{d4} \text{d6}^2+28 \text{c3} \text{d4} \text{d6}^2+4 \text{C1}^2 \text{d5} \text{d6}^2+\text{c2}^2 (24 \text{d4} \text{d5} \text{d6}-72
\text{d3} \text{d6}^2)+\text{c2} (4 \text{C1}^2 \text{d4} \text{d5} \text{d6}-12 \text{C1}^2 \text{d3} \text{d6}^2-148 \text{C1} \text{d4}
\text{d6}^2+24 \text{d5} \text{d6}^2))+\text{c2} (8 \text{C1}^9 \text{d5}^2 \text{d6}-16 \text{C1}^9 \text{d4} \text{d6}^2+42 \text{C1}^7
\text{d6}^3+\text{c3}^2 (19 \text{C1}^4 \text{d5}^3-57 \text{C1}^4 \text{d4} \text{d5} \text{d6}+7 \text{C1}^3 \text{d5}^2 \text{d6}+57 \text{C1}^4
\text{d3} \text{d6}^2-14 \text{C1}^3 \text{d4} \text{d6}^2)+\text{c3} (42 \text{C1}^7 \text{d5}^3-126 \text{C1}^7 \text{d4} \text{d5}
\text{d6}+57 \text{C1}^6 \text{d5}^2 \text{d6}+126 \text{C1}^7 \text{d3} \text{d6}^2-114 \text{C1}^6 \text{d4} \text{d6}^2+19 \text{C1}^4 \text{d6}^3))+\text{c1}^7
(52 \text{C1}^4 \text{d4} \text{d6}^2-24 \text{C1}^3 \text{d5} \text{d6}^2+4 \text{C1}^2 \text{d6}^3+\text{c2}^2 (-148 \text{C1} \text{d4}
\text{d5} \text{d6}+24 \text{d5}^2 \text{d6}+444 \text{C1} \text{d3} \text{d6}^2-36 \text{d4} \text{d6}^2)+\text{c3} (4 \text{C1}^2 \text{d5}^3-12
\text{C1}^2 \text{d4} \text{d5} \text{d6}+12 \text{C1}^2 \text{d3} \text{d6}^2-196 \text{C1} \text{d4} \text{d6}^2+28 \text{d5} \text{d6}^2)+\text{c2}
(-24 \text{C1}^3 \text{d4} \text{d5} \text{d6}+4 \text{C1}^2 \text{d5}^2 \text{d6}+72 \text{C1}^3 \text{d3} \text{d6}^2+334 \text{C1}^2 \text{d4}
\text{d6}^2-148 \text{C1} \text{d5} \text{d6}^2+24 \text{d6}^3+\text{c3} (24 \text{d5}^3-44 \text{d4} \text{d5} \text{d6}-12 \text{d3} \text{d6}^2)))+\text{c1}^6
(-40 \text{C1}^5 \text{d4} \text{d6}^2+52 \text{C1}^4 \text{d5} \text{d6}^2-24 \text{C1}^3 \text{d6}^3+\text{c2}^3 (12 \text{d4} \text{d5}
\text{d6}-36 \text{d3} \text{d6}^2)+\text{c3}^2 (28 \text{d5}^3-84 \text{d4} \text{d5} \text{d6}+84 \text{d3} \text{d6}^2)+\text{c2}^2
(342 \text{C1}^2 \text{d4} \text{d5} \text{d6}-148 \text{C1} \text{d5}^2 \text{d6}-1026 \text{C1}^2 \text{d3} \text{d6}^2+258 \text{C1} \text{d4}
\text{d6}^2+12 \text{d5} \text{d6}^2)+\text{c3} (-24 \text{C1}^3 \text{d5}^3+72 \text{C1}^3 \text{d4} \text{d5} \text{d6}-72 \text{C1}^3
\text{d3} \text{d6}^2+559 \text{C1}^2 \text{d4} \text{d6}^2-196 \text{C1} \text{d5} \text{d6}^2+28 \text{d6}^3)+\text{c2} (52 \text{C1}^4
\text{d4} \text{d5} \text{d6}-24 \text{C1}^3 \text{d5}^2 \text{d6}-156 \text{C1}^4 \text{d3} \text{d6}^2-278 \text{C1}^3 \text{d4} \text{d6}^2+342
\text{C1}^2 \text{d5} \text{d6}^2-148 \text{C1} \text{d6}^3+\text{c3} (-148 \text{C1} \text{d5}^3+248 \text{C1} \text{d4} \text{d5} \text{d6}+28
\text{d5}^2 \text{d6}+144 \text{C1} \text{d3} \text{d6}^2-48 \text{d4} \text{d6}^2)))+\text{c1}^5 (-20 \text{C1}^6 \text{d4}
\text{d6}^2-40 \text{C1}^5 \text{d5} \text{d6}^2+52 \text{C1}^4 \text{d6}^3+\text{c2}^3 (-38 \text{C1} \text{d4} \text{d5} \text{d6}+12 \text{d5}^2
\text{d6}+114 \text{C1} \text{d3} \text{d6}^2-24 \text{d4} \text{d6}^2)+\text{c3}^2 (-196 \text{C1} \text{d5}^3+588 \text{C1} \text{d4}
\text{d5} \text{d6}-588 \text{C1} \text{d3} \text{d6}^2-7 \text{d4} \text{d6}^2)+\text{c3} (52 \text{C1}^4 \text{d5}^3-156 \text{C1}^4
\text{d4} \text{d5} \text{d6}+156 \text{C1}^4 \text{d3} \text{d6}^2-838 \text{C1}^3 \text{d4} \text{d6}^2+559 \text{C1}^2 \text{d5} \text{d6}^2-196
\text{C1} \text{d6}^3)+\text{c2}^2 (-326 \text{C1}^3 \text{d4} \text{d5} \text{d6}+342 \text{C1}^2 \text{d5}^2 \text{d6}+978 \text{C1}^3
\text{d3} \text{d6}^2-652 \text{C1}^2 \text{d4} \text{d6}^2-38 \text{C1} \text{d5} \text{d6}^2+12 \text{d6}^3+\text{c3} (12 \text{d5}^3-28 \text{d4}
\text{d5} \text{d6}+12 \text{d3} \text{d6}^2))+\text{c2} (-40 \text{C1}^5 \text{d4} \text{d5} \text{d6}+52 \text{C1}^4 \text{d5}^2
\text{d6}+120 \text{C1}^5 \text{d3} \text{d6}^2-100 \text{C1}^4 \text{d4} \text{d6}^2-326 \text{C1}^3 \text{d5} \text{d6}^2+342 \text{C1}^2 \text{d6}^3+\text{c3}
(342 \text{C1}^2 \text{d5}^3-467 \text{C1}^2 \text{d4} \text{d5} \text{d6}-196 \text{C1} \text{d5}^2 \text{d6}-651 \text{C1}^2 \text{d3} \text{d6}^2+349
\text{C1} \text{d4} \text{d6}^2+8 \text{d5} \text{d6}^2)))+\text{c1}^4 (56 \text{C1}^7 \text{d4} \text{d6}^2-20 \text{C1}^6
\text{d5} \text{d6}^2-40 \text{C1}^5 \text{d6}^3+\text{c2}^3 (32 \text{C1}^2 \text{d4} \text{d5} \text{d6}-38 \text{C1} \text{d5}^2 \text{d6}-96
\text{C1}^2 \text{d3} \text{d6}^2+76 \text{C1} \text{d4} \text{d6}^2)+\text{c3}^2 (559 \text{C1}^2 \text{d5}^3-1677 \text{C1}^2 \text{d4}
\text{d5} \text{d6}+1677 \text{C1}^2 \text{d3} \text{d6}^2+21 \text{C1} \text{d4} \text{d6}^2-7 \text{d5} \text{d6}^2)+\text{c3} (-40
\text{C1}^5 \text{d5}^3+120 \text{C1}^5 \text{d4} \text{d5} \text{d6}-120 \text{C1}^5 \text{d3} \text{d6}^2+700 \text{C1}^4 \text{d4} \text{d6}^2-838
\text{C1}^3 \text{d5} \text{d6}^2+559 \text{C1}^2 \text{d6}^3)+\text{c2}^2 (4 \text{C1}^4 \text{d4} \text{d5} \text{d6}-326 \text{C1}^3
\text{d5}^2 \text{d6}-12 \text{C1}^4 \text{d3} \text{d6}^2+664 \text{C1}^3 \text{d4} \text{d6}^2+32 \text{C1}^2 \text{d5} \text{d6}^2-38 \text{C1}
\text{d6}^3+\text{c3} (-38 \text{C1} \text{d5}^3+71 \text{C1} \text{d4} \text{d5} \text{d6}+8 \text{d5}^2 \text{d6}+15 \text{C1} \text{d3} \text{d6}^2-16
\text{d4} \text{d6}^2))+\text{c2} (-20 \text{C1}^6 \text{d4} \text{d5} \text{d6}-40 \text{C1}^5 \text{d5}^2 \text{d6}+60 \text{C1}^6
\text{d3} \text{d6}^2+320 \text{C1}^5 \text{d4} \text{d6}^2+4 \text{C1}^4 \text{d5} \text{d6}^2-326 \text{C1}^3 \text{d6}^3+\text{c3}^2 (8 \text{d5}^3-31
\text{d4} \text{d5} \text{d6}+45 \text{d3} \text{d6}^2)+\text{c3} (-326 \text{C1}^3 \text{d5}^3+140 \text{C1}^3 \text{d4} \text{d5} \text{d6}+559
\text{C1}^2 \text{d5}^2 \text{d6}+1536 \text{C1}^3 \text{d3} \text{d6}^2-1037 \text{C1}^2 \text{d4} \text{d6}^2-43 \text{C1} \text{d5} \text{d6}^2+8
\text{d6}^3)))+\text{c1}^3 (8 \text{C1}^9 \text{d3} \text{d6}^2-36 \text{C1}^8 \text{d4} \text{d6}^2+56 \text{C1}^7 \text{d5} \text{d6}^2-20 \text{C1}^6 \text{d6}^3+\text{c3}^3
(-7 \text{d5}^3+21 \text{d4} \text{d5} \text{d6}-21 \text{d3} \text{d6}^2)+\text{c2}^3 (12 \text{C1}^3 \text{d4} \text{d5} \text{d6}+32
\text{C1}^2 \text{d5}^2 \text{d6}-36 \text{C1}^3 \text{d3} \text{d6}^2-64 \text{C1}^2 \text{d4} \text{d6}^2)+\text{c3}^2 (-838 \text{C1}^3
\text{d5}^3+2514 \text{C1}^3 \text{d4} \text{d5} \text{d6}-2514 \text{C1}^3 \text{d3} \text{d6}^2-21 \text{C1}^2 \text{d4} \text{d6}^2+21 \text{C1}
\text{d5} \text{d6}^2-7 \text{d6}^3)+\text{c3} (-20 \text{C1}^6 \text{d5}^3+60 \text{C1}^6 \text{d4} \text{d5} \text{d6}-60 \text{C1}^6
\text{d3} \text{d6}^2-310 \text{C1}^5 \text{d4} \text{d6}^2+700 \text{C1}^4 \text{d5} \text{d6}^2-838 \text{C1}^3 \text{d6}^3)+\text{c2}^2
(240 \text{C1}^5 \text{d4} \text{d5} \text{d6}+4 \text{C1}^4 \text{d5}^2 \text{d6}-720 \text{C1}^5 \text{d3} \text{d6}^2-36 \text{C1}^4 \text{d4}
\text{d6}^2+12 \text{C1}^3 \text{d5} \text{d6}^2+32 \text{C1}^2 \text{d6}^3+\text{c3} (32 \text{C1}^2 \text{d5}^3-15 \text{C1}^2 \text{d4} \text{d5}
\text{d6}-43 \text{C1} \text{d5}^2 \text{d6}-147 \text{C1}^2 \text{d3} \text{d6}^2+86 \text{C1} \text{d4} \text{d6}^2))+\text{c2} (56
\text{C1}^7 \text{d4} \text{d5} \text{d6}-20 \text{C1}^6 \text{d5}^2 \text{d6}-168 \text{C1}^7 \text{d3} \text{d6}^2-138 \text{C1}^6 \text{d4} \text{d6}^2+240
\text{C1}^5 \text{d5} \text{d6}^2+4 \text{C1}^4 \text{d6}^3+\text{c3}^2 (-43 \text{C1} \text{d5}^3+150 \text{C1} \text{d4} \text{d5} \text{d6}-7
\text{d5}^2 \text{d6}-192 \text{C1} \text{d3} \text{d6}^2+14 \text{d4} \text{d6}^2)+\text{c3} (4 \text{C1}^4 \text{d5}^3+688 \text{C1}^4
\text{d4} \text{d5} \text{d6}-838 \text{C1}^3 \text{d5}^2 \text{d6}-2088 \text{C1}^4 \text{d3} \text{d6}^2+1611 \text{C1}^3 \text{d4} \text{d6}^2+81
\text{C1}^2 \text{d5} \text{d6}^2-43 \text{C1} \text{d6}^3)))+\text{c1}^2 (8 \text{C1}^9 \text{d4} \text{d6}^2-36 \text{C1}^8
\text{d5} \text{d6}^2+56 \text{C1}^7 \text{d6}^3+\text{c3}^3 (21 \text{C1} \text{d5}^3-63 \text{C1} \text{d4} \text{d5} \text{d6}+63 \text{C1}
\text{d3} \text{d6}^2)+\text{c2}^3 (-28 \text{C1}^4 \text{d4} \text{d5} \text{d6}+12 \text{C1}^3 \text{d5}^2 \text{d6}+84 \text{C1}^4
\text{d3} \text{d6}^2-24 \text{C1}^3 \text{d4} \text{d6}^2)+\text{c3}^2 (700 \text{C1}^4 \text{d5}^3-2100 \text{C1}^4 \text{d4} \text{d5}
\text{d6}+2100 \text{C1}^4 \text{d3} \text{d6}^2+7 \text{C1}^3 \text{d4} \text{d6}^2-21 \text{C1}^2 \text{d5} \text{d6}^2+21 \text{C1} \text{d6}^3)+\text{c3}
(56 \text{C1}^7 \text{d5}^3-168 \text{C1}^7 \text{d4} \text{d5} \text{d6}+168 \text{C1}^7 \text{d3} \text{d6}^2+57 \text{C1}^6 \text{d4} \text{d6}^2-310
\text{C1}^5 \text{d5} \text{d6}^2+700 \text{C1}^4 \text{d6}^3)+\text{c2}^2 (-178 \text{C1}^6 \text{d4} \text{d5} \text{d6}+240 \text{C1}^5
\text{d5}^2 \text{d6}+534 \text{C1}^6 \text{d3} \text{d6}^2-470 \text{C1}^5 \text{d4} \text{d6}^2-28 \text{C1}^4 \text{d5} \text{d6}^2+12 \text{C1}^3
\text{d6}^3+\text{c3} (12 \text{C1}^3 \text{d5}^3-101 \text{C1}^3 \text{d4} \text{d5} \text{d6}+81 \text{C1}^2 \text{d5}^2 \text{d6}+231 \text{C1}^3
\text{d3} \text{d6}^2-162 \text{C1}^2 \text{d4} \text{d6}^2))+\text{c2} (-36 \text{C1}^8 \text{d4} \text{d5} \text{d6}+56 \text{C1}^7
\text{d5}^2 \text{d6}+108 \text{C1}^8 \text{d3} \text{d6}^2-70 \text{C1}^7 \text{d4} \text{d6}^2-178 \text{C1}^6 \text{d5} \text{d6}^2+240 \text{C1}^5
\text{d6}^3+\text{c3}^2 (81 \text{C1}^2 \text{d5}^3-264 \text{C1}^2 \text{d4} \text{d5} \text{d6}+21 \text{C1} \text{d5}^2 \text{d6}+306 \text{C1}^2
\text{d3} \text{d6}^2-42 \text{C1} \text{d4} \text{d6}^2)+\text{c3} (240 \text{C1}^5 \text{d5}^3-1030 \text{C1}^5 \text{d4} \text{d5}
\text{d6}+700 \text{C1}^4 \text{d5}^2 \text{d6}+1650 \text{C1}^5 \text{d3} \text{d6}^2-1381 \text{C1}^4 \text{d4} \text{d6}^2-65 \text{C1}^3 \text{d5}
\text{d6}^2+81 \text{C1}^2 \text{d6}^3)))+\text{c1} (8 \text{C1}^9 \text{d5} \text{d6}^2-36 \text{C1}^8 \text{d6}^3+\text{c3}^3
(-21 \text{C1}^2 \text{d5}^3+63 \text{C1}^2 \text{d4} \text{d5} \text{d6}-63 \text{C1}^2 \text{d3} \text{d6}^2)+\text{c2}^3 (10 \text{C1}^5
\text{d4} \text{d5} \text{d6}-28 \text{C1}^4 \text{d5}^2 \text{d6}-30 \text{C1}^5 \text{d3} \text{d6}^2+56 \text{C1}^4 \text{d4} \text{d6}^2)+\text{c3}^2
(-310 \text{C1}^5 \text{d5}^3+930 \text{C1}^5 \text{d4} \text{d5} \text{d6}-930 \text{C1}^5 \text{d3} \text{d6}^2+7 \text{C1}^3 \text{d5} \text{d6}^2-21
\text{C1}^2 \text{d6}^3)+\text{c3} (-36 \text{C1}^8 \text{d5}^3+108 \text{C1}^8 \text{d4} \text{d5} \text{d6}-108 \text{C1}^8 \text{d3}
\text{d6}^2+57 \text{C1}^6 \text{d5} \text{d6}^2-310 \text{C1}^5 \text{d6}^3)+\text{c2}^2 (42 \text{C1}^7 \text{d4} \text{d5} \text{d6}-178
\text{C1}^6 \text{d5}^2 \text{d6}-126 \text{C1}^7 \text{d3} \text{d6}^2+356 \text{C1}^6 \text{d4} \text{d6}^2+10 \text{C1}^5 \text{d5} \text{d6}^2-28
\text{C1}^4 \text{d6}^3+\text{c3} (-28 \text{C1}^4 \text{d5}^3+103 \text{C1}^4 \text{d4} \text{d5} \text{d6}-65 \text{C1}^3 \text{d5}^2 \text{d6}-141
\text{C1}^4 \text{d3} \text{d6}^2+130 \text{C1}^3 \text{d4} \text{d6}^2))+\text{c2} (8 \text{C1}^9 \text{d4} \text{d5} \text{d6}-36
\text{C1}^8 \text{d5}^2 \text{d6}-24 \text{C1}^9 \text{d3} \text{d6}^2+72 \text{C1}^8 \text{d4} \text{d6}^2+42 \text{C1}^7 \text{d5} \text{d6}^2-178
\text{C1}^6 \text{d6}^3+\text{c3}^2 (-65 \text{C1}^3 \text{d5}^3+202 \text{C1}^3 \text{d4} \text{d5} \text{d6}-21 \text{C1}^2 \text{d5}^2 \text{d6}-216
\text{C1}^3 \text{d3} \text{d6}^2+42 \text{C1}^2 \text{d4} \text{d6}^2)+\text{c3} (-178 \text{C1}^6 \text{d5}^3+591 \text{C1}^6 \text{d4}
\text{d5} \text{d6}-310 \text{C1}^5 \text{d5}^2 \text{d6}-705 \text{C1}^6 \text{d3} \text{d6}^2+620 \text{C1}^5 \text{d4} \text{d6}^2+19 \text{C1}^4
\text{d5} \text{d6}^2-65 \text{C1}^3 \text{d6}^3)))\)
\end{doublespace}







\end{document}